\documentclass[11pt]{article}
\usepackage{amsmath,amssymb,amsthm,amsfonts,verbatim}
\usepackage{enumerate}
\usepackage{microtype}
\usepackage{url}
\usepackage[letterpaper]{geometry}
\usepackage[all,arc]{xy}
\usepackage{mathrsfs}
\usepackage{listings, color} 
\usepackage[pdftex]{graphicx}
\usepackage[section]{placeins}
\usepackage{soul}
\usepackage[utf8]{inputenc}
\usepackage[english]{babel}
\usepackage{color}
\usepackage{xcolor}
\usepackage{mathtools}
\usepackage{hyperref}
\usepackage{listings}
\usepackage{color}
\usepackage{xcolor}
\usepackage[makeroom]{cancel}
\usepackage{hyperref}
\usepackage{chngcntr}

\newtheoremstyle{quest}{\topsep}{\topsep}{}{}{\bfseries}{}{ }{\thmname{#1}\thmnote{ #3}.}
\theoremstyle{quest}

\theoremstyle{plain}
\theoremstyle{definition}
\newtheorem{theorem}{Theorem}[section]
\newtheorem{corollary}[theorem]{Corollary}
\newtheorem{proposition}[theorem]{Proposition}

\newtheorem{lemma}[theorem]{Lemma}

\newtheorem{requirement}[theorem]{Requirement}
\newtheorem{definition}[theorem]{Definition}

\newtheorem{remark}[theorem]{Remark}

\definecolor{dkgreen}{rgb}{0,0.6,0}
\definecolor{gray}{rgb}{0.5,0.5,0.5}
\definecolor{mauve}{rgb}{0.58,0,0.82}
 
 \allowdisplaybreaks
 \numberwithin{equation}{section} 
 \counterwithin{part}{section}
 
\title{\vspace{-50pt} Global Steady Prandtl Expansion Over a Moving Boundary}
\author{ \Large Sameer Iyer \footnote{\url{sameer_iyer@brown.edu}. Division of Applied Mathematics, Brown University, 182 George Street, Providence, RI 02912, USA. Partially supported by NSF grant 1209437. }}
\date{September 15, 2016}

\newcommand{\ud}{\,\mathrm{d}}
\newcommand{\p}{\ensuremath{\partial}}
\newcommand{\n}{\ensuremath{\nonumber}}
\newcommand{\eps}{\ensuremath{\epsilon}}

\begin{document}
\maketitle
\vspace{-30pt}
\begin{center}
\end{center}

\begin{abstract}
In this three-part monograph, we prove that steady, incompressible Navier-Stokes flows posed over the moving boundary, $y = 0$, can be decomposed into Euler and Prandtl flows in the inviscid limit globally in $[1, \infty) \times [0,\infty)$, assuming a sufficiently small velocity mismatch. Sharp decay rates and self-similar asymptotics are extracted for both Prandtl and Eulerian layers. We then develop a functional framework to capture precise decay rates of the remainders, and prove the corresponding embedding theorems by establishing weighted estimates for their higher order tangential derivatives. These tools are then used in conjunction with a third order energy analysis, which in particular enables us to control the nonlinearity $vu_y$ globally. 
\end{abstract}

\tableofcontents

\section{Introduction}

We consider the steady, incompressible Navier-Stokes equations in two dimensions:
\begin{align} \label{NS.1}
&U^{NS}U^{NS}_X + V^{NS}U^{NS}_Y + P^{NS}_X = \epsilon \Delta U^{NS}, \\ \label{NS.2}
&U^{NS}V^{NS}_X + V^{NS}V^{NS}_Y + P^{NS}_Y = \epsilon \Delta V^{NS}, \\ \label{NS.3}
&U^{NS}_X + V^{NS}_Y = 0,
\end{align}
in the domain, 
\begin{align} \label{N}
\Omega = [1, \infty) \times \mathbb{R}_+.
\end{align}
The boundary $Y=0$ is moving with velocity $u_b > 0$. The no-slip boundary conditions are placed on this portion of the boundary: 
\begin{equation} \label{NS.noslip}
U^{NS}(X,0) = u_b = 1-\delta, \hspace{3 mm} V^{NS}(X,0) = 0. 
\end{equation}

The boundary conditions at $X = 1$ will be prescribed explicitly in the text. We take $X = 1$ for convenience (this enables us to replace weights of $(1+x)^k$ with $x^k$). Throughout this paper, we assume that prescribed Euler flow is the shear flow: 
\begin{equation} \label{E.1}
(U^E, V^E) = (1,0). 
\end{equation}

We are interested in the limit as $\epsilon \rightarrow 0$. Formally, one expects that the solutions to Navier-Stokes equations in (\ref{NS.1}) - (\ref{NS.3}) converges to the Euler shear flow in (\ref{E.1}). This does not happen, however, due to the mismatch at the boundary $Y = 0$, between the no-slip condition enforced for Navier-Stokes, (\ref{NS.noslip}), and $U^E(X, Y=0) = 1$. 

To account for the mismatch at the boundary, Prandtl in 1904 proposed a thin fluid boundary layer which connects the velocity of $u_b$ to the Euler velocity of $1$. The Prandtl hypothesis is that the Navier-Stokes solutions can be decomposed, up to leading order in $\epsilon$, as the sum of the prescribed Euler flow and a boundary layer, the latter of which corrects the disparity at the boundary between Euler and Navier-Stokes: 
\begin{equation} \label{bltheory}
U^{NS} = 1 + u^0_p + h.o.t(\epsilon), \hspace{3 mm} V^{NS} = 0 + \sqrt{\epsilon}v^0_p + \sqrt{\epsilon} v^1_e + h.o.t(\epsilon). \footnote{Here, ``h.o.t" is an acronym for ``higher order terms."}
\end{equation}

\textit{The contribution of this paper is to validate the boundary layer theory, equations (\ref{bltheory}), in the domain $\Omega$, which in particular implies that the tangential variable can be taken in $[1,\infty)$ if the mismatch is sufficiently small}: 
\begin{equation}
U^E - U^{NS}|_{Y=0} = 1 - u_b = 1 - (1-\delta) = \delta << 1. 
\end{equation}

\subsection*{Boundary Layer Expansion}

We will work with scaled, boundary layer variables: 
\begin{equation} \label{scaled.variables}
x = X, \hspace{3 mm} y = \frac{Y}{\sqrt{\epsilon}}.
\end{equation}

The scaled Navier-Stokes unknowns are then given by: 
\begin{equation}
U^\epsilon(x,y) = U^{NS}(X,Y), \hspace{3 mm} V^\epsilon(x,y) = \frac{V^{NS}(X,Y)}{\sqrt{\epsilon}}, \hspace{3 mm} P^\epsilon(x,y) = P^{NS}(X,Y). 
\end{equation}

These unknowns satisfy the following system: 
\begin{align} \label{scaled.NS.1}
&U^{\epsilon} U^{\epsilon}_x + V^{\epsilon} U^{\epsilon}_y + P^{\epsilon}_x = U^{\epsilon}_{yy} + \epsilon U^{\epsilon}_{xx}, \\ \label{scaled.NS.2}
&U^{\epsilon} V^{\epsilon}_x + V^{\epsilon} V^{\epsilon}_y + \frac{P^{\epsilon}_y}{\epsilon} = V^{\epsilon}_{yy} + \epsilon V^{\epsilon}_{xx}, \\ \label{scaled.NS.3}
&U^{\epsilon}_x + V^{\epsilon}_y = 0. 
\end{align}

Note that the steady Prandtl system is obtained by considering the leading order in $\eps$ of the above system (\ref{scaled.NS.1}) - (\ref{scaled.NS.3}). We start with the following asymptotic expansion: 
\begin{align} \label{expansion.1}
&U^\epsilon(x,y) = 1 + u^0_p + \sum_{i=1}^{n} \epsilon^{\frac{i}{2}} u^i_e + \epsilon^{\frac{i}{2}} u^i_p  + \epsilon^{\frac{n}{2}+\gamma} u(x,y), \\ \label{expansion.2}
&V^\epsilon(x,y) = \sum_{i=0}^{n-1} \epsilon^{\frac{i}{2}} v^i_p + \epsilon^{\frac{i}{2}}v^{i+1}_e + \epsilon^{\frac{n}{2}}v^n_p +  \epsilon^{\frac{n}{2}+\gamma} v(x,y), \\ \label{expansion.3}
& P^\epsilon(x,y) = \sum_{i=1}^n  \epsilon^{\frac{i}{2}}P^i_e + \epsilon^{\frac{i}{2}} P^i_p + \epsilon^i P^{i,a}_e + \epsilon^{\frac{i+1}{2}} P^{i,a}_p + \epsilon^{\frac{n}{2}+\gamma}P(x,y). 
\end{align}

Here, $\gamma \in [0,\frac{1}{4})$. We will use the word ``profiles" to refer to the terms which appear in the expansions (\ref{expansion.1}) - (\ref{expansion.2}), excluding the remainders, $[u,v,P]$. All of the profiles with subscript-$e$ are functions of Eulerian variables, $(x,Y)$, whereas all terms with subscript-$p$ are functions of boundary layer variables, $(x,y)$. Here, $[u^i_p, v^i_p]$ are boundary layers to be constructed. The number of intermediate layers, $n$, is dependent on universal constants. The pressures $P^i_e, P^i_p$ are the pressures associated with the $i'th$ Euler and Prandtl layers, respectively. We will show that $P^i_p = 0$, that is the leading-order pressure in the boundary layers is zero. The pressures $P_P^{i,a}, P^{1,a}_e$ are \textit{auxiliary} pressures, which are higher-order, whose purpose is to capitalize on the gradient structure of our problem (see \ref{grad.pres.2}). After these layers are constructed, the Navier-Stokes remainders $[u,v,P]$ are then constructed. Let us now designate names for the partial expansions:
\begin{align} \label{barus}
u_s^{(i)} := 1 + \sum_{j=0}^{i-1} \epsilon^{\frac{j}{2}} u^j_p + \sum_{j=1}^i \epsilon^{\frac{j}{2}}u^j_e, \hspace{3 mm} \bar{u}^{(i)}_s := u_s^{(i)} + \epsilon^{\frac{i}{2}} u^i_p = 1 + \sum_{j=0}^{i} \epsilon^{\frac{j}{2}} u^j_p + \sum_{j=1}^i \epsilon^{\frac{j}{2}}u^j_e \\ \label{barvs}
v_s^{(i)} := \sum_{j=0}^{i-1} \epsilon^{\frac{j}{2}} v^j_p + \sum_{j=1}^i \epsilon^{\frac{j}{2}-\frac{1}{2}} v^j_e, \hspace{3 mm} \bar{v}^{(i)}_s := v_s^{(i)} + \epsilon^{\frac{i}{2}} v^i_p =  \sum_{j=0}^{i} \epsilon^{\frac{j}{2}} v^j_p + \sum_{j=1}^i \epsilon^{\frac{j}{2}-\frac{1}{2}} v^j_e,
\end{align}

For the Pressure expansion: 
\begin{align}
P_{s}^{(i)} &:= \sum_{j=1}^{i-1} \epsilon^{\frac{j}{2}} P_p^j + \sum_{j=1}^{i-1} \epsilon^{\frac{j+1}{2}}P_p^{j,a} + \sum_{j=1}^i  \epsilon^{\frac{j}{2}} P^j_e + \sum_{j=1}^i \epsilon^j P^{j, a}_e. \\
\bar{P}_s^{(i)} &:= \sum_{j=1}^i \epsilon^{\frac{j}{2}} P_p^j + \sum_{j=1}^i \epsilon^{\frac{j}{2}} P^j_e + \sum_{j=1}^i \epsilon^j P^{j, a}_e + \sum_{j=1}^i \epsilon^{\frac{j+1}{2}}P_p^{j,a}.
\end{align}

We insert the expansions (\ref{expansion.1}) - (\ref{expansion.3}) into (\ref{scaled.NS.1}) - (\ref{scaled.NS.3}) and collect a heirarchy of equations in powers of $\epsilon$. Doing so yields the linearized Prandtl-equations: 
\begin{align} \label{sys.pr.i.intro}
&(1 + u^0_p) u_{px}^{i} + u_{sx}^{(i)}u_p^{i} + v_s^{(i)} u^i_{py} + u_{py}^{0}\Big(v^i_p - v^i_p(x,0)\Big) + P^i_{px} = u^i_{pyy} + f^{(i)}, \\
&u^i_p(x,0) = -u^i_e(x,0), \hspace{3 mm} \lim_{y \rightarrow \infty} u^i_p(x,y) = 0, \hspace{3 mm} u^i_p(1,y) = U_i(y). 
\end{align}

and the Euler equations: 
\begin{align} \label{i.Euler.intro}
u^{i}_{ex} + P^i_{ex} = 0, \hspace{3 mm} v^i_{ex} + P^i_{eY} = 0, \hspace{3 mm} u^i_{ex} + v^i_{eY} = 0. 
\end{align}

The forcing term $f^{(i)}$ will be defined precisely in (\ref{f.i}). These equations are derived rigorously in the analysis leading up to equations (\ref{euler.1.V.0}), (\ref{pr.eqn.1}),  (\ref{i.Euler}), and (\ref{sys.pr.i}).

\subsection*{Boundary Data:} \label{subsection.BD}

The no-slip boundary condition at the boundary $\{y = 0\}$ is the most important, and must be enforced at each order in $\epsilon$, which gives: 
\begin{align} \label{intro.BC.1}
u^0_p(x,0) = -\delta, \hspace{3 mm} u^i_e(x,0) + u^i_p(x,0) = 0 \text{ for } i \ge 1, \hspace{3 mm} u(x,0) = 0, \\
v^{i-1}_p(x,0) + v^{i}_e(x,0) = 0 \text{ for } i \ge 1, \hspace{3 mm} v^n_p(x,0) = 0, \hspace{3 mm} v(x,0) = 0. 
\end{align}

The in-flow ($x = 1$) boundary conditions for the leading order boundary layer, $u^0_p$, is: 
\begin{align}  \label{intro.BC0.1}
u^0_p(1,y) = U_0(y), \hspace{3 mm} U_0(0) = -\delta, \hspace{3 mm} \lim_{y \rightarrow \infty} U_0(y) = 0.
\end{align}

We assume the rapid decay of the profile: 
\begin{align}  \label{intro.BC0.2}
||\langle y \rangle^m \p_y^j U_0(y)||_{L^\infty} \le C(m,j), \text{ for any } m, j \ge 0.
\end{align}

We will in addition assume the following smallness condition:
\begin{align} \label{intro.BC0.3}
||\langle y \rangle^m \p_y^j U_0(y)||_{L^\infty} \le \mathcal{O}(\delta; j, m) \text{ for any } m \ge 0, j = 0,1,2.
\end{align}

The in-flow ($x = 1$) boundary conditions for the boundary-layer profiles are: 
\begin{align} \label{intro.BC.2}
u^i_p(1,y) = U_i(y), \hspace{3 mm} u(1,y) = 0, \hspace{3 mm} v(1,y) = 0, \text{ for } 1 \le i \le n - 1.
\end{align}

Here, $U_i(y)$, for $i \le 1 \le n-1$, will be prescribed to be rapidly decaying in $y$, so: 
\begin{align} \label{dd}
||\langle y \rangle^m u^i_p(1,y)||_{L^\infty} = ||\langle y \rangle^m U_i(y)||_{L^\infty} \le C(i,m) \text{ for } 1 \le i \le n-1.
\end{align}

For the final Prandtl layer, $u^n_p$, which occurs at order $\eps^{\frac{n}{2}}$, the in-flow data is determined through the analysis and is not explicitly prescribed. This is due to retaining that $[u^n_p, v^n_p]$ are divergence free, while cutting off $v^n_p$ for large values of $y$. The reader is invited to turn to equation (\ref{cutoff.defn}) and corresponding discussion for details regarding this matter. We will enforce $||\langle y \rangle^m U_n(y)||_{L^\infty} \le C(m)$. However $U_n$ is an auxiliary in-flow, which is used to construct $[u^n_p, v^n_p]$. That is: $u^n_p(1,y) \neq U_n(y)$, and the in-flow velocity, $u^n_p(1,y)$, is given by an implicit, bounded profile which decays as $y \rightarrow \infty$. We will need several compatibility conditions on the in-flow data for the Prandtl layers. The first of these is:
\begin{align} \label{compatibility.1}
U_0(0) &= -\delta, \hspace{3 mm} \partial_{yy} U_0(y) = 0, \hspace{3 mm} U_i(0) = -u^i_e(1,0).
\end{align}

However, we shall also need higher-order compatibility conditions on the $U_i(y)$ at $y = 0$, (see for instance Remarks \ref{hoc}, \ref{hoc2}) which we refrain from depicting explicitly here. Finally, the boundary conditions of the boundary-layer profiles as $y \rightarrow \infty$ are:
\begin{align} \label{intro.BC.4}
\lim_{y \rightarrow \infty} [u^i_p(x), v^i_p(x)] = \lim_{y \rightarrow \infty} [u(x),v(x)] = 0 \text{ for all } x \ge 1, \text{ and all }  0 \le i \le n. 
\end{align} 

These boundary conditions are known as the ``matching condition", and physically correspond to the Navier-Stokes flow matching the outer Euler flow away from the boundary, $y = 0$. According to our construction, we will have rapid matching up to order $\eps^{\frac{n-1}{2}}$: 
\begin{equation}
\lim_{y \rightarrow \infty} [\langle y \rangle^N u^i_p(x), \langle y \rangle^N v^i_p(x)] = 0 \text{ for all } x \ge 1, \text{ for } 1 \le i \le n-1.
\end{equation}

At the highest-order in $\eps$, we enforce the matching condition: 
\begin{equation} \label{intro.BC.match}
\lim_{y \rightarrow \infty} \eps^{\frac{n}{2}}[u^n_p(x,y), v^n_p(x,y)] = \lim_{y \rightarrow \infty} \eps^{\frac{n}{2}+\gamma}[u(x,y), v(x,y)] = 0 \text{ for all } x \ge 1.
\end{equation}

Let us now turn to the Euler flows. At leading order, we have the prescription: $[u^0_e, v^0_e] = [1,0]$. The higher-order Euler flows will be described starting in Section \ref{section.euler} and Subsection \ref{sub.EL}. The higher-order Euler flows are obtained as suitable Poisson extensions of the $Y=0$ boundary data, $v^{i-1}_p(x,0)$ (see (\ref{intro.BC.1})), which depend on the constructed Prandtl layers. For these higher-order Euler flows, we do not prescribe the in-flow data, $u^i_e(1,Y)$. Rather, the in-flow conditions are obtained through the analysis, so we state: 
\begin{equation} \label{IF.EUL}
\text{Eulerian In-Flow  =} 1 + \sum_{i=1}^n \epsilon^{\frac{i}{2}} u^i_e(1,\cdot). 
\end{equation}

\subsection*{Main Result:}

In order to state our main result, we need to introduce the norm $Z$ in which we control the remainder solutions, $[u,v]$:
\begin{definition} The norm $Z$ is defined through: 
\begin{align} \nonumber
||u,v||_Z := &||u,v||_{X_1 \cap X_2 \cap X_3} + \epsilon^{N_2} ||u,v||_{Y_2} + \epsilon^{N_3} ||u,v||_{Y_3} + \epsilon^{N_4} ||ux^{\frac{1}{4}} , \sqrt{\eps} v x^{\frac{1}{2}}||_{L^\infty} \\ \nonumber 
&+ \epsilon^{N_5} \sup_{x \ge 20} ||\sqrt{\eps} v_x x^{\frac{3}{2}} , u_x x^{\frac{5}{4}} ||_{L^\infty} + \eps^{N_6}  \sup_{x \ge 20} ||u_y x^{\frac{1}{2}}||_{L^2_y}\\ \label{norm.Z.INTRO}
& + \epsilon^{N_7} \Big[\int_{20}^\infty x^4 ||\sqrt{\eps} v_{xx}||_{L^\infty_y}^2 dx \Big]^{\frac{1}{2}} .
\end{align}
\end{definition}
Here, $N_i$, are large numbers which will be specified in (\ref{sel.N2}) - (\ref{sel.N4}). They depend only on universal constants. The parameter $n$ from (\ref{expansion.1}) - (\ref{expansion.2}) will be taken much larger than any of the $N_i$. The norms $||\cdot||_{X_i}$ are energy norms defined in (\ref{norm.x0}) - (\ref{norm.x2}). The norms $||\cdot||_{Y_i}$ are elliptic norms defined in (\ref{norm.Y2}) - (\ref{norm.Y3}). For the purposes of stating the main result, we can refrain from being too specific with regards to the definitions of these norms. The essential point that we will record concerns the uniform component: 
\begin{equation} \label{intro.comp.Z}
\epsilon^{N_4} ||ux^{\frac{1}{4}} , \sqrt{\eps} v x^{\frac{1}{2}}||_{L^\infty} \le ||u,v||_{Z}, 
\end{equation}

\vspace{2 mm}

The main result of this paper is:

\begin{theorem} \label{thm.m.1} Suppose the the outer Euler flow is prescribed with $u^0_e = 1$. Suppose the boundary and in-flow data are specified satisfying the conditions outlined in (\ref{intro.BC.1}) - (\ref{intro.BC.match}). Then there exists an $n$ depending on only universal constants such that the asymptotic expansions in (\ref{expansion.1}) - (\ref{expansion.3}) are valid globally on the domain $\Omega$, for $0 \le \gamma < \frac{1}{4}$, so long as the mismatch between the Eulerian boundary trace and the motion of the boundary, $\delta$, and the viscosity, $\eps$, are taken sufficiently small relative to universal constants, and $\eps << \delta$. The remainders, $[u,v]$, in the expansions (\ref{expansion.1}) - (\ref{expansion.2}) are uniquely determined in the space $Z$:
\begin{align}
||u,v||_{Z} \lesssim \eps^{\frac{1}{4}-\gamma - \kappa},
\end{align} 

where $\kappa$ is any fixed constant such that $\gamma + \kappa < \frac{1}{4}$.
\end{theorem}

Because $\frac{n}{2}$ is large relative to $N_4$ in (\ref{intro.comp.Z}), we immediately find: 

\begin{corollary}[Inviscid $L^\infty$ Convergence] Under the hypothesis of Theorem \ref{thm.m.1}, there exists a unique Navier-Stokes solution $[U^{NS}, V^{NS}, P^{NS}]$ on $\Omega$ such that:
\begin{align}
&\sup_{(X,Y) \in \Omega} \Big|U^{NS}(X,Y) - 1 - u^0_p(X,y) \Big| X^{\frac{1}{4}} \lesssim \epsilon^{\frac{1}{2}}, \\
& \sup_{(X,Y) \in \Omega} \Big|V^{NS}(X,Y) -  \sqrt{\epsilon}v^0_p(X,y) - \sqrt{\epsilon} v^1_e(X,Y)\Big| X^{\frac{1}{2}} \lesssim  \epsilon.
\end{align}
\end{corollary}

\subsection*{Existing Literature:}

Let us first discuss the issue of establishing wellposedness of the Prandtl equation, which becomes an issue in the unsteady setting (in contrast to the steady setting of the present paper). This program was initiated in the classic works \cite{Oleinik}, \cite{Oleinik1}, in which, under the monotonicity assumption $U^\eps_y(t = 0) > 0$, globally regular solutions are constructed on the $[0,L] \times \mathbb{R}_+$, where $L$ is sufficiently small, and local solutions are constructed for arbitrary, but finite $L$. This was extended in \cite{Xin}, in which global weak solutions were constructed for arbitrary $L$, under both monotonicity and favorable outer-Euler pressure ($\p_x P^E(t,x) \le 0$ for $t \ge 0$) assumptions. 

From a physical standpoint, the monotonicity and favorable pressure assumptions mentioned above are stabilizing and in particular prevent boundary layer separation. This phenomena was known to Prandtl, see Figure 2 in \cite{Prandtl}. More recently, it was announced in \cite{MD} that a proof of boundary layer separation in the steady setting has been obtained. 

The main tool used both in \cite{Oleinik1} and \cite{Xin} is the Crocco transform. Still under monotonicity hypothesis, local wellposedness was obtained in \cite{AL} and \cite{MW}, neither works using the Crocco transform. \cite{AL} use energy methods coupled with a Nash-Moser iteration, and \cite{MW} use energy methods applied to a good unknown which enjoys crucial cancellation properties. Generalizing to multiple monotonicity regions, \cite{KMVW} have shown the Prandtl equation is locally well-posed, if an analyticity assumption is made on the complement of the monotonicity regions. 

Indeed, when the assumption of monotonicity is removed, the wellposedness results are largely in the analytic or  Gevrey setting. The reader should consult \cite{Caflisch1} - \cite{Caflisch2},  \cite{Kuka}, \cite{Lom}, \cite{Vicol}, and \cite{GVM} for some results in this direction. In the Sobolev setting without monotonicity, the equations are linearly and nonlinearly ill-posed (see \cite{GVD} and \cite{GVN}). A finite-time blowup result was obtained in \cite{EE} when the outer Euler flow is taken to be zero, in \cite{KVW} for a particular, periodic outer Euler flow, and in \cite{Hunter} for both the inviscid and viscous Prandtl equations. The above discussion is not comprehensive: we refer the reader to the review articles, \cite{E}, \cite{Temam} and references therein for a more thorough review of the wellposedness theory. 

The question with which we are concerned is the validity of the asymptotic expansion (\ref{expansion.1}) - (\ref{expansion.3}) in the inviscid limit. Let us first discuss unsteady flows. Local-in-time convergence is established in \cite{Caflisch1}, \cite{Caflisch2} in the analyticity framework, in \cite{DMM} in the Gevrey setting, and in \cite{Mae} when the initial vorticity distribution is supported away from the boundary. The reader should see also \cite{Asano}, \cite{Taylor} for related results. Despite the boundary layer classically having thickness $\sqrt{\eps}$, an interesting criteria was given in \cite{Kato} which points to phenomena occurring in a sub-layer of size $\epsilon$. There are also several linear and nonlinear instability results (for instance, \cite{Grenier},  \cite{GGN1}, \cite{GGN2}, \cite{GGN3}, \cite{GN2}) which show the invalidity of Prandtl's expansion generically in Sobolev spaces in the unsteady setting. 

For steady flows, there are very few validity results. \cite{GN} is the first result in this direction, establishing validity of the boundary layer expansion for steady state flows in a rectangular domain over a moving boundary. Geometric effects of the boundary were subsequently considered in \cite{Iyer}. The crucial idea in \cite{GN} was the use of a positivity estimate, which is coupled with energy estimates and elliptic estimates. Both of these results are \textit{local} in the tangential variable. In the present work, we prove validity of the boundary layer expansion \textit{globally} in the tangential variable, $x$, in the setting of small data. 

One preliminary piece of our analysis is to obtain the asymptotics of the Prandtl layer, $u^0_p$. This has nontrivial dynamics due to the mismatched boundary conditions, $u^0_p(y = 0) = -\delta$, while $\lim_{y \rightarrow \infty} u^0_p(y) = 0$. These asymptotics were first studied in \cite[pg. 493, Inequality 5]{Serrin} using maximum principle techniques and are valid for large data. The result in \cite{Serrin} gives that the difference  between $u^0_p$ and a Gaussian``front" solution to the heat equation (call it $w$) is $o(1)$ in $x$, uniformly in $y$. Under the hypothesis of small data, we sharpen these asymptotics in the following sense: first, $w$ is shown to belong to a higher-order Sobolev space $H^k(m)$, where this weight is in the self-similar variable $z = \frac{y}{\sqrt{x}}$. Second, we obtain rates of decay of $w$ in $x$ in various norms. 

We now detail the main difficulties and ideas behind our analysis.

\subsection*{Sharp Decay of Profiles (Chapter I):}

The key issue that we must capture in our analysis is the decay as $x \rightarrow \infty$ of various quantities. More specifically, a central difficulty is to control contributions from the nonlinearity $V^\epsilon U^\epsilon_y$. Let us now introduce the equations for the remainders, $[u,v,P]$:
\begin{align} \label{intro.NSR.sys.1}
-\Delta_\epsilon u + S_u + P_x = f, \hspace{3 mm} -\Delta_\epsilon v + S_v + \frac{P_y}{\epsilon} = g, \hspace{3 mm} u_x + v_y = 0.
\end{align}

Here, the terms $S_u, S_v$ contain the linearizations of $[u,v]$ around the previously constructed profiles, $[\bar{u}_s^{(n)}, \bar{v}_s^{(n)}]$. These terms, together with $f, g$ are specifically defined in (\ref{defn.SU.SV}) - (\ref{defn.Su}). To organize this discussion, let us record the heuristic: 
\begin{equation} \label{heur.SF.1}
\text{ Difficult Contributions from $V^\eps U^\eps_y$  = } \bar{u}^{(n)}_{sy} v + \bar{v}^{(n)}_su_y + \eps^{\frac{n}{2}+\gamma} vu_y.
\end{equation}

First, let us discuss $\bar{u}_{sy}^{(n)}v$ from (\ref{heur.SF.1}), which will motivate the crucial decay rates appearing in (\ref{REQ}). Applying the scaled multiplier of $(u,\epsilon v)$ to the system (\ref{intro.NSR.sys.1}) requires controlling the large convective term, $\int \int u^0_{py}uv$. We do not have the ability to create a derivative through the Poincare inequality, and so we trade factors of $x$ and $y$ in the following manner:
\begin{equation} \label{intro.cvc}
\Big| \int \int u^0_{py} uv \Big| \le ||u^0_{py} y^2 x^{-\frac{1}{2}}||_{L^\infty}^2 ||\frac{u}{y}||_{L^2} ||\frac{v x^{\frac{1}{2}}}{y}||_{L^2} \le \mathcal{O}(\delta) ||u_y||_{L^2} ||v_y x^{\frac{1}{2}}||_{L^2}.
\end{equation}

The first crucial observation we make is the identification of a self-similar front$, \phi_\ast(\frac{y}{\sqrt{x}})$, which bridges the boundary conditions: $u^0_p(x,0) = -\delta$, $u^0_p(x,\infty) = 0$. Then, temporarily identifying $u^0_p \approx \phi_\ast(\frac{y}{\sqrt{x}})$, (\ref{intro.cvc}) will be satisfied:
\begin{requirement}[Self-Similarity of Prandtl profiles]
\begin{equation} \label{intro.s.s.absorb}
 ||y^2 x^{-\frac{1}{2}} u^0_{py}||_{L^\infty} = ||y^2 x^{-1} \phi_\ast'\Big(\frac{y}{\sqrt{x}} \Big) ||_{L^\infty} = ||z^2 \phi_\ast'(z)||_{L^\infty} \le \mathcal{O}(\delta). 
 \end{equation}
\end{requirement}

Summarizing the energy estimate that we obtain: 
\begin{equation} \label{intro.EE}
||u_y||_{L^2}^2 \le \mathcal{O}(\delta) ||\{\sqrt{\eps}v_x, v_y\} x^{\frac{1}{2}}||_{L^2}^2 + \text{ Forcing Terms.}
\end{equation}

Above, the key point is the loss of weight, $x^{\frac{1}{2}}$. The next ingredient is recovering this weight in the Positivity estimate, (see Proposition \ref{prop.pos}), which is summarized:  
\begin{equation} \label{intro.PE}
||\{\sqrt{\eps}v_x, v_y\} x^{\frac{1}{2}}||_{L^2}^2 \lesssim ||u_y||_{L^2}^2 + \text{ Forcing Terms.}
\end{equation}

In order to prove (\ref{intro.PE}), we must apply the weighted multilplier $v_y x$. Referring to the final two terms in (\ref{heur.SF.1}), this gives (temporarily ignoring factors of $\epsilon$): 
\begin{align} \label{intro.vab}
\Big| \int \int \bar{v}_s^{(n)} u_y \cdot v_y x \Big| + \Big| \int \int v u_y \cdot v_y x \Big| \le ||\{\bar{v}_s^{(n)}, v\} x^{\frac{1}{2}}||_{L^\infty} ||u_y||_{L^2} ||v_y x^{\frac{1}{2}}||_{L^2}. 
\end{align}

The latter two $L^2$ quantities are controlled by the left-hand sides of (\ref{intro.PE}) - (\ref{intro.vab}). From this, we obtain the requirements: 
\begin{requirement}[Uniform Decay]
\begin{equation} \label{REQ}
\bar{v}_s^{(n)} + v \sim x^{-\frac{1}{2}}, \text{ as } x \rightarrow \infty. 
\end{equation}
\end{requirement}

We emphasize that this requirement is inflexible, and we cannot sacrifice even a logarithmic factor of $x$ here. The main contribution of Chapter I is the construction of profiles $u^i_p, v^i_p, u^i_e, v^i_e$ which satisfy the requirement of $v^{(n)}_s$ in (\ref{REQ}). The most difficult task is to obtain the estimate (\ref{REQ}) for the Eulerian profiles, $v^1_e$, as these are solutions to elliptic boundary value problems in which the boundary condition exhibits exactly the required decay rate, $|v^1_e(x,0)| \le x^{-\frac{1}{2}}$. We refer the reader to the crucial Proposition \ref{Lemma.Sing.Unif} in which we introduce novel techniques centered around the explicit integral representation of $v^1_e$, enabling us to prove the required decay, $|v^1_e| \lesssim x^{-\frac{1}{2}}$. 

Note that we construct the expansions, (\ref{expansion.1}) - (\ref{expansion.3}) for any $n \in \mathbb{N}$, which is required as discussed in the paragraph following (\ref{below.1}). Our ability to do this relies on the Cauchy-Riemann structure of the Eulerian profiles, which we use in Lemma \ref{grad.pres.2}.

\subsection*{The Norm $Z$ (Chapter II):}

Let us now turn to the $v$ term in Requirement (\ref{REQ}): the main contribution of Chapter II is to prove the required decay estimate $|v| \lesssim x^{-\frac{1}{2}}$ by using the crucial norm $Z$ (see Lemmas \ref{L.Evol}, \ref{Lemma.UIMP}). The challenge is to extract this precise uniform decay information from the energy norms that are controlled. Bearing in mind the inflexibility of (\ref{REQ}), obtaining the decay for $v$ using the norms $X_i$ is an extremely delicate matter, in which key quantities must overcome the critical Hardy inequality. To see this, we first use the $H^1_y(\mathbb{R}_+) \hookrightarrow L^\infty_y(\mathbb{R}_+)$ Sobolev embedding (ignoring factors of $\eps$): 
\begin{align}
\sup_{x \ge 1} ||vx^{\frac{1}{2}}||_{L^\infty_y} \le \sup_{x \ge 1} ||v||_{L^2_y}^{\frac{1}{2}} \cdot \sup_{x \ge 1}  ||v_y x||_{L^2_y}^{\frac{1}{2}}.
\end{align}

For the first quantity on the right-hand side above, we write:
\begin{align}
\p_x \int v^2 \ud y = \int 2vv_x \ud y. 
\end{align}

Recall now the quantities, $||u_y,  \sqrt{\eps} v_x x^{\frac{1}{2}}, v_y x^{\frac{1}{2}}||_{L^2_{xy}}^2$, which are controlled on the left-hand sides (\ref{intro.EE}), (\ref{intro.PE}) and constitute the energy norm $X_1$.  The right-hand side above fails to be $x$-integrable, precisely because of criticality of Hardy's inequality with power $x^{-\frac{1}{2}}$ in $L^2$:
\begin{align} \label{HCrit.1}
|\int \int vv_x \ud y \ud x| \le ||v x^{-\frac{1}{2}}||_{L^2_{xy}} ||v_x x^{\frac{1}{2}}||_{L^2_{xy}} \xcancel{\le} ||v_x x^{\frac{1}{2}}||_{L^2_{xy}}^2. 
\end{align}

To avert this, we move to higher-order derivatives, which invokes the full strength of the norm $Z$. Indeed, suppose we knew $v_x \sim x^{-\frac{3}{2}}$, then coupled with the boundary condition $v \rightarrow 0$ as $x \rightarrow \infty$, this would immediately imply $v \sim x^{-\frac{1}{2}}$. Establishing the decay rate, $||v_x||_{L^\infty_y} \le x^{-\frac{3}{2}}$, then becomes the goal, which requires us to go to third-order energy estimates (thus explaining the presence of $X_1, X_2, X_3$ in the norm $Z$). Our main uniform estimates, given in Lemmas \ref{L.Evol}, \ref{Lemma.UIMP} are given by the following sequence:
\begin{align} \label{below.1}
||v x^{\frac{1}{2}}||_{L^\infty_{xy}} \lesssim ||v_x x^{\frac{3}{2}}||_{L^\infty_{xy}} \lesssim \sup_{x \ge 1} ||v_x x ||_{L^2_y}^{\frac{1}{2}} \cdot \sup_{x \ge 1} ||v_{xy} x^2||_{L^2_y}^{\frac{1}{2}} \lesssim ||u,v||_{X_1 \cap X_2 \cap X_3}. 
\end{align}

The key point is that the quantities $||v_x x||_{L^2_y}$ and $||v_{xy}x^2||_{L^2_y}$ appearing above \textit{do not} face issues of Hardy-criticality present in (\ref{HCrit.1}), which can be seen in Lemma \ref{L.Evol}. 

Applying $\p_x^k$ to the system creates singularities near the corner at $(1,0)$. To handle this we cutoff near the boundary, $x = 1$, when performing higher order energy estimates. Cutoff functions interact poorly with nonlinearities, and so we need to supplement energy estimates with elliptic estimates which retain additional control of $[u,v]$ near $x = 1$ (though not all the way up to the boundary, $x = 1$). These are characterized by the norms $||\cdot||_{Y_i}$ in (\ref{norm.Z.INTRO}). These $Y_i$ norms are controlled by invoking the elliptic theory, which in turn requires sacrificing factors of $\eps$. For this, we require a high power of $\eps^n$ to accompany nonlinear terms, which in turn requires us to go to high order expansions in (\ref{expansion.1}) - (\ref{expansion.3}).

\subsection*{Existence and Uniqueness (Chapter III):}

Upon proving our main \textit{a-priori} estimate, Theorem \ref{thm.m.2}, we prove existence and uniqueness of a solution in $Z$. As in (\ref{intro.cvc}), applying the multiplier $u$ produces the nonlinearity $vu_y \cdot u$, which is a perfect derivative and therefore vanishes. This cancellation property is destroyed upon taking differences, and so we cannot rely on a standard application of the contraction mapping theorem. A sequence of auxiliary, approximate systems are then carefully designed to produce enough compactness enabling us to show existence of a solution in the space $Z$. 

The uniqueness in $Z$ is a more delicate matter, again due to a lack of the perfect derivative structure. For this, we take the difference of two solutions in $Z$ and repeat the energy analysis with weaker weights. The reader is referred to Lemma \ref{L1U}, in which the weights must be selected carefully in a small interval below those weights appearing in the energy analysis, for instance in (\ref{intro.EE}), (\ref{intro.PE}). These methods are carried out in Chapter III.

\subsection*{Notation and Important Parameters}  
There are three important parameters in this paper: $\eps$, $\delta$, and $n$ (see (\ref{expansion.1}) - (\ref{expansion.3})). The notation $A \lesssim B$ means $A \le CB$, where $C$ is some constant which is independent of small $\delta, \eps$, and large $n$. Constants denoted by $\mathcal{O}(\delta)$ or $\mathcal{O}(\eps)$ satisfy $\mathcal{O}(\delta), \mathcal{O}(\eps) \rightarrow 0$ as $\delta, \eps \rightarrow 0$, respectively. Given any parameter, say $p$, constants denoted by $C(p)$ mean those constants which depend (perhaps poorly) on large values of $p$. Given two parameters, $\delta$ and $\sigma$, for instance, we will write $\mathcal{O}(\delta; \sigma)$ to denote a constant which depends on $\sigma$ and $\delta$, but such that for fixed $\sigma$, $\delta$ can be made small to make the constant small, for instance $\delta \times \sigma$. We define $||f||_{L^p_y}^p := \int f(x,y)^p dy$. When unspecified, $||\cdot||_{L_p}$ means the $L^p$ norm of two-variables. We define here the differential operators: 
\begin{align} \label{def.Scaled.Lap}
&\Delta_\eps u := (\p_{xx} + \eps \p_{yy}) u \text{ for any profile } u; \hspace{5 mm} \Delta [u^i_e, v^i_e] := (\p_{xx} + \p_{YY}) [u^i_e, v^i_e].
\end{align}

The variable $z$ will denote a self-similar variable, so typically $z = \frac{y}{\sqrt{x}}$ or $z = \frac{\eta}{\sqrt{x}}$. Finally, the word ``profiles" refers to terms in the expansion (\ref{expansion.1}) - (\ref{expansion.2}), excluding the remainders $[u,v]$, and ``profile terms" refers to the linearizations in $S_u, S_v$, defined in (\ref{defn.Su}).

\break

\part*{\centerline{Chapter I: Construction of Profiles}} 
\addcontentsline{toc}{part}{Chapter I: Construction of Profiles}

\section{Overview of Results}

The purpose of this chapter is to construct each of the profiles appearing in the expansions (\ref{expansion.1}) - (\ref{expansion.3}), with the exception of the final terms, $[u,v,P]$. This results of this chapter are used extensively in Chapter II. Let us first introduce the following notation for $\bar{u}^{(n)}_s, \bar{v}^{(n)}_s$: 
\begin{align} \label{defn.UR}
&u_R^P := \sum_{j=0}^n \epsilon^{\frac{j}{2}} u^j_p, \hspace{3 mm} u_R^E := 1 + \sum_{j=1}^n \epsilon^{\frac{j}{2}} u^j_e, \hspace{3 mm} u_R := \bar{u}_s^{(n)} = u_R^P + u_R^E, \\ \label{defn.VR}
&v_R^P := \sum_{j=0}^n \epsilon^{\frac{j}{2}} v^j_p, \hspace{3 mm} v_R^E := \sum_{j=1}^n \epsilon^{\frac{j}{2}-\frac{1}{2}} v^j_e, \hspace{3 mm} v_R := \bar{v}^{(n)}_s = v_R^P + v_R^E.
\end{align}

We shall also have occasion to further split $u^P_R$ to distinguish the final layer via: 
\begin{align}
u^{P,n-1}_R =  \sum_{j = 0}^n \epsilon^{\frac{j}{2}} u^j_p, \text{ so that } u^P_R = u^{P,n-1}_R + \epsilon^{\frac{n}{2}} u^n_p.
\end{align}

Inserting the expansion (\ref{expansion.1}) - (\ref{expansion.3}) into the scaled NS equations, (\ref{scaled.NS.1}) - (\ref{scaled.NS.3}), motivates the following definition:
\begin{definition} The $n$'th remainder is denoted by: 
\begin{align} \label{Rude.n.INTRO}
R^{u,n} &:= -\Delta_\epsilon \bar{u}_s^{(n)} + \bar{u}_s^{(n)} \bar{u}_{sx}^{(n)} + \bar{v}_s^{(n)} \bar{u}_{sy}^{(n)} + \bar{P}_{sx}^{(n)}, \\ \label{Rvde.n.INTRO}
R^{v,n} &:=  -\Delta_\epsilon \bar{v}_s^{(n)} + \bar{u}_s^{(n)} \bar{v}_{sx}^{(n)} +  \bar{v}^{(n)}_s \bar{v}^{(n)}_{sy}   + \frac{\partial_y}{\epsilon} \bar{P}_s^{(n)}.
\end{align}
\end{definition}

The main result of this chapter is: 
\begin{theorem} \label{thm.m.part.1} Let $n \ge 2 \in \mathbb{N}$. Let $\delta, \eps$ be sufficiently small relative to universal constants, and $\eps << \delta$. Let the boundary and in-flow data from (\ref{intro.BC.1}) - (\ref{IF.EUL}) be prescribed. Then there exist Prandtl profiles $[u^j_p, v^j_p, P^j_p]$ for $j = 1,...,n$, Euler profiles $[u^j_e, v^j_e, P^j_e]$ for $j = 1,...,n$, and auxiliary pressures $[P^{j,a}_p, P^{j,a}_e]$ for $j = 1,...,n$ such that for $R^{u,n}, R^{v,n}$ as defined in (\ref{Rude.n.INTRO}) - (\ref{Rvde.n.INTRO}), and for any $\gamma \in [0,\frac{1}{4})$, $n \ge 2$, and for $\sigma_n = \frac{1}{10,000}$, $\kappa>0$ arbitrarily small, the following remainder estimate holds for any $k \ge 0$: 
\begin{align} \label{Rnl2.INTRO.1}
&\epsilon^{-\frac{n}{2}-\gamma} \Big| \partial_x^k R^{u,n} + \sqrt{\epsilon} \partial_x^k R^{v,n} \Big| \le C(n, \kappa) \epsilon^{\frac{1}{4}-\gamma-\kappa} x^{-k-\frac{3}{2}+2\sigma_n}, \\  \label{Rnl2.INTRO}
&\epsilon^{-\frac{n}{2}-\gamma} ||\sqrt{\epsilon}\partial_x^k R^{u,n}, \sqrt{\epsilon} \partial_x^k R^{v,n}||_{L^2_y} \le C(n, \kappa) \epsilon^{\frac{1}{4}-\gamma-\kappa} x^{-k-\frac{5}{4}+2\sigma_n + \kappa}.
\end{align}

The following bounds hold on $[u_R,v_R]$ by construction, for any $[k, j, m] \ge 0$, so long as $n$ is sufficiently large relative to $m$.
\begin{align} \label{PE0.1}
&||\partial_x^k \partial_y^j v^P_R z^m x^{k + \frac{j}{2} + \frac{1}{2}}||_{L^\infty}  \le C(k, j, m) \text{ if } k \ge 1, \\  \label{PE0.2}
& || \partial_y^j v^P_R z^m x^{ \frac{j}{2} + \frac{1}{2}}||_{L^\infty}  \le C(j, m) \text{ if } j \ge 2, \\  \label{PE1}
 &|| \partial_y^j v^P_R z^m x^{ \frac{j}{2} + \frac{1}{2}}||_{L^\infty}  \le \mathcal{O}(\delta; m, j)\text{ if }j = 0,1. \\  \label{PE0.3}
 & || \partial_x^k \partial_y^j u^P_R z^m x^{k + \frac{j}{2}}||_{L^\infty}  \le C(k, j, m) \text{ for } k > 1, j \ge 0 \\  \label{PE0.4}
 &|| \partial_x u^P_R z^m x||_{L^\infty}  \le \mathcal{O}(\delta; m), \\ \label{PE2}
 & || \partial_x \p_y^j u^P_R z^m x||_{L^\infty}  \le C(m, j) \text{ for }j \ge 1 \\  \label{PE0.5}
 & || \p_y^j u_R^{P,n-1} y^j z^m||_{L^\infty} \le \mathcal{O}(\delta; m, j), \text{ for } 0 \le j \le 2, \\ \label{PE3.b}
 & || \p_y^j u_R^{P,n-1} y^j z^m||_{L^\infty} \le C(m,j), \text{ for } j > 2,\\ \label{PE3}
 & || \p_y^j u^n_p y^j x^{\frac{1}{2}-\sigma_n} ||_{L^\infty} \le C(n,j) \text{ for all } j \ge 0, \\  \label{PE4}
& ||\partial_x^k \partial_Y^j v^E_R x^{k+j + \frac{1}{2}}||_{L^\infty}  \le C(k, j) \text{ for } k + j > 0, \\ \label{PE5}
&  ||\partial_x^k \partial_Y^j u^E_R x^{k+j + \frac{1}{2}}||_{L^\infty}  \le \sqrt{\eps} C(k, j) \text{ for } k + j > 0 \\ \label{PE4.n.1}
& ||\partial_x^k v^E_R  x^{k-\frac{1}{2}}Y ||_{L^\infty} \le C(k,j) \text{ for } k \ge 1, \\ \label{PE4.new.2}
&  ||\{u^E_R - 1, v^E_R \} x^{\frac{1}{2}}, v^E_{RY}  x^{\frac{3}{2}}||_{L^\infty} \le \mathcal{O}(\delta). 
\end{align}
\end{theorem}

The profiles $u_R, v_R$ from (\ref{defn.UR}) - (\ref{defn.VR}) arise as coefficients in the linearized problem for the Navier-Stokes remainders $[u,v, P]$, which is to be analyzed in Chapter II. The estimates obtained in (\ref{PE0.1}) - (\ref{PE5}) are therefore essential to the analysis of Chapter II. In particular, we invite the reader to compare the requirements discussed in (\ref{intro.s.s.absorb}) and (\ref{REQ}) with the estimates we prove in (\ref{PE0.1}) - (\ref{PE4.new.2}).

The structure of this chapter is as follows: 

\begin{itemize}

\item[(Step 1)]  Construction of the zeroeth-order Prandtl layers, $[u^0_p, v^0_p]$ (Section \ref{section.u0p}): The distinguishing feature of $u^0_p$ are the mismatched boundary conditions, as seen from (\ref{intro.eqn.u0p}). As shown in Proposition \ref{thm.front}, this contributes a ``front"-profile which looks similar to $e_\delta$, defined in (\ref{defn.erf}). The decay rates for Prandtl profiles from estimates (\ref{PE0.1}) - (\ref{PE5}) are dictated by this front profile. For the zeroeth layer, this is formalized in Proposition \ref{prop.ww2}, Corollaries \ref{cp1} and \ref{cp2}. It is essential that we obtain estimates weighted in the self-similar variable, $z = \frac{y}{\sqrt{x}}$, as can be seen from (\ref{PE0.1}) - (\ref{PE4.new.2}) above.

\item[(Step 2)] Construction of Euler-1 layers, $[u^1_e, v^1_e]$ (Section \ref{section.euler}): Given the boundary conditions, which are known to satisfy the progressive estimate: $|\p_x^k v^1_e(x,0)| \lesssim x^{-\frac{1}{2}-k}$, we must obtain the sharp uniform estimate $|\p_x^k v^1_e(x,Y)| \lesssim  x^{-\frac{1}{2}-k}$. This is a delicate matter, as these profiles are elliptic, and as indicated in (\ref{REQ}), the required decay rates must be obtained \textit{exactly}. We introduce a method to obtain the required pointwise estimates using directly the Poisson integral formulation, which is carried out in Proposition \ref{Lemma.Sing.Unif}. This section is a major contribution of Chapter I.

\item[(Step 3)] Construction of Prandtl-1 layer, $[u^1_p, v^1_p]$ (Section \ref{section.prandtl.1}): This profile is controlled using coupled energy and positivity estimates, given in Lemmas \ref{lemma.pr.1.e} and Lemma \ref{lemma.pr.1.p}.

\item[(Step 4)] Construction of Intermediate Euler and Prandtl layers (Section \ref{Section.Inter}):. The essential mechanism here is as follows: as one consider linearizations for $[u^j_p, v^j_p]$ for $j > 1$, one encounters terms which scale poorly in $z = \frac{y}{\sqrt{x}}$, due to Euler-Euler interactions. However, due to the Cauchy-Riemann structure present in the Euler profiles (see (\ref{div.curl})), we may introduce auxiliary pressures $P^{i,a}_p, P^{i,a}_E$ which creates cancellations of all terms which are ``purely-Eulerian". This is seen in (\ref{grad.pres.1}) - (\ref{grad.pres.2}). 

\item[(Step 5)] Construction of Final Prandtl layer, $[u^n_p, v^n_p]$ (Subsection \ref{Final.Pr.Sub}): The final Prandtl layer satisfies the boundary condition $v^n_p(x,0) = 0$, and so has a contribution as $y \uparrow \infty$. We cut-off this layer in the region $z \le \frac{1}{\sqrt{\eps}}$, which honors the parabolic scaling of the Prandtl layers. This is a generalization of the cut-off used in \cite{GN}. It is delicate to ensure these cutoff layers obey desirable estimates, which is done in Lemma \ref{Lpf}. It is also delicate to ensure that this process contributes an error that satisfies estimates (\ref{Rnl2.INTRO.1}) - (\ref{Rnl2.INTRO}) above. This is proven in Lemma \ref{Lemma FE}.

\end{itemize}

\section{Asymptotics of Prandtl Layer, $u^0_p$:} \label{section.u0p}

Our starting point is the leading order terms from (\ref{sys.pr.i.intro}), which yields the following system for the Prandtl layer, $[u^0_p, v^0_p]$: 
\begin{align} \label{eqn.Pr.0.1}
&\Big(1 + u^0_p \Big) u^0_{px} + \Big(v^0_p + v^1_e(x,0) \Big) u^0_{py} = u^0_{pyy}, \hspace{5 mm} u^0_{px} + v^0_{py} = 0, \\ \label{eqn.Pr.0.2}
&u^0_p(x,0) = -\delta, \hspace{3 mm} v^0_p(x,0) = -v^1_e(x,0), \hspace{3 mm} u^0_p(1,y) = U_0(y). 
\end{align}

As shown in \cite[pg. 9]{GN}, taking $[u^0_p, v^0_p]$ to solve the system (\ref{eqn.Pr.0.1}) - (\ref{eqn.Pr.0.2}) creates an error:  
\begin{align} \label{errore.1}
R^{u,0} := \epsilon u^0_{pxx} + \sqrt{\epsilon}y v^1_{eY} u^0_{py} + \epsilon u^0_{py} \int_0^y \int_{y'}^y v^1_{eYY} dy'' dy' 
\end{align}

This $R^{u,0}$ contribution is higher-order in $\eps$, and so will be accounted for as a forcing term in the construction of the next Prandtl layer, $[u^1_p, v^1_p]$ (see (\ref{forcing.pr.1})). After introducing the von-Mises coordinates, 
\begin{equation} \label{VM}
\eta = \int_0^y \Big( 1 + u^0_p \Big) dy',
\end{equation}

the equation for $u^0_p$ becomes parabolic, with $x$ being the time-like variable: 
\begin{equation} \label{intro.eqn.u0p}
u^0_{px} = \partial_\eta((1+u^0_p)u^0_{p\eta}), \hspace{3 mm} u^0_p(x,0) = -\delta, \hspace{3 mm} u^0_p(x,\infty) = 0, \hspace{3 mm} u^0_p(1,\eta) = U_0, 
\end{equation}

Via the maximum principle, as in \cite{GN},
\begin{equation} \label{pr.pos}
1 + u^0_p \ge 1- \delta. 
\end{equation}

For the analysis in Section \ref{section.u0p}, it is convenient to introduce the shifted unknown 
\begin{equation} \label{defn.q}
q = u^0_p + \delta.
\end{equation}

The shifted unknown then satisfies the IBVP: 
\begin{align} \label{q.eqn}
q_x = \partial_\eta \Big( (1-\delta + q) q_\eta \Big), \hspace{3 mm} q(x,0) = 0, \hspace{3 mm} q(x,\infty) = \delta ,\hspace{3 mm} q(1,\eta) = u^0_p(1,\eta) + \delta. 
\end{align}

\subsection{Existence of Front Profile}

The dynamics of the solution to equations (\ref{eqn.Pr.0.1}) - (\ref{eqn.Pr.0.2}), or equivalently, (\ref{q.eqn}), are governed by a self-similar ``front", in the sense of \cite{BKL}. This is due to the mismatch in boundary conditions at $y = 0$ and $y = \infty$, as seen from (\ref{q.eqn}). Inserting a self-similar anzatz $\phi_\ast(z) = \phi_\ast(\frac{\eta}{\sqrt{x}})$ into equation (\ref{q.eqn}) gives the following ODE:  
\begin{equation} \label{ODE1}
(1-\delta+\phi_\ast) \phi_\ast'' + \Big| \phi_\ast' \Big|^2 + \frac{z}{2} \phi_\ast'= 0, \hspace{3 mm} \phi_\ast(0) = q(x,0) = 0, \hspace{3 mm} \phi_\ast(\infty) = q(x,\infty) = \delta. 
\end{equation}

Here the $'$ denotes $\partial_z$, where $z$ is the self-similar variable for $\phi_\ast$. As the profile $\phi_\ast$ that we seek is a nonlinear variant of the Gaussian error function, we study: 
\begin{equation} \label{diffofdiff}
\psi = \phi_\ast - e_\delta, 
\end{equation}

where $e_\delta$ is the Gaussian front profile with value $\delta$ at $+\infty$: 
\begin{equation} \label{defn.erf}
e_\delta(z) = \frac{\delta}{\sqrt{\pi}}\int_0^z e^{-\frac{t^2}{4}} dt.
\end{equation}

Let us record the following: 
\begin{lemma} \label{edelta} With $e_\delta$ being defined as (\ref{defn.erf}), $e_\delta(\frac{\eta}{\sqrt{x}})$ is an explicit solution of the heat equation, which bridges two distinct boundary conditions at $y = 0$ and $y = \infty$: 
\begin{align} \label{heateqn1}
\Big(\p_x - \p_{\eta \eta}\Big) e_\delta(\frac{\eta}{\sqrt{x}}) = 0, \hspace{3 mm} e_\delta(0) = 0, \hspace{3 mm} e_\delta(\infty) = \delta. 
\end{align}
\end{lemma}
\begin{proof}

We first record the identities: 
\begin{align} \label{basic.id}
\frac{\partial z}{\partial \eta} = \frac{1}{\sqrt{x}}, \hspace{3 mm} \frac{\partial z}{ \partial x} = -\frac{z}{2x}.
\end{align} 

Differentiating (\ref{defn.erf}) gives the identities: 
\begin{align} \label{equiv.front}
e_\delta'(z) = \frac{ \delta}{\sqrt{\pi}} e^{-\frac{z^2}{4}}, \hspace{3 mm} e_{\delta}''(z) = -\frac{z}{2x} \frac{\delta}{\sqrt{\pi}} e^{-\frac{z^2}{4}} = -\frac{z}{2x} e_\delta'(z). 
\end{align}

One then checks that: $\p_{x} e_\delta(\frac{\eta}{\sqrt{x}}) = \p_{\eta \eta} e_\delta(\frac{\eta}{\sqrt{x}})$ is equivalent to (\ref{equiv.front}). The boundary condition at $0$ is trivial from (\ref{defn.erf}), and the boundary condition at $\infty$ arises from: $e_\delta(\infty) = \frac{ \delta}{\sqrt{\pi}} \int_0^\infty e^{-\frac{t^2}{4}} dt = \delta$.  The lemma is proven.

\end{proof}

Using the heat equation for $e_\delta$, coupled with (\ref{diffofdiff}) we obtain: 
\begin{align} \n
&\Big( 1 - \delta + \psi \Big) \psi'' + \Big| \psi' \Big|^2 + \frac{z}{2} \psi' = -2\psi' e_\delta' - |e_\delta'|^2 - \psi e_\delta'' - e_\delta \psi'' - e_\delta e_{\delta}'', \\ \label{eqn.front.diff}
& \psi(0) = \psi(\infty) = 0. 
\end{align}

The first task is to obtain existence of a solution, $\psi$, to the above boundary value problem in a suitable Sobolev space.\footnote{It is clear by rescaling $z \rightarrow (1-\delta)^{\frac{1}{2}}z$, we can replace the factor of $1-\delta$ in front of $\psi''$ by simply $1$. This rescaling would change the main linear operator, $(1-\delta) \psi'' + \frac{z}{2} \psi'$ to $\psi'' + \frac{z}{2} \psi'$. For notational ease, then, we work simply with the $\psi''$ instead of $(1-\delta)\psi''$. The actual self-similar variable, then, is really $(1-\delta)\frac{\eta}{\sqrt{x}}$, but as $(1-\delta)$ is near $1$, this causes no confusion in the analysis to follow. }

\begin{proposition} \label{thm.front} For $\delta$ sufficiently small, there exists a unique solution to the equation (\ref{eqn.front.diff}) in $H^2_w(\mathbb{R}_+)$ satisfying $||\psi||_{H^2_w} \lesssim \delta$, where $H^2_w$ is a weighted variant of $H^2$ which is formally defined in (\ref{defn.h2w})
\end{proposition}
\begin{proof}

For this argument, fix $r > 0$. We will eventually let $r \rightarrow \infty$. Our starting point is to establish existence of solutions to the linear operator  
\begin{equation}
\Big(-\partial_z^2 - \frac{z}{2}\partial_z \Big).
\end{equation}

We recall the identity given in \cite{BKL}: 
\begin{equation}
e^{\frac{z^2}{8}} \Big(-\partial_z^2 - \frac{z}{2}\partial_z  \Big) e^{-\frac{z^2}{8}} \psi = \Big( -\partial_z^2 + \frac{z^2}{16} + \frac{1}{4} \Big) \psi,
\end{equation}
which in turn implies 
\begin{equation} \label{front.eqn.ur}
\Big(-\partial_z^2 - \frac{z}{2}\partial_z  \Big) \psi = f, \hspace{3 mm} \psi(0) = \psi(r) = 0
\end{equation}

if and only if 
\begin{align} \label{prb.tilde.front}
& \Big(-\partial_z^2 + \frac{z^2}{16} + \frac{1}{4} \Big) \tilde{\psi}_{(r)}  = \tilde{f}, \hspace{3 mm} \tilde{\psi}_{(r)}(0) = \tilde{\psi}_{(r)}(r) = 0, \hspace{3 mm} \tilde{\psi}_{(r)} = e^{\frac{z^2}{8}} \psi_{(r)}, \hspace{3 mm} \tilde{f} = e^{\frac{z^2}{8}}f. 
\end{align}

The subscripts are included to emphasize the domain, $(0,r) \subset \mathbb{R}$. Consider: 
\begin{equation}
B[\tilde{\psi}_{(r)},v] := \int_0^r \partial_z \tilde{\psi}_{(r)} \cdot \partial_z v + \int_0^r \Big(\frac{z^2}{16} + \frac{1}{4} \Big) \tilde{\psi}_{(r)} v, \hspace{3 mm} H^1_0(0,r) \times H^1_0(0,r) \rightarrow \mathbb{R}.
\end{equation}

$B$ is clearly bounded and coercive on $H^1_0(0,r)$. One obtains the existence of a unique $H^1$ weak solution to the system (\ref{prb.tilde.front}), and correspondingly to (\ref{front.eqn.ur}) from the Lax-Milgram Lemma. Via elliptic regularity, this implies $H^2$ regularity which, in original unknowns, translates to the existence of a unique $\psi_{(r)} \in H^2(0,r)$ for each $f \in L^2(0,r)$ such that:
\begin{equation} \label{front.lin.1}
||\psi_{(r)}||_{H^2(0,r)} \le C(r) ||f||_{L^2(0,r)},\text{ and } \psi_{(r)}(0) = \psi_{(r)}(r) = 0. 
\end{equation}

Let us rewrite the nonlinear problem (\ref{eqn.front.diff}) as a fixed point to:
\begin{equation} \label{front.nonlin.stab.1}
-\psi_{(r)}'' -\frac{z}{2} \psi_{(r)}' = f(\overline{\psi}_{(r)}) + (e_\delta + \delta)\overline{\psi}_{(r)}'', \hspace{3 mm} \psi_{(r)}(0) = \psi_{(r)}(r) = 0,
\end{equation}

where
\begin{equation} \label{front.nonlin.stab.2}
f(\overline{\psi}_{(r)}) = \overline{\psi}_{(r)}\overline{\psi}_{(r)}'' + \Big| \overline{\psi}_{(r)}' \Big|^2 + 2e_\delta \overline{\psi}_{(r)}' + \Big|e_\delta' \Big|^2 + e_{\delta}'' \overline{\psi}_{(r)} +  e_\delta e_{\delta}''.
\end{equation}

Define now the norm: 
\begin{equation} \label{defn.h2w}
||\psi_{(r)}||_{H^2_w}^2 :=  \int |\psi_{(r)}''|^2 + \int (1+z^2) \Big|\psi_{(r)}, \psi_{(r)}'\Big|^2,
\end{equation}

and the parameter: 
\begin{equation}
R(\delta) = \max \Big\{ ||e_\delta||_{L^\infty}, ||e_\delta'||_{L^\infty}, ||e_\delta''||_{L^\infty}, ||(1+z^2)^{\frac{1}{2}}e_\delta''||_{L^2} , ||e_\delta' (1+z^2)^{\frac{1}{2}}||_{L^2}\Big\}
\end{equation}

Via a standard integration by parts argument applied to (\ref{front.nonlin.stab.1}), we obtain the stabilizing estimate: 
\begin{align}
||\psi_{(r)}||_{H^2_w}^2 \le R(\delta)^4 + R(\delta)^2 ||\overline{\psi}_{(r)}||_{H^2_w}^2 + ||\overline{\psi}_{(r)}||_{H^2_w}^4,
\end{align}

which proves that if $||\overline{\psi}_{(r)}||_{H^2_w} \le R(\delta)$, then 
\begin{equation}
||\psi_{(r)}||_{H^2_w}^2 \le R(\delta)^4 + 2R(\delta)^2 ||\overline{\psi}_{(r)}||_{H^2_w}^2
\end{equation}

Selecting $\delta$ small enough then ensures $3R(\delta)^4 < R(\delta)^2$, ensuring that 
\begin{equation}
\psi_{(r)} = N(\overline{\psi_{(r)}}) \in B_{R(\delta)} \subset H^2_w \text{ whenever } \overline{\psi_{(r)}} \in B_{R(\delta)} \subset H^2_w.
\end{equation}

The second step is to prove the nonlinear map $N$ is a contraction map on $B_{R(\delta)} \subset H^2_w$. As such, label the pairs: 
\begin{equation} \label{front.nonlin.stab.3}
-\psi_{(i,r)}'' -\frac{z}{2} \psi_{(i,r)}' = f(\overline{\psi_{(i, r)}}) + \Big(e_\delta + \delta\Big) \overline{\psi}_{(r)}'', \hspace{3 mm} \psi_{(i, r)}(0) = \psi_{(i, r)}(r) = 0, \hspace{3 mm} i = 1,2. 
\end{equation}

Taking differences yields and performing a standard integration by parts argument gives:  
\begin{align}
||\psi_{1,r} - \psi_{2,r}||_{H^2_w}^2 \le R(\delta)^2 ||\overline{\psi}_{1,r} - \overline{\psi}_{2,r}||_{H^2_w}^2, 
\end{align} 

By the contraction mapping theorem there exists a unique solution to the nonlinear problem (\ref{front.nonlin.stab.1}) - (\ref{front.nonlin.stab.2}). Moreover, as this solution lies in the ball $B_{R(\delta)}$, it obeys the estimate 
\begin{equation}
||\psi_{(r)}||_{H^2_w(0,r)} \le R(\delta),
\end{equation}

uniformly in $r$. Therefore, we may let $r \rightarrow \infty$ to obtain a solution to the problem (\ref{eqn.front.diff}). 

\end{proof}

We may repeat the procedure in the above theorem with any weight, giving: 
\begin{equation}
||\langle z \rangle^M  \psi||_{H^2} \le \mathcal{O}(\delta; M). 
\end{equation}

By further differentiating equation (\ref{eqn.front.diff}), we can obtain 
\begin{equation} \label{GL.Front.est.1}
||\langle z \rangle^M \psi ||_{H^k} \le \mathcal{O}(\delta; M, k). 
\end{equation}

Our front from here on will be denoted as $\phi^\ast = \psi + e_\delta$. We will abuse notation and depict
\begin{equation} \label{GLFront}
\phi^\ast(z) = \phi^\ast(x,\eta). 
\end{equation}

We have the following corollary to the preceding analysis:

\begin{corollary} \label{corfront} The front $\phi^\ast$ obeys the following bounds, for any $m, k, l \ge 0$:
\begin{align} \label{corphi}
||z^m \partial_\eta^l \partial_x^k \Big(\phi^\ast - \delta \Big) ||_{L^p_\eta} \le \mathcal{O}(\delta) x^{-k-\frac{l}{2}+\frac{1}{2p}}
\end{align} 
\end{corollary}
\begin{proof}
This follows immediately from writing $\phi^\ast - \delta = \psi + (e_\delta - \delta)$, the definition of $e_\delta$ in (\ref{defn.erf}), the bounds in (\ref{GL.Front.est.1}), and then the chain rule together with the identities (\ref{basic.id}).
\end{proof}

Let us now record the following fact:
\begin{corollary} A solution $\phi_\ast$ to the system (\ref{ODE1}) must satisfy $\phi_\ast'(0) > 0$. 
\end{corollary}
\begin{proof}
Define $u_f = \phi_\ast, v_f = \phi_\ast'$. By (\ref{corphi}), $|u_f| \le c_0 \delta$ for some constant $c_0$. The system (\ref{ODE1}) is then: $u_f' = v_f$, $(1-\delta + u_f)v_f' + |v_f|^2 + \frac{z}{2}v_f = 0$. The invariant set $\Gamma = \{u_f = C, v_f = 0\}$, for $-c_0\delta \le C \le c_0 \delta$ contains equilibria. The solution $\phi_\ast$ corresponds to an orbit with initial condition $(u_f(0) = 0, v_f(0))$ and final condition $(u_f(\infty) = \delta, v_f(\infty))$. By ODE uniqueness, the trajectory cannot cross the set $\Gamma$ unless at $(\delta, 0)$. Thus, if $v_f(0) \le 0$, $v_f(z) \le 0$ for all $z$, which violates $u_f(\infty) = \delta$. Thus, we must have $v_f(0) > 0$. 
\end{proof}

\subsection{Zeroeth Prandtl Layer, $u^0_p$}

We define the remainder, $w$, via:  
\begin{equation} \label{rep.}
w = q - \phi^\ast(\cdot), \hspace{3 mm}  g := w(1,\cdot) = q(1,\cdot) - \phi^\ast(\cdot) = U_0(1,\cdot) + \delta - \phi^\ast(\cdot). 
\end{equation}

The initial data in (\ref{rep.}) are clearly rapidly decaying after recalling (\ref{intro.BC0.1}) - (\ref{intro.BC0.3}), via the relation: $g = U_0 - (e_\delta -  \delta) - \psi$. Moreover, the following smallness is obtained, again by recalling the assumptions (\ref{intro.BC0.1}) - (\ref{intro.BC0.3}): 
\begin{align}
||\langle y \rangle^m \p_y^j g ||_{L^\infty} \le \mathcal{O}(\delta; j, m) \text{ for any }m, \text{ and } j = 0,1,2. 
\end{align}

The equation satisfied by $w$ is the following (after substituting $q = w + \phi^\ast$): 
\begin{align} \label{RNG.eqn.w}
w_x = (1-\delta) w_{\eta \eta} + \partial_\eta \Big( qw_\eta + w \partial_\eta \phi^\ast \Big) = (1-\delta) w_{\eta \eta} + K, \\  
w(x,0) = 0, \hspace{3 mm} w(x,\infty) = 0, \hspace{3 mm} w(1,\eta) = g(\eta), 
\end{align}

where 
\begin{equation} \label{RNG.eqn.K}
K = w_\eta^2 + 2 \phi^\ast_\eta w_\eta + w w_{\eta \eta} + \phi^\ast w_{\eta \eta} + w \phi^\ast_{\eta \eta}.
\end{equation}

Let us first make the following basic observation regarding our initial data: 
\begin{lemma}[Compatibility of Initial Data] Suppose the compatibility conditions in (\ref{compatibility.1}) are enforced. Then the initial data of $w_x$ respects the boundary condition $w_x(1,0) = 0$. 
\end{lemma}
\begin{proof}

This follows upon evaluating (\ref{eqn.Pr.0.1}) at $y = 0$ and (\ref{q.eqn}) at $\eta = 0, x= 1$.
\end{proof}

\begin{remark}[Higher Order Compatibility] \label{hoc}
This same calculation for higher-order compatibility conditions at $x = 1, y = 0$ can be made. We omit displaying them explicitly, but will feel free to assume higher-order compatibility of the initial data at $y = 0$ as needed. 
\end{remark}

Let us now define a series of norms in which we control the remainder, $w$:
\begin{align} \label{defn.Q}
||w||_{Q(0,0)}^2 := \sup_{x \ge 1} \int w^2, w_\eta^2 x, w_{\eta \eta}^2 x^2 + \int_1^\infty \int w_\eta^2,  w_{\eta \eta}^2 x,  w_{\eta \eta \eta}^2 x^2. 
\end{align}

We shall need the general, differentiated and weighted variant of the above norms:
\begin{align} \nonumber
||w||_{Q(\sigma_0, k)}^2 &:= \sup_{x \ge 1} \int  \Big|\partial_x^k w\Big|^2 x^{2k-2\sigma_0}, \Big| \partial_x^k w_\eta\Big|^2x^{2k+1-2\sigma_0} ,  \Big| \partial_x^k w_{\eta \eta}\Big|^2 x^{2k + 2- 2\sigma_0} \\ \label{defn.Q.2} 
& + \int_1^\infty \int \Big|\partial_x^k  w_\eta\Big|^2  x^{2k-2\sigma_0},  \Big|\partial_x^k  w_{\eta \eta}\Big|^2 x^{2k+1-2\sigma_0}, \Big|\partial_x^k  w_{\eta \eta \eta}\Big|^2 x^{2k+2-2\sigma_0} . 
\end{align}

\begin{remark}
One should take note of several points. First, each application of $\partial_x$ to the system (\ref{RNG.eqn.w}) adds enhanced decay of $x^{-1}$, which is reflected in the norms above. This is because of the behavior of the profiles $\phi^\ast$ under the application of $\partial_x$, see estimate (\ref{corphi}).

Second, upon controlling $w$ in the $Q$-norms in (\ref{defn.Q}), the dynamics of $q = w + \phi^\ast$ will be dominated by those of the front $\phi^\ast$ (so long as the parameter $\sigma_0 < \frac{1}{4}$, which we will enforce). The small parameter $\sigma_0$ arises for technical reasons when performing weighted estimates, but should be ignored for the unweighted estimates. 

The final point is that upon heuristically identifying $w_x \approx w_{\eta \eta}$ through the equation (\ref{RNG.eqn.w}), one sees that the norms $Q(\sigma_0, k)$ have an iterative structure:  
\begin{align}
\int |\p_x^{k+1} w|^2 x^{2(k+1)} \ud \eta \approx \int |\p_x^k w_{\eta \eta}|^2 x^{2k+2} \ud \eta,
\end{align}

the quantity on the left-hand side above being in $||w||_{Q(0,k+1)}$ and the quantity on the right-hand side above being in $||w||_{Q(0,k)}$ (upon taking $\sup$ in $x$). The reason for this structure is that the equation (\ref{RNG.eqn.w}) is quasilinear.
\end{remark}

Through standard Sobolev interpolation, it is clear that: 
\begin{lemma} For any $ \sigma_0 \ge 0$, and $k \ge 0$, 
\begin{align} \label{q.sobolev}
\sup_x || \partial_x^k w x^{\frac{1}{4}-\sigma_0 + k}||_{L^\infty_\eta} + \sup_x  || \partial_x^k w_{\eta} x^{\frac{3}{4}-\sigma_0 + k}||_{L^\infty_\eta} \lesssim ||w||_{Q(\sigma_0, k)}. 
\end{align}
\end{lemma}
\begin{proof}
As the profile $w$ decays at $\eta \rightarrow \infty$ for each fixed $x$, we have: 
\begin{align} \nonumber
|\partial_x^k w|^2  x^{\frac{1}{2}-2\sigma_0 + 2k} &= \Big|2\int_\eta^\infty  \partial_x^k w \partial_x^k w_\eta x^{\frac{1}{2}-2\sigma_0 + 2k} d\eta'\Big| \lesssim ||w x^{k-\sigma_0}||_{L^2_\eta} ||w_\eta x^{k-\sigma_0 +\frac{1}{2}}||_{L^2_\eta} \\ 
& \lesssim ||w||_{Q(\sigma_0, k)}^2.
\end{align}

A similar computation works for the $\partial_x^k w_\eta$ term in (\ref{q.sobolev}).

\end{proof}

We now give the energy estimates for $w$:
\begin{proposition} \label{prop.ww2} For any $\sigma_0 > 0$, and $m \ge 0$, $w$ satisfies the following estimate: 
\begin{align} \label{w.w.2}
&||z^m w||_{Q(\sigma_0, 0)} \le \mathcal{O}(\delta; m, \sigma_0), \\
&||z^m w||_{Q(\sigma_0, k)} \le C(k, m, \sigma_0),  \text{ for }k > 0.
\end{align}
\end{proposition}

\begin{remark} Notice that the weights, $z = \frac{\eta}{\sqrt{x}}$, honor the parabolic scaling of (\ref{RNG.eqn.w}), and are propagated by the linear flow. 
\end{remark}

\begin{proof}

The proof proceeds in several steps. 

\subsubsection*{Step 1: Multiplier $M = w$:}

Applying the multiplier $M = w$ to the system (\ref{RNG.eqn.w}) generates the following positive terms:
\begin{align}
\int \Big( w_x - (1-\delta) w_{\eta \eta} \Big) \cdot w = \frac{\partial_x}{2} \int w^2 + (1-\delta) \int w_\eta^2. 
\end{align}

Next, we must treat the nonlinear and linearized terms in $K$. The boundary condition $w(x,0) = 0$ enables us to integrate by parts the quasilinear term:
\begin{align}
&\Big| \int w_\eta^2 \cdot w + w w_{\eta \eta} \cdot w \Big| = \Big| \int w w_\eta^2 \Big| \lesssim ||w||_{L^\infty_\eta} \Big( \int w_\eta^2 \Big), \\ \nonumber
& \Big| \int \phi^\ast w_{\eta \eta} \cdot w + \int \phi^\ast_{\eta \eta} w \cdot w + \int 2 \phi^\ast_\eta w_\eta \cdot w \Big| = \Big| \int \phi^\ast_{\eta \eta} w^2 + \int \phi^\ast w_\eta^2 \Big| \\ 
& \hspace{20 mm} \le ||\phi^\ast_{\eta \eta} \eta^2||_{L^\infty_\eta}  \int \frac{w^2}{\eta^2} +  ||\phi^\ast||_{L^\infty} ||w_\eta||_{L^2_\eta}^2 \le \mathcal{O}(\delta) \int w_\eta^2. 
\end{align}

Above, we have used Corollary \ref{corfront} to absorb two factors of $\eta$ to into $\phi^\ast_{\eta \eta}$. We have also used the Hardy inequality, which is available as $w(x,0) = 0$. Summarizing, we have: 
\begin{align} \label{new.1}
\partial_x \int w^2 + \int w_\eta^2 \lesssim \Big[\mathcal{O}(\delta) + ||w||_{L^\infty_\eta} \Big] \int w_\eta^2. 
\end{align}

\subsubsection*{Step 2: Multiplier $M = -w_{\eta \eta} x$}

Next, we will apply the multiplier $M = -w_{\eta \eta} x$. This generates the positive terms: 
\begin{align}
-\int \Big( \partial_x - (1-\delta) \partial_{\eta \eta} \Big) w \cdot w_{\eta \eta} x = \frac{\partial_x}{2} \int w_\eta^2 x - \int w_\eta^2 + (1-\delta) \int w_{\eta \eta}^2 x.
\end{align}

Let us now turn to the terms in $K$:
\begin{align} \label{imrp.u}
&\Big| \int w_\eta^2 w_{\eta \eta} x  +  \int w w_{\eta \eta}^2 x \Big| \le ||w, w_\eta x^{\frac{1}{2}}||_{L^\infty_\eta} \Big( ||w_\eta||_{L^2_\eta}^2 + ||x^{\frac{1}{2}}w_{\eta \eta}||_{L^2_\eta}^2 \Big), \\ \nonumber
&\Big| \int \phi_\eta^\ast w_\eta w_{\eta \eta} x +  \int \phi^\ast w_{\eta \eta}^2 x +  \int \phi^\ast_{\eta \eta} w w_{\eta \eta} x\Big| \\ \nonumber
&\hspace{20 mm} \le ||\phi^\ast, \eta x^{\frac{1}{2}} \phi^\ast_{\eta \eta}, x^{\frac{1}{2}}\phi^\ast_\eta||_{L^\infty} \Big( ||w_\eta||_{L^2_\eta}^2 + ||w_{\eta \eta} x^{\frac{1}{2}}||_{L^2_\eta}^2 \Big) \\
& \hspace{20 mm} \le \mathcal{O}(\delta) \Big( ||w_\eta||_{L^2_\eta}^2 + ||w_{\eta \eta}x^{\frac{1}{2}}||_{L^2_\eta}^2 \Big).
\end{align}

Summarizing, we have: 
\begin{align} \label{new.2}
\partial_x \int w_{\eta}^2 x + \int w_{\eta \eta}^2 x \lesssim \Big(1 + ||w_\eta x^{\frac{1}{2}}||_{L^\infty_\eta} \Big) \int w_\eta^2 + ||w, w_\eta x^{\frac{1}{2}}||_{L^\infty_\eta} \int w_{\eta \eta}^2 x. 
\end{align}

\subsubsection*{Step 3: Multiplier $w_x x^2$}

The third step is to take $\partial_x$ of the equation (\ref{RNG.eqn.w}). Doing so gives the system: 
\begin{align} \label{w.s.1}
\Big(\partial_x - (1-\delta)\partial_{\eta \eta} \Big) w_x = \partial_x K, \hspace{3 mm} w_x(x,0) = 0, \\ \label{w.s.2} 
w_x(1,\cdot) = (1-\delta) g_{\eta \eta} + g_\eta^2 + 2\phi^\ast_\eta g_\eta + g g_{\eta \eta} + \phi^\ast g_{\eta \eta} + \phi^\ast_{\eta \eta} g.
\end{align}

First, we will record that $||\langle \eta \rangle^m w_x(1,\eta)||_{L^\infty} \lesssim \mathcal{O}(\delta; m)$, for any $m \ge 0$, by (\ref{w.s.2}) and (\ref{intro.BC0.1}) - (\ref{intro.BC0.3}). Let us now apply the multiplier $M = w_x x^2$ to the system (\ref{w.s.1}). Again, we will remain cognizant of the boundary condition $w_x(x,0) = 0$ when integrating by parts. Doing so gives the positive terms: 
\begin{align}
\int \Big(\partial_x - (1-\delta) \partial_{\eta \eta} \Big)w_x \cdot w_x x^2 = \partial_x \int w_{x}^2 x^2 - \int w_x^2 x + (1-\delta) \int w_{\eta x}^2 x^2.  
\end{align}

Next, we come to the nonlinearity in $K_x$, where integrating by parts as necessary: 
\begin{align} \nonumber
&\Big| \int w_\eta w_{\eta x} w_x x^{2} + w_x w_{\eta \eta} w_x x^2 + w w_{\eta \eta x} w_x x^2 \Big| \\ &
 \hspace{10 mm} \le ||w_\eta x^{\frac{1}{2}}||_{L^\infty_\eta} ||w_{\eta x} x||_{L^2_\eta} ||w_x x^{\frac{1}{2}}||_{L^2_\eta} + ||w||_{L^\infty} ||w_{\eta x} x||_{L^2_\eta}^2, \\ \nonumber
&\Big| \int \phi^\ast_{\eta x} w_\eta w_x x^2 + \int \phi^\ast_\eta w_{\eta x} w_x x^2  \Big| \\ \n
& \hspace{10 mm} \le ||\phi^\ast_{\eta x} x \eta ||_{L^\infty_\eta} ||w_\eta||_{L^2_\eta} ||w_{x\eta} x||_{L^2_\eta} + ||\phi^\ast_\eta \eta||_{L^\infty_\eta} ||w_{\eta x} x||_{L^2_\eta}^2, \\
& \hspace{10 mm} \le \mathcal{O}(\delta) ||w_\eta||_{L^2_\eta} ||w_{x\eta} x||_{L^2_\eta} + \mathcal{O}(\delta) ||w_{\eta x} x||_{L^2_\eta}^2  \\ \nonumber
& \Big| \int \phi^\ast_x w_{\eta \eta} w_x x^2 + \int \phi^\ast w_{\eta \eta x} w_x x^2\Big| \\ \n
& \hspace{10 mm} \le ||\phi^\ast_x x||_{L^\infty_\eta} ||w_{\eta \eta} x^{\frac{1}{2}}||_{L^2_\eta} ||w_x x^{\frac{1}{2}}||_{L^2_\eta} + ||\phi^\ast, \phi^\ast_\eta \eta||_{L^\infty_\eta} ||w_{\eta x}x||_{L^2_\eta}^2, \\ 
& \hspace{10 mm} \le \mathcal{O}(\delta) ||w_{\eta \eta} x^{\frac{1}{2}}||_{L^2_\eta} ||w_x x^{\frac{1}{2}}||_{L^2_\eta} + \mathcal{O}(\delta) ||w_{\eta x}x||_{L^2_\eta}^2 \\ \n
& \Big| \int \phi^\ast_{\eta \eta x} w w_x x^2 + \int \phi^\ast_{\eta \eta} w_x^2 x^2 \Big| \\ \n
& \hspace{10 mm}\le ||\phi^\ast_{\eta \eta x} \eta x^{\frac{3}{2}}||_{L^\infty_\eta} ||w_x x^{\frac{1}{2}}||_{L^2_\eta} ||w_\eta||_{L^2_\eta} + ||\phi^\ast_{\eta \eta} x||_{L^\infty} ||w_x x^{\frac{1}{2}}||_{L^2}^2 \\
& \hspace{10 mm} \le \mathcal{O}(\delta) ||w_x x^{\frac{1}{2}}||_{L^2_\eta} ||w_\eta||_{L^2_\eta} + \mathcal{O}(\delta) ||w_x x^{\frac{1}{2}}||_{L^2_\eta}^2. 
\end{align}

Summarizing this piece: 
\begin{align} \nonumber
\partial_x \int w_x^2 x^2 &+ \int w_{\eta x}^2 x^2 \lesssim \Big[\mathcal{O}(\delta) + ||w, w_\eta x^{\frac{1}{2}}||_{L^\infty_\eta} \Big] \int w_{\eta x}^2x^2  \\ \label{new.3}
& + \Big[1 + \mathcal{O}(\delta) + ||w, w_\eta x^{\frac{1}{2}}||_{L^\infty_\eta} \Big] \int w_x^2 x + \mathcal{O}(\delta) \int w_\eta^2, w_{\eta \eta}^2 x. 
\end{align}

With these estimates in hand, we may now apply a standard continuous induction argument to conclude the global existence of $w$ satisfying (\ref{w.w.2}) with $k = m = \sigma_0 = 0$.  

\subsubsection*{Step 4: Weighted Estimates}

We now apply the weighted multiplier $M = x^{-2\sigma_0} z^m w$ to the system (\ref{RNG.eqn.w}). First, we have the following positive terms: 
\begin{align} \nonumber
\int \Big(w_x - w_{\eta \eta} \Big) \cdot wz^{2m} x^{-2\sigma_0} &=  \frac{\partial_x}{2} \int w^2 z^{2m} x^{-2\sigma_0} + \Big( \frac{m}{2} + \sigma_0 \Big) \int w^2 z^{2m} x^{-1-2\sigma_0} \\ \label{LINHON}
& + \int w_\eta^2z^{2m} x^{-2\sigma_0} -m(2m-1) \int w^2 z^{2m-2} x^{-1-2\sigma_0}.
\end{align}

The last term above has been estimated inductively at the $m-1$'th iterate. Next, we address the terms in $K$: 
\begin{align}
&\Big| \int \Big( w_\eta^2 + ww_{\eta \eta} \Big) \cdot wz^{2m}x^{-\sigma_0} \Big| \lesssim ||w||_{L^\infty_\eta} ||w_\eta z^m x^{-\sigma_0}||_{L^2_\eta}^2 , \\
& \Big| \int \Big( 2\phi^\ast_\eta w_\eta + \phi^\ast w_{\eta \eta} + w \phi^\ast_{\eta \eta} \Big) \cdot z^{2m} x^{-2\sigma_0} w \Big| \lesssim ||\eta \phi^\ast_\eta, \eta^2 \phi_{\eta \eta}^\ast||_{L^\infty_\eta} ||w_\eta z^m x^{-\sigma_0}||_{L^2_\eta}^2. 
\end{align}

Summarizing: 
\begin{align} \label{summer.1}
\partial_x \int w^2 z^{2m}x^{-2\sigma_0} &+ \int w_\eta^2 z^{2m}x^{-2\sigma_0} + \int w^2 z^{2m} x^{-1-2\sigma_0} \lesssim \int w^2 z^{2m-2} x^{-1-2\sigma_0}.
\end{align}

By integrating in $x$:
\begin{align} \nonumber
\sup_{x \ge 1} \int w^2 z^{2m}x^{-2\sigma_0} &+ \int \int w_\eta^2 z^{2m}x^{-2\sigma_0} + \int \int w^2 z^{2m}x^{-1-2\sigma_0} \\
& \lesssim \int_{x=1} w^2 y^{2m} + \int \int w^2 z^{2m-2} x^{-1-2\sigma_0} \le \mathcal{O}(\delta, m, \sigma_0). 
\end{align}

The next step is to apply the multiplier $M = -w_{\eta \eta} x^{1-2\sigma_0} z^{2m}$. First, this gives: 
\begin{align} \nonumber
- \int \Big( \partial_x& - \partial_{\eta \eta} \Big) w \cdot w_{\eta \eta} z^{2m} x^{1-2\sigma_0} = \\ \nonumber
&\frac{\partial_x}{2} \int w_\eta^2 z^{2m} x^{1-2\sigma_0} + \int w_{\eta \eta}^2 z^{2m} x^{1-2\sigma_0} + \frac{m}{2} \int w_\eta^2 z^{2m} x^{-2\sigma_0} \\ 
&- \frac{1-2\sigma_0}{2} \int w_\eta^2 z^{2m}x^{-2\sigma_0} + 2m \int w_\eta w_x z^{2m-1} x^{\frac{1}{2}-2\sigma_0}.
\end{align} 

Next, let us turn to the nonlinear terms in $K$:
\begin{align} \nonumber
\Big| \int \Big(w_\eta^2 &+ ww_{\eta \eta} \Big) \cdot w_{\eta \eta} z^{2m}x^{1-2\sigma_0} \Big| \le ||w, w_\eta x^{\frac{1}{2}}||_{L^\infty_\eta} \Big( ||w_{\eta} z^mx^{-\sigma_0}||_{L^2_\eta}^2 \times \\ 
& ||w_{\eta \eta} z^m x^{\frac{1}{2}-\sigma_0}||_{L^2_\eta}^2 \Big), 
\end{align}

and the linearized terms: 
\begin{align} \nonumber
\Big| \int \Big(\phi^\ast_\eta w_\eta &+ \phi^\ast w_{\eta \eta} + \phi^\ast_{\eta \eta} w \Big) \cdot w_{\eta \eta} z^{2m} x^{1-2\sigma} \Big| \\ 
&\le ||x^{\frac{1}{2}} \phi^\ast_\eta, x^{\frac{1}{2}} \eta \phi^\ast_{\eta \eta}, \phi^\ast||_{L^\infty} ||w_\eta x^{-\sigma_0} z^m||_{L^2_\eta} ||w_{\eta \eta} z^m x^{\frac{1}{2}-\sigma_0}||_{L^2_\eta}.
\end{align}

Summarizing, 
\begin{align} \label{summer.2}
\partial_x \int w_\eta^2 z^{2m} x^{1-2\sigma_0} &+ \int w_{\eta \eta}^2 z^{2m}x^{1-2\sigma_0} + \int w_\eta^2 z^{2m}x^{-2\sigma_0} \lesssim 
\int w_\eta w_x z^{2m-1} x^{\frac{1}{2}-2\sigma_0} \\
& \lesssim \frac{1}{10,000} \int w_{\eta}^2 z^{2m} x^{-2\sigma_0} + C \int w_x^2 x^{1-2\sigma_0} z^{2m-2}.
\end{align}

Integrating in $x$: 
\begin{align} \nonumber
\sup_{x \ge 1} \int w_\eta^2 z^{2m} x^{1-2\sigma_0} &+ \int_1^{\infty} \int w_{\eta \eta}^2 z^{2m}x^{1-2\sigma_0} + \int_1^\infty \int w_\eta^2 z^{2m} x^{-2\sigma_0}\\ 
& \lesssim \int_{x=1} w_\eta^2 y^{2m} + \int \int w_x^2 x^{1-2\sigma_0} z^{2m-2} \le \mathcal{O}(\delta; \sigma_0, m). 
\end{align}

The final step is to apply the multiplier $M = w_x x^{2-2\sigma_0} z^{2m}$ to the system (\ref{w.s.1}). This generates the following terms: 
\begin{align} \nonumber
\int \Big(\partial_x &- \partial_{\eta \eta} \Big)w_x \cdot w_x x^{2-2\sigma_0} z^{2m} = \frac{\partial_x}{2} \int w_x^2 x^{2-2\sigma_0} z^{2m} + \int w_{\eta x}^2 x^{2-2\sigma_0} z^{2m} \\ &+ \Big(\frac{m}{2} - 1 + \sigma_0\Big) \int w_x^2 x^{1-2\sigma_0} z^{2m} - m(2m-1) \int w_x^2 x^{1-2\sigma_0} z^{2m-2}.
\end{align}

We turn to the nonlinearities contained in $K_x$, for which we integrate by parts as needed:
\begin{align} 
\Big|\int \Big(w_\eta w_{\eta x} &+ w_x^2 w_{\eta \eta} + w w_{\eta \eta x} w \Big) \cdot w_x x^{2-2\sigma_0} z^{2m} \Big| \\ \nonumber
&\lesssim ||w, w_\eta x^{\frac{1}{2}}||_{L^\infty} \Big( ||w_x x^{\frac{1}{2}-\sigma_0}z^m||_{L^2_\eta}^2 + ||w_x x^{\frac{1}{2}-\sigma_0}z^{m-1}||_{L^2_\eta}^2 + ||w_x x^{\frac{1}{2}-\sigma_0} z^m||_{L^2_\eta}^2 \Big).
\end{align}

Next, we come to the linearizations around the front profile, $\phi^\ast$:
\begin{align} \nonumber
&\Big| \int \phi^\ast_{\eta x} w_\eta w_x x^{2-2\sigma_0} z^{2m} + \int \phi^\ast_\eta w_{\eta x} w_x x^{2-2\sigma_0} z^{2m}  \Big| \\ 
& \hspace{10 mm} \le ||\phi^\ast_{\eta x} \eta x z^{2m}, \phi^\ast_\eta \eta z^{2m} ||_{L^\infty} \Big( ||w_\eta ||_{L^2_\eta}^2 +  ||w_{x \eta} x||_{L^2_\eta}^2 \Big), \\ 
& \Big| \int \phi^\ast_x w_{\eta \eta} w_x x^{2-2\sigma_0} z^{2m} \Big| \le ||z^{2m} x \phi^\ast_x ||_{L^\infty} \Big( ||w_{\eta \eta} x^{\frac{1}{2}-\sigma_0} ||_{L^2_\eta}^2 + ||w_x x^{\frac{1}{2}-\sigma_0}||_{L^2_\eta}^2 \Big), \\ \nonumber
& \Big| \int \phi^\ast w_{x \eta \eta} w_x x^{2-2\sigma_0} z^{2m} \Big| = \Big| \int \phi^\ast_\eta w_x w_{x\eta} x^{2-2\sigma_0} z^{2m} + 2m \int \phi^\ast w_x w_{x\eta} x^{\frac{3}{2}-2\sigma_0} z^{2m-1} \Big| \\
& \hspace{10 mm} \lesssim ||\eta \phi^\ast_\eta z^{2m}, \phi^\ast ||_{L^\infty} \Big( ||x w_{x\eta}||_{L^2_\eta}^2 + ||x^{1-\sigma_0} w_{x\eta}z^m||_{L^2_\eta}^2 + ||z^{m-1} w_x x^{\frac{1}{2}-\sigma_0}||_{L^2_\eta}^2 \Big), \\ \nonumber
& \Big| \int \phi^\ast_{\eta \eta x} w w_x x^{2-2\sigma_0} z^{2m} + \int \phi^\ast_{\eta \eta} w_x^2 x^{2-2\sigma_0} z^{2m} \Big| \\
& \hspace{10 mm} \lesssim ||z^{2m} x \eta^2 \phi^\ast_{\eta \eta x}, z^{2m} \eta^2 \phi^\ast_{\eta \eta} ||_{L^\infty} \Big( ||w_\eta||_{L^2_\eta}^2 + ||w_{x\eta} x||_{L^2_\eta}^2 \Big).
\end{align}

Summarizing the results of this multiplier: 
\begin{align} \nonumber
\frac{\partial_x}{2} \int w_x^2 x^{2-2\sigma_0} z^{2m} &+ \int w_{\eta x}^2 x^{2-2\sigma_0} z^{2m} \lesssim
 \int w_x^2 x^{1-2\sigma_0} z^{2m} \\ \label{summer.3}
 &+ \int w_x^2 x^{1-2\sigma_0}z^{2m-2} + \mathcal{O}(\delta) \int w_\eta^2, w_{\eta \eta}^2x, w_{x \eta}^2 x^2,  
\end{align}

Again, we may integrate in $x$ to obtain: 
\begin{align} \nonumber
\sup_{x \ge 1} \int &w_x^2 x^{2-2\sigma_0} z^{2m} + \int_1^\infty \int w_{\eta x}^2 x^{2-2\sigma_0} z^{2m} \\ \nonumber
& \lesssim \int_{x=1} w_x(1)^2 y^{2m} + \int_1^\infty \int w_x^2 x^{1-2\sigma_0} z^{2m} + \int_1^{\infty} \int w_x^2 x^{1-2\sigma_0} z^{2m-2} \\ 
&+ \mathcal{O}(\delta) \int_1^\infty \int w_\eta^2, w_{\eta \eta}^2 x^{\frac{1}{2}}, w_{x\eta}^2 x.
\end{align}

The third and fourth terms on the right-hand side are controlled by $||w||_Q^2$, and the first term is controlled  by the initial data. For the second term, we must use the equation (\ref{RNG.eqn.w}): 
\begin{align} \nonumber
||w_x &x^{\frac{1}{2}-\sigma_0} z^{m}||_{L^2} \le ||w_{\eta \eta} x^{\frac{1}{2}-\sigma_0} z^{m}||_{L^2} + ||\Big(w_\eta^2 + 2\phi^\ast_\eta w_\eta + w w_{\eta \eta} + \phi^\ast w_{\eta \eta} + w \phi^\ast_{\eta \eta} \Big) x^{\frac{1}{2}-\sigma_0} z^{m}||_{L^2} \\ \nonumber
& \le ||w_{\eta \eta} x^{\frac{1}{2}-\sigma_0} z^{m}||_{L^2} + ||w_\eta x^{\frac{1}{2}}, \phi^\ast_\eta x^{\frac{1}{2}}, w, \phi^\ast, \phi^\ast_{\eta \eta} \eta x^{\frac{1}{2}}||_{L^\infty} ||w_\eta x^{-\sigma_0} z^m, w_{\eta \eta} x^{\frac{1}{2}-\sigma_0} z^m||_{L^2} \\
& \lesssim ||w_{\eta \eta} x^{\frac{1}{2}-\sigma_0} z^{m}, w_{\eta} x^{-\sigma_0}z^m||_{L^2}. 
\end{align}

As the majorizing terms above have been controlled, we can conclude: 
\begin{align}
\sup_{x \ge 1} \int &w_x^2 x^{2-2\sigma_0} z^{2m} + \int_1^\infty \int w_{\eta x}^2 x^{2-2\sigma_0} z^{2m} \le \mathcal{O}(\delta). 
\end{align}

By subsequently relating $w_x$ to $w_{\eta \eta}$ and $w_{\eta x}$ to $w_{\eta \eta \eta}$, we have shown that $||z^mw||_{Q(\sigma_0, 0)} \le \mathcal{O}(\delta, m, \sigma_0)$, which is the $k =0$ case of (\ref{w.w.2}). We may upgrade the previous set of estimates to higher order $x$ derivatives by iterating Steps 2, 3, and 4. The mechanism for doing so is that each application of $\p_x$ to the profile terms $\phi^\ast$ produces an extra factor of $x^{-1}$, and similarly each time $\p_x$ hits $z$, one extra factor of $x^{-1}$ is produced by (\ref{basic.id}). As this is similar to the previous steps, we omit the details. We remark that controlling higher order derivatives does not require smallness of the initial datum (hence, the smallness assumption (\ref{intro.BC0.3}) is only for $j = 0,1,2$).

\end{proof}

We now give the following corollaries:

\begin{corollary}[$u^0_p$ Asymptotics] \label{cp1} For any $m \ge 0$, $2 \le p \le \infty$,
\begin{align} \label{main.decay.B}  
&||z^m \partial_y^l \partial_x^k u^0_p||_{L^p_y} \le C(m, l, k) x^{-k-\frac{l}{2}+ \frac{1}{2p}} \text{ for }2k+l > 2,\\ \label{main.decay}  
&||z^m \partial_y^l \partial_x^k u^0_p||_{L^p_y} \le \mathcal{O}(\delta; m, l, k) x^{-k-\frac{l}{2}+ \frac{1}{2p}} \text{ for }2k+l \le 2.
\end{align}
\end{corollary}
\begin{proof}

We start with the representation of $u^0_p$ via (\ref{rep.}):
\begin{align} \label{transfer}
||z^m \partial_\eta^l \partial_x^k u^0_p||_{L^p_\eta} = ||z^m \partial_\eta^l \partial_x^k \Big(\psi + e_\delta - \delta\Big)||_{L^p_\eta} + ||z^m \partial_\eta^l \partial_x^k w||_{L^p_\eta}.  
\end{align}

The first quantity on the right-hand side above is controlled by Corollary \ref{corfront}. For the second quantity on the right-hand side, we simply use (\ref{w.w.2}) with $0 < \sigma_0 < \frac{1}{4}$, coupled with the standard Sobolev interpolation. Recall now: $\eta = \int_0^y 1 + u^0_p(x,\theta)d\theta$, from which we have the equivalence: $(1-\delta)y \le \eta \le (1+\delta)y$. Moreover, we have the Jacobians of the change of coordinates are given by: $\eta'(y) = 1+u^0_p, \hspace{3 mm} y'(\eta) = \frac{1}{1+u^0_p}$, both of which are bounded and bounded away from zero. Therefore, we can transfer (\ref{transfer}) to $(x,y)$ coordinates.

\end{proof}

We may give bounds for $v^0_p$: 
\begin{corollary}[Estimates for $v^0_p$] \label{cp2} We have the following bounds for $v$, for any $m \ge 0$:
\begin{align} \label{main.v.l2.B}
&||z^m \partial_x^k  v^0_p(x, \cdot)||_{L^2_y} \le C(k,m) x^{-k -\frac{1}{4}} \text{ for }k \ge 1, \\ \label{main.v.l2}
&||z^m  v^0_p(x, \cdot)||_{L^2_y} \le \mathcal{O}(\delta; m) x^{-\frac{1}{4}}, \\ \label{main.v.B}
&||z^m \partial_x^k v^0_p(x,\cdot)||_{L^\infty_y}  \le C(k,m) x^{-k-\frac{1}{2}} \text{ for }k \ge 1, \\ \label{main.v}
&||z^m v^0_p(x,\cdot)||_{L^\infty_y}  \le \mathcal{O}(\delta; m) x^{-\frac{1}{2}}.
\end{align}
\end{corollary}
\begin{proof}

First, we give $L^2$ via the Hardy inequality, which is available for $m \ge 0$ as $\partial_x^k v(x,y = \infty) = 0$:
\begin{align} \label{one.one}
||z^m \partial_x^k v^0_p(x,\cdot)||_{L^2_y} \le ||z^m y \partial_x^k v^0_{py}(x,\cdot)||_{L^2_y} = ||z^m y \partial_x^k u^0_{px}(x,\cdot)||_{L^2_y} \lesssim x^{-k-\frac{1}{4}}.
\end{align}

We have used estimate (\ref{main.decay.B}) for $k > 0$, and one obtains the smallness of $\mathcal{O}(\delta)$ for $k = 0$ by applying (\ref{main.decay}). The $L^\infty$ estimate then follows via standard Sobolev interpolation. Indeed, as $\lim_{y \rightarrow \infty} v^0_p(x,y) = 0$, with rapid decay for fixed $x$, one writes: 
\begin{align} \n
|y^m \p_x^k v^0_p(x,y)|^2 &= \int_y^\infty 2 (y^m \p_x^k v^0_p) \p_y (y^m \p_x^k v^0_p) \ud y' \\
& = \int_y^\infty 2 y^m \p_x^k v^0_p m y^{m-1} \p_x^k v^0_p \ud y' + \int_y^\infty 2 (y^m \p_x^k v^0_p) y^m \p_x^k v^0_{py} \ud y'.
\end{align}

Dividing both sides by $x^m$ then gives: 
\begin{align} 
||z^m \partial_x^k v^0_p(x,\cdot)||_{L^\infty_y} &\lesssim ||z^m \partial_x^k v^0_p(x)||_{L^2_y}^{\frac{1}{2}}||z^m \partial_x^k v^0_{py}(x)||_{L^2_y}^{\frac{1}{2}} + x^{-\frac{1}{4}}||z^{m-\frac{1}{2}}\p_x^k v^0_p||_{L^2_y},
\end{align}

from which the desired results follow upon consultation with (\ref{one.one}) and (\ref{main.decay.B}) - (\ref{main.decay}).
\end{proof}

\section{Euler-1 Layer} \label{section.euler}

\subsection{Derivation of Equations}

In this section we construct the next order in the expansion, $[u^1_e, v^1_e, P^1_e]$. The analysis will be taking place in Euler coordinates, that is $(x,Y)$, where $Y = \sqrt{\epsilon}y$ as in (\ref{scaled.variables}). We make the notational convention that differential operators applied to $[u^1_e, v^1_e, P^1_e]$ will always be in $(x,Y)$ coordinates, so for example $\Delta u^1_e := \partial_x^2 u^1_e + \partial_Y^2 u^1_e$. Let us now expand the partial expansion in order to find the lowest-order terms which we then take to solve the Euler-1 equation. The first equation is:  
\begin{align} \nonumber
-\Delta_\epsilon u_s^{(1)} &+ u_s^{(1)} u_{sx}^{(1)} + v_s^{(1)}u_{sy}^{(1)} + P_{sx}^{(1)} = -\Delta_\epsilon u^0_p - \epsilon^{\frac{3}{2}} \Delta u^1_e + \Big(\bar{u}^0_s + \sqrt{\epsilon} u^1_e \Big)\Big( u^0_{px} + \sqrt{\epsilon}u^1_{ex} \Big) \\ &+ \Big(v^0_p + v^1_e \Big)\Big( u^0_{py} + \epsilon u^1_{eY} \Big) + \sqrt{\epsilon} P^1_{ex} + \epsilon P^{1,a}_{ex}.
\end{align}

Removing now the terms found in equation (\ref{eqn.Pr.0.1}), and retaining the lowest-order purely Euler terms gives: 
\begin{align} \label{deri.e.1}
\sqrt{\epsilon}\Big[ u^1_{ex} + P^1_{ex} \Big] = 0. 
\end{align}

The remaining terms from the above expansion are placed into the next order, which is discussed in detail in equations (\ref{deri.pr.1.st}) - (\ref{deri.pr.1.end}). Let us now turn to the second equation:
\begin{align} \nonumber
-\Delta_\epsilon v_s^{(1)} + u_s^{(1)} v_{sx}^{(1)} &+ v_s^{(1)} v_{sy}^{(1)} + \frac{\partial_y}{\epsilon} P_s^{(1)} = -\Delta_\epsilon v^0_p - \epsilon \Delta v^1_e + \Big( \bar{u}^0_s + \sqrt{\epsilon}u^1_e \Big) \Big( v^0_{px} + v^1_{ex} \Big) \\
&+ \Big( v^0_p + v^1_e \Big)\Big(v^0_{py} + \sqrt{\epsilon}v^1_{eY} \Big) + P^1_{eY} + \sqrt{\epsilon}P^{1,a}_{eY}. 
\end{align}

Taking the order-1 purely Eulerian terms gives: 
\begin{align} \label{deri.e.2}
v^1_{ex} + P^1_{eY} = 0.
\end{align}

The remaining terms are contributed to the next-order error in (\ref{deri.pr.1.st}) - (\ref{deri.pr.1.end}). Putting (\ref{deri.e.1}) - (\ref{deri.e.2}) together with the divergence-free condition yields the following system for the Euler-1 layer:
\begin{equation} \label{euler.1.V.0}
u^1_{ex} + P^1_{ex} = 0, \hspace{3 mm} v^1_{ex} + P^1_{eY} = 0, \hspace{3 mm} u^1_{ex} + v^1_{eY} = 0, \text{ in } \Omega. 
\end{equation}

with prescribed boundary data 
\begin{equation} \label{euler.1.V.BC}
v^1_e(x,0) = - v^0_p(x,0).
\end{equation}

Without loss of generality, taking $P^1_e = -u^1_e$, gives the div-curl system 
\begin{equation} \label{div.curl}
 v^1_{ex} -u^1_{eY} = 0, \hspace{3 mm} u^1_{ex} + v^1_{eY} = 0. 
\end{equation}

Notice that these equations are the Cauchy-Riemann equations for $v^1_e = Re(f)$, $u^1_e = Im(f)$, $f$ holomorphic on $\Omega$.  Then by taking the curl of the equation,
\begin{equation} \label{euler.1.V}
\Delta v^1_e = 0 \text{ in } \Omega, \hspace{3 mm} v^1_{e}(x,0) = -v^0_{p}(x,0). 
\end{equation}

The following boundary decay rates follow from (\ref{main.v.B}) - (\ref{main.v}):
\begin{equation} \label{e.trace}
\Big| \partial_x^k v^0_p(x,0) \Big| \le C(k) x^{-\frac{1}{2} - k}, \hspace{3 mm} \Big| v^0_p(x,0) \Big| \le \mathcal{O}(\delta) x^{-\frac{1}{2}}
\end{equation}

\subsection{Uniform Decay Estimates}

Motivated by (\ref{e.trace}), let us define for this section $f := \chi_{x \ge -10} E_{\mathbb{R}_+}^{(m)} v^0_p$, where $E_{\mathbb{R}_+}^{(m)}$ is the $m$'th order Sobolev extension operator on the half-line $\mathbb{R}_+$ (see \cite[Chapter 5]{Adams}). Here, $m$ is selected as large as required by the remainder of our analysis. Then, $f$ agrees with $v^0_p$ on $[0,\infty)$, and is cut-off past $x \le -10$. Then by standard Sobolev extension theory, $||f||_{H^m} \le C(m) ||v^0_p||_{H^m}$. Our setup for this subsection is: 
\begin{align}
\Delta v = 0 \text{ on } \mathbb{H}, \hspace{5 mm} v(x,0) = f(x), \hspace{3 mm} \Big| \partial_x^k f(x) \Big| \le x^{-k-\frac{1}{2}},  \hspace{3 mm} \Big| f(x) \Big| \le \mathcal{O}(\delta) x^{-\frac{1}{2}}.
\end{align}

Using the Poisson Kernel, we have: 
\begin{equation} \label{defn.PE}
v(x, Y) = P_Y \ast f = \int_{-\infty}^\infty P_Y(x - t) f(t) dt = \int_{-\infty}^\infty \frac{Y}{Y^2 + (x-t)^2} f(t) dt. 
\end{equation}

We will need the following pointwise estimates on the Poisson Kernel:
\begin{lemma} \label{Lemma.pw.PY} For $x > 0$, $k \ge 0$, and $0 \le s \le k$, we have: 
\begin{align} \label{elem}
|\p_x^k P_Y(x)| \lesssim x^{-s-1}Y^{-(k-s)}. 
\end{align}
\end{lemma}
\begin{proof}
First, let us address the $k = 0$ case. Indeed, 
\begin{align} 
x P_Y(x) &= \frac{Yx}{Y^2 + x^2} \le \frac{Y^2 + x^2}{Y^2 + x^2} \le 1.
\end{align}

For $k \ge 1$, we record the basic identities: 
\begin{align} \label{stp.even}
\p_x^k P_Y(x) = \sum_{j=0}^{\frac{k}{2}} c_{j,k} Yx^{2j}  (Y^2 + x^2)^{-(\frac{k}{2}+1)-j} \text{ if $k$ even,} \\ \label{stp.odd}
\p_x^k P_Y(x) = \sum_{j=0}^{\frac{k-1}{2}} c_{j,k} Yx^{2j+1} (Y^2 + x^2)^{-\frac{k+1}{2}-1-j} \text{ if $k$ odd.}
\end{align}

For the $k$ even case, we multiply (\ref{stp.even}) by $x^{k+1}$, use the binomial formula, and then apply Young's inequality for products in the following manner: 
\begin{align} \n
|x^{s+1} Y^{k-s} \p_x^k P_Y(x)| &= | \sum_{j=0}^{\frac{k}{2}} c_{j,k} Y^{1+k-s}x^{2j + s + 1}  (Y^2 + x^2)^{-(\frac{k}{2}+1)-j}| \\ \n
& \lesssim  \sum_{j=0}^{\frac{k}{2}} c_{j,k} \frac{Y^{1+k-s}x^{2j + s + 1}}{Y^{k+2+2j} + x^{k+2+2j}} \\
& \lesssim \sum_{j = 0}^{\frac{k}{2}} \frac{Y^{(1+k-s)(\frac{k+2+2j}{1+k-s})} + x^{(s+1+2j)(\frac{k+2+2j}{s+1+2j})}}{Y^{k+2+2j} + x^{k+2+2j}} \lesssim 1,
\end{align}

where the Young's conjugates are: 
\begin{align}
1 = \frac{1+k-s}{k+2+2j} + \frac{s+1+2j}{k+2+2j}. 
\end{align}

Performing the same calculation using (\ref{stp.odd}), we have: 
\begin{align} \n
|x^{s+1} Y^{k-s} \p_x^k P_Y(x)| &= |\sum_{j=0}^{\frac{k-1}{2}} c_{j,k} Y^{1+k-s}x^{s+2j+2} (Y^2 + x^2)^{-\frac{k+1}{2}-1-j}| \\ \n
& \lesssim \sum_{j=0}^{\frac{k-1}{2}} c_{j,k} \frac{Y^{1+k-s}x^{s+2j+2}}{ Y^{k+3+ 2j} + x^{k+3+2j}} \\
& \lesssim \sum_{j=0}^{\frac{k-1}{2}} c_{j,k} \frac{Y^{(1+k-s)(\frac{k+2j+3}{1+k-s})} + x^{(s+2j+2)(\frac{k+2j+3}{s+2j+2})}}{ Y^{k+3+ 2j} + x^{k+3+2j}} \lesssim 1. 
\end{align}

Here, the Young's conjugates are: 
\begin{align}
1 = \frac{1+k-s}{k+2j+3} + \frac{s+2j+2}{k+2j+3}.
\end{align}

This proves (\ref{elem}).

\end{proof}

We now prove the following uniform estimates for the $v^1_e(x, Y)$, which now drop superscripts and denote by $v$ for notational ease: 
\begin{proposition} \label{Lemma.Sing.Unif}
Let $v$ be the Poisson extension defined via (\ref{defn.PE}). Then $v$ satisfies the following bounds: 
\begin{equation} \label{v.cases}
\sup_{Y \ge 0} \Big| \partial_{x}^kv(x, Y) \Big| \le C(k) x^{-k - \frac{1}{2}}, \hspace{3 mm} \sup_{Y \ge 0} \Big| v(x, Y) \Big| \le \mathcal{O}(\delta) x^{- \frac{1}{2}}.
\end{equation}
\end{proposition}

\begin{proof}

The $Y = 0$ case is clear by (\ref{e.trace}), and so we must treat the case when $Y > 0$. In this regime, several \textit{qualitative} facts are available for use. In particular, $P_Y$ is smooth, has unit mass, and is positive everywhere. As is standard, we shall exploit these qualitative facts in order to extract quantitative bounds independent of $Y > 0$. Fixing any $\alpha > 0$, we shall split the integral (\ref{defn.PE}) into three pieces:
\begin{align} \nonumber
v(x,Y) &= \int_{-\infty}^{-\frac{x}{1+\alpha}} P_Y(x-t)f(t) dt + \int_{-\frac{x}{1+\alpha}}^{\frac{x}{1+\alpha}} P_Y(x-t)f(t) dt + \int_{\frac{x}{1+\alpha}}^\infty P_Y(x-t) f(t) dt \\ \label{Im}
& = I_1 + I_2 + I_3.  
\end{align}

For the sake of concreteness, fix: $\alpha = \frac{1}{10,000}$. $\delta$ and $\eps$ will be taken small relative to this universal constant, $\alpha$. Let us first turn to $I_3$: 
\begin{align} 
x^{\frac{1}{2}} \Big| I_3 \Big| &= \Big| \int_{\frac{x}{1+\alpha}}^\infty P_Y(x-t) x^{\frac{1}{2}} f(t) dt \Big| \\
& \le \Big| \int_{\frac{x}{1+\alpha}}^\infty P_Y(x-t) \{ x^{\frac{1}{2}} - |t|^{\frac{1}{2}} \} f(t) dt \Big| + \Big| \int_{\frac{x}{1+\alpha}}^\infty P_Y(x-t) |t|^{\frac{1}{2}} f(t) dt \Big| \\
& \le \int_{\frac{x}{1+\alpha}}^\infty P_Y(x-t) \{ \Big(\frac{x}{|t|} \Big)^{\frac{1}{2}} - 1 \} |t|^{\frac{1}{2}}f(t) dt + ||f x^{\frac{1}{2}}||_{L^\infty} \int_{-\infty}^\infty P_Y(x-t) dt \\ \label{region.B}
& \le \sup_{t \ge \frac{x}{1+\alpha}} \Big| \Big( \frac{x}{|t|} \Big)^{\frac{1}{2}} - 1 \Big| ||fx^{\frac{1}{2}}||_{L^\infty} \int_{-\infty}^\infty P_Y(x-t) dt  + \mathcal{O}(\delta) \le \mathcal{O}(\delta).
\end{align}

where we have used for each fixed $Y > 0$, $P_Y \ge 0$ and $\int P_Y = 1$. By symmetry, $I_1$ works the same way. Therefore, we are left with $I_2$, where the main mechanism are pointwise estimates on $P_Y$. According to (\ref{elem}), with $k = s = 0$,  
\begin{equation}
\sup_{ -\frac{x}{1+\alpha} \le t \le \frac{x}{1+\alpha} }P_Y(x-t) \lesssim \frac{1}{x},
\end{equation}

as $x - t > 0$ int the region $-\frac{x}{1+\alpha} \le t \le \frac{x}{1+\alpha}$. Thus, we have: 
\begin{align} \nonumber
\Big| I_2 \Big| &\le \int_{-\frac{x}{1+\alpha}}^{\frac{x}{1+\alpha}} P_Y(x-t) |f(t)| dt \lesssim \frac{1}{x} \int_{-\frac{x}{1+\alpha}}^{\frac{x}{1+\alpha}} |f(t)| dt \le \frac{\mathcal{O}(\delta)}{x} \int_{-\frac{x}{1+\alpha}}^{\frac{x}{1+\alpha}} |1+t|^{-\frac{1}{2}}dt \\ \label{region.A}
& \le \mathcal{O}(\delta) x^{-\frac{1}{2}}.
\end{align}

Combining (\ref{region.B}) - (\ref{region.A}), we may conclude that 
\begin{equation}
\Big| x^{\frac{1}{2}}v \Big| \le \mathcal{O}(\delta). 
\end{equation}

We now need to establish this procedure for higher-order $x$ derivatives for $v$. We continue with the $I_1, I_2, I_3$ splitting used above. Upon differentiating (\ref{Im}) with respect to $x$, let us treat each term $I_i$ individually. 
\begin{align} \nonumber
\partial_{x} I_1 &= \frac{-1}{1+\alpha}P_Y(x + \frac{x}{1+\alpha})f(-\frac{x}{1+\alpha}) + \int_{-\infty}^{-\frac{x}{1+\alpha}} \partial_{x} P_Y(x - t) f(t) dt \\ \nonumber
&= \frac{-1}{1+\alpha}P_Y(x + \frac{x}{1+\alpha})f(-\frac{x}{1+\alpha})  - \int_{-\infty}^{-\frac{x}{1+\alpha}} \partial_t P_Y(x - t) f(t) dt \\ \label{stp.1}
&=  B_1 + \int_{-\infty}^{-\frac{x}{1+\alpha}}P_Y(x - t) f'(t) dt.
\end{align}

Here, 
\begin{align}
B_1 = \frac{-1}{1+\alpha}P_Y(x + \frac{x}{1+\alpha})f(-\frac{x}{1+\alpha}) - P_Y(x + \frac{x}{1+\alpha})f(\frac{-x}{1+\alpha}).
\end{align}

Note now in a similar manner to (\ref{elem}), 
\begin{align}
\Big| B_1 \Big| \lesssim P_Y(\frac{2+\alpha}{1+\alpha} x)|f(-\frac{x}{1+\alpha})| \lesssim \frac{1}{ x} \mathcal{O}(\delta) x^{-\frac{1}{2}} \lesssim x^{-\frac{3}{2}}. 
\end{align}

Using the same method as evaluating (\ref{region.B}): 
\begin{align} \nonumber
x^{\frac{3}{2}} \Big|& \int_{-\infty}^{-\frac{x}{1+\alpha}} P_Y(x - t) f'(t) dt \Big| \\ \nonumber
& \le \int_{-\infty}^{-\frac{x}{1+\alpha}} P_Y(x - t) |t|^{\frac{3}{2}} \Big| f'(t) \Big| dt + \int_{-\infty}^{-\frac{x}{1+\alpha}} P_Y(x - t) \Big| \{ x^{\frac{3}{2}} - |t|^{\frac{3}{2}} \} \Big| \Big|f'(t)\Big| dt \\ \nonumber
& \le \int_{-\infty}^\infty P_Y(x - t) |t|^{\frac{3}{2}} \Big| f'(t) \Big| dt + \int_{-\infty}^{-\frac{x}{1+\alpha}} P_Y(x - t) \Big| \{ x^{\frac{3}{2}} - |t|^{\frac{3}{2}} \}\Big| \Big|f'(t)\Big| dt \\ \nonumber 
& \le ||x^{\frac{3}{2}}f'||_{L^\infty} + \int_{-\infty}^{-\frac{x}{1+\alpha}} P_Y(x-t) \Big| \{ \Big( \frac{x}{|t|} \Big)^{\frac{3}{2}} - 1 \} \Big| |t|^{\frac{3}{2}} \Big| f'(t) \Big| dt \\ \label{regBplus}
& \lesssim ||x^{\frac{3}{2}}f'||_{L^\infty} + \sup_{t \le -\frac{x}{1+\alpha}}  \Big| \{ \Big( \frac{x}{|t|} \Big)^{\frac{3}{2}} - 1 \} \Big| || x^{\frac{3}{2}} f' ||_{L^\infty} \int_{-\infty}^\infty P_Y(x - t) dt \le C.
\end{align}

Next, we turn to $I_2$, which after differentiating gives:
\begin{align} \nonumber
\partial_{x} I_2 &= \frac{1}{1+\alpha}P_Y(x - \frac{x}{1+\alpha})f(\frac{x}{1+\alpha}) + \frac{1}{1+\alpha}  P_Y(x + \frac{x}{1+\alpha})f(-\frac{x}{1+\alpha}) \\ \label{stp.3}
& + \int_{-\frac{x}{1+\alpha}}^{\frac{x}{1+\alpha}} \partial_{x} P_Y(x - t) f(t) dt.
\end{align}

Via (\ref{elem}) with $k = 1, s= 0$, 
\begin{align}
\sup_{Y > 0, t \in [-\frac{x}{1+\alpha}, \frac{x}{1+\alpha}]} \Big| \partial_{x} P_Y(x - t) \Big| \le \sup_{Y > 0, t \in [-\frac{x}{1+\alpha}, \frac{x}{1+\alpha}]}  \frac{1}{(x - t)^2} \lesssim \frac{1}{x^2}. 
\end{align}

For $I_2$, this then facilitates the bound: 
\begin{align}
\Big| \int_{-\frac{x}{1+\alpha}}^{\frac{x}{1+\alpha}} \partial_{x}P_Y(x - t) f(t) dt \Big| \lesssim \frac{1}{x^2} \int_{-\frac{x}{1+\alpha}}^{\frac{x}{1+\alpha}} (1+t)^{-\frac{1}{2}} dt \lesssim x^{-\frac{3}{2}}. 
\end{align}

For the $\p_x I_3$ term, we must integrate by parts in the following manner: 
\begin{align} \nonumber
\partial_{x} I_3 &= \frac{1}{1+\alpha}P(x - \frac{x}{1+\alpha})  f(\frac{x}{1+\alpha}) + \int_{\frac{x}{1+\alpha}}^\infty \partial_{x} P_Y(x - t) f(t) dt \\ \label{e.sing.i3}
& = B_3 + \int_{\frac{x}{1+\alpha}}^\infty P_Y(x - t) f'(t) dt.
\end{align}

Here, $B_3$ contains the boundary terms: 
\begin{align} \n
B_3 &= \frac{1}{1+\alpha}P_Y(x - \frac{x}{1+\alpha})  f(\frac{x}{1+\alpha}) - P_Y(x - \frac{x}{1+\alpha})f(\frac{x}{1+\alpha}) \\
& = \frac{1}{1+\alpha}P_Y(\frac{\alpha}{1+\alpha}x)   f(\frac{x}{1+\alpha}) - P_Y(\frac{\alpha}{1+\alpha}x) f(\frac{x}{1+\alpha}). 
\end{align}

Again, via direct calculation using the definition of $P_Y$, it is clear that: 
\begin{equation}
\Big| B_3 \Big| \lesssim x^{-\frac{3}{2}}. 
\end{equation}

We may estimate the integral in (\ref{e.sing.i3}) in identical fashion to (\ref{regBplus}). This procedure can be iterated for higher-order derivatives. To see this, let us start with the splitting given in (\ref{Im}). We will apply $\p_x^k$, where $k \ge 2$: 
\begin{align}
\p_x^k v(x,Y) = \p_x^k I_1 + \p_x^k I_2 + \p_x^k I_3. 
\end{align}

Our starting point is the expression (\ref{stp.1}), from which it becomes clear that applying $\p_x^k$ to $I_1$, for generic constants $c_\alpha > 0$ (which can change from term to term), we have: 
\begin{align} \label{stp.2}
\p_x^k I_1 = \sum_{j=0}^{k-1} c_j \p_x^j P_Y(c_\alpha x) \p_x^{k-1-j} f(c_\alpha x) + \int_{-\infty}^{-\frac{x}{1+\alpha}} P_Y(x-t) \p_t^k f(t) dt. 
\end{align}

Using (\ref{elem}), we estimate: 
\begin{align} \label{stp.4}
|\p_x^j P_Y(c_\alpha x) \p_x^{k-1-j} f(c_\alpha x)| \le  c_\alpha x^{-j-1} x^{-(k-1-j) - \frac{1}{2}} \lesssim x^{-k-\frac{1}{2}}. 
\end{align}

For the integration term in (\ref{stp.2}), we follow (\ref{regBplus}) to give: 
\begin{align}
x^{k+\frac{1}{2}} \Big|& \int_{-\infty}^{-\frac{x}{1+\alpha}} P_Y(x - t) \p_t^k f(t) dt \Big| \\ \nonumber
& \le \int_{-\infty}^{-\frac{x}{1+\alpha}} P_Y(x - t) |t|^{k+\frac{1}{2}} \Big| \p_t^k f(t) \Big| dt + \int_{-\infty}^{-\frac{x}{1+\alpha}} P_Y(x - t) \Big| \{ x^{k+\frac{1}{2}} - |t|^{k+\frac{1}{2}} \} \Big| \Big|\p_t^k f(t)\Big| dt \\ \nonumber
& \le \int_{-\infty}^\infty P_Y(x - t) |t|^{k+\frac{1}{2}} \Big| \p_t^k f(t) \Big| dt + \int_{-\infty}^{-\frac{x}{1+\alpha}} P_Y(x - t) \Big| \{ x^{k+\frac{1}{2}} - |t|^{k+\frac{1}{2}} \}\Big| \Big|\p_t^k f(t)\Big| dt \\ \nonumber 
& \le ||x^{k+\frac{1}{2}}\p_x^k f||_{L^\infty} + \int_{-\infty}^{-\frac{x}{1+\alpha}} P_Y(x-t) \Big| \{ \Big( \frac{x}{|t|} \Big)^{k+\frac{1}{2}} - 1 \} \Big| |t|^{k+\frac{1}{2}} \Big| \p_t^k f(t) \Big| dt \\ \label{stp.5}
& \lesssim ||x^{k+\frac{1}{2}}\p_x^k f||_{L^\infty} + \sup_{t \le -\frac{x}{1+\alpha}}  \Big| \{ \Big( \frac{x}{|t|} \Big)^{k+\frac{1}{2}} - 1 \} \Big| || x^{k+\frac{1}{2}} \p_x^k f ||_{L^\infty} \int_{-\infty}^\infty P_Y(x - t) dt \le C.
\end{align}

By symmetry, the same estimate applies to $I_3$, and so we turn to $I_2$. Referring to (\ref{stp.3}), one sees the following expression: 
\begin{align}
\p_x^k I_2 = \sum_{j=0}^{k-1} c_j \p_x^j P_Y(c_\alpha x) \p_x^{k-1-j} f(c_\alpha x) + \int_{-\frac{x}{1+\alpha}}^{\frac{x}{1+\alpha}} \p_x^k P_Y(x-t) f(t) \ud t,
\end{align}

where $c_\alpha > 0$ denotes generic constants which depend on $\alpha$, not on $\delta, \eps$, and are strictly positive. Referring to (\ref{stp.4}), it remains to estimate the integral term above: 
\begin{align} \n
|\int_{-\frac{x}{1+\alpha}}^{\frac{x}{1+\alpha}} \p_x^k P_Y(x-t) f(t) \ud t| &\le \int_{-\frac{x}{1+\alpha}}^{\frac{x}{1+\alpha}} |\p_x^k P_Y(x-t)| |f(t)| \ud t \\ \n
& \lesssim \sup_{-\frac{x}{1+\alpha} \le t \le \frac{x}{1+\alpha}} |\p_x^k P_Y(x-t)| \int_{-\frac{x}{1+\alpha}}^{\frac{x}{1+\alpha}} |(1+t)|^{\frac{1}{2}} \ud t \\
& \lesssim x^{-k-1} x^{\frac{1}{2}} = x^{-k - \frac{1}{2}}. 
\end{align}

This then proves the desired result. 

\end{proof}

By repeating the above proof, we actually have the following: 
\begin{corollary} \label{cor.w} Consider boundary values $g(x)$ satisfying: 
\begin{align}
\Big| \partial_x^k g(x) \Big| \le x^{-k-w}, \text{ where } w \in (0,1).
\end{align}
Then the Poisson extension satisfies:
\begin{align} \label{w.weight}
\sup_Y \Big| \partial_x^k \{ P_Y \ast g \} \Big| \lesssim x^{-k-w}. 
\end{align}
\end{corollary}

We now give an estimate more suitable to obtaining $L^2$ in $Y$ estimates.  
\begin{lemma} \label{lem.e.ss} For any $k \ge 1$, and $0 \le k-s \le 1$, we have:
\begin{align} \label{e.ss.1}
\Big| \partial_{x}^k v(x, Y)\Big| \lesssim x^{-s-\frac{1}{2}} Y^{-(k-s)}.
\end{align}
\end{lemma}
\begin{proof}

By Lemma \ref{Lemma.Sing.Unif}, the claim is only meaningful for $Y \ge x$, so make this assumption to begin. Again, we start with the splitting (\ref{Im}). By symmetry it suffices to consider $I_2, I_3$. Applying $\partial_x^k$, we have the expression (for generic constants $c_{\alpha} > 0$):
\begin{align} \label{I2.1}
\partial_x^k I_2 = \sum_{j = 0}^{k-1} c_{j,k,\alpha} \p_x^j P_Y(c_\alpha x) \p_x^{k-1-j} f(c_\alpha x) + \int_{-\frac{x}{1+\alpha}}^{\frac{x}{1+\alpha}} \p_x^k P_Y(x - t) f(t) dt, 
\end{align}

First, 
\begin{align} \n
|\p_x^j P_Y(c_\alpha x) \p_x^{k-1-j}f(c_\alpha x)| &\lesssim  Y^{-(k-s)} x^{-j-1+(k-s)} x^{-(k-1-j)-\frac{1}{2}} \\ \label{stp.7}
& \lesssim Y^{-(k-s)} x^{-s-\frac{1}{2}}.
\end{align}

Above, we have used (\ref{elem}) for the case $j \ge 1$. For the $j = 0$ case we use the assumption that $Y \ge x$, and the following estimate on $P_Y$, which is valid so long as $k-s \le 1$:
\begin{equation}
P_Y(x - t) = \frac{Y}{Y^2 + (x - t)^2} \le \frac{Y}{Y^2} = \frac{1}{Y} \le Y^{-(k-s)} x^{-1+(k-s)}.
\end{equation}

For the integral term in (\ref{I2.1}), via (\ref{elem}), 
\begin{align} \label{maxim.1}
\sup_{\substack{Y \ge 0 \\ t \in [-\frac{x}{1+\alpha}, \frac{x}{1+\alpha}]}} \Big| \partial_{x}^{k}P_Y(x - t) \Big|  \lesssim  Y^{-(k-s)} x^{-s-1}. 
\end{align}

From here, the estimate on the integral in (\ref{I2.1}) follows: 
\begin{align}
\Big| \int_{-\frac{x}{1+\alpha}}^{\frac{x}{1+\alpha}} \partial_{x} P_Y(x - t) f(t) dt \Big| \lesssim Y^{-(k-s)}x^{-s-1} \int_{-\frac{x}{1+\alpha}}^{\frac{x}{1+\alpha}} (1 + t)^{-\frac{1}{2}}dt \lesssim Y^{-(k-s)}x^{-s-\frac{1}{2}}. 
\end{align}

For the $I_3$ contribution, we apply $\partial_{x}^k$ and integrate by parts in $t$, reads: 
\begin{align} \label{dxi3}
\partial_{x} I_3 &= \sum_{j=0}^{k-1} c_{j,k,\alpha} \p_x^j P_Y(c_\alpha x) \p_x^{k-1-j} f(c_\alpha x) + \int_{\frac{x}{1+\alpha}}^\infty P_Y(x - t) \p_t^k f(t) dt. 
\end{align}

First, as shown in (\ref{stp.7})
\begin{align}
|\p_x^j P_Y(c_\alpha x) \p_x^{k-1-j} f(c_\alpha x)| \lesssim  Y^{-(k-s)} x^{-s-\frac{1}{2}}.
\end{align}

Turning to the integration in (\ref{dxi3}), our point of view on the kernel: 
\begin{equation}
P_Y(x - t) = \frac{Y}{Y^2 + (x - t)^2} \le \frac{Y}{Y^2} = \frac{1}{Y}. 
\end{equation}

Then in the regime $Y \ge x$, we have (since $k-s \le 1$): 
\begin{align} \n
\Big| \int_{\frac{x}{1+\alpha}}^\infty P_Y(x - t) \p_t^k f(t) dt \Big| &\le \frac{1}{Y} \int_{\frac{x}{1+\alpha}}^\infty t^{-k-\frac{1}{2}} dt \le \frac{1}{Y}x^{-k+\frac{1}{2}} \\  \label{tradeoff.Euler}
&\le \frac{1}{Y}x^{-k+\frac{1}{2}} \Big( \frac{Y}{x} \Big)^{1-(k-s)}  \le Y^{-(k-s)}x^{-s-\frac{1}{2}}.
\end{align}

This establishes the desired result.  
\end{proof}

\begin{lemma}[Interpolation Estimates] \label{interp.Euler} We have the following decay estimates: 
\begin{equation}
x^{k + \frac{1}{2}} ||\partial_Y^k v(x, \cdot)||_{L^\infty_Y} \le C, \hspace{3 mm} x^{ \frac{3}{2}} ||\partial_Y v(x, \cdot)||_{L^\infty_Y} \le \mathcal{O}(\delta). 
\end{equation}
\end{lemma}
\begin{proof}
We treat the case when $k = 1$, with the extension for higher $k$ values being obvious. This follows via standard interpolation arguments: fix an $x$, and consider $v_Y(x, \cdot)$ as a function of the $Y$ variable, defined on the half-space $Y \ge 0$. Then via Gagliardo-Nirenberg interpolation estimates, and via harmonicity $v_{xx} = -v_{YY}$,  
\begin{equation} \label{GNS.1}
x^{\frac{3}{2}}||v_Y(x, \cdot)||_{L^\infty_Y} \lesssim \Big(x^{\frac{1}{2}} ||v(x, \cdot)||_{L^\infty_Y} \Big)^{\frac{1}{2}}\Big(x^{\frac{5}{2}} ||v_{xx}(x,\cdot)||_{L^\infty_Y} \Big)^{\frac{1}{2}}.
\end{equation}

Importantly, via Remark 4, Page 125 in \cite{GNS1}, we need not include a lower-order term on the right-hand side of (\ref{GNS.1}), which is due to the fact that our domain, for each fixed $x$, is the half-space $Y \ge 0$. Taking the supremum over $x$ and applying (\ref{v.cases}) yields the desired result. The smallness of $\mathcal{O}(\delta)$ for $k = 1$ case is guaranteed due to the $v$ term in (\ref{GNS.1}). 

\end{proof}

\begin{proposition}[Euler Correctors]  \label{thm.euler.1} Suppose $[u^1_e, v^1_e]$ solves the boundary value problem: 
\begin{align}
\Delta v^1_e = 0, \hspace{3 mm} v^1_e(x,0) = -v^0_p(x,0), \hspace{3 mm} u^1_e = \int_x^\infty v^1_{eY}(x', Y) dx', \text{ for } x \ge 1.
\end{align}
Then the following bounds holds for any $k,j \ge 0$: 
\begin{align} \label{euler.DECAY.main}
&\sup_{Y \ge 0} \Big| \partial_x^k \partial_Y^j v^1_e \Big|x^{k + j + \frac{1}{2}} \le C,  \hspace{5 mm} \sup_{Y \ge 0} \Big[ \Big| u^1_e, v^1_e \Big|x^{\frac{1}{2}} + \Big| u^1_{ex}, v^1_{eY} \Big|x^{\frac{3}{2}} \Big]  \le \mathcal{O}(\delta).
\end{align}
\end{proposition}
\begin{proof}

We shall take $v^1_e = v$ from the above lemmas.  For the $u^1_e$ profile, we define for $x \ge 1$: 
\begin{align}
u^1_e(x,Y) := \int_x^\infty v^1_{eY}(x', Y) dx'.
\end{align}

From here it is clear that the Cauchy-Riemann equations hold: 
\begin{align}
u^1_{ex} = - v^1_{eY}, \hspace{3 mm} u^1_{eY} = \int_x^\infty v^1_{eYY} = - \int_x^\infty v^1_{exx} = v^1_{ex}. 
\end{align}

We have used that $v^1_{ex}(x, Y) \rightarrow 0$ as $x \rightarrow \infty$. From Lemma \ref{interp.Euler}, it is then clear that:
\begin{align} \label{int.hilbert}
 \Big| u^1_{e} \Big| = \Big| \int_x^\infty u^1_{ex} dx' \Big| = \Big| \int_x^\infty v^1_{eY} dx' \Big| \le \mathcal{O}(\delta) \int_x^\infty |x'|^{-\frac{3}{2}} dx' \le \mathcal{O}(\delta) x^{-\frac{1}{2}}. 
\end{align}

\end{proof}

\begin{remark}[In-flow Conditions] Note that we do not solve for the Euler correctors given an arbitrary in-flow condition at $x = 1$. Rather, we take $[u^1_e, v^1_e, P^1_e]$ to be the explicit profiles obtained by applying the Poisson extension to $f$. In this sense, the set-up we consider here is distinct from the construction of Euler profiles in \cite{GN}. 
\end{remark}

\section{Prandtl Layer 1}  \label{section.prandtl.1}

\subsection{Derivation of Linearized Prandtl Equations:}

In this section, we will construct the Prandtl correctors, $[u^1_p, v^1_p, P^1_p]$. Let us now obtain the equations that they will satisfy. We do this in a manner which can be easily generalized in the next section. We expand the nonlinear terms: 
\begin{align} \nonumber
\bar{u}_s^{(1)} \bar{u}^{(1)}_{sx} &= \Big(\bar{u}_s^{(0)} + \sqrt{\epsilon}u^1_e + \sqrt{\epsilon}u^1_p \Big) \Big(\bar{u}_{sx}^{(0)} + \sqrt{\epsilon}u^1_{ex} + \sqrt{\epsilon}u^1_{px} \Big) \\ \nonumber
& = \bar{u}_s^{(0)} \bar{u}_{sx}^{(0)} + \sqrt{\epsilon} u^{(1)}_{sx} u^1_p + \sqrt{\epsilon} u^{(1)}_s u^1_{px} + \epsilon u^1_p u^1_{px} + \sqrt{\epsilon}u^1_{ex}u^0_p + \sqrt{\epsilon}u^1_e u^0_{px} \\ \label{deri.pr.1.st} 
& + \sqrt{\epsilon}u^1_{ex} + \epsilon u^1_e u^1_{ex}, \\ \nonumber
\bar{v}_s^{(1)} \bar{u}_{sy}^{(1)} & = \Big(v^0_p + v^1_e + \sqrt{\epsilon}v^1_p \Big)\Big( u^0_{py} + \epsilon u^1_{eY} + \sqrt{\epsilon} u^1_{py} \Big) \\
& = v_s^{(1)} u^0_{py} + \sqrt{\epsilon} u_{sy}^{(1)}v^1_p + \sqrt{\epsilon} v_s^{(1)} u^1_{py} + \epsilon v^1_p u^1_{py}  + \epsilon v^0_p u^1_{eY} + \epsilon v^1_e u^1_{eY}, \\ \nonumber
\bar{u}^{(1)}_s \bar{v}^{(1)}_{sx} &= \Big( 1 + u^0_p + \sqrt{\epsilon}u^1_e + \sqrt{\epsilon}u^1_p \Big) \Big( v^0_{px} + v^1_{ex} + \sqrt{\epsilon}v^1_{px} \Big) \\ \nonumber
& = \bar{u}_s^{(0)} v^0_{px} + \sqrt{\epsilon} v^1_{px} u_s^{(1)} + \sqrt{\epsilon} v^{(1)}_{sx} u^1_p + \epsilon u^1_p v^1_{px} + u^0_p v^1_{ex} + \sqrt{\epsilon}u^1_e v^0_{px} \\ 
& + v^1_{ex} + \sqrt{\epsilon}u^1_e v^1_{ex}, \\ \nonumber
\bar{v}^{(1)}_s \bar{v}^{(1)}_{sy} &= \Big( v^0_p + v^1_e + \sqrt{\epsilon}v^1_p \Big) \Big( v^0_{py} + \sqrt{\epsilon} v^1_{eY} + \sqrt{\epsilon}v^1_{py} \Big) \\
& = v^0_p v^0_{py} + \sqrt{\epsilon} v^1_{py} v^{(1)}_s + \sqrt{\epsilon} v^{(1)}_{sy} v^1_p + \epsilon v^1_p v^1_{py} + \sqrt{\epsilon} v^0_p v^1_{eY} + v^1_e v^0_{py} + \sqrt{\epsilon} v^1_e v^1_{eY}. 
\end{align}

Let us now denote: 
\begin{align} \label{Ru1E}
R^{u,1}_E &:= v^1_e u^0_{py} + \sqrt{\epsilon} u^1_{ex} u^0_p + \sqrt{\epsilon} u^1_e u^0_{px} + \epsilon v^0_p u^1_{eY}, \hspace{3 mm} E^{u,1}_E :=   \epsilon v^1_e u^1_{eY} + \epsilon u^1_e u^1_{ex}, \\  \label{Rv1E}
R^{v,1}_E &= u^0_p v^1_{ex} + \sqrt{\epsilon} u^1_e v^0_{px} + \sqrt{\epsilon} v^0_p v^1_{eY} + v^1_e v^0_{py}, \hspace{3 mm} E^{v,1}_E :=  \sqrt{\epsilon} u^1_e v^1_{ex} + \sqrt{\epsilon} v^1_e v^1_{eY}.
\end{align}

The purpose of the terms above is to separate the Euler-Prandtl terms and the pure-Euler terms. The pure-Euler terms are harmful to our analysis, because they scale differently than the Euler-Prandtl terms. From a practical point of view, their presence prevents the application of weights of the form $z = \frac{y}{\sqrt{x}}$, and therefore obstructs self-similarity. It turns out that all pure-Euler terms are of ``gradient-type", see (\ref{grad.pres.1}). Thus, by introducing appropriate potential functions in the pressure expansion, we can force these terms to vanish identically. With this notation, we may write the Navier-Stokes expansion as: 
\begin{align} \nonumber
-\Delta_\epsilon \bar{u}^{(1)}_s &+ \bar{u}^{(1)}_s \bar{u}^{(1)}_{sx} + \bar{v}^{(1)}_s \bar{u}^{(1)}_{sy} + \bar{P}^{(1)}_x = -\Delta_\epsilon \bar{u}^{(0)}_s + \bar{u}^{(0)}_s \bar{u}^{(0)}_{sx} + \bar{v}^{(0)}_s \bar{u}^{(0)}_{sy} + \bar{P}^{(0)}_x \\ \nonumber
& + \sqrt{\epsilon}\Big[ -\Delta_\epsilon u^1_p + u_s^{(1)} u^1_{px} + u^{(1)}_{sx} u^1_p + u^{(1)}_{sy} v^1_p + v^{(1)}_s u^1_{py} + \sqrt{\epsilon} u^1_p u^1_{px} \\ \label{deri.pr.1.e.1}
& + \sqrt{\epsilon} v^1_p u^1_{py} + P^1_{px} \Big] + \epsilon P^{1,a}_{px} + R^{u,1}_E + E^{u,1}_E - \epsilon^{\frac{3}{2}} \Delta u^1_e + \epsilon P^{1,a}_{Ex} + \sqrt{\epsilon} (u^1_{ex} + P^1_{ex}).
\end{align}

The normal equation is expanded via: 
\begin{align} \nonumber
-\Delta_\epsilon \bar{v}^{(1)}_s &+ \bar{u}^{(1)}_s \bar{v}^{(1)}_{sx} + \bar{v}^{(1)}_s \bar{v}^{(1)}_{sy} + \frac{\bar{P}^{(1)}_y}{\epsilon} = -\Delta_\epsilon \bar{v}^{(0)}_s + \bar{u}^{(0)}_s \bar{v}^{(0)}_{sx} + \bar{v}^{(0)}_s \bar{v}^{(0)}_{sy} + \frac{\bar{P}^{(0)}_y}{\epsilon} \\ \nonumber
 & + \sqrt{\epsilon}\Big[-\Delta_\epsilon v^1_p + u_s^{(1)} v^1_{px} + v^{(1)}_{sx} u^1_p + v_s^{(1)} v^1_{py} + v^{(1)}_{sy} v^1_p +  \sqrt{\epsilon} u^1_p v^1_{px} \\ \label{deri.pr.1.end}
 & + \sqrt{\epsilon} v^1_p v^1_{py} + \frac{P^1_{py}}{\epsilon} \Big] + \frac{\partial_y}{\epsilon} \epsilon P^{1,a}_p  + R^{v,1}_E + E^{v,1}_E - \epsilon \Delta v^1_e + \frac{\partial_y}{\epsilon} \epsilon P^{1,a}_E + (\frac{\partial_y}{\epsilon} \sqrt{\epsilon} P^1_e + v^1_{ex}).
\end{align}

In the above expressions, the $[\bar{u}_s^{(0)}, \bar{v}_s^{(0)}]$, $[u_s^{(1)}, v_s^{(1)}]$ are known at this stage, and the $u^1_p, v^1_p$ are unknowns, to be constructed in this step. Similarly for the pressures, the the $P^1_e$ is known, but the auxiliary pressures $P^{1,a}_E, P^{1,a}_p$ are to be defined in this section. Finally, as can be seen from the definitions in (\ref{Ru1E}), (\ref{Rv1E}), the $R^{u,1}_E, R^{v,1}_E, E^{u,1}_E, E^{v,1}_E$ are all knowns.  Let us now simplify the above expressions. First, via the construction of $[u^0_p, v^0_p]$ we may write: 
\begin{align} \nonumber
R^{u,0}  &:= -\Delta_\epsilon \bar{u}^{(0)}_s + \bar{u}^{(0)}_s \bar{u}^{(0)}_{sx} + \bar{v}^{(0)}_s \bar{u}^{(0)}_{sy} + \bar{P}^{(0)}_x + v^1_e u^0_{py} \\ \label{ru0.1} &= -\epsilon u^0_{pxx} + \sqrt{\epsilon}y v^1_{eY}u^0_{py} + \epsilon u^0_{py} \int_0^y \int_y^{y'}v^1_{eYY} dy'' dy'. 
\end{align}

This then accounts for the lowest order term from $R^{u,1}_E$ in (\ref{Ru1E}), allowing us to redefine:  
\begin{align}
R^{u,1}_E = \sqrt{\epsilon} u^1_{ex} u^0_p + \sqrt{\epsilon} u^1_e u^0_{px} + \epsilon v^0_p u^1_{eY}.
\end{align}

Let us now define the auxiliary Euler pressure: 
\begin{align} \label{P1aasin}
P^{1,a}_E := -\frac{1}{2}\Big(\Big| u^1_e \Big|^2 + \Big|v^1_e \Big|^2 \Big),
\end{align}

so that, combined with the equations we have taken for $[u^1_e, v^1_e]$, we have:
\begin{lemma}  With $P^{1,a}_E$ as in (\ref{P1aasin}), and $E_E^{u,1}, E_E^{v,1}$ as in (\ref{Ru1E}) - (\ref{Rv1E}),
\begin{align} \label{pot.1}
 \partial_x \epsilon P^{1,a}_E + E^{u,1}_E =  \frac{\partial_y}{\epsilon} \epsilon P^{1,a}_E + E^{v,1}_E = 0. 
\end{align}
\end{lemma}
\begin{remark}Denoting by $\nabla_\epsilon := (\partial_x, \frac{\partial_Y}{\sqrt{\eps}})$, the above lemma reads:  
\begin{equation}
\nabla_\eps \eps P^{1,a}_E + (E^{u,1}_E, E^{v,1}_E) = 0.
\end{equation}

Thus, the purely Eulerian terms are of gradient structure, which is then exploited with the introduction of our auxiliary pressure. 
\end{remark}
\begin{proof}

The proof follows by direct calculation and an appeal to (\ref{euler.1.V.0}):
\begin{align}
\eps \partial_x P^{1,a}_E =  -\eps u^1_e u^1_{ex} - \eps v^1_e v^1_{ex} = -\eps u^1_e u^1_{ex} - \eps v^1_e u^1_{eY} = - E^{u,1}_E.
\end{align}

Similarly, 
\begin{align}
\frac{\p_Y}{\sqrt{\eps}} \eps P^{1,a}_E = -\sqrt{\eps} u^1_e u^1_{eY} - \eps v^1_e v^1_{eY} = -\sqrt{\eps} u^1_e v^1_{ex} - \eps v^1_e v^1_{eY} = - E^{v,1}_E. 
\end{align}

The claim has been proven. 

\end{proof}

\begin{corollary} \label{cor.van} All ``purely Eulerian" terms from the expansions (\ref{deri.pr.1.e.1}), (\ref{deri.pr.1.end}) vanish identically. That is: 
\begin{align}
&E^{u,1}_E - \epsilon^{\frac{3}{2}} \Delta u^1_e + \epsilon P^{1,a}_{Ex} + \sqrt{\epsilon} (u^1_{ex} + P^1_{ex}) = 0, \\
& E^{v,1}_E - \epsilon \Delta v^1_e + \frac{\partial_y}{\epsilon} \epsilon P^{1,a}_E + (\frac{\partial_y}{\epsilon} \sqrt{\epsilon} P^1_e + v^1_{ex}) = 0. 
\end{align}
\end{corollary}
\begin{proof}
First, by harmonicity of $[u^1_e, v^1_e]$, we have: 
\begin{align} \label{dropout.1}
\epsilon^{\frac{3}{2}} \Delta u^1_e = \epsilon \Delta v^1_e = 0. 
\end{align}

Next, regarding the final terms in both (\ref{deri.pr.1.e.1}) and (\ref{deri.pr.1.end}), by the equations (\ref{deri.e.1}) and (\ref{deri.e.2}), we see that these terms drop out: 
\begin{align} \label{dropout.2}
u^1_{ex} + P^1_{ex} = 0, \hspace{3 mm} v^1_{ex} + P^1_{eY} = 0.
\end{align} 

Coupled with (\ref{pot.1}), this establishes the desired claim. 

\end{proof}

In the above calculation, we are using crucially the Cauchy-Riemann structure of $[u^1_e, v^1_e]$ in \ref{div.curl}; harmonicity alone does not suffice here. Next, define the auxiliary Prandtl pressure: 
\begin{align} \label{p1pa}
P^{1,a}_p :=  \int_y^\infty \Big[-\Delta_\epsilon \bar{v}_s^{(0)} + \bar{u}_s^{(0)} \bar{v}_{sx}^{(0)} + \bar{v}_s^{(0)} \bar{v}_{sy}^{(0)}  \Big] +  \int_y^\infty R^{v,1}_E. 
\end{align}

Combining all this, we take our Prandtl-1 equation to be: 
\begin{align} \label{pr.eqn.1}
 -u^1_{pyy} + (1 + u^0_p) u^1_{px} + u^{(1)}_{sx} u^1_p + u^{0}_{py} \Big( v^1_p - v^1_p(x,0)\Big) + v^{(1)}_s u^1_{py}  + P^1_{px} = f^{(1)},
\end{align}

together with the divergence free condition $u^1_{px} + v^1_{py} = 0$, and the boundary conditions: 
\begin{align} \label{pr.1.BC}
u^1_p(x,0) = 0, \hspace{2 mm} \lim_{y \rightarrow \infty} [u^1_p(x,y), v^1_p(x,y)] = 0, \hspace{2 mm} v^1_p(x,0) = -v^2_e(x,0), \hspace{2 mm} u^1_p(1,y) = U_1(y). 
\end{align}

Here, the forcing term is defined as: 
\begin{align} \label{forcing.pr.1}
f^{(1)} := \epsilon^{-\frac{1}{2}}\Big[  R^{u,0} + R^{u,1}_E + \epsilon P_{px}^{1,a}\Big].
\end{align}

The boundary contribution of $v^1_p(x,0)$ in (\ref{pr.eqn.1}) again arises from the calculation: 
\begin{align} \nonumber
v^2_e(x,Y) u^0_{py} &= v^2_e(x,0) u^0_{py} + \sqrt{\epsilon} y u^0_{py}v^2_{eY} + \epsilon u^0_{py} \int_0^y \int_{y}^{y'} v^2_{eYY} \\ \label{take.1.1}
&= -v^{1}_p(x,0) u^0_{py} +  \sqrt{\epsilon} y u^0_{py}v^2_{eY} + \epsilon u^0_{py} \int_0^y \int_{y}^{y'} v^2_{eYY}.
\end{align}

Consolidating (\ref{dropout.1}), (\ref{dropout.2}), (\ref{pot.1}), (\ref{p1pa}), (\ref{pr.eqn.1}), (\ref{take.1.1}) with the expressions (\ref{deri.pr.1.e.1}) and (\ref{deri.pr.1.end}) then shows that the following remainder is contributed: 
\begin{align}
&R^{u,1}  = \sqrt{\epsilon}\Big[-\epsilon u^1_{pxx} + \sqrt{\epsilon} y v^2_{eY} u^0_{py} + \epsilon u^0_{py} \int_0^y \int_y^{y'} v^2_{eYY} + (\epsilon u^1_{eY} + \sqrt{\epsilon} u^1_{py}) v^1_p  \\ \n
& \hspace{20 mm} + \sqrt{\epsilon} (u^1_e + u^1_{p}) u^1_{px}  \Big], \\ \label{rem.pr.1}
&R^{v,1}= \sqrt{\epsilon}\Big[-\Delta_\epsilon v^1_p + u_s^{(1)} v^1_{px} + v^{(1)}_{sx} u^1_p + v_s^{(1)} v^1_{py} + v^{(1)}_{sy} v^1_p +  \sqrt{\epsilon} u^1_p v^1_{px}  + \sqrt{\epsilon} v^1_p v^1_{py} \Big]. 
\end{align}

Let us emphasize the boundary condition at $y \rightarrow \infty$ for $v^1_p$ means that we will define: 
\begin{align} \label{emphv1p}
v^1_p(x,y) = \int_y^\infty u^1_{px}(x,y') dy'.
\end{align}

We may evaluate the equation (\ref{pr.eqn.1}) at $y = \infty$ to see that the pressure term drops out. That is the pressure in the Prandtl layer is constant, so we may WLOG take $P^1_p = 0$.

\subsection{Global in $x$ Existence and Decay:}

The first step is to homogenize the boundary conditions by introducing the new unknowns: 
\begin{align} \label{pr.1.hom}
u = u^1_p + \chi(y) u^1_e(x,0), \hspace{3 mm} v = v^1_p - v^1_p(x,0)+ u^1_{ex}(x,0)I_\chi(y), \hspace{3 mm} I_\chi(y) = \int_y^\infty \chi(\theta) d\theta,
\end{align}

where $\chi(y)$ is a cutoff function satisfying: 
\begin{equation}
\chi(0) = 1, \hspace{3 mm} \partial_y^k \chi(0) = 0, \hspace{3 mm} \int_0^\infty \chi(y) dy = 0.
\end{equation}

The mean-zero condition is meant to ensure that $v(0) = 0$. The homogenized profiles now satisfy the following system: 
\begin{align} \label{pr.1.eqn.bar}
(1+u^0_p) u_x &- u_{yy} + \mathcal{P}(u,v) = f^{(1)} + \mathcal{J}; \hspace{3 mm} u(x,0) = v(x,0) = 0, \hspace{3 mm} \lim_{y \rightarrow \infty} u(x,y) = 0. 
\end{align}

along with the divergence free condition $u_x + v_y = 0$. Here: 
\begin{align} \label{calP}
\mathcal{P} &:= u^{(1)}_{sx}u + u^0_{py} v + v_s^{(1)} u_y, \\ \n
\mathcal{J} &:= u_s^{(1)}\chi(y) u^1_{ex}(x,0) - \chi'' u^1_e(x,0) - \chi(y) u^{(1)}_{sx} u^1_e(x,0) - u^0_{py}I_\chi(y) u^1_{ex}(x,0) \\ \label{calJ} &+ v_s^{(1)}\chi'(y) u^1_e(x,0).
\end{align}

We note that $v(x,y)$ does not vanish as $y \rightarrow \infty$ due to the definition in (\ref{pr.1.hom}). The essential feature of $v(x,y)$ that will be in use is that $v(x,0) = 0$. Examining (\ref{calP}), one observes that $v$ is always accompanied by $u^0_p$, which decays rapidly in $y$ for each fixed $x$. Let us now define the norm in which we shall control the Prandtl solutions:
\begin{equation} \label{norm.P}
||u||_{P(X_1, \sigma)}^2 := \sup_{1 \le x \le X_1} \{ \int u^2  x^{-2\sigma} + u_y^2 x^{1-2\sigma} \} + \int_1^{X_1} \int \{ u^2  x^{-1 - 2\sigma} +  u_y^2 x^{-2\sigma} + u_x^2 x^{1-2\sigma}\}. 
\end{equation}

We shall also need the following, differentiated version, of the above Prandtl-layer norm: 
\begin{align} \label{norm.Pk}
||u||_{P_k(X_1, \sigma)}^2 &:= \sup_{1 \le x \le X_1} \{ \int |\partial_x^k u|^2 x^{2k-2\sigma} + |\partial_x^k u_y|^2  x^{2k+1-2\sigma} \} \\ & \nonumber \hspace{20 mm} + \int_1^{X_1} \int \{ |\partial_x^k u|^2 x^{2k-1 - 2\sigma}  + |\partial_x^k u_y|^2   x^{2k-2\sigma}+ |\partial_x^{k+1}u|^2 x^{2k+1-2\sigma}\}. 
\end{align}

Through the $H^1_y \hookrightarrow L^\infty_y$ embedding, it is clear that: 
\begin{equation}
\sup_{1 \le x \le X_1} x^{k+\frac{1}{4}-\sigma}||\partial_x^k u||_{L^\infty} \le ||u||_{P_k(X_1, \sigma)}. 
\end{equation}

Finally, we shall write the global norms as: 
\begin{equation}
||u||_{P(\sigma)} := ||u||_{P(\infty, \sigma)}, \hspace{3 mm} ||u||_{P_k(\sigma)} := ||u||_{P_k(\infty, \sigma)}. 
\end{equation}

\begin{remark} A comparison of the norms $P_k$ with the estimates valid for $[u^0_p, v^0_p]$ in (\ref{main.decay}), (\ref{main.v}) show that $u^1_p$ is roughly ``$x^{-\frac{1}{4}}$-better" than $u^0_p$. There is no front-profile as in the case of $[u^0_p, v^0_p]$, which is because the present boundary condition, $u^1_p(x,0) = -u^1_e(x,0)$ decays as $x \rightarrow \infty$, due to (\ref{int.hilbert}), in contrast with the boundary condition for $u^0_p(x,0) = -\delta$. The secondary reason is because $f^{(1)}$ in (\ref{forcing.pr.1}) contains derivative and nonlinear terms from the previous layers, which enhances the decay in $x$.
\end{remark}

The first step is to give the following estimates on the forcing terms, which capitalize on the structure of these terms either having many derivatives (thereby enhancing decay in $x$) or are of product form (also enhancing decay in $x$):
\begin{lemma}[Forcing Estimates] For any $k,m \ge 0$, and arbitrary $N > 0$
\begin{align} \label{n.forcing.1}
&\Big| z^m \partial_x^k \mathcal{J} \Big| \le C(k,m) \langle y \rangle^{-N} x^{-k-\frac{1}{2}}, \\ \label{n.forcing.2}
&||z^m \partial_x^k f^{(1)}||_{L^2_y} \le C(k, m) \langle x \rangle^{-k -\frac{5}{4}}. 
\end{align}
\end{lemma}
\begin{proof}
The estimate for the $\mathcal{J}$ terms is direct, thanks to Proposition \ref{thm.euler.1}, estimate (\ref{euler.DECAY.main}). For $f^{(1)}$, in consultation with the definition (\ref{forcing.pr.1}), we start with the $R^{u,0}$, defined in (\ref{ru0.1}): 
\begin{align}
&\sqrt{\epsilon}||u^0_{pxx}||_{L^2_y} \lesssim \sqrt{\epsilon}  x^{-\frac{7}{4}}, \\
&||y u^0_{py} v^1_{eY}||_{L^2_y} \le ||x^{-\frac{1}{4}} y u^0_{py}||_{L^2_y} \sup_y \Big| v^1_{eY}\Big| x^{\frac{1}{4}} \lesssim  x^{-\frac{5}{4}}, \\
& \sqrt{\epsilon} ||u^0_{py} \int_0^y \int_y^{y'} v^1_{exx} dy'' dy' ||_{L^2_y} \le \sqrt{\epsilon} ||y^2 u^0_{py}||_{L^2_y} ||v^1_{exx}||_{L^\infty_y} \lesssim \sqrt{\epsilon} x^{-\frac{7}{4}}. 
\end{align}

Second, we control $R^{u,1}_E$ via: 
\begin{align} \nonumber
\epsilon^{-\frac{1}{2}}||R^{u,1}_E||_{L^2_y} &\le ||u^1_{ex}||_{L^\infty_y} ||u^0_p||_{L^2_y} + ||u^1_e||_{L^\infty_y} ||u^0_{px}||_{L^2_y} + \sqrt{\epsilon} ||v^0_p||_{L^2_y} ||u^1_{eY}||_{L^\infty_y} \\
&\lesssim x^{-\frac{5}{4}}.
\end{align}

Next, according to (\ref{p1pa}), we have: 
\begin{align}
||P^{1,a}_{px}||_{L^2_y} \le  || \langle y \rangle^{\frac{3}{2}+\kappa}  \partial_x \Big(-\Delta_\epsilon \bar{v}_s^{(0)} + \bar{u}_s^{(0)} \bar{v}_{sx}^{(0)} + \bar{v}_s^{(0)} \bar{v}_{sy}^{(0)} \Big) ||_{L^\infty_y} +  ||\langle y \rangle^{\frac{3}{2}+\kappa} R^{v,1}_{Ex}||_{L^\infty_y}.  
\end{align}

For each of these terms, the ability to trade $x^{\frac{3}{4}+\frac{\kappa}{2}}$ for $y^{\frac{3}{2}+\kappa}$ is in constant use: 
\begin{align} \nonumber
& ||\langle y \rangle^{\frac{3}{2}+\kappa} \Delta_\epsilon v^0_{px}||_{L^\infty_y} +  || \langle y \rangle^{\frac{3}{2}+\kappa} \partial_x \{ \bar{u}_s^{(0)} \bar{v}_{sx}^{(0)} \}||_{L^\infty_y} +  || \langle y \rangle^{\frac{3}{2}+\kappa} \partial_x\{ \bar{v}^{(0)}_s \bar{v}^{(0)}_{sy} \}||_{L^\infty_y} \\ \nonumber
&+  || \langle y \rangle^{\frac{3}{2}+\kappa} \partial_x\{u^0_p v^1_{ex}  \}  ||_{L^\infty_y} +  || \langle y \rangle^{\frac{3}{2}+\kappa} \partial_x\{ \sqrt{\epsilon} u^1_e v^0_{px} \}  ||_{L^\infty_y} +  || \langle y \rangle^{\frac{3}{2}+\kappa} \partial_x\{ \sqrt{\epsilon} v^0_p v^1_{eY}  \}  ||_{L^\infty_y} \\ &+  || \langle y \rangle^{\frac{3}{2}+\kappa} \partial_x\{ v^1_e v^0_{py} \}  ||_{L^\infty_y} \lesssim  x^{-\frac{3}{2}}.
\end{align}

Above, we are using the established estimates in (\ref{main.decay}), (\ref{main.v}) for the Prandtl-0 profiles, and (\ref{euler.DECAY.main}) for the Euler-1 profiles. The desired estimates are proven for $m = 0$. For general $m$ the estimates work in an identical manner, after noticing that powers of $z$ play no role when accompanied by $[u^0_p, v^0_p]$, which appear in every term above. 

\end{proof}

The above lemma relies crucially on Corollary \ref{cor.van} in order to apply the weight $z^m$. We now give the following energy estimate: 
\begin{lemma}\label{lemma.pr.1.e} Let $\sigma > 0 $ and fix any $X_1 > 1$. Then:
\begin{align} \nonumber
\sup_{x \in [1, X_1]} \int x^{-2\sigma} u^2 + \int_1^{X_1} \int x^{-1-2\sigma} u^2&+ \int_1^{X_1} \int x^{-2\sigma} u_y^2  \\ \label{pr.1.energy} 
& \lesssim \mathcal{O}(\delta; \sigma) \int_1^{X_1} \int v_y^2 x^{1-2\sigma} + C(\sigma). 
\end{align}
The constant above depends poorly on small $\sigma$. 
\end{lemma}

\begin{remark} The need for this $\sigma > 0$ is to avoid certain $x$-integrations being critical. We make the notational convention that we will not rename different values for $\sigma$ (for instance $2\sigma$), as we can always redefine $\sigma$ to be smaller. However, \textit{within} a single calculation (for instance the upcoming proof) we fix a $\sigma > 0$.
\end{remark}

\begin{proof}
Applying the multiplier $M = ux^{-2\sigma}$ to the system (\ref{pr.1.eqn.bar}) gives the following terms: 
\begin{equation}
\int \Big((1+u^0_p) \partial_x - \partial_{yy} \Big) u \cdot ux^{-2\sigma} \gtrsim  \partial_x  \int (1+u^0_p) x^{-2\sigma} u^2 + \int x^{-1-2\sigma} u^2 + \int x^{-2\sigma} u_y^2 - \int u^{0}_{px} u^2 x^{-2\sigma}. 
\end{equation}

The constant in the above estimate depends poorly on $\sigma$, as there is a factor of $\sigma$ accompanying the $\int x^{-1-2\sigma} u^2$ term. The final term  above then gets placed into the contributions from $\mathcal{P}$, to which we now turn (see the definition in (\ref{calP})): 
\begin{align} \n
\Big| \int \mathcal{P} \cdot u x^{-2\sigma} \Big| &\le ||x u^{(1)}_{sx}, y u^0_{py}, v^{(1)}_{sy} x||_{\infty} \Big( \int u^2 x^{-1-2\sigma} + \int v_y^2 x^{1-2\sigma} \Big) \\
& \le \mathcal{O}(\delta) \Big( \int u^2 x^{-1-2\sigma} + \int v_y^2 x^{1-2\sigma} \Big). 
\end{align}

Above, we have used the Prandtl-0 bounds in (\ref{main.decay}), which crucially provides the smallness of $\{u^0_{px}, u^0_{py}\}$ in terms of $\mathcal{O}(\delta)$. No smallness of Eulerian profiles is required due to the extra factor of $\sqrt{\eps}$. The key estimate which forces a loss of $\partial_x$ derivative is the following: 
\begin{align} \label{key.1}
\Big|\int u^0_{py} v \cdot u x^{-2\sigma} \Big| \le ||u^0_{py}y||_{L^\infty} \int \Big| \frac{v}{y} u x^{-2\sigma} \Big| \le \mathcal{O}(\delta) ||v_y x^{\frac{1}{2}-\sigma}||_{L^2_y} ||u x^{-\frac{1}{2}-\sigma}||_{L^2_y}. 
\end{align}

The structure of the key estimate above, (\ref{key.1}), is omnipresent in our analysis: we use the $y$-absorption of $u^0_{py}$ to produce a $v_y$ term via Hardy's inequality (which is valid as $v(x,0) = 0$). This then forces a loss of $x^{\frac{1}{2}-\sigma}$ in the decay, which must then be regained in the next lemma. Again, the estimate $||yu^0_{py}||_{L^\infty} \le \mathcal{O}(\delta)$ arises from (\ref{main.decay}). Next, we arrive at the forcing terms. First, via (\ref{n.forcing.1}):
\begin{align} \nonumber
\Big| \int \mathcal{J} \cdot u x^{-2\sigma} \Big| &\lesssim \int \langle y \rangle^{-N} x^{-\frac{1}{2}} |u|x^{-2\sigma} \lesssim  ||x^{-\frac{1}{2}-\sigma} \langle y \rangle^{-\frac{N}{2}}||_{L^2_y} ||x^{-\sigma} \frac{u}{\langle y \rangle^{\frac{N}{2}}}||_{L^2_y} \\
&\le \frac{1}{100,000} ||u_y x^{-\sigma}||_{L^2_y}^2 + C x^{-1-2\sigma}. 
\end{align}

Upon taking an integration in $dx$, the majorizing terms above are finite. Next, according to (\ref{n.forcing.2}), via Young's inequality:
\begin{align} \nonumber
\Big| \int f^{(1)}\cdot ux^{-2\sigma} \Big| &\le ||f^{(1)}x^{\frac{1}{2}-\sigma}||_{L^2_y} ||u x^{-\frac{1}{2}-\sigma}||_{L^2_y} \le C x^{-\frac{3}{2}} + \frac{1}{100,000}  \int u^2 x^{-1 -2\sigma} dy. 
\end{align}

Placing these estimates together: 
\begin{align} \nonumber
\partial_x \int x^{-2\sigma} u^2 &+ \int x^{-1-2\sigma} u^2 + \int x^{-2\sigma} u_y^2 \lesssim \mathcal{O}(\delta) \int u_x^2 x^{1-2\sigma} +  x^{-1-2\sigma} .
\end{align}

Integrating above from $x = 1$ to $x = X_1$ yields:
\begin{align} \nonumber
 \int X_1^{-2\sigma} u^2(X_1) dy + \int_1^{X_1} \int x^{-1-2\sigma} u^2 &+ \int_1^{X_1} \int x^{-2\sigma} u_y^2 \\  \label{non.sup}
&\lesssim \mathcal{O}(\delta) \int_1^{X_1} \int v_y^2 x^{1-2\sigma} + C.
\end{align}

Finally, we reason as follows: the second and third terms on the left-hand side above are positive, which then gives for this $X_1$: 
\begin{align}
\int X_1^{-2\sigma} u^2(X_1) dy \lesssim  \mathcal{O}(\delta) \int_1^{X_1} \int v_y^2 x^{1-2\sigma} + C. 
\end{align}

For any $X_2 \in [1,X_1]$, the same estimate holds, namely: 
\begin{align} 
\int X_2^{-2\sigma} u^2(X_2) dy &\lesssim  \mathcal{O}(\delta) \int_1^{X_2} \int v_y^2 x^{1-2\sigma} + C \lesssim \mathcal{O}(\delta) \int_1^{X_1} \int v_y^2 x^{1-2\sigma} + C.
\end{align}

Above, we have used the monotonicity of the right-hand side as $X_2 < X_1$. This allows us to replace in (\ref{non.sup}) the first term on the left-hand side with $\sup_{1 \le x \le X_1} \int x^{-2\sigma} u^2 dy$. This then gives the desired estimate in (\ref{pr.1.energy}). 

\end{proof}

We now recover the $v_y$ term on the right-hand side of (\ref{pr.1.energy}) via: 
\begin{lemma} \label{lemma.pr.1.p} Let $\sigma > 0 $, and fix any $X_1 > 1$. Then:
\begin{align} \nonumber
\sup_{1 \le x \le X_1} \int u_y^2 x^{1-2\sigma} + \int_1^{X_1} \int u_x^2 x^{1-2\sigma} &\lesssim C +\mathcal{O}(\delta) \int_1^{X_1} \int u^2 x^{-1-2\sigma} \\  \label{evolution.pos.1}
&+ \int_1^{X_1} \int u_y^2 x^{-2\sigma}. 
\end{align}
\end{lemma}
\begin{proof}
We now apply the multiplier $M = u_x x^{1-2\sigma}$ to the system in (\ref{pr.1.eqn.bar}). Doing so yields the following positive terms: 
\begin{align}
\int \Big( (1+u^0_p) \partial_x - \partial_{yy} \Big)u \cdot u_x x^{1-2\sigma} \gtrsim \partial_x \int u_y^2 x^{1-2\sigma} - \int u_y^2 x^{-2\sigma} + \int u_x^2 x^{1-2\sigma}. 
\end{align}

Note crucially that the middle term in the above estimate, $||u_y x^{-\sigma}||_{L^2_y}^2$, has been estimated in (\ref{pr.1.energy}) upon taking an $x$-integration. To close this sequence of estimates, therefore, it is crucial that this small parameter $\mathcal{O}(\delta)$ is attached to the $||v_y x^{\frac{1}{2}-\sigma}||_{L^2}^2$ term in (\ref{pr.1.energy}). Next, we come to the profile terms, $\mathcal{P}$ (see (\ref{calP}) for the definition): 
\begin{align} \n
\Big| \int \mathcal{P} \cdot u_x x^{1-2\sigma} \Big| &\le ||xu^{(1)}_{sx}, y u^0_{py}, v^{(1)}_s x^{\frac{1}{2}}||_{L^\infty} \Big( ||u x^{-\frac{1}{2}-\sigma}||_{L^2_y}^2 + ||u_x x^{\frac{1}{2}-\sigma}||_{L^2_y}^2 \Big) \\
& \le \mathcal{O}(\delta) \Big( ||u x^{-\frac{1}{2}-\sigma}||_{L^2_y}^2 + ||u_x x^{\frac{1}{2}-\sigma}||_{L^2_y}^2 \Big) .
\end{align}

We have used estimates (\ref{main.decay}), (\ref{main.v}), and (\ref{euler.DECAY.main}), which provide the smallness of $\mathcal{O}(\delta)$. Next, we come to $f^{(1)}$, for which we use the estimate in (\ref{n.forcing.2}) (with $k = 0$):
\begin{align}
\Big| \int f^{(1)} u_x x^{1-2\sigma} \Big| \le ||x^{\frac{1}{2}-\sigma}f^{(1)}||_{L^2_y} ||u_x x^{\frac{1}{2}-\sigma}||_{L^2_y} \le Cx^{-\frac{3}{2}} +  \frac{1}{100,000}||u_x x^{\frac{1}{2}-\sigma}||_{L^2_y}^2. 
\end{align}

Piecing the above estimates together gives: 
\begin{align} \nonumber
\partial_x \int u_y^2 x^{1-2\sigma} &+ \int u_x^2 x^{1-2\sigma}  \\ 
&\le \int u_y^2 x^{-2\sigma} + \mathcal{O}(\delta) \int u^2 x^{-1-2\sigma}  + Cx^{-1-\sigma} + \int \mathcal{J} \cdot u_x x^{1-2\sigma}.
\end{align}

Now, we take an integration from $x = 1$ to $x = X_1$: 
\begin{align} \nonumber
\int u_y^2(X_1) X_1^{1-2\sigma} &+ \int_1^{X_1} \int u_x^2 x^{1-2\sigma} \lesssim \int_1^{X_1} \int u_y^2 x^{-2\sigma} \\ \label{pre.J}  & + \mathcal{O}(\delta) \int_1^{X_1} \int u^2 x^{-1-2\sigma}  + C + \int_1^{X_1} \int \mathcal{J} \cdot u_x x^{1-2\sigma}.
\end{align}

The last step is to come to the terms in $\mathcal{J}$, from (\ref{calJ}). The most delicate term here requires successive integration by parts: 
\begin{align} \label{suc.1}
 &\int_1^{X_1} \int \chi''(y) x^{1-2\sigma} u^1_e(x,0) u_x \\ \nonumber
 &= -\int_1^{X_1} \int \chi''(y) u \partial_x\{ u^1_e(x,0) x^{1-2\sigma} \} - \int_{x=1}  u^1_e(1,0) u \chi''(y) dy \\ \label{suc.2} 
 & \hspace{40 mm} + \int_{x = X_1} X_1^{1-2\sigma} u^1_e(X_1,0) \chi''(y) u(X_1,y) dy \\ \nonumber
 & = -\int_1^{X_1} \int \chi''(y) u \partial_x\{ u^1_e(x,0) x^{1-2\sigma} \} + C\\ \label{suc.3}
 & \hspace{40 mm} - \int_{x=X_1} X_1^{1-2\sigma} u^1_e(X_1, 0)\chi'(y) u_y(X_1,y) dy \\ \label{suc.4}
 &\le \int_1^{X_1} \int \chi'(y) u_y \partial_x  \{u^1_e(x,0) x^{1-2\sigma} \} + C + \frac{1}{100,000} \int_{x=X_1} X_1^{1-2\sigma} u_y^2(X_1) \\ \label{suc.5}
 & \lesssim ||u_y x^{-\sigma}||_{L^2} ||\chi'(y) x^{-\frac{1}{2}-\sigma}||_{L^2} + C \le C + \frac{1}{100,000} ||u_y x^{-\sigma}||_{L^2}^2. 
\end{align}

Going from (\ref{suc.3}) to (\ref{suc.4}), we have used $|u^1_e(X_1,0)| \lesssim X_1^{-\frac{1}{2}}$, according to (\ref{euler.DECAY.main}). Going from (\ref{suc.4}) to (\ref{suc.5}), we have used $|u^1_{ex}(x,0)| \lesssim x^{\frac{3}{2}}$, also according to (\ref{euler.DECAY.main}). Finally, we have absorbed the boundary contribution at $x = X_1$, $\int u_y^2 X_1^{1-2\sigma}$ into the left-hand side of (\ref{pre.J}). A consultation with (\ref{calJ}) shows that we can perform a similar calculation with the remaining terms in $\mathcal{J}$ because these terms either have one extra $x$-derivative, or are accompanied by $v^{(1)}_s$, which contributes additional decay of $x^{-\frac{1}{2}}$. This then gives: 
\begin{align} \nonumber
\int u_y^2&(X_1) X_1^{1-2\sigma} + \int_1^{X_1} \int u_x^2 x^{1-2\sigma} \lesssim \int_1^{X_1} \int u_y^2 x^{-2\sigma} + \mathcal{O}(\delta) \int_1^{X_1} \int u^2 x^{-1-2\sigma} +  C.
\end{align}

Using a similar line of reasoning as in the Lemma \ref{lemma.pr.1.e}, we can replace the $\int u_y^2(X_1) X_1^{1-2\sigma}$ with the supremum over all $x \in [1,X_1]$, thereby yielding the desired result.

\end{proof}

Consolidating the results of the previous two lemmas and applying contraction mapping: 
\begin{corollary} For $\delta, \epsilon$ sufficiently small relative to $\sigma$, a solution to the Prandtl system in (\ref{pr.1.eqn.bar}) satisfies the following a-priori estimate in the space $P$:
\begin{align} \label{contraction.X0}
||u||_{P(X_1, \sigma)}^2 \lesssim C(\sigma). 
\end{align}
For $\delta, \epsilon$ sufficiently small, there exists a unique solution to (\ref{pr.1.eqn.bar}), satisfying $||u||_{P(X_1)} \le C(\sigma)$. The constant above in (\ref{contraction.X0}) is independent of $X_1$, and so we can immediately  send $X_1 \rightarrow \infty$, thereby yielding a global solution on $x \in [1,\infty)$ satisfying: $||u||_{P} \le C(\sigma)$.
\end{corollary}

It is possible to successively differentiate the system in $\partial_x$, and re-apply the previous estimates, noting that the added $x-$derivative adds a factor of $x^{-1}$ to each term above, enabling us to enhance the multiplier to $[x^{k-2\sigma}\partial_x^k u, x^{1+k-2\sigma}\partial_x^{k+1}u]$. This is the reason for the differentiated version of the Prandtl-norm in (\ref{norm.Pk}). To do so, we simply need: 
\begin{lemma}[Initial Conditions] For each $k \ge 0$, the initial data $\partial_x^k u(1,y)$ is of order $\delta$ and decays rapidly in $y$. 
\end{lemma}
\begin{proof}
This follows from using the equation (\ref{pr.eqn.1}) to write: 
\begin{align} \nonumber
u^1_{px}(1,y) &= U_{1yy}(y) + u^{(1)}_{sx}(1,y) U_1(y) + y u^0_{py}(1,y) \frac{v^1_p(1,y) - v^1_p(1,0)}{y}  \\ \label{icup1}&+ v_s^{(1)}(1,y) U_{1y}(y) + f^{(1)}(1,y).  
\end{align}

We may now multiply by $(1+y)^n$, use the inequality $||\frac{v(1,y)}{y}||_{L^\infty} \lesssim ||v_y(1,y)||_{L^\infty}$ for functions $v$ satisfying $v(x,0) = 0$ (here we take $v = v^1_p(1,y) - v^1_p(1,0)$), and use $||\langle y \rangle^n U_1(y) u^0_{px}(1,y)||_{L^\infty} \le \mathcal{O}(\delta) ||u^0_{px}(1,y)||_{L^\infty}$, to obtain:  $||\langle y \rangle^n u^1_{px}(1,y)||_{L^\infty_y} \le C$. It is clear that the same procedure may be applied to higher order $x$-derivatives. 

\end{proof}

\begin{remark} \label{hoc2}[Higher Order Compatibility] In order to apply the procedure described, we require high-order compatibility condition such that the initial data of $u^1_{px}(1,0) = 0$, thereby honoring the boundary condition. These condition can be ascertained inductively from (\ref{icup1}). For instance: 
\begin{align}
0 =  u^1_{px}(1,0) = U_{1yy}(0) + u^{(1)}_{sx}(1,0)U_1(0) + v_s^{(1)} U_{1y}(0) + f^{(1)}(1,0). 
\end{align}

We suppose these compatibility conditions for large $k$. 
\end{remark}

Repeating the previous set of estimates after applying the self-similar weight $z^M$, and upon applying $\p_x^k$ to the system, we arrive at the following:
\begin{lemma} Given any $\sigma > 0$, let $\delta, \epsilon$ be sufficiently small relative to $\sigma$. Consider the system given in (\ref{pr.eqn.1}), together with the boundary conditions (\ref{pr.1.BC}). Let all derivatives of the prescribed data $U_1(y)$ be exponentially decaying in its argument. Then there exists a unique, global in $x$ solution $[u_p, v_p]$, satisfying: 
\begin{equation} \label{pr.1.close.1}
||z^M \partial_x^k u||_{P_k(\sigma)} \le C(M, k, \sigma). 
\end{equation}
\end{lemma}

It is now our task to extract similar estimates for the profiles $u^1_p, v^1_p$ from (\ref{pr.1.close.1}). First, 
\begin{lemma} For any $m,k \ge 0, \sigma > 0$, 
\begin{align} \label{korr}
\sup_{x \ge 1}  \int z^{2m} |\partial_x^k u^1_p|^2 x^{2k-2\sigma} \le C(M, m, k, \sigma). 
\end{align}
\end{lemma}
\begin{proof}
This follows by writing $u^1_p = u - \chi(y) u^1_e(x,0)$, and using $\Big| \partial_x^k  u^1_e(x,0) \Big| \lesssim x^{-k-\frac{1}{2}}$.
\end{proof}

Next, we may give the following uniform decay estimate for $v^1_p$: 
\begin{corollary}[Uniform Estimates for $v^1_p$] For any $m \ge 0$, 
\begin{equation} \label{pr.v.unif.1}
||z^m \p_x^k v^1_p||_{L^\infty_y} \lesssim x^{-k-\frac{3}{4}+\sigma} C(\sigma, m, k).
\end{equation}
\end{corollary}
\begin{proof}
First, according to (\ref{emphv1p}), we have the rapid decay $v^1_p \rightarrow 0$ as $y \rightarrow \infty$. By the divergence-free condition, and the trace inequality,  and for any small $\kappa > 0$, 
\begin{align} \label{initial.integration.vp}
\Big|v^1_p\Big|^2 & \lesssim \Big| \int_y^\infty v^1_pv_{py}\Big| \le  \int_0^\infty \Big| \frac{v^1_p}{y^{\frac{1}{2}-\kappa}} y^{\frac{1}{2}-\kappa}v^1_{py} \Big| \le ||\frac{v^1_p}{y^{\frac{1}{2}-\kappa}}||_{L^2_y}||y^{\frac{1}{2}-\kappa} v^1_{py}||_{L^2_y}\\ 
&\le ||y^{\frac{1}{2}+\kappa} v^1_{py}||_{L^2_y} ||y^{\frac{1}{2}-\kappa} v^1_{py}||_{L^2_y} = ||x^{\frac{1}{4}+\kappa} z^{\frac{1}{2}+\kappa} v^1_{py}||_{L^2_y} ||x^{\frac{1}{4}-\kappa} z^{\frac{1}{2}-\kappa} v^1_{py}||_{L^2_y} \\
& \le x^{\frac{1}{2}} x^{-1+\sigma} x^{-1+\sigma} = x^{-\frac{3}{2}+2\sigma}. 
\end{align}

We have used the $\kappa > 0$ to avoid the critical Hardy inequality. The Hardy inequality we have used (with power $y^{-\frac{1}{2}+\kappa} v^1_p$) relies on the vanishing of $v^1_p$ at $y = \infty$. From here the desired bound follows for $m = 0, k = 0$, and $k,m \ge 1$ works analogously. 

\end{proof}

The final ingredient we will need is to understand the connection between the norms we have controlled, $P_k$, and the quantities $u^1_{py}, u^1_{pyy}$. This is the content of the following: 

\begin{lemma} 
\begin{align} \label{P.embed}
x^{\frac{3}{4}-\sigma}||z^m u^1_{py}||_{L^\infty_y} + x^{\frac{1}{2}-\sigma} ||z^m u^1_{py}||_{L^2_y} + x^{1-\sigma}||z^m u^1_{pyy}||_{L^2_y}  \le C < \infty. 
\end{align}
\end{lemma}
\begin{proof}

First, let us record: 
\begin{align} \nonumber
x^{1-2\sigma} \int  z^{2m} \Big|u^1_{py}\Big|^2 &\le x^{1-2\sigma} \int z^{2m} u_y^2 + x^{1-2\sigma} \int z^{2m} \chi^2 |u^1_e|^2(x,0) \\ \label{korr.2}
& \le ||z^m u||_{P(\sigma)} + C.
\end{align}

Via the equation (\ref{pr.eqn.1}), we have: 
\begin{align} \nonumber
x^{1-\sigma}||u^1_{pyy}||_{L^2_y} &\le x^{1-\sigma} ||u^1_{px}||_{L^2_y} +x^{1-\sigma} || \mathcal{P}||_{L^2_y} + x^{1-\sigma} ||f^{(1)}||_{L^2_y} \\ \label{uyy.embed}
& \lesssim ||u||_{P_1(\sigma)} + ||u||_{P(\sigma)} + C. 
\end{align}

We are using the decay rates established for $u^0_p$ in (\ref{main.decay}), and the pointwise decay of the Euler profiles established in (\ref{e.ss.1}) and the relations in (\ref{korr}). Let us give $\mathcal{P}$ in detail,
\begin{align}
&x^{1-\sigma}||u^{(1)}_{sx} u^1_p||_{L^2_y} \le ||x u^{(1)}_{sx}||_{L^\infty} x^{-\sigma} ||u^1_p||_{L^2_y} \le \mathcal{O}(\delta) ||u||_P + \mathcal{O}(\delta), \\ \nonumber
&x^{1-\sigma} ||u^0_{py} \Big(v^1_p(x,y) - v^1_p(x,0) \Big)||_{L^2_y} \le x^{1-\sigma}||y u^0_{py}||_{L^\infty} || \frac{v^1_p(x,y) - v^1_p(x,0) }{y}||_{L^2_y} \\ 
& \hspace{20 mm} \le \mathcal{O}(\delta) x^{1-\sigma}||u^1_{px}||_{L^2_y} \le \mathcal{O}(\delta) ||u||_{P_1}, \\
&x^{1-\sigma} ||v^{(1)}_s u^1_{py}||_{L^2_y} \le ||x^{\frac{1}{2}} v^{(1)}_s||_{L^\infty} x^{\frac{1}{2}-\sigma} ||u^1_{py}||_{L^2_y} \le \mathcal{O}(\delta) ||u||_{P} + \mathcal{O}(\delta). 
\end{align}

For the final line we have used (\ref{korr.2}). Next, via Sobolev interpolation in the $y$ direction, 
\begin{align}
x^{\frac{3}{4}-\sigma} ||u^1_{py}||_{L^\infty_y} \le \Big(x^{1-\sigma} ||u^1_{pyy}||_{L^2_y} \Big)^{\frac{1}{2}} \Big(x^{\frac{1}{2}-\sigma} ||u^1_{py}||_{L^2_y} \Big)^{\frac{1}{2}} \le C ||u||_P + \frac{1}{100}x^{1-\sigma} ||u^1_{pyy}||_{L^2_y}. 
\end{align}

Coupled with the $u^1_{pyy}$ estimate in (\ref{uyy.embed}), this then establishes the desired bounds. The weighted estimate in $z$ follows analogously. 
\end{proof}

We will select now, 
\begin{align} \label{sigma1}
\sigma = \sigma_1 = \frac{3}{3^n \times 10,000}.
\end{align}

Summarizing the results of this section: 
\begin{proposition}[Prandtl-1 Layer Bounds] \label{Pp1} Given any $n \in \mathbb{N}$, let $\sigma_1$ be as in (\ref{sigma1}), and let $\delta, \epsilon$ be sufficiently small relative to $n$ and $\sigma_1$. Consider the system given in (\ref{pr.eqn.1}), together with the boundary conditions (\ref{pr.1.BC}). Let all derivatives of the prescribed data $U_0(y)$ be exponentially decaying in its argument. Then there exists a unique, global in $x$ solution $[u^1_p, v^1_p]$, satisfying: 
\begin{equation} \label{pr.1.close}
||z^M \partial_x^k u^1_p||_{P_k(\sigma_1)} + x^{k+\frac{3}{4}-\sigma_1}||z^m \p_x^k v^1_p||_{L^\infty_y} \le C(M, k, n), \text{ for any } k,m \ge 0. 
\end{equation}
\end{proposition}

\section{Intermediate Layers} \label{Section.Inter}

We now construct intermediate layers $i = 2$ through $i = n-1$. This is achieved inductively, starting with the construction of the Euler Layer, $u^i_e, v^i_e$. Let us fix the parameters: 
\begin{align} \label{sigma.i}
\sigma_i = \frac{3^i}{3^n \times 10,000} \text{ for } i = 2,...,n. 
\end{align}

The reason for this selection will be seen in (\ref{choice}). 

\subsection{Construction of Euler Layer, $[u^i_e, v^i_e]$} \label{sub.EL}

For this step in the construction, we suppose that $[\bar{u}^{(i-1)}_s, \bar{v}^{(i-1)}_s]$ have already been constructed. The inductive hypothesis on the $1,...,i-1$ Prandtl profiles are as follows: 
\begin{align} \label{ind.1}
|| z^m \partial_x^k u^j_p ||_{P_k(\sigma_j)}  \le C(k, m, n), \text{ for } j = 1,...,i-1,
\end{align}

where $[u^j_p, v^j_p]$ satisfy the system given in (\ref{sys.pr.i}) for $2 \le j \le i-1$, and for $i = 2$, that $[u^1_p, v^1_p]$ satisfy (\ref{pr.eqn.1}). In order to obtain the equations for $[u_e^i, v_e^i]$, we expand the nonlinear terms including the new Euler terms: 
\begin{align} \nonumber
u_s^{(i)} u^{(i)}_{sx} &= \Big( \bar{u}^{(i-1)}_s + \epsilon^{\frac{i}{2}}u_e^i  \Big)\Big( \bar{u}^{(i-1)}_{sx} + \epsilon^{\frac{i}{2}}u_{ex}^i  \Big) \\ 
&= \bar{u}_s^{(i-1)} \bar{u}_{sx}^{(i-1)} + \epsilon^{\frac{i}{2}} \bar{u}_{sx}^{(i-1)} u^i_e + \epsilon^{\frac{i}{2}} \bar{u}_s^{(i-1)} u^i_{ex} + \epsilon^i u^i_e u^i_{ex}, \\ \nonumber
v_s^{(i)}u^{(i)}_{sy} &= \Big( \bar{v}^{(i-1)}_s + \epsilon^{\frac{i-1}{2}}v_e^i  \Big)\Big( \bar{u}^{(i-1)}_{sy} + \sqrt{\epsilon} \epsilon^{\frac{i}{2}}u_{eY}^i  \Big) \\ 
&=  \bar{v}^{(i-1)}_s \bar{u}^{(i-1)}_{sy} + \epsilon^{\frac{i-1}{2}} \bar{u}^{(i-1)}_{sy} v^i_e + \sqrt{\epsilon}\epsilon^{\frac{i}{2}}\bar{v}_s^{(i-1)}u^i_{eY}  + \epsilon^i v^i_e u^i_{eY}. \\ \nonumber
u_s^{(i)} v_{sx}^{(i)} &= \Big(\bar{u}_s^{(i-1)} + \epsilon^{\frac{i}{2}} u^i_e \Big)\Big( \bar{v}_{sx}^{(i)} + \epsilon^{\frac{i-1}{2}} v^i_{ex} \Big) \\
&= \bar{u}_s^{(i-1)} \bar{v}_{sx}^{(i-1)} + \epsilon^{\frac{i}{2}} \bar{v}^{(i-1)}_{sx} u^i_e + \epsilon^{\frac{i-1}{2}} \bar{u}_s^{(i-1)} v^i_{ex} + \epsilon^{i-\frac{1}{2}} u^i_e v^i_{ex}, \\ \nonumber
 v_s^{(i)} v^{(i)}_{sy} &= \Big(\bar{v}^{(i-1)}_s + \epsilon^{\frac{i-1}{2}} v^i_e \Big) \Big( \bar{v}^{(i-1)}_{sy} + \epsilon^{\frac{i}{2}} v^i_{eY} \Big) \\
 &= \bar{v}^{(i-1)}_s \bar{v}^{(i-1)}_{sy} + \epsilon^{\frac{i-1}{2}} \bar{v}^{(i-1)}_{sy} v^i_e + \epsilon^{\frac{i}{2}} \bar{v}^{(i-1)}_s v^i_{eY} + \epsilon^{i-\frac{1}{2}} v^i_e v^i_{eY}. 
\end{align}

We will now define several terms: 
\begin{definition} The $i-1$'th remainder is denoted by: 
\begin{align} \label{Rude}
R^{u,i-1} &:= -\Delta_\epsilon \bar{u}_s^{(i-1)} + \bar{u}_s^{(i-1)} \bar{u}_{sx}^{(i-1)} + \bar{v}_s^{(i-1)} \bar{u}_{sy}^{(i-1)} + \bar{P}_{sx}^{(i-1)} + \eps^{\frac{i-1}{2}} v^i_e u^0_{py}, \\ \label{Rvde}
R^{v,i-1} &:=  -\Delta_\epsilon \bar{v}_s^{(i-1)} + \bar{u}_s^{(i-1)} \bar{v}_{sx}^{(i-1)} +  \bar{v}^{(i-1)}_s \bar{v}^{(i-1)}_{sy}   + \frac{\partial_y}{\epsilon} \bar{P}_s^{(i-1)}.
\end{align}
\end{definition}

We will also split the Euler-Euler interaction terms and the Euler-Prandtl terms via: 
\begin{definition}
\begin{align} \n
R^{u,i}_E &:= \epsilon^{\frac{i}{2}} u^{i}_{ex} \Big[ \sum_{j=0}^{i-1} \epsilon^{\frac{j}{2}} u^j_p \Big] + \epsilon^{\frac{i}{2}} u^{i}_e \Big[ \sum_{j=0}^{i-1} \epsilon^{\frac{j}{2}} u^j_{px} \Big] \\ \label{danger.1}
& \hspace{33 mm} + \epsilon^{\frac{i}{2}} \sqrt{\epsilon} u^i_{eY} \Big[ \sum_{j=0}^{i-1} \epsilon^{\frac{j}{2}} v^j_p \Big] + \epsilon^{\frac{i-1}{2}} v^{i}_e \Big[ \sum_{j=1}^{i-1} \epsilon^{\frac{j}{2}} u^j_{py} \Big], \\ \n
R^{v,i}_E &:= \epsilon^{\frac{i-1}{2}} v^i_{ex} \sum_{j=0}^{i-1} \epsilon^{\frac{j}{2}} u^j_p + \epsilon^{\frac{i}{2}} u^i_e \sum_{j=0}^{i-1} \epsilon^{\frac{j}{2}} v^j_{px} \\ \label{RVE}
& \hspace{32 mm} + \epsilon^{\frac{i-1}{2}} v^i_e \sum_{j=0}^{i-1} \epsilon^{\frac{j}{2}} v^j_{py} + \sqrt{\epsilon} \epsilon^{\frac{i-1}{2}} v^i_{eY} \sum_{j=0}^{i-1} \epsilon^{\frac{j}{2}} v^j_p, \\ \n
E^{u,i}_e &:= \epsilon^i \Big[ u^i_e u^i_{ex} + v^i_eu^i_{eY} \Big] + \epsilon^{\frac{i}{2}} u^i_{ex} \sum_{j=1}^{i-1} \epsilon^{\frac{j}{2}} u^j_e + \epsilon^{\frac{i}{2}} u^{i}_e \sum_{j=1}^{i-1} \epsilon^{\frac{j}{2}} u^j_{ex} \\ \label{PureEul1}
& \hspace{35 mm} + \sqrt{\epsilon} \epsilon^{\frac{i}{2}} u^i_{eY} \Big[ \sum_{j=1}^{i-1} \epsilon^{\frac{j-1}{2}} v^j_e \Big] + \epsilon^{\frac{i-1}{2}} v^i_e \Big[ \sum_{j=1}^{i-1} \epsilon^{\frac{1}{2}} \epsilon^{\frac{j}{2}} u^j_{eY} \Big], \\ \n
E^{v,i}_e &:= \epsilon^{i - \frac{1}{2}} \Big[u^i_e v^i_{ex} + v^i_{e}v^i_{eY} \Big] + \epsilon^{\frac{i}{2}} u^i_e \Big[ \sum_{j=1}^{i-1} \epsilon^{\frac{j-1}{2}} v^j_{ex} \Big] + \epsilon^{\frac{i-1}{2}}v^i_{ex} \Big[ \sum_{j=1}^{i-1} \epsilon^{\frac{j}{2}} u^j_e \Big] \\ \label{PureEul2}
& \hspace{39 mm} + \sqrt{\epsilon}\epsilon^{\frac{i-1}{2}} v^i_e \sum_{j=1}^{i-1} \epsilon^{\frac{j-1}{2}} v^j_{eY} + \sqrt{\epsilon}\epsilon^{\frac{i-1}{2}} v^i_{eY} \sum_{j=1}^{i-1} \epsilon^{\frac{j-1}{2}} v^j_{e}.
\end{align}
\end{definition}

\begin{remark} \label{rmk.rem} The natural definition of the remainder term, $R^{u,i-1}$ should be: 
\begin{align}
R^{u,i-1} ``=" -\Delta_\epsilon \bar{u}_s^{(i-1)} + \bar{u}_s^{(i-1)} \bar{u}_{sx}^{(i-1)} + \bar{v}_s^{(i-1)} \bar{u}_{sy}^{(i-1)} + \bar{P}_{sx}^{(i-1)},
\end{align}

and the natural definition of $R^{u,i}_E$ would contain the lowest-order term, $\eps^{\frac{i-1}{2}} v^i_e u^0_{py}$. However, for $i < n$, the quantity of interest in (\ref{f.g.1}) is the sum, $R^{u,i-1} + R^{u,i}_E$, that is contributed to the next order (see the definition of the forcing in (\ref{INFI})). Thus, for convenience (see calculation \ref{CALi.5} and (\ref{take.1})), we add and subtract one factor $\eps^{\frac{i-1}{2}} v^i_e u^0_{py}$, which explains the definitions of (\ref{Rude}) and (\ref{danger.1}). This, however, will not be done for $i = n$. 
\end{remark}

The Navier-Stokes expansion reads: 
\begin{align} \nonumber
-\Delta_\epsilon u_s^{(i)} &+ u_s^{(i)} u_{sx}^{(i)} + v_s^{(i)} u_{sy}^{(i)} + P_{sx}^{(i)} = -\Delta_\epsilon \bar{u}_s^{(i-1)} + \bar{u}_s^{(i-1)} \bar{u}_{sx}^{(i-1)} + \bar{v}_s^{(i-1)} \bar{u}_{sy}^{(i-1)} + \bar{P}_{sx}^{(i-1)} \\ \nonumber
&- \epsilon^{\frac{i}{2}+1}\Delta u^i_e +  \epsilon^{\frac{i}{2}} \bar{u}_{sx}^{(i-1)} u^i_e + \epsilon^{\frac{i}{2}} \bar{u}_s^{(i-1)} u^i_{ex} + \epsilon^i u^i_e u^i_{ex} + \epsilon^{\frac{i-1}{2}} \bar{u}^{(i-1)}_{sy} v^i_e \\ \label{read.u}
& + \sqrt{\epsilon}\epsilon^{\frac{i}{2}}\bar{v}_s^{(i-1)}u^i_{eY}  + \epsilon^i v^i_e u^i_{eY} + \epsilon^{\frac{i}{2}} P^i_{ex} + \epsilon^i P^{1,a}_{ex} \\ \label{read.u.1}
& = \eps^{\frac{i}{2}} \Big[ u^i_{ex} + P^i_{ex} \Big]  + \eps^{\frac{i}{2}+1} \Delta u^i_e + R^{u,i-1}+ R^{u,i}_E + E^{u,i}_e + \eps^{i}P^{i,a}_{ex}. 
\end{align}

For the normal equation, the expansions read: 
\begin{align} \nonumber
-\Delta_\epsilon v_s^{(i)} &+ u_s^{(i)} v_{sx}^{(i)} + v_s^{(i)} v_{sy}^{(i)} + \frac{\partial_y}{\epsilon}P_{s}^{(i)} = -\Delta_\epsilon \bar{v}_s^{(i-1)} + \bar{u}_s^{(i-1)} \bar{v}_{sx}^{(i-1)} +  \bar{v}^{(i-1)}_s \bar{v}^{(i-1)}_{sy}   + \frac{\partial_y}{\epsilon} \bar{P}_s^{(i-1)} \\ \nonumber
&-\epsilon^{\frac{i+1}{2}}\Delta v^i_e + \epsilon^{\frac{i}{2}} \bar{v}^{(i-1)}_{sx} u^i_e + \epsilon^{\frac{i-1}{2}} \bar{u}_s^{(i-1)} v^i_{ex} + \epsilon^{i-\frac{1}{2}} u^i_e v^i_{ex}  + \epsilon^{\frac{i-1}{2}} \bar{v}^{(i-1)}_{sy} v^i_e \\ \label{read.v}
&+ \epsilon^{\frac{i}{2}} \bar{v}^{(i-1)}_s v^i_{eY} + \epsilon^{i-\frac{1}{2}} v^i_e v^i_{eY}+ \epsilon^{\frac{i-1}{2}} P^i_{eY} + \epsilon^{i-\frac{1}{2}} P^{1,a}_{eY} \\ \label{read.v.1}
& = \eps^{\frac{i-1}{2}} \Big[ v^i_{ex} + P^i_{eY} \Big] + \eps^{\frac{i+1}{2}} \Delta v^i_e + R^{v,i-1} + R^{v,i}_E + E^{v,i}_e + \eps^{i - \frac{1}{2}}P^i_{eY}. 
\end{align}

From here, we simply read off the highest order terms that are ``purely-Eulerian". All of the remaining terms will be treated in the next subsection. In (\ref{read.u}), these are at order $\epsilon^{\frac{i}{2}}$, and in (\ref{read.v}), these are at order $\epsilon^{\frac{i-1}{2}}$: 
\begin{equation}
\epsilon^{\frac{i}{2}}\Big[ u^i_{ex} + P^i_{ex} \Big] = 0, \hspace{3 mm} \epsilon^{\frac{i-1}{2}}\Big[ v^i_{ex} + P^i_{eY} \Big] = 0
\end{equation}

When paired with the divergence-free condition, we arrive at the equations that are taken for the $[u_e^i, v_e^i]$, which are the Cauchy-Riemann equations: 
\begin{align} \label{i.Euler}
u^{i}_{ex} + P^i_{ex} = 0, \hspace{3 mm} v^i_{ex} + P^i_{eY} = 0, \hspace{3 mm} u^i_{ex} + v^i_{eY} = 0. 
\end{align}

The boundary conditions for the $i'th$ Euler layer is $v^i_e(x,0) = -v^{i-1}_p(x,0)$. According to the inductive hypothesis, the decay rate of this boundary condition is: 
\begin{align} \nonumber
\Big|v^{i}_e(x,0)\Big| = \Big|v^{i-1}_p(x,0)\Big| &\le ||v^{i-1}_p(x,y)||_{L^2_y}^{\frac{1}{2}} ||v^{i-1}_{py}(x,y)||_{L^2_y}^{\frac{1}{2}} \le ||yv^{i-1}_{py}(x,y)||_{L^2_y}^{\frac{1}{2}} ||v^{i-1}_{py}(x,y)||_{L^2_y}^{\frac{1}{2}} \\ \label{improve}
&\le C(\sigma, n) x^{-\frac{3}{4}+\sigma_{i-1}}, \text{ for }i \ge 2. 
\end{align}

A comparison of (\ref{improve}) to (\ref{e.trace}) shows that the decay rate of the boundary condition has improved, enabling us to improve the Euler decay rates. Indeed, as the system (\ref{i.Euler}) is the identical system to the first Euler layer (and is in particular the Cauchy-Riemann equations), we may simply repeat the analysis given there to conclude: 
\begin{proposition}[Euler-i Layer] \label{i.Euler.lay} Let $i \ge 2$. The $i'$th Euler layer, defined by the Cauchy-Riemann equations (\ref{i.Euler}) taken with boundary conditions (\ref{improve}) satisfy the enhanced decay rates:
\begin{equation} \label{i.Euler.decay}
x^{k+m + \frac{3}{4}-\sigma_{i-1}}\Big| \partial_x^k \partial_Y^m v^i_e\Big| + x^{\frac{3}{4}-\sigma_{i-1}} \Big| u^i_e \Big| \le C(k, m, n).
\end{equation}
\end{proposition}
\begin{proof}
This follows by repeating the arguments in Section \ref{section.euler} with the enhanced boundary condition (\ref{improve}).
\end{proof}

\begin{remark} It is not possible to enhance the rates (\ref{i.Euler.decay}) much more (for instance, past $x^{-1}$ for $k = m = 0$. This is due to the restriction of $w \in (0,1)$ in Corollary \ref{cor.w}. 
\end{remark}

We have the simplified expression for (\ref{read.u.1}), (\ref{read.v.1})

\begin{lemma} According to definitions in (\ref{Rude}) - (\ref{PureEul2}), one has
\begin{align}  \label{f.g.1}
-\Delta_\epsilon u_s^{(i)} &+ u_s^{(i)} u_{sx}^{(i)} + v_s^{(i)} u_{sy}^{(i)} + P_{sx}^{(i)} = R^{u,i-1} + R^{u,i}_E + E^{u,i}_e + \eps^i P^{i,a}_{ex}, \\ \label{f.g.2}
-\Delta_\epsilon v_s^{(i)} &+ u_s^{(i)} v_{sx}^{(i)} + v_s^{(i)} v_{sy}^{(i)} + \frac{\partial_y}{\epsilon}P_{s}^{(i)} = R^{v,i-1} + R^{v,i}_E + E^{v,i}_e + \eps^{i-\frac{1}{2}}P^i_{eY}.
\end{align}
\end{lemma}
\begin{proof}
By construction in Proposition \ref{i.Euler.lay}, the $u^i_e, v^i_e$ are harmonic, and so the $\Delta u^i_e, \Delta v^i_e$ terms vanish identically. This then accounts for all of the terms in (\ref{read.u}) - (\ref{read.v}).
\end{proof}

\subsection{Construction of Prandtl Layer, $[u^i_p, v^i_p]$}

For this step in the construction, we suppose that $[u_s^{(i)}, v_s^{(i)}]$ have been constructed. The inductive hypothesis on these profiles are that the following remainders (according to the definitions in (\ref{Rude}) for $R^{u,i-1}$ and (\ref{Rvde}) for $R^{v,i-1}$) have been accumulated:
\begin{align} \n
R^{u,i-1} &= \epsilon^{\frac{i-1}{2}}\Big[ \epsilon u^{i-1}_{pxx}  +\sqrt{\epsilon} yu^0_{py} v^i_{eY} + u^{i-1}_{px} \sum_{j=1}^{i-1} \eps^{\frac{j}{2}} \{u^j_e + u^j_p \} \\ \label{comp.1}  & \hspace{15 mm} + \epsilon u^0_{py} \int_0^y \int_y^{y'} v^i_{eYY} dy'' dy' + v^{i-1}_p\Big(\sum_{j=1}^{i-1} \epsilon^{\frac{j}{2}} \{u^j_{py} + \sqrt{\epsilon}u^j_{eY}\} \Big) \Big], \\ \nonumber
R^{v, i-1} &= \epsilon^{\frac{i-1}{2}}\Big[ -\Delta_\epsilon v^{i-1}_p + u_s^{(i-1)} v^{i-1}_{px} + v^{(i-1)}_{sx} u^{i-1}_p + v_s^{(i-1)} v_{py}^{i-1} + v_{sy}^{(i-1)} v_p^{i-1} \\  \label{comp.2}
& \hspace{15 mm} + \epsilon^{\frac{i-1}{2}} u^{i-1}_p v^{i-1}_{px} + \epsilon^{\frac{i-1}{2}} v^{i-1}_p v^{i-1}_{py} \Big].
\end{align}

The induction will start at $i = 2$, and so (\ref{comp.1}) - (\ref{comp.2}) should be compared to (\ref{rem.pr.1}) for this case and to (\ref{comp.next}) for the general $i$ case. The relevant profile estimates, according to Propositions \ref{thm.euler.1}, \ref{i.Euler.lay} and \ref{Pp1} which hold inductively are: 
\begin{align}
&x^{k + m + \frac{1}{2}} |\p_x^k \p_Y^m v^1_e| + x^{\frac{1}{2}}|u^1_e| \le C(k, m, n), \\
&x^{k+m + \frac{3}{4}-\sigma_{j-1}}\Big| \partial_x^k \partial_Y^m v^j_e\Big| + x^{\frac{3}{4}-\sigma_{j-1}} \Big| u^j_e \Big| \le C(k, m, n) \text{ for } j = 2,..,i, \\
&||z^m \p_x^k u^j_p||_{P_k(\sigma_j)} \le C(k, m, n) \text{ for } j = 1,...,i-1.
\end{align}

By writing $\bar{u}_s^{(i)} = u_s^{(i)} + \epsilon^{\frac{i}{2}} u^{(i)}_p$, $\bar{v}_s^{(i)} = v_s^{(i)} + \epsilon^{\frac{i-1}{2}}v^{(i)}_p$, and expanding the Navier-Stokes equations, we obtain the two expansions: 
\begin{align} \nonumber
-\Delta_\eps \bar{u}_s^{(i)} &+ \bar{u}_s^{(i)} \bar{u}_{sx}^{(i)} +  \bar{v}_s^{(i)} \bar{u}_{sy}^{(i)} + \bar{P}_{sx}^{(i)} = -\Delta_\epsilon u_s^{(i)} + u_s^{(i)} u_{sx}^{(i)} + v_s^{(i)} u_{sy}^{(i)} + P^{(i)}_{sx} + \epsilon^{\frac{i+1}{2}} P^{i,a}_{px} \\ \label{read.u.p}
& + \epsilon^{\frac{i}{2}}\Big[ - \Delta_\eps u^i_p + u_{sx}^{(i)} u^i_p + u^{(i)}_s u^i_{px} + u^{(i)}_{sy} v^i_p + v^{(i)}_s u^i_{py} + \epsilon^{\frac{i}{2}} u^i_p u^i_{px} +\epsilon^{\frac{i}{2}} v^i_p u^i_{py} + P^i_{px} \Big], 
\end{align}

and:
\begin{align} \nonumber
-\Delta_\epsilon \bar{v}^{(1)}_s &+ \bar{u}^{(i)}_s \bar{v}^{(i)}_{sx} + \bar{v}^{(i)}_s \bar{v}^{(i)}_{sy} + \frac{\partial_y}{\epsilon} \bar{P}^{(i)}_{s}  = -\Delta_\epsilon v^{(i)}_s + u_s^{(i)} v_{sx}^{(i)} + v_s^{(i)} v^{(i)}_{sy} + \frac{P^{(i)}_{sy}}{\epsilon} + \epsilon^{\frac{i-1}{2}} P^{i,a}_{py} \\ \label{read.v.p}
& + \epsilon^{\frac{i}{2}}\Big[-\Delta_\eps v^i_p + u_s^{(1)} v^i_{px} + v^{(i)}_{sx} u^i_p + v^{(i)}_{sy}v^i_p + v^{(i)}_s v^i_{py} + \epsilon^{\frac{i}{2}} u^i_p v^i_{px} + \epsilon^{\frac{i}{2}} v^i_p v^i_{py} + \frac{P^{i}_{py}}{\epsilon} \Big] .
\end{align}

Define the forcing term to be those terms from (\ref{read.u.p}) which do not appear in the bracket: 
\begin{align} \nonumber
-\epsilon^{\frac{i}{2}}f^{(i)} &:=  -\Delta_\epsilon u_s^{(i)} + u_s^{(i)} u_{sx}^{(i)} + v_s^{(i)} u_{sy}^{(i)} + P^{(i)}_{sx} + \epsilon^{\frac{i+1}{2}} P^{i,a}_{px} \\ \label{INFI}
&= R^{u,i-1} + R^{u,i}_E + E^{u,i}_e + \epsilon^{i} P^{1,a}_{ex} + \epsilon^{\frac{i+1}{2}} P^{i,a}_{px},
\end{align}

where we have used (\ref{f.g.1}) to simplify the expression. The first step, here, is to introduce a potential Pressure which eliminates the ``purely" Eulerian terms from above:

\begin{definition} The i'th auxiliary Euler pressure, $P^{1,a}_e$, is defined by:
\begin{align} \label{grad.pres.1}
P^{1,a}_e := - \sum_{j=1}^{i-1} \epsilon^{\frac{j-i}{2}} v^i_e v^j_e - \sum_{j=1}^{i-1} \epsilon^{\frac{j-i}{2}} u^i_e u^j_e - \frac{1}{2}\Big|v^i_e\Big|^2 - \frac{1}{2}\Big| u^i_e \Big|^2.
\end{align}
\end{definition}

We may now check that: 
\begin{lemma} With the definition above (\ref{grad.pres.1}), $P^{1,a}_e$ serves as a gradient potential to eliminate the purely-Eulerian terms, $[E^{u,i}_e, E^{v,i}_e]$ from the expansion 
\begin{align} \label{grad.pres.2}
 \Big(\partial_x, \frac{\partial_y}{\epsilon} \Big) \epsilon^i P^{1,a}_e + \Big(E^{u,i}_e, E^{v,i}_e \Big) = 0. 
\end{align}
\end{lemma}
\begin{proof}
By scaling, we will write: 
\begin{equation}
 \Big(\partial_x, \frac{\partial_y}{\epsilon} \Big) \epsilon^i P^{1,a}_e = \Big(\partial_x, \frac{\partial_Y}{\sqrt{\epsilon}} \Big) \epsilon^i P^{1,a}_e. 
\end{equation}

We will go term by term through the definition in (\ref{grad.pres.1}), starting with:
\begin{align} \label{tbt.1}
\eps^i \partial_x \Big( -\sum_{j=1}^{i-1} \eps^{\frac{j-i}{2}} v^i_e v^j_e \Big) = -\sum_{j=1}^{i-1} \eps^{\frac{j+i}{2}} \Big( v^i_{ex} v^j_e + v^i_e v^j_{ex} \Big) = -\sum_{j=1}^{i-1} \eps^{\frac{j+i}{2}} \Big( u^i_{eY} v^j_e + v^i_e u^j_{eY} \Big).
\end{align}

Next, 
\begin{align} \label{tbt.2}
\eps^i \partial_x \Big( - \sum_{j=1}^{i-1} \eps^{\frac{j-i}{2}} u^i_e u^j_e \Big) = - \sum_{j = 1}^{i-1} \eps^{\frac{j+i}{2}} \Big( u^i_{ex} u^j_e + u^i_e u^j_{ex} \Big). 
\end{align}

Third, 
\begin{align} \label{tbt.3}
\eps^i \partial_x \Big( -\frac{1}{2} |v^i_e|^2 - \frac{1}{2}|u^i_e|^2 \Big) = -\eps^i v^i_e v^i_{ex} - \eps^i u^i_e u^i_{ex} = -\eps^i v^i_e u^i_{eY} - \eps^i u^i_e u^i_{ex}. 
\end{align}

Comparing these expressions, (\ref{tbt.1}) - (\ref{tbt.3}) to the expression (\ref{PureEul1}), one observes the exact cancellation: 
\begin{align}
\partial_x \eps^i P^{1,a}_e + E^{u,i}_E = 0. 
\end{align}

Next, we will move to the $\frac{\p_Y}{\sqrt{\eps}}$ terms: 
\begin{align} \label{tbt.4}
-\frac{\p_Y}{\sqrt{\eps}} \eps^i \Big( \sum_{j=1}^{i-1} \eps^{\frac{j-i}{2}} v^i_e v^j_e \Big) = -\sum_{j=1}^{i-1} \eps^{\frac{i+j}{2}-\frac{1}{2}} \Big( v^i_{eY}v^j_e + v^i_e v^j_{eY} \Big).
\end{align}

Next, 
\begin{align} \n
- \frac{\p_Y}{\sqrt{\eps}} \eps^i \Big( \sum_{j=1}^{i-1} \eps^{\frac{j-i}{2}} u^i_e u^j_e \Big) &= -\sum_{j=1}^{i-1} \eps^{\frac{i+j}{2}-\frac{1}{2}} \Big( u^i_{eY} u^j_e + u^i_{e} u^j_{eY} \Big) \\ \label{tbt.5}
& = -\sum_{j=1}^{i-1} \eps^{\frac{i+j}{2}-\frac{1}{2}} \Big( v^i_{ex} u^j_e + u^i_{e} v^j_{ex} \Big).
\end{align}

Third, 
\begin{align} \label{tbt.6}
\frac{\p_Y}{\sqrt{\eps}} \eps^i \Big(- |v^i_e|^2 - |u^i_e|^2 \Big) = -\eps^{i-\frac{1}{2}} v^i_e v^i_{eY} -\eps^{i-\frac{1}{2}} u^i_e u^i_{eY} =  -\eps^{i-\frac{1}{2}} v^i_e v^i_{eY} -\eps^{i-\frac{1}{2}} u^i_e v^i_{ex}.
\end{align}

Comparing these expressions, (\ref{tbt.4}) - (\ref{tbt.6}) to the expression (\ref{PureEul1}), one observes the exact cancellation: 
\begin{align}
\frac{\partial_Y}{\sqrt{\eps}} \eps^i P^{1,a}_e + E^{v,i}_E = 0. 
\end{align}

This establishes the desired result, (\ref{grad.pres.2}). 

\end{proof}
\begin{remark} One should notice the essential role played by the Cauchy-Riemann equations, $u^j_{eY} = v^j_{ex}$ in the equalities above. 
\end{remark}

Let us now turn to the terms outside of the bracket in (\ref{read.v.p}), which we also simplify via (\ref{f.g.2}) and subsequently via (\ref{grad.pres.2}):
\begin{align} \nonumber
-\Delta_\epsilon v^{(i)}_s + u_s^{(i)} v_{sx}^{(i)} &+ v_s^{(i)} v^{(i)}_{sy} + \frac{P^{(i)}_{sy}}{\epsilon} + \epsilon^{\frac{i-1}{2}} P^{i,a}_{py} \\ \nonumber
&= R^{v,i-1} + R^{v,i}_E + E^{v,i}_e + \epsilon^{i-\frac{1}{2}}P^{1,a}_{eY} + \epsilon^{\frac{i-1}{2}} P^{i,a}_{py} \\ \label{Rvdiss}
&= R^{v,i-1} + R^{v,i}_E + \epsilon^{\frac{i-1}{2}} P^{i,a}_{py}. 
\end{align}

Motivated by this, define the auxiliary Pressure via: 

\begin{definition} The i'th auxiliary Prandtl pressure, $P^{i,a}_P$ is defined via: 
\begin{align} \label{aux.pressure.1}
\epsilon^{\frac{i+1}{2}}P^{i,a}_P := \epsilon \int_y^{\infty} R^{v,i-1} + R^{v,i}_E. 
\end{align}
\end{definition}

Immediately from this definition, we have: 
\begin{lemma} According to the Definition \ref{aux.pressure.1}, the following identity holds: 
\begin{align} \label{Rvdiss.1}
-\Delta_\epsilon v^{(i)}_s + u_s^{(i)} v_{sx}^{(i)} &+ v_s^{(i)} v^{(i)}_{sy} + \frac{P^{(i)}_{sy}}{\epsilon} + \epsilon^{\frac{i-1}{2}} P^{i,a}_{py} = 0. 
\end{align}
\end{lemma}
\begin{proof}
By direct calculation from (\ref{aux.pressure.1}), 
\begin{align}
\eps^{\frac{i-1}{2}} P_P^{i,a} = - R^{v,i-1} - R_E^{v,i}. 
\end{align}

Combined with (\ref{Rvdiss}) then implies the desired result. 
\end{proof}

We are now able to rewrite the forcing for our equation, which was defined in (\ref{INFI}):
\begin{align} \label{f.i}
f^{(i)} =- \epsilon^{-\frac{i}{2}}\Big[ R^{u,i-1} + R_E^{u,i} + \epsilon^{\frac{i+1}{2}} P^{i,a}_{px} \Big]. 
\end{align}

The $\epsilon^{-\frac{i}{2}}$ factor arises as we are in the $i'th$ order of the construction. Let us briefly comment on the orders of the three-terms on the right-hand side of (\ref{f.i}). According to (\ref{comp.1}), it is evident that $R^{u,i-1}$ is order $\epsilon^{\frac{i}{2}}$. According to (\ref{danger.1}), it is clear that $R^{u,i}_E$ is order $\epsilon^{\frac{i}{2}}$. According to (\ref{aux.pressure.1}) coupled with (\ref{comp.2}) and (\ref{RVE}), it is clear that $P_P^{i,a}$ is order $1$. Thus, all of the terms in (\ref{f.i}) are order $1$ or higher. Reading off from (\ref{read.u.p}) - (\ref{read.v.p}), the system we will now be considering is: 
\begin{align} \label{sys.pr.i}
&(1+u^0_p) u_{px}^{i} + u_{sx}^{(i)}u_p^{i} + v_s^{(i)} u^i_{py} + u_{py}^{0}\Big(v^i_p - v^i_p(x,0)\Big) + P^i_{px} = u^i_{pyy}  + f^{(i)}, \\ \label{sys.pr.i.PIP}
&u^i_p(x,0) = -u^i_e(x,0), \hspace{3 mm} \lim_{y \rightarrow \infty} u^i_p(x,y) = 0, \hspace{3 mm} u^i_p(1,y) = U_i(y), \hspace{3 mm} P^i_{py} = 0.
\end{align}

Again, by evaluating the equation at $y = \infty$ gives $P^i_{px} = 0$. Coupled with the normal equation $P^i_{py} = 0$ then shows that $P^i_p = 0$. After this construction, $R^{u,i}, R^{v,i}$ contain the terms from (\ref{read.u.p}) - (\ref{read.v.p}) which were omitted in the construction of $[u^i_p, v^i_p]$:
\begin{lemma} For each $i$, with $[R^{u,i}, R^{v,i}]$ defined as in (\ref{Rude}) - (\ref{Rvde}), with $[u^i_p, v^i_p]$ taken to solve the system (\ref{sys.pr.i}), the following identities hold:
\begin{align} \label{comp.next}
R^{u,i} &= \epsilon^{\frac{i}{2}}\Big[ \epsilon u^{i}_{pxx} +\sqrt{\epsilon} yu^0_{py} v^{i+1}_{eY} + \epsilon u^0_{py} \int_0^y \int_y^{y'} v^{i+1}_{eYY} dy'' dy' \\ \nonumber & \hspace{15 mm} + v^{i}_p \sum_{j=1}^{i} \epsilon^{\frac{j}{2}} \{u^j_{py} + \sqrt{\epsilon}u^j_{eY}\} + u^{i}_{px} \sum_{j=1}^{i} \eps^{\frac{j}{2}} \{u^j_e + u^j_p \}  \Big], \\  \label{comp.2.next}
R^{v, i} &= \epsilon^{\frac{i}{2}}\Big[ -\Delta_\epsilon v^{i}_p + u_s^{(i)} v^{i}_{px} + v^{(i)}_{sx} u^{i}_p + v_s^{(i)} v_{py}^{i} + v_{sy}^{(i)} v_p^{(i)} + \epsilon^{\frac{i}{2}} u^{i}_p v^{i}_{px} + \epsilon^{\frac{i}{2}} v^{i}_p v^{i}_{py} \Big].
\end{align}
\end{lemma}
\begin{proof}
Starting with the definition in equation (\ref{Rude}), we have: 
\begin{align} \label{CALi.1}
R^{u,i} &:= -\Delta_\epsilon \bar{u}_s^{(i)} + \bar{u}_s^{(i)} \bar{u}_{sx}^{(i)} + \bar{v}_s^{(i)} \bar{u}_{sy}^{(i)} + \bar{P}_{sx}^{(i)} + \eps^{\frac{i}{2}} v^{i+1}_e u^0_{py} \\ \n
& =  -\Delta_\epsilon u_s^{(i)} + u_s^{(i)} u_{sx}^{(i)} + v_s^{(i)} u_{sy}^{(i)} + P^{(i)}_{sx} + \epsilon^{\frac{i+1}{2}} P^{i,a}_{px} \\ \n
& \hspace{30 mm} + \epsilon^{\frac{i}{2}}\Big[ - \Delta_\eps u^i_p + u_{sx}^{(i)} u^i_p + u^{(i)}_s u^i_{px} + u^{(i)}_{sy} v^i_p + v^{(i)}_s u^i_{py}  \\ \label{CALi.2}
& \hspace{30 mm} + \epsilon^{\frac{i}{2}} u^i_p u^i_{px} +\epsilon^{\frac{i}{2}} v^i_p u^i_{py} + P^i_{px} \Big] + \eps^{\frac{i}{2}} v_e^{i+1} u^0_{py} \\ \n
& = -\eps^{\frac{i}{2}}f^{(i)} + \epsilon^{\frac{i}{2}}\Big[ - \Delta_\eps u^i_p + u_{sx}^{(i)} u^i_p + u^{(i)}_s u^i_{px} + u^{(i)}_{sy} v^i_p + v^{(i)}_s u^i_{py}  \\ \label{CALi.3}
& \hspace{30 mm} + \epsilon^{\frac{i}{2}} u^i_p u^i_{px} +\epsilon^{\frac{i}{2}} v^i_p u^i_{py} + P^i_{px} \Big] + -\eps^{\frac{i}{2}} v_e^{i+1} u^0_{py} \\ \n
& =  \eps^{\frac{i}{2}}f^{(i)} + \epsilon^{\frac{i}{2}}\Big[ - u^i_{pyy} + u_{sx}^{(i)} u^i_p + (1+u^0_p) u^i_{px} + u^{0}_{py} v^i_p + v^{(i)}_s u^i_{py}  \\ \n
& \hspace{30 mm} + \epsilon^{\frac{i}{2}} u^i_p u^i_{px} +\epsilon^{\frac{i}{2}} v^i_p u^i_{py} + P^i_{px} - \eps u^i_{pxx} + \sum_{j =1}^i \eps^{\frac{j}{2}}\{u^j_e + u^j_p\} u^i_{px} \\ \label{CALi.4}
& \hspace{30 mm} + \sum_{j = 1}^{i} \eps^{\frac{j}{2}} \{u^j_{py} + \sqrt{\eps} u^j_{eY} \} v^i_p \Big] + \eps^{\frac{i}{2}} v_e^{i+1} u^0_{py} \\ \n
& =  -\eps^{\frac{i}{2}}f^{(i)} + \epsilon^{\frac{i}{2}}\Big[ - u^i_{pyy} + u_{sx}^{(i)} u^i_p + (1+u^0_p) u^i_{px} + u^{0}_{py} v^i_p + v^{(i)}_s u^i_{py}  \\ \n
& \hspace{30 mm} + \epsilon^{\frac{i}{2}} u^i_p u^i_{px} +\epsilon^{\frac{i}{2}} v^i_p u^i_{py} + P^i_{px} - \eps u^i_{pxx} + \sum_{j =1}^i \eps^{\frac{j}{2}}\{u^j_e + u^j_p\} u^i_{px} \\ \n
& \hspace{30 mm} + \sum_{j = 1}^{i} \eps^{\frac{j}{2}} \{u^j_{py} + \sqrt{\eps} u^j_{eY} \} v^i_p \Big] + \eps^{\frac{i}{2}} \Big[ -v^{i}_p(x,0) u^0_{py} + \sqrt{\eps}y u^0_{py} v^{i+1}_{eY} \\ \label{CALi.5}
& \hspace{30 mm} + \eps u^0_{py} \int_0^y \int_{y}^{y'} v^{i+1}_{eYY} \ud y'' \ud y' \Big] \\ \n
& = \eps^{\frac{i}{2}}\Big[ -\eps u^i_{pxx} + \sum_{j =1}^i \eps^{\frac{j}{2}}\{u^j_e + u^j_p\} u^i_{px} + \sum_{j = 1}^{i} \eps^{\frac{j}{2}} \{u^j_{py} + \sqrt{\eps} u^j_{eY} \} v^i_p  +  \sqrt{\eps}y u^0_{py} v^{i+1}_{eY} \\ \label{CALi.6}
& \hspace{30 mm} + \eps u^0_{py} \int_0^y \int_{y}^{y'} v^{i+1}_{eYY} \ud y'' \ud y' \Big].
\end{align}

For the calculation in (\ref{CALi.2}), we have used the expansion (\ref{read.u.p}). For the calculation in (\ref{CALi.3}), we have used (\ref{INFI}). For the calculation in (\ref{CALi.4}), we have simply rearranged terms. For the calculation in (\ref{CALi.5}), we have used: 
\begin{align}  \label{take.1}
\eps^{\frac{i}{2}} v^{i+1}_e u^0_{py} = \eps^{\frac{i}{2}} \Big[ -v^{i}_p(x,0) u^0_{py} + \sqrt{\eps}y u^0_{py} v^{i+1}_{eY} + \eps u^0_{py} \int_0^y \int_{y}^{y'} v^{i+1}_{eYY} \ud y'' \ud y' \Big].
\end{align}

Finally, for the calculation in (\ref{CALi.6}), we used (\ref{sys.pr.i}). The identity (\ref{comp.next}) then follows. For the $R^{v,i}$ contribution, we combine (\ref{Rvdiss.1}) with the expression (\ref{read.v.p}), and finally with the condition that $P^i_p = 0$ from (\ref{sys.pr.i.PIP}). 
\end{proof}

A comparison with (\ref{comp.1}) - (\ref{comp.2}) shows that this closes the construction. We now give estimates on the forcing term based on the inductively assumed decay rates and regularity in  (\ref{ind.1}), together with the $[u^0_p, v^0_p]$ bounds in (\ref{main.decay}) and the $[u^1_e, v^1_e]$ bounds in (\ref{euler.DECAY.main}).
\begin{lemma}[Forcing Estimates] Let $i \ge 2$. For any $m,k \ge 0$, with $f^{(i)}$ as in (\ref{f.i}),
\begin{align} \label{f.decay.1}
||z^m\partial_x^k f^{(i)}||_{L^2_y} \le C(k, m, \sigma, n) x^{-k -\frac{5}{4}+2\sigma_{i-1}}.
\end{align}
\end{lemma}
\begin{proof}

We start with the $R^{u,i}_E$ terms, which are defined in (\ref{danger.1}): 
\begin{align}
&\sum_{j=0}^{i-1} \epsilon^{\frac{j}{2}} ||u^i_{ex} u_p^j||_{L^2_y} \le \sum_{j=0}^{i-1} \epsilon^{\frac{j}{2}} ||u^i_{ex} x^{\frac{3}{2}}||_{L^\infty} x^{-\frac{3}{2}}||u^j_p||_{L^2_y} \lesssim x^{-\frac{5}{4}} ,\\
&\sum_{j=0}^{i-1}|| u^{i}_e  \epsilon^{\frac{j}{2}} u^j_{px}||_{L^2_y} \lesssim \sum_{j=0}^{i-1} \epsilon^{\frac{j}{2}} ||u^i_e x^{\frac{1}{2}} ||_{L^\infty} || x^{-\frac{1}{2}} u^j_{px}||_{L^2_y} \le x^{-\frac{5}{4}}, \\
 &\sum_{j=0}^{i-1} \epsilon^{\frac{j+1}{2}} ||  u^i_{eY}v^j_p ||_{L^2_y} \lesssim \sum_{j=0}^{i-1} \epsilon^{\frac{j+1}{2}} ||u^i_{eY} x^{\frac{3}{2}} ||_{L^\infty} x^{-\frac{3}{2}}||z^m v^j_p||_{L^2_y} \lesssim x^{-\frac{7}{4}} , \\ \label{control}
& \epsilon^{\frac{i-1}{2}}   \sum_{j=1}^{i-1} \epsilon^{\frac{j}{2}} || v^{i}_e u^j_{py}||_{L^2_y} \lesssim  \sum_{j=1}^{i-1} ||v^i_e x^{\frac{3}{4}-\sigma_{i-1}}||_{L^\infty}  ||x^{-\frac{3}{4}+\sigma_{i-1}} u^j_{py}||_{L^2_y} \le  x^{-\frac{5}{4} + 2\sigma_{i-1}}.
\end{align}

In (\ref{control}), we have used the enhanced Eulerian decay rate in (\ref{i.Euler.decay}). Next, we have the Prandtl contributions from $R^{u,i-1}$, according to (\ref{comp.1}):
\begin{align}
&\sqrt{\epsilon} || u^{i-1}_{pxx} ||_{L^2_y} \le \sqrt{\epsilon} x^{-\frac{7}{4}+\sigma_{i-1}}, \\
&||y u^0_{py}v^i_{eY}||_{L^2_y}  \le ||x^{-\frac{1}{4}} y u^0_{py}||_{L^2_y} x^{\frac{1}{4}} ||v^i_{eY}||_{L^\infty_y} \lesssim x^{-\frac{5}{4}}, \\
&\sqrt{\epsilon} || u^0_{py} \int_0^y \int_y^{y'} v^i_{eYY} dy' dy'' ||_{L^2_y} \le \sqrt{\epsilon} ||y^2 x^{-\frac{3}{4}} u^0_{py}||_{L^2_y} x^{\frac{3}{4}}||v^i_{eYY}||_{L^\infty_y} \le \sqrt{\epsilon} x^{-\frac{7}{4}}, \\
& \epsilon^{\frac{j-1}{2}} ||v^{i-1}_p u^j_{py}||_{L^2_y} \le  \epsilon^{\frac{j-1}{2}} x^{-\frac{5}{4}+2\sigma_{i-1}} , j \ge 1 \\
&  \epsilon^{\frac{j}{2}} ||v^{i-1}_p u^j_{eY}||_{L^2_y} \le  \epsilon^{\frac{j}{2}} x^{-2 +2\sigma_{i-1}}, j \ge 1, \\
&  \eps^{\frac{j}{2}} ||u^{i-1}_{px} \{u^j_e + u^j_p \}||_{L^2_y} \lesssim \sqrt{\eps} x^{-\frac{5}{4}+2\sigma_{i-1}} , j \ge 1.
\end{align}

We now move to the term $P_{px}^{i,a}$. For this we note: 
\begin{align} \label{P.aux.calc.1}
\epsilon^{-\frac{i}{2}}\epsilon^{\frac{i+1}{2}} ||P_{px}^{i,a}||_{L^2_y} \le \sqrt{\epsilon} \epsilon^{-\frac{i-1}{2}} ||\langle y \rangle^{\frac{3}{2}+\kappa} \partial_x \Big(R^{v,i-1} + R^{v,i}_E\Big)||_{L^\infty_y}.
\end{align}

We now turn to evaluating the right-hand side of (\ref{P.aux.calc.1}). Let us comment that it is essential at this stage that the pressure $P^{1,a}_e$ was introduced to eliminate all of the purely Eulerian contributions, as those terms cannot handle the weight of $y^{\frac{3}{2}+\kappa}$. Each term below has at least one Prandtl ($[u^i_p, v^i_p]$) factor, and so we may exchange the weight of $y^{\frac{3}{2}+\kappa}$ for decay, $x^{\frac{3}{4}+\frac{\kappa}{2}}$. For the calculations below, we simply pair this with the decay rates in (\ref{ind.1}):
\begin{align}
&|| y^{\frac{3}{2}+\kappa} \partial_x \Big( v^i_{ex} u^j_p \Big)||_{L^\infty_y} + || y^{\frac{3}{2}+\kappa} \partial_x \Big( u^i_{e} v^j_{px} \Big)||_{L^\infty_y} \lesssim x^{-\frac{3}{2}+2\sigma_{i-1} }, \\
 & || y^{\frac{3}{2}+\kappa} \partial_x \Big(v^i_e v^j_{py} \Big)||_{L^\infty_y}  \lesssim x^{-\frac{3}{2} + 2\sigma_{i-1}}\\
& ||  y^{\frac{3}{2}+\kappa} \partial_x \Big(v^i_{eY} v^j_{p} \Big)   ||_{L^\infty_y}\lesssim x^{-2 +2\sigma_{i-1} }, \\
& ||  y^{\frac{3}{2}+\kappa} \partial_x \Big( \Delta_\epsilon v^{i-1}_p \Big) \Big)   ||_{L^\infty_y}\lesssim x^{-\frac{5}{4}+2\sigma_{i-1}}, \\
& ||y^{\frac{3}{2}+\kappa} \partial_x \Big( u_s^{(i-1)} v^{i-1}_{px} \Big) ||_{L^\infty_y} + ||y^{\frac{3}{2}+\kappa} \partial_x \Big(v_{sx}^{(i-1)}u_p^{i-1} \Big) ||_{L^\infty_y} \lesssim x^{-\frac{3}{2}+2\sigma_{i-1} }, \\
& ||y^{\frac{3}{2}+\kappa} \partial_x \Big( v_s^{(i-1)} v_{py}^{i-1} \Big) ||_{L^\infty_y} + ||y^{\frac{3}{2}+\kappa} \partial_x \Big( v_{sy}^{(i-1)} v_p^{i-1} \Big) ||_{L^\infty_y} \lesssim x^{-\frac{3}{2}+2\sigma_{i-1}}, \\
& ||y^{\frac{3}{2}+\kappa} \partial_x \Big( \epsilon^{\frac{i-1}{2}} u_p^{i-1} v_{px}^{i-1} + v_p^{i-1} v_{py}^{i-1} \Big) ||_{L^\infty_y} \lesssim \epsilon^{\frac{i-1}{2}} x^{-\frac{5}{4}+2\sigma_{i-1} }. 
\end{align}

This proves the claim for $k = 0, m = 0$. The general case can be obtained in an identical fashion, upon noticing that powers of $z$ will not influence the estimates when accompanied by Prandtl profiles, which are present in each term above. 

\end{proof}

Indeed, the boundary condition in (\ref{sys.pr.i}) has improved, referring to (\ref{i.Euler.decay}): 
\begin{align} \label{en.e}
\Big| u^i_e(x,0) \Big| \le x^{-\frac{3}{4}+\sigma_{i-1}}, 
\end{align}

according to (\ref{i.Euler.decay}). We can expect enhanced decay due to the enhanced decay of the boundary data. This is the content of the following:

\begin{proposition}[Construction of $i'th$ Prandtl Layer] \label{Lpi} For $i \ge 2$, there exists a solution $[u^i_p, v^i_p]$ to the system (\ref{sys.pr.i}), satisfying for any $m, k \ge 0$: 
\begin{align} \label{pr.i.b}
|| z^m x^{\frac{1}{4}} \partial_x^k u^i_p||_{P_k(\sigma_i)} + x^{k+1-\sigma_i} ||z^m \p_x^k v^i_p||_{L^\infty_y} \le C(k, m, n). 
\end{align}
\end{proposition}

This will be proven in several steps. We homogenize the boundary conditions by defining: 
\begin{align} \nonumber
&u = u_p + \chi(y) u^{i}_e(x,0), \hspace{2 mm} v = v_p - v_p(x,0) + u^{n}_{ex}(x,0) I_\chi(y), \hspace{2 mm} I_\chi(y) = \int_y^\infty \chi(\theta) d\theta,
\end{align}

where $\chi$ is a localized cut-off function selected to have mean-zero. These profiles satisfy: 
\begin{align} \label{npr.1}
(1+u^0_p) u_x - u_{yy} + \mathcal{P}(u,v) = f^{(i)}  + \mathcal{J},
\end{align} 

where 
\begin{align} \nonumber
\mathcal{J} := u_s^{(i)} \chi(y) u^{i}_{ex}(x,0) &- \chi'' u^{i}_e(x,0) - \chi(y) u^{(i)}_{sx} u^{n}_e(x,0) \\ 
&- u^0_{py} I_\chi(y) u^{i}_{ex}(x,0) + v_s^{(i)}\chi'(y) u^{i}_e(x,0). 
\end{align}

We will record the estimate on $\mathcal{J}$, which follows directly from (\ref{en.e}):
\begin{align} \label{est.math.J}
| \mathcal{J} | \lesssim \langle y \rangle^{-N} x^{-\frac{3}{4}+\sigma_{i-1}}. 
\end{align}

We introduce the stream function, 
\begin{align}
\psi(x,y) := \int_y^\infty u(x,y') dy', \hspace{3 mm} \psi_y = -u, \hspace{3 mm} \psi_x = v - v(x,\infty). 
\end{align}

The stream function satisfies the following system: 
\begin{align} \label{str.sys.1.1}
\Big(\partial_x - \partial_{yy}\Big) \psi = \int_y^\infty u^0_p u_x + \int_y^\infty \mathcal{P} + \int_y^\infty f^{(n)} + \int_y^\infty \mathcal{J}, \\
\psi_y(x,0) = u(x,0) = 0, \hspace{3 mm} \psi(x,\infty) = 0, \hspace{3 mm} \psi(1,y) = \int_y^\infty U_{n}(y') dy'.
\end{align}

For technical reasons (due to certain integrals being critical), it is necessary to start with the stream-function formulation in order to obtain the desired, enhanced decay.  
\begin{lemma} The stream function $\psi$ solving the system (\ref{str.sys.1.1}) satisfies:
\begin{align} \nonumber
\sup_x \int x^{-\frac{1}{2}-2\sigma_i} \psi^2 + \int \int x^{-\frac{3}{2}-2\sigma_i} \psi^2 &+ \int \int x^{-\frac{1}{2}-2\sigma_i} u^2 \\ \label{see.2}
& \lesssim C + \mathcal{O}(\delta) \int \int v_y^2 x^{\frac{3}{2}-2\sigma_i}.  
\end{align}
\end{lemma}
\begin{proof}

Applying the multiplier $M = x^{-\frac{1}{2}-2\sigma_i} \psi$ to the above system yields the following terms on the left-hand side:
\begin{align}
\partial_x \int x^{-\frac{1}{2}-2\sigma_i} \psi^2 + \int x^{-\frac{3}{2}-2\sigma_i} \psi^2 + \int x^{-\frac{1}{2}-2\sigma_i} \psi_y^2 \lesssim \int \Big(\partial_x - \partial_{yy} \Big) \psi \cdot \psi x^{-\frac{1}{2}-2\sigma_i}. 
\end{align}

For the first term on the right-hand side of (\ref{str.sys.1.1}), we will now give the bound via Hardy's inequality, and (\ref{main.decay}): 
\begin{align} \nonumber
\int x^{-\frac{1}{2}-2\sigma_i} &\psi \int_y^\infty u^0_p u_x dy' dy \le ||x^{-\frac{3}{4}-\sigma_i} \psi||_{L^2_y} ||x^{\frac{1}{4}} \int_y^\infty u^0_p u_x dy' ||_{L^2_y} \\ 
& \le ||x^{-\frac{3}{4}-\sigma_i} \psi||_{L^2_y} ||x^{\frac{1}{4}-\sigma_i} y u^0_p u_x||_{L^2_y} \le \mathcal{O}(\delta) ||x^{-\frac{3}{4}-\sigma_i} \psi||_{L^2_y} ||x^{\frac{3}{4}-\sigma_i} u_x||_{L^2_y}. 
\end{align}

Next, let us address the profile terms in $\mathcal{P}$. For the first term from $\mathcal{P}$, we will split: 
\begin{align} \label{SNEW}
\int x^{-\frac{1}{2}-2\sigma_i} &\psi \int_y^\infty u^{(i)}_{sx} u \ud y' \ud y \\ \n
& =  \sum_{j=0}^{i-1} \int x^{-\frac{1}{2}-2\sigma_i} \psi \int_y^\infty \eps^\frac{j}{2} u^{j}_{px} u \ud y' \ud y + \sum_{j=1}^{i} \int x^{-\frac{1}{2}-2\sigma_i} \psi \int_y^\infty  \eps^{\frac{j}{2}} u^{j}_{ex} u \ud y' \ud y 
\end{align}

For the first term above, we apply the Hardy inequality in $y$ (it suffices to consider $j = 0$):
\begin{align} \nonumber 
\Big| \int x^{-\frac{1}{2}-2\sigma_i} \psi \int_y^\infty u^{0}_{px} u \ud y' \ud y \Big| &\le ||x^{-\frac{3}{4}-\sigma_i} \psi||_{L^2_y} || x^{\frac{1}{4}-\sigma_i} \int_y^\infty u^{0}_{px} u \ud y' ||_{L^2_y} \\ \n
& \le ||x^{-\frac{3}{4}-\sigma_i} \psi||_{L^2_y} ||x^{\frac{1}{2}} y u_{px}^{0}||_{L^\infty} || x^{-\frac{1}{4}-\sigma_i} u ||_{L^2_y} \\ \label{str.pr.1}
& \le \mathcal{O}(\delta)  ||x^{-\frac{3}{4}-\sigma_i} \psi||_{L^2_y}^2 + \mathcal{O}(\delta)  || x^{-\frac{1}{4}-\sigma_i} u ||_{L^2_y}^2. 
\end{align}

We have used the smallness which is guaranteed by (\ref{main.decay}). For the second term in (\ref{SNEW}), we will integrate by parts in $y$ and recall $\psi \xrightarrow{y \rightarrow \infty} 0$, (again it suffices to consider $j = 1$):
\begin{align} \label{SNEW.2}
\int x^{-\frac{1}{2}-2 \sigma_i} \psi \int_y^\infty \sqrt{\eps} u^1_{ex} \p_y \psi = \int x^{-\frac{1}{2}-2\sigma_i} \sqrt{\eps} u^1_{ex} \psi^2 + \int x^{-\frac{1}{2}-2\sigma_i} \psi \int_y^\infty \eps u^{1}_{exY} \psi.
\end{align}

The first term in (\ref{SNEW.2}):
\begin{align}
\Big| \int x^{-\frac{1}{2}-2\sigma_i} \sqrt{\eps} u^1_{ex} \psi^2 \Big| \le \sqrt{\eps} ||u^1_{ex} x||_{L^\infty}||x^{-\frac{3}{4}-\sigma_i} \psi||_{L^2_y}^2 \lesssim \sqrt{\eps} ||x^{-\frac{3}{4}-\sigma_i} \psi||_{L^2_y}^2. 
\end{align}

For the second term in (\ref{SNEW.2}), we apply Hardy in $y$ and appeal to estimate (\ref{e.ss.1}) with $k = 2, s = 1$:
\begin{align} \n
\Big| \int x^{-\frac{1}{2}-2\sigma_i} \psi \int_y^\infty \eps u^1_{exY} \psi \Big| &\le \sqrt{\eps} ||u^1_{exY} Yx||_{L^\infty} ||x^{-\frac{3}{4}-\sigma_i} \psi||_{L^2_y}^2 \\
& = \sqrt{\eps} ||v^1_{exx} Yx||_{L^\infty} ||x^{-\frac{3}{4}-\sigma_i} \psi||_{L^2_y}^2 \lesssim \sqrt{\eps} ||x^{-\frac{3}{4}-\sigma_i} \psi||_{L^2_y}^2. 
\end{align}

Let us now move to the second term in $\mathcal{P}$. One integration by parts in $y$ produces:
\begin{align} \label{str.pr.2}
& \int x^{-\frac{1}{2}-2\sigma_i} \psi \int_y^\infty v^{(i)}_s u_y dy' dy = \int x^{-\frac{1}{2}-2\sigma_i}  v_s^{(i)} \psi \psi_y + \int x^{-\frac{1}{2}-2\sigma_i} \psi \int_y^\infty u^{(i)}_{sx} u dy' dy.
\end{align}

For the second term in (\ref{str.pr.2}), one notices this is exactly the term on the left-hand side of (\ref{SNEW}). For the first term in (\ref{str.pr.2}), we estimate:  
\begin{align} \nonumber
\Big| \int x^{-\frac{1}{2}-2\sigma_i} v^{(i)}_s \psi \psi_y \Big| &\le ||v^{(i)}_{s} x^{\frac{1}{2}}||_{L^\infty} \Big( ||x^{-\frac{3}{4}-\sigma_i} \psi||_{L^2_y}^2 + ||x^{-\frac{1}{4}-\sigma_i} \psi_y||_{L^2_y}^2 \Big) \\
& \le \mathcal{O}(\delta) \Big( ||x^{-\frac{3}{4}-\sigma_i} \psi||_{L^2_y}^2 + ||x^{-\frac{1}{4}-\sigma_i} \psi_y||_{L^2_y}^2 \Big).
\end{align}

We have used the smallness guaranteed by (\ref{main.v}). The final term from $\mathcal{P}$ is:
\begin{align} \nonumber
\Big| \int x^{-\frac{1}{2}-2\sigma_i} \psi \int_y^\infty u^0_{py} v dy' dy \Big| &\le ||x^{-\frac{3}{4}-\sigma_i}\psi||_{L^2_y} ||x^{\frac{1}{4}-\sigma_i} \int_y^\infty u^0_{py} v||_{L^2_y} \\ \nonumber
&\le ||x^{-\frac{3}{4}-\sigma_i} \psi||_{L^2_y} x^{\frac{1}{4}-\sigma_i} ||yu^0_{py} v||_{L^2_y} & \\ \n
&\le ||y^2x^{-\frac{1}{2}}u^0_{py}||_{L^\infty} ||x^{-\frac{3}{4}-\sigma_i} \psi||_{L^2_y}||v_y x^{\frac{3}{4}-\sigma_i}||_{L^2_y} \\
& \le \mathcal{O}(\delta) ||x^{-\frac{3}{4}-\sigma_i} \psi||_{L^2_y}||v_y x^{\frac{3}{4}-\sigma_i}||_{L^2_y}.
\end{align}

Consolidating the previous estimates, taking $\delta$ small enough relative to $\sigma_i$ (and consequently relative to large $n$, according to (\ref{sigma.i})) and applying Young's inequality gives: 
\begin{align} \nonumber
\partial_x \int x^{-\frac{1}{2}-2\sigma_i} \psi^2 &+ \int x^{-\frac{3}{2}-2\sigma_i} \psi^2 + \int x^{-\frac{1}{2}-2\sigma_i} \psi_y^2 \\ \label{see.1}
& \lesssim \mathcal{O}(\delta) \int v_y^2 x^{\frac{3}{2}-2\sigma_i} + \int x^{-\frac{1}{2}-2\sigma_i} \psi \int_y^\infty \Big[ f^{(i)} + \mathcal{J} \Big].  
\end{align}

For the forcing terms, we appeal to the bounds in (\ref{f.decay.1}), coupled with (\ref{est.math.J}):
\begin{align}  
||x^{-\frac{1}{2}-2\sigma_i} \psi \int_y^\infty f^{(i)} \ud y' ||_{L^2_y} &\le ||x^{-\frac{3}{4}-\sigma_i} \psi||_{L^2_y} || x^{\frac{1}{4}-\sigma_i} \int_y^\infty f^{(i)} \ud y' ||_{L^2_y} \\ 
& \le ||x^{-\frac{3}{4}-\sigma_i} \psi||_{L^2_y} || x^{\frac{1}{4}-\sigma_i} y f^{(i)}  ||_{L^2_y} \\ \label{abnot}
& \le ||x^{-\frac{3}{4}-\sigma_i} \psi||_{L^2_y} || x^{\frac{3}{4}-\sigma_i} z f^{(i)}  ||_{L^2_y} \\ 
& \le || x^{-\frac{3}{4}-\sigma_i} \psi||_{L^2_y}  x^{-\frac{1}{2}-\sigma_i} \\
& \le \frac{1}{100,000} || x^{-\frac{3}{4}-\sigma_i} \psi||_{L^2_y}^2 + C x^{-1-\sigma'},
\end{align}

where $\sigma' > 0$. We have used (\ref{sigma.i}) via: $\sigma_i - 2\sigma_{i-1} = 3\sigma_{i-1} - 2\sigma_{i-1} = \sigma_{i-1}$, to calculate: 
\begin{align}
||x^{\frac{3}{4}-\sigma_i} zf^{(i)}||_{L^2_y} \lesssim x^{\frac{3}{4}-\sigma_i} x^{-\frac{5}{4}+2\sigma_{i-1}} =  Cx^{-\frac{1}{2}-\sigma_i + 2\sigma_{i-1}} = C x^{-\frac{1}{2}-\sigma_{i-1}}. 
\end{align}

Next, through Young's inequality
\begin{align} \n
| \int x^{-\frac{1}{2}-2\sigma_i} \psi \int_y^\infty \mathcal{J} \ud y' | &\lesssim \int x^{-\frac{1}{2}-2\sigma_i} \psi \langle y \rangle^{-N} x^{-\frac{3}{4}+\sigma_{i-1}} \\ 
& \lesssim \frac{1}{100,000} || x^{-\frac{3}{4}-\sigma_i} \psi||_{L^2_y}^2 + Cx^{-1-\sigma'},
\end{align}

where $\sigma' > 0$. Inserting these into (\ref{see.1}), relabeling $\psi_y = u$, and integrating in $x$, one obtains the desired result in (\ref{see.2}).

\end{proof}

From here, we may repeat the calculations in Lemmas \ref{lemma.pr.1.e} - \ref{lemma.pr.1.p}:
\begin{lemma} The solutions $[u,v]$ to the system (\ref{sys.pr.i}) satisfy the following inequality:
\begin{align} \label{see.3}
\sup_x \int u^2 x^{\frac{1}{2}-2\sigma_i} + \int \int u_y^2 x^{\frac{1}{2}-2\sigma_i} \lesssim C  + \mathcal{O}(\delta) \int \int v_y^2 x^{\frac{3}{2}-2\sigma_i} + \int \int u^2 x^{-\frac{1}{2}-2\sigma_i}.
\end{align}
\end{lemma}
\begin{proof}
We multiply both sides of equation (\ref{npr.1}) by $ux^{\frac{1}{2}-2\sigma_i}$ and integrate by parts. This gives on the left-hand side: 
\begin{align} \n
\int \Big((1+u^0_p)\partial_x - \partial_{yy} \Big) u \cdot u x^{\frac{1}{2}-2\sigma_i} &\gtrsim \partial_x \int (1+u^0_p) u^2 x^{\frac{1}{2}-2\sigma_i} + \int u_y^2 x^{\frac{1}{2}-2\sigma_i} \\ \label{cruc.PN}
&  - \int u^0_{px} u^2 x^{\frac{1}{2}-2\sigma_i} - \int  (1 + u^0_p)  u^2 x^{-\frac{1}{2}-2\sigma_i}. 
\end{align}

One now sees that it is crucial to have controlled the final term in (\ref{cruc.PN}), which is the purpose of (\ref{see.1}). Moving next to the sequence of terms in $\mathcal{P}(u,v)$:
\begin{align}
&| \int u^{(i)}_{sx} u^2 x^{\frac{1}{2}-2\sigma_i} | \le ||u^{(i)}_{sx} x||_{L^\infty}  \int u^2 x^{-\frac{1}{2}-2\sigma_i} \le \mathcal{O}(\delta) ||u x^{-\frac{1}{4}-\sigma_i}||_{L^2_y}^2, \\ \n
&| \int v^{(i)}_s u_y u x^{\frac{1}{2}-2\sigma_i} | = | \int \frac{v^{(i)}_{sy}}{2} u^2 x^{\frac{1}{2}-2\sigma_i} | \\ 
& \hspace{30 mm} \lesssim ||u^{(i)}_{sx} x||_{L^\infty}  ||u x^{-\frac{1}{4}-\sigma_i}||_{L^2_y}^2 \le \mathcal{O}(\delta)  ||u x^{-\frac{1}{4}-\sigma_i}||_{L^2_y}^2, \\ \n
& | \int u^0_{py} uv x^{\frac{1}{2}-2\sigma_i} | \lesssim ||u^0_{py} y||_{L^\infty} ||\frac{v}{y} x^{\frac{3}{4}-\sigma_i}||_{L^2_y} ||ux^{-\frac{1}{4}-\sigma_i}||_{L^2_y} \\ 
&  \hspace{30 mm} \lesssim \mathcal{O}(\delta) ||v_y x^{\frac{3}{4}-\sigma_i}||_{L^2_y} ||ux^{-\frac{1}{4}-\sigma_i}||_{L^2_y}.
\end{align}

We can summarize the above terms by writing: 
\begin{align}
| \int \mathcal{P} \cdot u x^{\frac{1}{2}-\sigma_i} | \le \mathcal{O}(\delta) ||v_y x^{\frac{3}{4}-\sigma_i} ||_{L^2_y}^2 + \mathcal{O}(\delta) ||ux^{-\frac{1}{4}-\sigma_i}||_{L^2_y}^2. 
\end{align}

Note that the smallness in the above estimates is guaranteed by (\ref{main.decay}), (\ref{main.v}), and (\ref{euler.DECAY.main}).  We now move to the forcing terms, $f$ and $\mathcal{J}$:
\begin{align} \n
| \int f^{(i)} \cdot ux^{\frac{1}{2}-2\sigma_i} | &\le ||f^{(i)} x^{\frac{3}{4}-\sigma_i}||_{L^2_y} ||ux^{-\frac{1}{4}-\sigma_i}||_{L^2_y} \lesssim  x^{-\frac{1}{2}-\sigma'} ||ux^{-\frac{1}{4}-\sigma_i}||_{L^2_y} \\
& \le  Cx^{-1-\sigma'} + \frac{1}{100,000} ||ux^{-\frac{1}{4}-\sigma_i}||_{L^2_y}^2. 
\end{align}

Next, we use the Hardy inequality:
\begin{align} \nonumber
| \int \mathcal{J} \cdot ux^{\frac{1}{2}-2\sigma_i} | &\lesssim \int \langle y \rangle^{-N} x^{-\frac{3}{4}+\sigma_{i-1}} |u|x^{\frac{1}{2}-2\sigma_i} \\ \n
& \lesssim ||\langle y \rangle^{-N+1}||_{L^2_y} x^{-\frac{1}{2}-\sigma'}  ||\frac{u}{y}||_{L^2_y} x^{\frac{1}{4}-\sigma_i} \\
& \le C x^{-1-\sigma'} + \frac{1}{100,000} ||u_y x^{\frac{1}{4}-\sigma_i}||_{L^2_y}^2. 
\end{align}

Combining the above estimates together, one obtains, for some $\sigma' > 0$ small,
\begin{align}
\partial_x \int (1+u^0_p) u^2 x^{\frac{1}{2}-2\sigma_i} + \int u_y^2 x^{\frac{1}{2}-2\sigma_i} \lesssim C x^{-1-\sigma'} + \mathcal{O}(\delta) \int v_y^2 x^{\frac{3}{2}-2\sigma_i} + \int u^2 x^{-\frac{1}{2}-2\sigma_i}, 
\end{align}

and so integrating in $x$ gives the desired result.

\end{proof}

The third step in establishing the desired bounds is: 
\begin{lemma} For solutions $[u,v]$ to the system in (\ref{npr.1}) one has the following positivity estimate: 
\begin{align} \label{see.5}
\sup_x \int u_y^2 x^{\frac{3}{2}-2\sigma_i} + \int \int u_x^2 x^{\frac{3}{2}-2\sigma_i} \lesssim C + \int \int u_y^2 x^{\frac{1}{2}-2\sigma_i} + \mathcal{O}(\delta) \int \int u^2 x^{-\frac{1}{2}-2\sigma_i}.
\end{align} 
\end{lemma}
\begin{proof}
We apply the multiplier $u_x x^{\frac{3}{2}-2\sigma_i}$ to the system (\ref{npr.1}). This gives on the left-hand side: 
\begin{align}
\int \Big( (1+u^0_p) \partial_x - \partial_{yy} \Big) u \cdot u_x x^{\frac{3}{2}-2\sigma_i} \gtrsim \int (1+u^0_p) u_x^2 x^{\frac{3}{2}-2\sigma_i} + \partial_x \int u_y^2 x^{\frac{3}{2}-2\sigma_i} - \int u_y^2 x^{\frac{1}{2}-2\sigma_i}
\end{align}

Let us now move to the terms in $\mathcal{P}$:
\begin{align} \n
&| \int u^{(i)}_{sx} u u_x x^{\frac{3}{2}-2\sigma_i} | \le ||u^{(i)}_{sx} x||_{L^\infty} || u x^{-\frac{1}{4}-\sigma_i}||_{L^2_y} ||u_x x^{\frac{3}{4}-\sigma_i}||_{L^2_y} \\ 
& \hspace{30 mm} \lesssim \mathcal{O}(\delta) || u x^{-\frac{1}{4}-\sigma_i}||_{L^2_y} ||u_x x^{\frac{3}{4}-\sigma_i}||_{L^2_y},\\ \n
&| \int v^{(i)}_s u_y u_x x^{\frac{3}{2}-2\sigma_i} | \le ||v_s^{(i)} x^{\frac{1}{2}}||_{L^\infty} ||u_y x^{\frac{1}{4}-\sigma}||_{L^2_y} ||u_x x^{\frac{3}{4}-\sigma_i}||_{L^2_y} \\  
& \hspace{30 mm}  \le \mathcal{O}(\delta) ||u_y x^{\frac{1}{4}-\sigma_i}||_{L^2_y} ||u_x x^{\frac{3}{4}-\sigma_i}||_{L^2_y},  \\ \n
& | \int u^0_{py} u_x v x^{\frac{3}{2}-2\sigma_i} | \lesssim ||u^0_{py} y||_{L^\infty} ||\frac{v}{y} x^{\frac{3}{4}-\sigma_i}||_{L^2_y} ||u_xx^{\frac{3}{4}-\sigma_i}||_{L^2_y} \\ 
&  \hspace{30 mm} \lesssim \mathcal{O}(\delta) ||v_y x^{\frac{3}{4}-\sigma_i}||_{L^2_y}^2.
\end{align}

We can summarize this contribution via: 
\begin{align}
| \int \mathcal{P} \cdot u_x x^{\frac{3}{2}-2\sigma_i} | \le \mathcal{O}(\delta) ||u_x x^{\frac{3}{4}-\sigma_i}||_{L^2_y}^2 + \mathcal{O}(\delta) ||u_y x^{\frac{1}{4}-\sigma_i}||_{L^2_y}^2 + \mathcal{O}(\delta) ||ux^{-\frac{1}{4}-\sigma_i}||_{L^2_y}^2. 
\end{align}

Again, the smallness is guaranteed by (\ref{main.decay}), (\ref{main.v}), and (\ref{euler.DECAY.main}). Next, the forcing is given in the same way, due to the choice of $\sigma_i$ relative to $\sigma_{i-1}$:
\begin{align} \n
| \int f^{(i)} \cdot u_x x^{\frac{3}{2}-2\sigma_i} | &\le ||f^{(i)} x^{\frac{3}{4}-\sigma_i}||_{L^2_y}  ||u_x x^{\frac{3}{4}-\sigma_i}||_{L^2_y} \\
& \lesssim x^{-\frac{1}{2}-\sigma'} ||u_x x^{\frac{3}{4}-\sigma_i}||_{L^2_y} \lesssim x^{-1-\sigma'} + \frac{1}{100,000} ||u_x x^{\frac{3}{4}-\sigma_i}||_{L^2_y}^2. 
\end{align}

We now integrate up in $x$ up to some fixed point $X_1$:
\begin{align}  \label{PREN.J}
\int u_y^2(X_1) X_1^{\frac{3}{2}-2\sigma_i} &+ \int_1^{X_1} \int u_x^2 x^{\frac{3}{2}-2\sigma_i} \\ \n
& \lesssim C + \int_1^{X_1} \int u_y^2 x^{\frac{1}{2}-2\sigma_i} + \mathcal{O}(\delta) \int_1^{X_1} \int u^2 x^{-\frac{1}{2}-2\sigma_i} + \int_1^{X_1} \int \mathcal{J} \cdot u_x x^{\frac{3}{2}-2\sigma_i}. 
\end{align}

The last part is to control the $\mathcal{J}$ term above:
\begin{align} 
 &\int_1^{X_1} \int \chi''(y) x^{\frac{3}{2}-2\sigma_i} u^{i}_e(x,0) u_x \\ \nonumber
 &= -\int_1^{X_1} \int \chi''(y) u \partial_x\{ u^i_e(x,0) x^{\frac{3}{2}-2\sigma_i} \} - \int_{x=1}  u^{i}_e(0,0) u \chi''(y) dy \\ 
 & \hspace{40 mm} + \int_{x = X_1} X_1^{\frac{3}{2}-2\sigma_i} u^{i}_e(X_1,0) \chi''(y) u(X_1,y) dy \\ \nonumber
 & = -\int_1^{X_1} \int \chi''(y) u \partial_x\{ u^{i}_e(x,0) x^{\frac{3}{2}-2\sigma_i} \} + C \\ 
 & \hspace{40 mm} - C \int_{x=X_1} X_1^{\frac{3}{2}-2\sigma_i} u^{i}_e(X_1, 0)\chi'(y) u_y(X_1,y) dy \\ 
 &\le \int_1^{X_1} \int \chi'(y) u_y \partial_x  \{u^{i}_e(x,0) x^{\frac{3}{2}-2\sigma_i} \} + C + \frac{1}{100,000} \int_{x=X_1} X_1^{\frac{3}{2}-2\sigma_i} u_y^2(X_1) \\
 & \le C ||u_y x^{\frac{1}{4}-\sigma_i}||_{L^2} ||\chi'(y) x^{-\frac{1}{2}-\sigma_i}||_{L^2} + C \le C + \frac{1}{100,000} ||u_y x^{\frac{1}{4}-\sigma_i}||_{L^2}^2. 
\end{align}

We have absorbed the $x$ boundary contribution, $\int u_y^2 X_1^{\frac{3}{2}-2\sigma_i}$ into the left-hand side of (\ref{PREN.J}). For the other terms in $\mathcal{J}$, we can perform a similar calculation. As $X_1$ is arbitrary, this then implies the desired result, (\ref{see.5}).

\end{proof}

\begin{proof}[Proof of Proposition \ref{Lpi}]

Consolidating estimates the three estimates (\ref{see.2}), (\ref{see.3}), and (\ref{see.5}), and subsequently taking $\delta$ small enough, one obtains:
\begin{align} \n
\sup_x \int \psi^2 x^{-\frac{1}{2}-2\sigma_i} + &u^2 x^{\frac{1}{2}-2\sigma_i} + u_y^2 x^{\frac{3}{2}-2\sigma_i}  \\ 
& + \int \int \psi^2 x^{-\frac{3}{2}-2\sigma_i} + u^2 x^{-\frac{1}{2}-2\sigma_i} + u_y^2 x^{\frac{1}{2}-2\sigma_i} + u_x^2 x^{\frac{3}{2}-2\sigma_i} \lesssim \mathcal{O}(\delta; n). 
\end{align}

It is a standard matter now to obtain weighted in $z$ estimates, and to also successively differentiate the system in $x$ and repeat the previously established bounds. This procedure establishes the desired result fpr $u^i_p$ in (\ref{pr.i.b}). To estimate $v^i_p$, we take: 
\begin{align} \label{defnvip}
v^i_p = \int_y^\infty u^i_{px} \ud y',
\end{align}

and repeat the procedure as in estimate (\ref{pr.v.unif.1}). 

\end{proof}

\subsection{Final Prandtl Layer} \label{Final.Pr.Sub}

We now construct the final Prandtl layer, $[u^n_p, v^n_p]$. This final, $n$'th layer is slightly different because $v^{n}_p$ will be taken to satisfy the boundary condition: $v^n_p(x,0) = 0$.  According to (\ref{sys.pr.i}), the system for the $n$'th layer is: 
\begin{align} \label{sys.pr.minus.0}
&(1+u^0_p) u_{px} + u_{sx}^{(n)}u_p + v_s^{(n)} u_{py} + u^0_{py} v_p + P^{n}_{px} = u_{pyy} + f^{(n)}, \\ \label{pr.n.bc.1}
&[u_p(x,0), v_p(x,0)] = [-u^{n}_e(x,0), 0], \hspace{3 mm} \lim_{y \rightarrow \infty} u_p(x,y) = 0, \hspace{3 mm} u_p(1,y) = U_{n}(y), \\
& P^{n}_{py} = 0, \hspace{3 mm} u_{px} + v_{py} = 0. 
\end{align}

As usual from the Prandtl layers, by evaluating the equation at $y = \infty$, it is clear that the leading order Prandtl pressure is constant, that is  $P^{n}_{p} = 0$. Once $u_p, v_p$ are constructed to solve (\ref{sys.pr.minus.0}), we shall cut them off, thereby defining $[u^n_p, v^n_p]$. The relevant remainders are given in (\ref{Rude.n.INTRO}) - (\ref{Rvde.n.INTRO}), which we recall here for convenience: 
\begin{definition} The $n$'th remainder is denoted by: 
\begin{align} \label{Rude.n}
R^{u,n} &:= -\Delta_\epsilon \bar{u}_s^{(n)} + \bar{u}_s^{(n)} \bar{u}_{sx}^{(n)} + \bar{v}_s^{(n)} \bar{u}_{sy}^{(n)} + \bar{P}_{sx}^{(n)}, \\ \label{Rvde.n}
R^{v,n} &:=  -\Delta_\epsilon \bar{v}_s^{(n)} + \bar{u}_s^{(n)} \bar{v}_{sx}^{(n)} +  \bar{v}^{(n)}_s \bar{v}^{(n)}_{sy}   + \frac{\partial_y}{\epsilon} \bar{P}_s^{(n)}.
\end{align}
\end{definition}

\begin{remark} We refer the reader to Remark \ref{rmk.rem}. In this case, there is no $n+1$'th Euler construction, which explains the definition in (\ref{Rude.n}). This is seen as the remainder which is contributed to the next order in the specification of system (\ref{EQ.NSR.1}) - (\ref{defn.Nu}). 
\end{remark}

Using similar arguments to (\ref{comp.next}) - (\ref{comp.2.next}), we have the following errors: 
\begin{align} \label{runreef}
R^{u,n} &= \epsilon^{\frac{n}{2}}\Big[ \epsilon u^{n}_{pxx} + v^{n}_p\Big(\sum_{j=1}^{n} \epsilon^{\frac{j}{2}} \{u^j_{py} + \sqrt{\epsilon}u^j_{eY}\} \Big) + u^{n}_{px} \sum_{j=1}^{n} \epsilon^{\frac{j}{2}}\{u^j_e + u^j_p \}+ \mathcal{E}^{(n)}\Big], \\ \label{rvnreef}
R^{v,n} &= \epsilon^{\frac{n}{2}}\Big[ -\Delta_\epsilon v^{n}_p + u_s^{(n)} v^{n}_{px} + v^{(n)}_{sx} u^{n}_p + v_s^{(n)} v_{py}^{n} + v_{sy}^{(n)} v_p^{n} + \epsilon^{\frac{n}{2}} u^{n}_p v^{n}_{px} + \epsilon^{\frac{n}{2}} v^{n}_p v^{n}_{py} \Big],
\end{align}

Here $\mathcal{E}^{(n)} = \mathcal{E}^{(n)}(u_p, v_p)$ is an error term created by cutting off these layers, which will be defined in (\ref{error.n}). One can repeat the procedure used to construct the previous Prandtl layers to conclude:

\begin{proposition} \label{PropNN} Let $\sigma_n$ be defined according to (\ref{sigma.i}). The Prandtl layer, $u_p$ constructed to satisfy the equation (\ref{sys.pr.minus.0}) together with the boundary conditions in (\ref{pr.n.bc.1}), satisfies: 
\begin{align} \label{pr.n.bound}
||x^{\frac{1}{4}} z^m \partial_x^k u_p||_{P_k(\sigma_n)} \le C(k,m,n). 
\end{align}
\end{proposition}

Let us summarize the decay rates which result from this construction, by recalling the definition of our Prandtl norms, $P_k$, given in (\ref{norm.P}) - (\ref{norm.Pk}):
\begin{corollary} For any $k, j, M \ge 0$,  the profiles $u_p$ satisfy the following decay rates: 
\begin{align} \label{precut.1}
x^{k+\frac{j}{2}+\frac{1}{2}-\sigma_n}|| z^M \partial_x^k \partial_y^j u_p||_{L^\infty_y} + x^{k+\frac{j}{2}+\frac{1}{4}-\sigma_n} ||z^M \partial_x^k \partial_y^j u_p||_{L^2_y} \le C(k,j,M, n). 
\end{align}
\end{corollary}

In order to satisfy $v_p(x,0) = 0$, $v_p$ is obtained from $u_p$ via: $v_p(x,y) = -\int_0^y u_{px}(x,y') \ud y'$. This is distinct from (\ref{emphv1p}) and (\ref{defnvip}), and distinguishes the final, $n-$th Prandtl layer from the previous layers. Then, 
\begin{corollary} $v_p$ obeys the following uniform decay estimate:
\begin{align}
 x^{k+1-\sigma_n} ||\partial_x^k v_p||_{L^\infty_y} \le C(k,n). 
\end{align}
\end{corollary}

\begin{proof}

Using the boundary condition $v_p(x,0) = 0$, one can write: 
\begin{align} \label{vpsup}
\Big|v_p\Big|^2 & \lesssim \Big| \int_0^y v_pv_{py}\Big| \le  \int_0^\infty \Big| \frac{v_p}{y^{\frac{1}{2}+\kappa}} y^{\frac{1}{2}+\kappa}v_{py} \Big| \le ||\frac{v_p}{y^{\frac{1}{2}+\kappa}}||_{L^2_y}||y^{\frac{1}{2}+\kappa} v_{py}||_{L^2_y}\\ 
&\le ||y^{\frac{1}{2}-\kappa} v_{py}||_{L^2_y} ||y^{\frac{1}{2}+\kappa} v_{py}||_{L^2_y} = ||x^{\frac{1}{4}-\kappa} z^{\frac{1}{2}-\kappa} v_{py}||_{L^2_y} ||x^{\frac{1}{4}+\kappa} z^{\frac{1}{2}+\kappa} v_{py}||_{L^2_y} \\
& \le x^{\frac{1}{2}} x^{-1+\sigma_n} x^{-1+\sigma_n} = x^{-\frac{3}{2}+2\sigma_n}. 
\end{align}

We have used the $\kappa > 0$ to avoid the critical Hardy inequality. The Hardy inequality we have used (with power $y^{-\frac{1}{2}-\kappa} v^1_p$) relies on the vanishing of $v^n_p$ at $y = 0$. 
\end{proof}

Next, we introduce a cutoff function which honors the scaling, $z = \frac{y}{\sqrt{x}}$:
\begin{align} \label{cutoff.defn}
v^n_p &:= \chi(\sqrt{\epsilon}z)v_p, \hspace{3 mm} u^n_p := \int_x^\infty \partial_y \Big[ \chi \Big( \sqrt{\epsilon} \frac{y}{\sqrt{x'}} \Big) v_p(x',y)  \Big) \Big] dx'.
\end{align}

Here, we make the notational convention for integrands:  $z' = \frac{y}{\sqrt{x'}}$  or $z' = \frac{y'}{\sqrt{x}}$ where $'$ denotes the integration variable. It is clear, then, that $\partial_x u^n_p + \partial_y v^n_p = 0$. The following boundary conditions are also clear: 
\begin{align}
[u^n_p, v^n_p] \rightarrow 0 \text{ as } y \rightarrow \infty, \hspace{3 mm} [u^n_p, v^n_p]|_{y = 0} = [u_p, v_p]|_{y = 0}. 
\end{align}

\begin{remark}This cut-off honors the parabolic scaling of the Prandtl layers for all $x > 0$. The cut-off introduced in \cite{GN} was in the region $y \ge \frac{1}{\sqrt{\eps}}$. Locally in $x$, $z$ is equivalent to $y$, so these cut-offs agree locally in $x$.  
\end{remark}

We must now record two properties about the cut-off layers. First, the uniform estimates of the cut-off layers remain unchanged: 

\begin{lemma} \label{L.CL} The following pointwise decay bounds hold for the cut-off layers, for any $ k \ge 0$: 
\begin{align} \label{pw.stable.1}
\Big| \partial_x^k v^n_{p}\Big| \lesssim C(k, n)  x^{-k -1+\sigma_n}, \\ \label{pw.stable.3}
\Big| \partial_x^k v^n_{py}\Big| \lesssim C(k, n) x^{-k -\frac{3}{2}+\sigma_n}, \\ \label{pw.stable.2}
\Big|  \partial_x^k u^n_{py}\Big| \lesssim C(k, n) x^{-k-1+\sigma_n}, \\ \label{m=0}
\Big|  \partial_x^k u^n_p \Big| \lesssim C(k, n) x^{-k-\frac{1}{2}+\sigma_n}.
\end{align}
\end{lemma}
\begin{proof}

First, 
\begin{align}
|v^n_{py}| \le \frac{\sqrt{\epsilon}}{\sqrt{x}} \chi' |v_p| + \chi |v_{py}| \lesssim x^{-\frac{3}{2}+\sigma_n}.
\end{align}

Next, 
\begin{align} \nonumber
\Big|u^n_{py}\Big| &=\Big| \int_x^\infty \partial_{yy} \Big[ \chi(\sqrt{\epsilon}z') v_p \Big] dx'\Big| \le \Big| \int_x^\infty \frac{\epsilon}{x'} \chi'' v_p \Big| + \Big| 2\int_x^\infty \frac{\sqrt{\epsilon}}{\sqrt{x'}} v_{py} \Big| + \Big| \int_x^\infty \chi v_{pyy}  \Big| \\ \label{unpyuni}
& \lesssim x^{-1+\sigma_n}.
\end{align}

Finally, for $u_p^n$,
\begin{align}
\Big| u^n_p \Big| \le \int_x^\infty \sqrt{\epsilon} \frac{1}{\sqrt{x'}} \chi' |v_p| + \int_x^\infty |v_{py}|  \le x^{-\frac{1}{2}+\sigma_n}.
\end{align}

By differentiating successively in $x$, the result for $k \ge 1$ follows in the same manner. 
\end{proof}

\begin{lemma}[$L^2_y$ Estimates] \label{L.CL.2} The following $L^2_y$ decay bounds hold for the cut-off layers, for any $ k \ge 0$: 
\begin{align} \label{L2.stable.1}
||  \partial_x^k \partial_y^j u^n_p ||_{L^2_y} \lesssim C(k, \sigma, n) x^{-k- \frac{j}{2}-\frac{1}{4}+\sigma_n}.
\end{align}
\end{lemma}
\begin{proof}

In order to compute the $L^2$ norms, we must trade in the following way: 
\begin{align} \label{trade.cut}
o(\chi) \Big| \eps^{\frac{1}{4}+\frac{\kappa}{2}} \langle y \rangle^{\frac{1}{2}+\kappa} x^{-\frac{1}{4}-\frac{\kappa}{2}} \Big| \lesssim 1,
\end{align}

where $o(\chi)$ means $\chi$ or any derivative of $\chi$, the essential feature being that $z \le \frac{1}{\sqrt{\eps}}$. Using this: 
\begin{align} 
||u^n_p||_{L^2_y} &\le \int_x^\infty ||\frac{\sqrt{\eps}}{\sqrt{x'}} \chi' v_p||_{L^2_y}  + \int_x^\infty ||\chi v_{py}||_{L^2_y} \lesssim  x^{-\frac{1}{4}+\sigma_n}.
\end{align}

For $u^n_{py}$, again according to the expression in (\ref{unpyuni}), and the decay rates in (\ref{precut.1}): 
\begin{align} 
||u^n_{py}||_{L^2_y} &\lesssim \int_x^\infty ||\frac{\epsilon \chi''}{x'} v_p||_{L^2_y}  + \int_x^\infty ||\frac{\sqrt{\eps}}{\sqrt{x'}} v_{py}||_{L^2_y} + \int_x^\infty ||v_{pyy}||_{L^2_y}  \lesssim x^{-\frac{3}{4}+\sigma_n}. 
\end{align}

It is now clear we can repeat these calculations for higher $x$ and $y$ derivatives, thereby obtaining the desired result. 

\end{proof}

The next task is to control the error made by cutting off the layers. This error is obtained by inserting the new, cut-off layers into the equation (\ref{sys.pr.minus.0}): 
\begin{align} \label{error.n}
\mathcal{E}^{(n)} &:= (1+u^0_p) u^{n}_{px} + u^{(n)}_{sx} u^n_p + v^{(n)}_s u^n_{py} + u^0_{py} v^n_p - u^n_{pyy} - f^{(n)}. 
\end{align}

We will proceed to expand (\ref{error.n}), term-by-term: 
\begin{align}
(1+u^0_p) u^n_{px} &= \chi (1 + u^0_p) u_{px} - (1 + u^0_p) \frac{\sqrt{\epsilon}}{\sqrt{x}} \chi' v_p, \\ \label{ibper}
u^{(n)}_{sx} u^n_p &= \chi u^{(n)}_{sx} u_p +  u^{(n)}_{sx}\int_x^\infty \frac{\sqrt{\epsilon}}{\sqrt{x'}} \chi' v_p + \chi' \frac{\sqrt{\epsilon}z}{x'}u_p, \\ \label{ibper3}
v^{(n)}_s u^n_{py} &= \chi v^{(n)}_s u_{py} + v^{(n)}_s \int_x^\infty \frac{\epsilon}{x'} \chi'' v_p +2  \frac{\sqrt{\epsilon}}{\sqrt{x'}} \chi' v_{py} -  \frac{\sqrt{\epsilon}}{\sqrt{x'}} \chi' z u_{py}, \\ \label{ibper4}
u^0_{py} v^n_p &= \chi u^0_{py} v_p, \\ \n
u^n_{pyy} &= \chi u_{pyy} - \int_x^\infty \frac{\sqrt{\epsilon}}{x'} \chi' \frac{z}{2} u_{pyy} +C_1 \frac{\epsilon}{x'} \chi'' v_{py} + C_2 \chi' \frac{\sqrt{\epsilon}}{\sqrt{x'}} v_{pyy} \\ \label{ibper2} 
& + \frac{\epsilon^{\frac{3}{2}}}{(x')^{\frac{3}{2}}} \chi''' v_p. 
\end{align}

Let us provide justification to the above expressions, first turning to (\ref{ibper}). Applying the definition in (\ref{cutoff.defn}) yields:
\begin{align} \n
u^{(n)}_{sx} u^n_p &= u^{(n)}_{sx} \int_x^\infty \frac{\sqrt{\epsilon}}{\sqrt{x'}} \chi' v_p + u^{(n)}_{sx} \int_x^\infty \chi v_{py} \\ \n
& = u^{(n)}_{sx} \int_x^\infty \frac{\sqrt{\epsilon}}{\sqrt{x'}} \chi' v_p - u^{(n)}_{sx} \int_x^\infty \chi u_{px} dx' \\
& = u^{(n)}_{sx} \int_x^\infty \Big[ \frac{\sqrt{\epsilon}}{\sqrt{x'}} \chi' v_p + u^{(n)}_{sx} \frac{z\sqrt{\epsilon}}{x'} \chi' u_p \Big] dx' + \chi u^{(n)}_{sx} u_p,
\end{align}

the final equality following from an integration by parts in $x$. Similarly, for (\ref{ibper3}): 
\begin{align} \nonumber
v^{(n)}_s u^n_{py} &= v^{(n)}_s \int_x^\infty \partial_{yy} \Big( \chi v_p \Big) dx' \\ \nonumber
& = v^{(n)}_s \int_x^\infty \Big( \frac{\epsilon}{x'} \chi'' v_p + 2 \chi' \frac{\sqrt{\epsilon}}{\sqrt{x'}} v_{py} + \chi v_{pyy} \Big)\ud x' \\ \nonumber
& = v^{(n)}_s \int_x^\infty \frac{\epsilon}{x'} \chi'' v_p + 2\chi' \frac{\sqrt{\epsilon}}{x'} v_{py} - v^{(n)}_s \int_x^\infty \chi u_{pxy} dx' \\
& =  v^{(n)}_s \int_x^\infty \Big[ \frac{\epsilon}{x'} \chi'' v_p + 2\chi' \frac{\sqrt{\epsilon}}{x'} v_{py} -  \sqrt{\epsilon} \frac{\chi'}{x'} z u_{py} \Big] \ud x' + v^{(n)}_s \chi u_{py}.
\end{align}

This same computation is performed for (\ref{ibper2}): 
\begin{align} \n
u^n_{pyy} &= \int_x^\infty \partial_y^3 \Big( \chi v_p \Big) dx' \\ \n
& = \int_x^\infty \Big[ \frac{\eps^{\frac{3}{2}}}{(x')^{\frac{3}{2}}} \chi''' v_p + C_1 \frac{\eps}{x'} \chi'' v_{py} + C_2 \frac{\sqrt{\eps}}{\sqrt{x'}} v_{pyy} \chi' \Big] + \int_x^\infty \chi v_{pyyy} \\ \n
& = \int_x^\infty \Big[ \frac{\eps^{\frac{3}{2}}}{(x')^{\frac{3}{2}}} \chi''' v_p + C_1 \frac{\eps}{x'} \chi'' v_{py} + C_2 \frac{\sqrt{\eps}}{\sqrt{x'}} v_{pyy} \chi' \Big] - \int_x^\infty \chi u_{pxyy} \\ \n
& = \int_x^\infty \Big[ \frac{\eps^{\frac{3}{2}}}{(x')^{\frac{3}{2}}} \chi''' v_p + C_1 \frac{\eps}{x'} \chi'' v_{py} + C_2 \frac{\sqrt{\eps}}{\sqrt{x'}} v_{pyy} \chi' \Big] - \int_x^\infty \Big[ \chi' u_{pyy} \frac{z}{2x'} \sqrt{\eps} \Big] \ud x'\\
& \hspace{20 mm} + \chi u_{pyy}. 
\end{align}

Summing the above terms together, we arrive at our expression:
\begin{align}  \label{EN1}
\mathcal{E}^{(n)} &= - (1 + u^0_p) \frac{\sqrt{\epsilon}}{\sqrt{x}} \chi' v_p +  u^{(n)}_{sx}\int_x^\infty \frac{\sqrt{\epsilon}}{\sqrt{x'}} \chi' v_p + u^{(n)}_{sx} \int_x^\infty \chi' \frac{\sqrt{\epsilon}z}{x'}u_p \\  \label{EN2}
& + v^{(n)}_s \int_x^\infty \frac{\epsilon}{x'} \chi'' v_p +2  \frac{\sqrt{\epsilon}}{\sqrt{x'}} \chi' v_{py} -  \frac{\sqrt{\epsilon}}{x'} \chi'  z u_{py} - \int_x^\infty \frac{\sqrt{\epsilon}}{x'} \chi' \frac{z}{2} u_{pyy}  \\ \label{expEN}
& + \int_x^\infty  \Big[ C_2 \chi' \frac{\sqrt{\epsilon}}{\sqrt{x'}} v_{pyy} + C_1 \frac{\epsilon}{x'} \chi'' v_{py} + \frac{\epsilon^{\frac{3}{2}}}{(x')^{\frac{3}{2}}} \chi''' v_p \Big]  + (1-\chi) f^{(n)}. 
\end{align}

\begin{lemma} For $\kappa > 0$ arbitrarily small, the error, $\mathcal{E}^{(n)}$ obeys the following bound: 
\begin{align} \label{error.estimates}   
\Big| \mathcal{E}^{(n)}  \Big| \le C(n, \kappa) \epsilon^{\frac{1}{4} - \kappa} x^{-\frac{3}{2}+\sigma_n + \kappa}, \hspace{3 mm} ||\mathcal{E}^{(n)} ||_{L^2_y} \le C(n, \kappa) \epsilon^{\frac{1}{4}-\kappa} x^{-\frac{5}{4}+\sigma_n + \kappa}.  
\end{align}
\end{lemma}
\begin{proof}
We will proceed in order from (\ref{EN1}) - (\ref{expEN}), and we will focus on the $L^2_y$ estimates, as the uniform estimates are straight-forward. First, 
\begin{align} \nonumber
||(1+u^0_p) \frac{\sqrt{\epsilon}}{\sqrt{x}} \chi'(z) v_p ||_{L^2_y} &\lesssim \sqrt{\epsilon} ||\langle y \rangle^{\frac{1}{2}+\kappa} \langle y \rangle^{-\frac{1}{2}-\kappa} \chi' v_p \frac{1}{\sqrt{x}}||_{L^2_y} \\ \nonumber
& \lesssim \sqrt{\epsilon} ||\langle y \rangle^{\frac{1}{2}+\kappa} \chi' x^{-\frac{1}{2}} v_p ||_{L^\infty_y} \lesssim \sqrt{\epsilon} ||z^{\frac{1}{2}+\kappa} \chi' x^{-\frac{1}{4}+\frac{\kappa}{2}} v_p||_{L^\infty_y} \\
& \lesssim \epsilon^{\frac{1}{4}-\kappa} x^{-\frac{5}{4}+\frac{\kappa}{2} + \sigma_n}. 
\end{align}

The essential characteristic in the above calculation is the ability to pay a factor of $\epsilon^{\frac{1}{4}+\frac{\kappa}{2}} x^{-\frac{1}{4}-\frac{\kappa}{2}}$ in order to obtain an $L^2_y$ quantity due to the presence of the cut-off function $\chi'$ (see \ref{trade.cut}). In similar manner, we have: 
\begin{align} \nonumber
||u^{(n)}_{sx} \int_x^\infty \frac{\sqrt{\epsilon}}{\sqrt{x'}} \chi' v_p dx' ||_{L^2_y} &\le ||u^{(n)}_{sx}||_{L^\infty} \int_x^\infty ||\frac{\sqrt{\epsilon}}{\sqrt{x'}} \chi' v_p||_{L^\infty_y} ||\langle y \rangle^{-\frac{1}{2}-\kappa} ||_{L^2_y} dx' \\ \nonumber
& \lesssim x^{-1} \int_x^\infty (x')^{-1+\sigma_n} \epsilon^{\frac{1}{4}-\frac{\kappa}{2}} x^{-\frac{1}{4}+\frac{\kappa}{2}}||\epsilon^{\frac{1}{4}+\frac{\kappa}{2}} x^{-\frac{1}{4}-\frac{\kappa}{2}} \langle y \rangle^{\frac{1}{2}+\kappa} \chi' ||_{L^\infty} \\ 
& \lesssim \epsilon^{\frac{1}{4}+\kappa} x^{-\frac{5}{4}+\sigma_n + \frac{\kappa}{2}}. 
\end{align}

The third term on line (\ref{EN1}) can be estimated analogously. Coming now to the first term in (\ref{EN2}), 
\begin{align} \nonumber
||v^{(n)}_s \int_x^\infty \frac{\epsilon}{x} \chi'' v_p ||_{L^2_y} &\le ||v^{(n)}_s ||_{L^\infty} \int_x^\infty \epsilon^{\frac{3}{4}-\frac{\kappa}{2}} (x')^{-\frac{3}{4}+\frac{\kappa}{2}} (x')^{-1+\sigma_n} ||\epsilon^{\frac{1}{4}+\frac{\kappa}{2}} x^{-\frac{1}{4}-\frac{\kappa}{2}} \langle y \rangle^{\frac{1}{2}+\kappa} \chi'' ||_{L^\infty} \\
& \lesssim \epsilon^{\frac{3}{4}-\frac{\kappa}{2}} x^{-\frac{5}{4}+\frac{\kappa}{2} + \sigma_n}. 
\end{align}

The second term in (\ref{EN2}): 
\begin{align} \nonumber
||v^{(n)}_s \int_x^\infty \frac{\sqrt{\epsilon}}{\sqrt{x'}} \chi' v_{py} ||_{L^2_y} &\le ||v^{(n)}_s||_{L^\infty} \int_x^\infty \epsilon^{\frac{1}{4}-\frac{\kappa}{2}} (x')^{-\frac{1}{4}+\frac{\kappa}{2}} (x')^{-\frac{3}{2}+\sigma_n} dx' \\
& \lesssim \epsilon^{\frac{1}{4}-\frac{\kappa}{2}} x^{-\frac{5}{4}+\sigma_n + \frac{\kappa}{2}}. 
\end{align}

Moving to the third, which is slightly different due to the weight $z$, but this causes no harm as $zu_{py}$ is known to be in $L^\infty_y$ according to (\ref{pr.n.bound}):
\begin{align}
||v^{(n)}_s \int_x^\infty \frac{\sqrt{\epsilon}}{x'} \chi' z u_{py} ||_{L^2_y} \lesssim x^{-\frac{1}{2}} \int_x^\infty \frac{\sqrt{\epsilon}}{x'} (x')^{-\frac{3}{4}+\sigma_n} dx' \lesssim \sqrt{\epsilon} x^{-\frac{5}{4}-\sigma_n}. 
\end{align}

Next, we move to the $u^n_{pyy}$ contributions: 
\begin{align}
\int_x^\infty || C_1 \frac{\eps}{x'} \chi'' v_{py} + C_2 \frac{\sqrt{\eps}}{\sqrt{x'}} v_{pyy} \chi' -  \chi' u_{pyy} \frac{z}{2x'} \sqrt{\eps} ||_{L^2_y}  \lesssim \sqrt{\eps} x^{-\frac{5}{4}+\sigma_n}, 
\end{align}

according to the rates, $||u_{pyy} z x^{\frac{5}{4}-\sigma_n}, v_{pyy} x^{\frac{7}{4}-\sigma_n}, v_{py} x^{\frac{5}{4}-\sigma_n}||_{L^2_y} \le C$ from (\ref{precut.1}). Finally, for the $v_p$ term in (\ref{expEN}), one must use the tradeoff in (\ref{trade.cut}), to obtain: 
\begin{align} \n
\int_{x}^\infty ||\frac{\eps^{\frac{3}{2}}}{(x')^{\frac{3}{2}}} \chi''' v_p ||_{L^2_y} dx' &\lesssim \eps^{\frac{5}{4}-\frac{\kappa}{2}} \int_x^\infty (x')^{\frac{1}{4}+\frac{\kappa}{2}} (x')^{-\frac{3}{2}} ||v_p||_{L^\infty_y} \\
& \lesssim \eps^{\frac{5}{4}-\frac{\kappa}{2}} \int_x^\infty (x')^{\frac{1}{4}+\frac{\kappa}{2}} (x')^{-\frac{3}{2}} (x')^{-1+\sigma_n} \lesssim  \eps^{\frac{5}{4}-\frac{\kappa}{2}} x^{-\frac{5}{4}+\frac{\kappa}{2}+\sigma_n}. 
\end{align}

The $f^{(n)}$ term can then be estimated upon observing that $f^{(n)}$ is exponentially small in the support of $1-\chi$:
\begin{align} \n
||(1-\chi) f^{(n)}||_{L^2_y} &= ||(1-\chi) z^{N} z^{-N} f^{(n)}||_{L^2_y} \\ & \lesssim \epsilon^{\frac{N}{2}} ||z^{N} f^{(n)}||_{L^2_y}  \lesssim  \epsilon^{\frac{N}{2}} x^{-\frac{5}{4}+\sigma_{n-1}} \lesssim  \epsilon^{\frac{N}{2}} x^{-\frac{5}{4}+\sigma_{n}},
\end{align}

for any $N$ large, according to estimate (\ref{f.decay.1}). This now concludes the proof of  (\ref{error.estimates}). 
\end{proof}

By construction, one has the following: 
\begin{lemma}[Remainder Estimates]\label{Lemma FE} For $R^{u,n}, R^{v,n}$ as defined in (\ref{rvnreef}), and for any $\gamma \in [0,\frac{1}{4})$, $n \ge 2$, and for $\sigma_n$ as in (\ref{sigma.i}), $\kappa>0$ arbitrarily small, 
\begin{align}
&\epsilon^{-\frac{n}{2}-\gamma} \Big| \partial_x^k R^{u,n} + \sqrt{\epsilon} \partial_x^k R^{v,n} \Big| \lesssim C(n, \kappa) \epsilon^{\frac{1}{4}-\gamma-\kappa} x^{-k-\frac{3}{2}+2\sigma_n}, \\  \label{Rnl2}
&\epsilon^{-\frac{n}{2}-\gamma} ||\sqrt{\epsilon}\partial_x^k R^{u,n}, \sqrt{\epsilon} \partial_x^k R^{v,n}||_{L^2_y} \lesssim C(n, \kappa) \epsilon^{\frac{1}{4}-\gamma-\kappa} x^{-k-\frac{5}{4}+2\sigma_n + \kappa}.
\end{align}
\end{lemma}
\begin{proof}
This follows from (\ref{error.estimates}) and those bounds established in (\ref{pr.n.bound}). We shall proceed term by term from (\ref{rvnreef}). First, 
\begin{align} \n
\eps^{1-\gamma} ||u^n_{pxx}||_{L^2_y} & = \eps^{1-\gamma} ||\partial_x \Big( \chi v_p \Big)_y ||_{L^2_y}  =  \eps^{1-\gamma} ||\p_x \Big( \frac{\chi'}{\sqrt{x}} \sqrt{\eps} v_p + \chi v_{yp} \Big)||_{L^2_y} \\ \label{record.1}
&= \eps^{1-\gamma} || \eps \frac{\chi'' }{x^{\frac{3}{2}}} z v_p + \frac{\chi' \sqrt{\eps}}{x^{\frac{3}{2}}} v_p + \frac{\chi' }{\sqrt{x}} \sqrt{\eps} v_{px} + \sqrt{\eps} \frac{\chi'}{x} z v_{py} + \chi v_{pxy} ||_{L^2_y} \\ \n
& = (\ref{record.1}.1) + ... (\ref{record.1}.5).
\end{align}

First, 
\begin{align} \n
| (\ref{record.1}.1) | &\lesssim \eps^{1-\gamma} ||\sqrt{\eps} z \cdot \Big( \langle y \rangle^{\frac{1}{2}+\kappa} x^{-\frac{1}{4}-\frac{\kappa}{2}} \eps^{\frac{1}{4}+\frac{\kappa}{2}} \Big) \chi''||_{L^\infty_y} \eps^{\frac{1}{4}-\frac{\kappa}{2}} ||v_p||_{L^\infty} x^{-\frac{5}{4}+\frac{\kappa}{2}} \\
& \lesssim \eps^{\frac{5}{4}-\gamma - \frac{\kappa}{2}} x^{-\frac{9}{4}+\frac{\kappa}{2}+\sigma_n}. 
\end{align}

It is clear that (\ref{record.1}.1), (\ref{record.1}.2), and (\ref{record.1}.3) are identical, and so we move to: 
\begin{align}
| (\ref{record.1}.4) | &\lesssim \eps^{1-\gamma} \frac{\sqrt{\eps}}{x} ||zv_{py}||_{L^2_y} \lesssim \eps^{1-\gamma} \frac{\sqrt{\eps}}{x} x^{-\frac{5}{4}+\sigma_n},
\end{align}

and finally: 
\begin{align}
| (\ref{record.1}.5) | &= \eps^{1-\gamma} ||\chi v_{pxy}||_{L^2_y} \lesssim  \eps^{1-\gamma} x^{-\frac{9}{4}+\sigma_n}. 
\end{align}

We next move to the second term in $R^{u,n}$, for which we immediately use (\ref{trade.cut})
\begin{align} \n
\eps^{\frac{n}{2}-\gamma} ||v^n_p u^n_{py}||_{L^2_y} &\le \eps^{\frac{n}{2}-\gamma - \frac{1}{4}-\frac{\kappa}{2}} x^{\frac{1}{4}+\frac{\kappa}{2}} ||v^n_p||_{L^\infty_y} ||u^n_{py}||_{L^\infty_y} \\
& \lesssim \eps^{\frac{n}{2}-\gamma - \frac{1}{4}-\frac{\kappa}{2}} x^{\frac{1}{4}+\frac{\kappa}{2}} x^{-2+2\sigma_n} = \eps^{\frac{n}{2}-\gamma - \frac{1}{4}-\frac{\kappa}{2}} x^{-\frac{7}{4} + \frac{\kappa}{2} + 2\sigma_n}. 
\end{align}

where we have used (\ref{pw.stable.1}), (\ref{pw.stable.2}). Similarly, for $j \ge 1$, one has: 
\begin{align}
\eps^{\frac{1}{2}-\gamma} ||v^n_p u^j_{py}||_{L^2_y} \le \eps^{\frac{1}{2}-\gamma} ||v^n_p||_{L^\infty} ||u^j_{py} ||_{L^2_y} \lesssim \eps^{\frac{1}{2}-\gamma} x^{-1+\sigma_n} x^{-\frac{1}{2}+\sigma_n} \lesssim \eps^{\frac{1}{2}-\gamma} x^{-\frac{3}{2}+2\sigma_n} 
\end{align}

where we have used the specification of the norm $||\cdot||_{P}$ given in (\ref{norm.P}). For the term $v^n_p u^j_{eY}$, we calculate, for $\kappa > 0$: 
\begin{align}
\epsilon^{\frac{1}{2}+\frac{j}{2}-\gamma}||v^n_p u^j_{eY}||_{L^2_y} \le \epsilon^{\frac{1}{4}+\frac{j}{2}-\gamma}||v^n_p||_{L^\infty_y} ||u^j_{eY}||_{L^2_Y} \lesssim C(n) \epsilon^{\frac{1}{4}+\frac{j}{2}-\gamma} x^{-1+\sigma_n} x^{-1+\kappa}.
\end{align}

The last estimate follows from applying $L^2_Y$ to the Eulerian self-similar estimate, (\ref{e.ss.1}):
\begin{align}
\Big| u^1_{eY} \Big| = \Big| v^1_{ex} \Big| \le C(\kappa) x^{-1+\kappa} Y^{-\frac{1}{2}-\kappa}. 
\end{align}

The next term in $R^{u,n}$ is, for $j \ge 1$, for which we use (\ref{pr.1.close}) for the $u^j_p$ term, (\ref{i.Euler.decay}) for the $u^j_e$ term, and (\ref{L2.stable.1}) for the $u^{(n)}_{px}$ term in $L^2$:
\begin{align}
||\eps^{\frac{j}{2}} u^{(n)}_{px} (u^j_e + u^j_p)||_{L^2_y} \lesssim \eps^{\frac{j}{2}} x^{-\frac{1}{4}+\sigma_n} ||u^{n}_{px}||_{L^2_y} \lesssim \eps^{\frac{1}{4}-\frac{\kappa}{2}} x^{-\frac{1}{4}+\sigma_n}  x^{-\frac{5}{4}+\sigma_n}. 
\end{align}

The final term in $R^{u,n}$ is the error term, $\mathcal{E}^{(n)}$, which has been controlled in (\ref{error.estimates}). We may now move to $R^{v,n}$. In so doing, we first note that $\eps v^{n}_{pxx}$ can be estimated identically to (\ref{vpx.hard}). Second, using (\ref{L2.stable.1}), we have: 
\begin{align} \n
\eps^{\frac{1}{2}-\gamma}||v^n_{pyy}||_{L^2_y} &\lesssim \eps^{\frac{1}{2}-\gamma} x^{-\frac{7}{4}+\sigma_n}. 
\end{align} 

For the $L^2$ bound on the $v^n_{px}$ term in (\ref{rvnreef}), we employ (\ref{trade.cut}):
\begin{align} \nonumber
\epsilon^{\frac{1}{2}-\gamma} ||u_s^{(n)}v^n_{px}||_{L^2_y} &\lesssim \epsilon^{\frac{1}{2}-\gamma} ||v^n_{px}||_{L^2_y} \le \epsilon^{\frac{1}{2}-\gamma} ||1_{z \le \frac{1}{\sqrt{\epsilon}}} v^n_{px} ||_{L^2_y} \\ \nonumber
&\le \epsilon^{\frac{1}{2}-\gamma} ||1_{z \le \frac{1}{\sqrt{\epsilon}}} x^{-2+\sigma} ||_{L^2_y} =  \epsilon^{\frac{1}{2}-\gamma} ||1_{z \le \frac{1}{\sqrt{\epsilon}}} x^{-2+\sigma_n} (1+y)^{\frac{1}{2}+\kappa} (1+y)^{-\frac{1}{2}-\kappa} ||_{L^2_y} \\ \nonumber
&\le \epsilon^{\frac{1}{2}-\gamma} ||1_{z \le \frac{1}{\sqrt{\epsilon}}} (1+y)^{\frac{1}{2}+\kappa} x^{-\frac{1}{4}-\frac{\kappa}{2}}||_{L^\infty_y} x^{-\frac{7}{4}+\sigma_n + \frac{\kappa}{2}} ||(1+y)^{-\frac{1}{2}-\kappa}||_{L^2_y} \\ \label{vpx.hard}
&\le \epsilon^{\frac{1}{2}-\gamma} ||1_{z \le \frac{1}{\sqrt{\epsilon}}} (1+z)^{\frac{1}{2}+\kappa}||_{L^\infty_y} x^{-\frac{7}{4}+\sigma_n + \frac{\kappa}{2}}  \le \epsilon^{\frac{1}{4}-\gamma-\kappa} x^{-\frac{7}{4}+\sigma_n + \frac{\kappa}{2}}. 
\end{align}

The next two terms: 
\begin{align}
&\eps^{\frac{1}{2}-\gamma} ||v^{(n)}_{sx} u^n_p||_{L^2_y} \le \sqrt{\eps} ||v^{(n)}_{sx}||_{L^\infty} ||u^n_p||_{L^2_y} \lesssim  \sqrt{\eps} x^{-\frac{3}{2}} x^{-\frac{1}{4}}, \\
&\eps^{\frac{1}{2}-\gamma} ||v^{(n)}_s v^n_{py}||_{L^2_y} \le \eps^{\frac{1}{2}-\gamma} x^{-\frac{7}{4}+2\sigma_n}.
\end{align}

Now, by using the trade-off in (\ref{trade.cut}): 
\begin{align}
\eps^{\frac{1}{2}-\gamma} ||v^{(n)}_{sy} v^n_p||_{L^2_y} \le \eps^{\frac{1}{4}-\gamma - \kappa} ||v^{(n)}_{sy}||_{L^\infty_y} ||v^n_p||_{L^\infty_y} x^{\frac{1}{4}+\frac{\kappa}{2}} \lesssim \eps^{\frac{1}{4}-\gamma - \kappa} x^{-\frac{7}{4}+2\sigma_n + \frac{\kappa}{2}}.
\end{align}

We will now move to: 
\begin{align}
&\eps^{\frac{n}{2}-\gamma} \sqrt{\eps} ||u^n_p v^n_{px}||_{L^2_y} \lesssim \eps^{\frac{n}{2}-\gamma + \frac{1}{2}} ||u^n_p||_{L^2_y} ||v^n_{px}||_{L^\infty_y} \lesssim  \eps^{\frac{n}{2}-\gamma + \frac{1}{2}} x^{-\frac{9}{4}+2\sigma_n }, \\
&\eps^{\frac{n}{2}-\gamma} \sqrt{\eps} ||v^n_p v^n_{py}||_{L^2_y} \lesssim \eps^{\frac{n}{2}-\gamma + \frac{1}{2}} x^{-\frac{9}{4}+2\sigma_n }. 
\end{align}

where we have used (\ref{pw.stable.1}) and (\ref{L2.stable.1}). This completes all of the terms in (\ref{rvnreef}), thereby establishing (\ref{Rnl2}). The uniform estimates follow in an analogous manner.  

\end{proof}

For the energy estimates in Section \ref{section.NSR.Linear}, we will need to retain the self-similarity of $u^n_p$ (the ability to absorb factors of $z$). As the profiles $u^n_p, v^n_p$ are higher order in $\epsilon$, we will be happy to pay factors of $\epsilon$ in order to retain this ability:
\begin{lemma} \label{Lpf} The following point-wise decay estimates holds for any $m \ge 0$, 
\begin{align}
&\Big| z^m \partial_x^k v^n_{p}\Big| \lesssim \epsilon^{-\frac{m}{2}} C(k, n, m)  x^{-k -1+\sigma_n}, \\
&\Big| z^m \partial_x^k v^n_{py}\Big| \lesssim \epsilon^{-\frac{m}{2}} C(k, n,  m) x^{-k -\frac{3}{2}+\sigma_n}, \\ \label{youpewhy}
&\Big| y^j \partial_y^j  u^n_{p}\Big| \lesssim C(k, n, m)x^{-\frac{1}{2}+\sigma_n} , \\ \label{youpewhy.2}
&\Big| z^m \partial_x^k u^n_{py} \Big| \lesssim \epsilon^{-\frac{m}{2}} C(k, n, m)x^{-1+\sigma_n - k}, \text{ for } k \ge 1.
\end{align}
\end{lemma}
\begin{proof}

Via the definition: 
\begin{align}
z^m |v^n_p| \le z^m |\chi(\sqrt{\epsilon}z) v_p| \le \epsilon^{-\frac{m}{2}} |v_p| \lesssim \epsilon^{-\frac{m}{2}} x^{-1+\sigma_n}. 
\end{align}

Next, applying $\partial_y$ yields: 
\begin{align}
z^m |v^n_{py}| = z^m \sqrt{\epsilon} |\chi' \frac{v_p}{\sqrt{x}} |+ \chi z^m |v_{py}| \lesssim  \epsilon^{-\frac{m}{2}} x^{-\frac{3}{2}+\sigma_n}.
\end{align}

The estimate for $u^n_{py}$ is more complicated. The $m = 0$ case was treated in (\ref{m=0}). Suppose $m = 1$. Then, 
\begin{align}
y u^n_{py} = y \int_x^\infty \partial_{yy} (\chi v_p) dx' &=  \int_x^\infty \frac{\epsilon}{x'} \chi'' y v_p + \int_x \frac{\sqrt{\epsilon}}{\sqrt{x'}} \chi' y v_{py} + \int_x^\infty \chi y v_{pyy} \\
& = I_1 + I_2 + I_3. 
\end{align}

We have: 
\begin{align}
|I_1| \le ||\chi'' z||_{L^\infty} \int_x^\infty \epsilon \frac{|v_p|}{\sqrt{x'}} dx' \le \sqrt{\epsilon} \int_x^\infty (x')^{-\frac{3}{2}+\sigma} dx' \lesssim \sqrt{\epsilon} x^{-\frac{1}{2}+\sigma_n}. 
\end{align}

Next, 
\begin{align}
|I_2| \le \int_x^\infty \sqrt{\epsilon} \chi' |z v_{py}| dx' \lesssim \int_x^\infty \sqrt{\epsilon} (x')^{-\frac{3}{2}} dx' \lesssim \sqrt{\epsilon} x^{-\frac{1}{2}+\sigma_n}. 
\end{align}

Lastly, 
\begin{align}
\int_x^\infty |y \chi v_{pyy} | dx' \le \int_x^\infty  (x')^{\frac{1}{2}} zv_{pyy} dx' \lesssim x^{-\frac{1}{2}+\sigma_n}. 
\end{align}

Piecing these estimates together gives: 
\begin{align}
y|u^n_{py}| \lesssim x^{-\frac{1}{2}+\sigma_n}, 
\end{align}

as desired. Higher $y$-derivatives of $u^n_p$ follow an identical calculation. Let us now move to $x$-derivatives of $u^n_{py}$: 
\begin{align}
u^n_{pyx} = \partial_{yy} (\chi v_p) = \frac{\epsilon}{x} \chi'' v_p + \frac{\sqrt{\epsilon}}{\sqrt{x}} \chi' v_{py} + \chi v_{pyy}. 
\end{align}

From here, once can estimate: 
\begin{align} \nonumber
z^m |u^n_{pxy}| &\lesssim \frac{\epsilon}{x}| \chi''| z^m |v_p |+ \frac{\sqrt{\epsilon}}{\sqrt{x}}| \chi'| z^m  |v_{py}|+ \chi z^m|v_{pyy}| \\
& \le \epsilon^{-\frac{m}{2}} x^{-2+\sigma_n}.
\end{align}

\end{proof}

We are now ready to conclude Chapter I: 

\begin{proof}[Proof of Theorem \ref{thm.m.part.1}]

Consolidating Corollaries \ref{cp1} - \ref{cp2}, Propositions \ref{thm.euler.1}, \ref{Pp1}, \ref{i.Euler.lay}, \ref{Lpi}, Lemma \ref{Lemma FE}, and Lemma \ref{Lpf} gives Theorem \ref{thm.m.part.1}.
\end{proof}

\break

\part*{\centerline{Chapter II: a-Priori Estimates for Remainder}}
\addcontentsline{toc}{part}{Chapter II: a-Priori Estimates for Remainder}

\section{Overview of Results}

We will now consider the system satisfied by the remainders, $[u,v,P]$ as defined in (\ref{expansion.1}) - (\ref{expansion.3}). Expanding the Navier-Stokes equations in (\ref{scaled.NS.1}) - (\ref{scaled.NS.3}), one obtains: 
\begin{align} \label{EQ.NSR.1}
-\Delta_\epsilon u + S_u(u,v) + P_{x} = f(u,v) ,\\ \label{EQ.NSR.2}
- \Delta_\epsilon v + S_v(u,v) + \frac{P_y}{\epsilon} = g(u,v), \\ \label{EQ.NSR.3}
u_x + v_y = 0,
\end{align}

which are taken together with the boundary conditions:
\begin{align} \label{nsr.bc.1}
[u,v]|_{\{y=0\}} =  [u,v]|_{\{x = 1\}} = \lim_{y \rightarrow \infty} [u,v] = \lim_{x \rightarrow \infty} [u,v] = 0.
\end{align}

The terms in equations (\ref{EQ.NSR.1}) - (\ref{EQ.NSR.2}) are defined:
\begin{align} \label{defn.SU.SV}
&f(u,v) := \epsilon^{-\frac{n}{2}-\gamma} R^{u,n} + \mathcal{N}^u(u,v), \hspace{3 mm} g(u,v) := \epsilon^{-\frac{n}{2}-\gamma} R^{v,n} + \mathcal{N}^v(u,v), \\ \label{defn.Su}
&S_u(u,v) := u_R u_{x} + u_{Rx}u + v_R u_{y} + u_{Ry}v , \hspace{3 mm} S_v(u,v) := u_R v_x + v_{Rx}u + v_R v_y + v_{Ry}v, \\ \label{defn.Nu}
& \mathcal{N}^u(u,v) := \epsilon^{\frac{n}{2}+\gamma} uu_x + \epsilon^{\frac{n}{2}+\gamma}vu_y, \hspace{3 mm} \mathcal{N}^v(u,v) := \epsilon^{\frac{n}{2}+\gamma} uv_x + \epsilon^{\frac{n}{2}+\gamma} vv_y. 
\end{align}

The forcing terms, $R^{u,v}, R^{v,n}$ are given by the expressions in (\ref{runreef}) - (\ref{rvnreef}), and have been controlled in Theorem \ref{thm.m.part.1}. The coefficients $u_R, v_R$ in $S_u(u,v), S_v(u,v)$ given in (\ref{defn.Su}) have been defined in (\ref{defn.UR}) - (\ref{defn.VR}), and controlled according to the estimates in Theorem \ref{thm.m.part.1}. Due to the eventual need to perform a fixed-point argument, we will actually consider the slightly generalized system, where $f, g$ above are replaced by:
\begin{align} \label{bar.f}
f(u, \bar{u}, \bar{v}) &= \epsilon^{-\frac{n}{2}-\gamma} R^{u,n} + \epsilon^{\frac{n}{2}+\gamma}\bar{u} \bar{u}_x + \epsilon^{\frac{n}{2}+\gamma}\bar{v} u_y, \\ \label{bar.g}
g(\bar{u}, \bar{v}) &= \epsilon^{-\frac{n}{2}-\gamma} R^{v,n} + \epsilon^{\frac{n}{2}+\gamma}\bar{u} \bar{v}_x + \epsilon^{\frac{n}{2}+\gamma}\bar{v} \bar{v}_y.
\end{align}

When $u = \bar{u}, v = \bar{v}$ (which corresponds to a fixed point of an appropriately defined map), one obtains the actual system of interest, with $f$ as in (\ref{defn.SU.SV}). For technical purposes, in this chapter we will consider the system (\ref{EQ.NSR.1}) - (\ref{EQ.NSR.3}) on the domain:
\begin{align} \label{omegaN}
\Omega^N := \{(x,y) : x > 0, 0 < y < N \}.
\end{align}

All estimates will be made independent of $N$, allowing us to eventually send $N \rightarrow \infty$. When working on this domain, we will use the boundary conditions: 
\begin{align} \label{BCN}
[u,v]|_{\{y=0\}} =  [u,v]|_{\{x = 1\}} =  [u,v]|_{\{y = N\}} = \lim_{x \rightarrow \infty} [u,v] = 0.
\end{align}

The norms which we will work in are defined in the next section, Section \ref{Section.Z}, starting with (\ref{norm.x0}) - (\ref{norm.Z}). We invite the reader to read these definitions at this point. The main result of this chapter is then: 
\begin{theorem}[Complete $Z$ Estimate] \label{thm.m.2} Fix any $N  > 0$. Let $\delta, \epsilon$ be sufficiently small based on universal constants, with $\epsilon << \delta$, and $n \in \mathbb{N}$ sufficiently large relative to universal constants. Let $N_i$ be parameters in the definition of $Z$, given in (\ref{norm.Z}). Let $\gamma \in (0, \frac{1}{4})$, and parameter $\kappa > 0$ arbitrarily small. Suppose $||\bar{u}, \bar{v}||_{Z(\Omega^N)} \le 1$. Then $[u,v] \in Z(\Omega^N)$, solutions to the system (\ref{EQ.NSR.1}) - (\ref{EQ.NSR.3}), (\ref{defn.SU.SV}) - (\ref{bar.g}), with boundary conditions (\ref{BCN}), obey the following \textit{a-priori} estimate:
\begin{align}
||u,v||_{Z(\Omega^N)}^2 & \lesssim \epsilon^{\frac{1}{4}-\gamma - \kappa} + \epsilon^{\frac{n}{2}-\omega(N_i)} \Big( ||\bar{u}, \bar{v}||_{Z(\Omega^N)}^4 \Big),
\end{align}

for some fixed function $\omega(N_i)$, where the constants above are independent of $N$.
\end{theorem}

Let us give a brief overview of the steps to prove Theorem \ref{thm.m.2}. 
\begin{itemize}

\item[(Step 1)] The Space $Z$ (Section \ref{Section.Z}): We introduce the various components of the crucial norm $Z$. The energy norms, $X_1, X_2, X_3$ are to be controlled using energy and positivity estimates (in the next steps). The elliptic norms, $Y_2, Y_3$, provide additional controls near the boundary, $x = 1$, and are related to the $X_i$ norms in subsection \ref{subsection.sing}. Most importantly, there are uniform-type norms, defined in (\ref{norm.U}), of which the $||vx^{\frac{1}{2}}||_{L^\infty}$ is the most crucial ingredient. These are controlled in Lemmas \ref{L.Evol}, \ref{Lemma.UIMP}. The relation between all of the norms is summarized in Theorem \ref{thm.z}. This analysis of $Z$ is the main novelty of Chapter II. 

\item[(Step 2)] Energy Estimates: The lowest-order energy estimates are performed in Proposition \ref{thm.energy}, for which the sharp profile estimates from (\ref{PE0.1}) - (\ref{PE4.new.2}) are essential. An examination of Proposition \ref{thm.energy} shows that the energy estimate loses a factor of $x^{\frac{1}{2}}$ which must be recovered in the next step.

\item[(Step 3)] Positivity Estimates: The lowest-order positivity estimates are performed in Proposition \ref{prop.pos}. Here again, the estimates from (\ref{PE0.1}) - (\ref{PE5}) are used in an essential way. In particular, the estimate (\ref{imp.vR}) forces the requirement shown in (\ref{REQ}). 

\item[(Step 4)] Higher-Order Energy/ Positivity Estimates: In Propositions \ref{prop.ho.1}, \ref{prop.ho.2}, \ref{prop.ho.3}, \ref{prop.ho.4}, we control the higher-order energy norms $X_2, X_3$ from definitions (\ref{norm.x1}) - (\ref{norm.x2}). This is achieved by applying appropriate iterations of $\p_x$ to the system (\ref{EQ.NSR.1}) - (\ref{EQ.NSR.3}). 

\item[(Step 5)] Nonlinear Estimates: The estimates of the nonlinearities present in $(f,g)$ from (\ref{bar.f}) - (\ref{bar.g}) is performed in Section \ref{Sec.NL}, in particular in Lemma \ref{LemmaW}. Here, there are several delicate matters. First, the sharp decay of $||x^{\frac{1}{2}}v||_{L^\infty}$, which is present in (\ref{norm.Z}) due to Lemma \ref{Lemma.UIMP}, is essential in order to control the term $\bar{v} \bar{u}_y$. This estimate is (\ref{exact.NL}). Second, estimate (\ref{order.NL}) capitalizes on a cancellation structure which enables us to control a term that would be otherwise out of reach. Third, the top-order nonlinear terms in estimates (\ref{TON.1}) and (\ref{TON.2}) are controlled by using mixed-norms. 

\end{itemize}

\section{The Function Space $Z$} \label{Section.Z}

In this section, we define and analyze the high-order, weighted norm, $Z$, on which we will close our nonlinear analysis. The functional framework of this section is required in order to follow the calculations in the upcoming sections. First, we need to define the energy norms in which we obtain energy estimates, and also several auxiliary norms that will supplement the energy norms. To define these, first define the cut-off functions: 
\begin{align}
  \zeta_3(x) &= \begin{cases}
0 \text{ for } 1 \le x \le \frac{3}{2},\\
1 \text{ for } x \ge 2.
\end{cases} \label{zeta} \\
  \rho_k(x) &= \begin{cases}
0 \text{ for } 1 \le x \le 50 + 50(k-2),\\
1 \text{ for } x \ge 60 + 50(k-2).
\end{cases} \label{rho}
\end{align}

The energy norms are defined as follows:
\begin{align} \label{norm.x0}
||u,v||_{X_1}^2 &:= ||u_y||_{L^2}^2 + ||\{\sqrt{\epsilon}v_x, v_y \} x^{\frac{1}{2}}||_{L^2}^2\\ \label{norm.x1}
||u,v||_{X_2}^2 &:= ||  u_{xy} \cdot \rho_2 x||_{L^2}^2 + ||  \{ \sqrt{\epsilon} v_{xx},  v_{xy} \} \cdot (\rho_2 x)^{\frac{3}{2}}||_{L^2}^2, \\ \label{norm.x2}
||u,v||_{X_3}^2 &:= ||  u_{xxy} \cdot (\rho_3 x)^2||_{L^2}^2 +||  \{ \sqrt{\epsilon} v_{xxx},  v_{xxy} \} \cdot (\rho_3 x)^{\frac{5}{2}}||_{L^2}^2.
\end{align}

\begin{definition} The norms $Y_2, Y_3$ are strengthenings of $X_2, X_3$ near the boundary, $x = 1$, and defined through: 
\begin{align} \label{norm.Y2}
&||u,v||_{Y_2}^2 := ||u_{xy} x||_{L^2}^2 + ||\{\sqrt{\epsilon} v_{xx}, v_{xy} \} x^{\frac{3}{2}}||_{L^2}^2 + ||u_{yy}||_{L^2(x \le 2000)}, \\ \label{norm.Y3}
&||u,v||_{Y_3}^2 := || u_{xxy} \cdot \zeta_3 x^{2}||_{L^2}^2 + || \{ \sqrt{\epsilon} v_{xxx}, v_{xxy} \} \cdot \zeta_3 x^{\frac{5}{2}}||_{L^2}^2.
\end{align}
\end{definition}

\begin{definition} The norm $Z$ is defined through: 
\begin{align} \nonumber
||u,v||_Z := &||u,v||_{X_1 \cap X_2 \cap X_3} + \epsilon^{N_2} ||u,v||_{Y_2} + \epsilon^{N_3} ||u,v||_{Y_3} + \epsilon^{N_4} ||ux^{\frac{1}{4}} , \sqrt{\eps} v x^{\frac{1}{2}}||_{L^\infty} \\ \nonumber 
&+ \epsilon^{N_5} \sup_{x \ge 20} ||\sqrt{\eps} v_x x^{\frac{3}{2}} , u_x x^{\frac{5}{4}} ||_{L^\infty} + \eps^{N_6}  \sup_{x \ge 20} ||u_y x^{\frac{1}{2}}||_{L^2_y}\\ \label{norm.Z}
& + \epsilon^{N_7} \Big[\int_{20}^\infty x^4 ||\sqrt{\eps} v_{xx}||_{L^\infty_y}^2 dx \Big]^{\frac{1}{2}} .
\end{align}
Here, $N_i$, are some large numbers which will be specified in (\ref{sel.N2}) - (\ref{sel.N4}). They depend only on universal constants. It will be understood that $N_i$ are much larger than any of the quantities $M_i$ appearing in the forthcoming lemmas, and that $n$ is much larger than any of the $N_i, M_i$. 
\end{definition}

Now that the norms we will be working in have been specified, we will define corresponding spaces. First, some basic notations: 
\begin{definition} For any open set $U \subset \Omega$, the space $C^\infty_0(U)$ denotes the space of smooth functions with compact support in $U$. We will also use $C^\infty_0(U)$ to denote the space of smooth vector fields with compact support in $U$, which will be clear from context. The space $C^\infty_{0,D}(U)$ denotes the space of smooth, divergence-free vector fields with compact support in $U$. When the set $U$ is clear from context, we shall suppress it.  
\end{definition}

\begin{definition} \label{defn.Z.space} Given any open set $U \subset \Omega$, the space $Z(U)$ is defined to be space of those divergence-free vector fields which lie in the closure of $C_{0,D}^\infty(U)$ under the norm $X_1$, such that $||u,v||_Z < \infty$.
\end{definition}

\begin{definition} \label{defn.X.space}  Given any open set $U \subset \Omega$, the space $(X_1 \cap X_2 \cap X_3)(U)$ is defined to be space of those divergence-free vector fields which lie in the closure of $C_{0,D}^\infty(U)$ under the norm $X_1$, such that $||u,v||_{X_1 \cap X_2 \cap X_3} < \infty$.
\end{definition}

Due to the weights in the norm $Z$ and the energy norms $X_1 \cap X_2 \cap X_3$, there is no ``$H = W$ Theorem" generically available, which would equate the density of smooth functions with compact support with the class of functions satisfying $||u,v||_Z < \infty$. This is the reason must specify that: 
\begin{align} \label{clos.1}
Z(U) \subset \overline{C_{0,D}^\infty}^{||\cdot||_{X_1}}(U). 
\end{align}

We will first record the boundary behavior of $Z(\Omega)$-vector fields:
\begin{lemma} \label{Lemma.BCZ} For $[u,v] \in Z(\Omega)$, the following boundary conditions are satisfied:
\begin{align} \label{BCENC}
[u,v]|_{x = 1} = [u,v]|_{y = 0} = \lim_{x \rightarrow \infty} [u,v] = \lim_{y \rightarrow \infty} [u,v] = 0. 
\end{align} 
\end{lemma}
\begin{proof}

The boundary conditions at $x = 1$ and $y = 0$ are satisfied by the density specification (\ref{clos.1}). The boundary conditions as $x \rightarrow \infty$ is satisfied by the definition of norm Z, (\ref{norm.Z}) because the decay is encoded (with rates) in the norm. The boundary conditions as $y \rightarrow \infty$ is enforced because $[u,v] \in Z(\Omega)$ implies that $[u,v] \in H^1_{(x \le A)}$, and candidacy in such a Sobolev space automatically encodes decay as $y \rightarrow \infty$.  
\end{proof}

The above lemma shows that $Z(\Omega)$ is a suitable space to obtain solutions to our boundary-value problem, after a consultation with (\ref{nsr.bc.1}). That $X_1 \cap X_2 \cap X_3$ also encodes the boundary conditions from (\ref{BCENC}) is less obvious, and is proven in Lemma \ref{L.XiBC}. 

\begin{lemma} \label{lemma.complete} For any open set $U \subset \Omega$, $Z(U)$ and $(X_1 \cap X_2 \cap X_3)(U)$ are Banach spaces.
\end{lemma}
\begin{proof}

This follows from standard arguments. 
\end{proof}

For the energy norm, $X_1 \cap X_2 \cap X_3$, we have Hilbertian structure. Given two divergence-free vector-fields:
\begin{align}
\bold{u} := (u,v), \hspace{3 mm} \bold{a}:= (a,b), 
\end{align}

define the following inner-products:
\begin{align} \label{IP}
&\langle \bold{u}, \bold{a} \rangle_{X_1} := \int \int u_y a_y + \int \int u_x a_x x + \int \int \eps v_x b_x x, \\
&\langle \bold{u}, \bold{a} \rangle_{X_2} := \int \int u_{xy} a_{xy} (\rho_2x)^2 + \int \int u_{xx} a_{xx} (\rho_2x)^3 + \int \int \eps v_{xx} b_{xx} (\rho_2 x)^3, \\
&\langle \bold{u}, \bold{a} \rangle_{X_3} := \int \int u_{xxy} a_{xxy} (\rho_3x)^4 + \int \int u_{xxx} a_{xxx} (\rho_3x)^5 + \int \int \eps v_{xxx} b_{xxx} (\rho_3x)^5, \\ \label{IP.4}
&\langle \bold{u}, \bold{a} \rangle_{X_1 \cap X_2 \cap X_3} := \langle \bold{u}, \bold{a} \rangle_{X_1} +  \langle \bold{u}, \bold{a} \rangle_{X_2} +  \langle \bold{u}, \bold{a} \rangle_{X_3}. 
\end{align}

Temporarily accepting Lemma \ref{L.XiBC}, we have: 

\begin{lemma} With the inner product (\ref{IP.4}), $(X_1 \cap X_2 \cap X_3)(\Omega)$ is a Hilbert space.
\end{lemma}
\begin{proof}

The inner-product (\ref{IP.4}) is non-degenerate due to the boundary conditions $[u,v]|_{y = 0}$. The remaining inner-product axioms are straightforward to check, which coupled with the completeness in Lemma \ref{lemma.complete} gives the result. 

\end{proof}

Notationally, it sometimes is convenient to refer at once to all of the ``uniform-type" quantities in the norm $Z$, so we designate the following notation:
\begin{definition} The norm $\mathcal{U}$ is defined by:
\begin{align} \n
||u,v||_{\mathcal{U}} &:= \epsilon^{N_4} ||ux^{\frac{1}{4}}, \sqrt{\eps} vx^{\frac{1}{2}}||_{L^\infty} + \epsilon^{N_5} \sup_{x \ge 20} ||\sqrt{\eps} v_x x^{\frac{3}{2}} ,u_x  x^{\frac{5}{4}} ||_{L^\infty} \\ \label{norm.U} &+ \eps^{N_6}  \sup_{x \ge 20} ||u_y x^{\frac{1}{2}}||_{L^2_y} + \epsilon^{N_7} \Big[\int_{20}^\infty x^4 ||\sqrt{\eps} v_{xx}||_{L^\infty_y}^2 dx \Big]^{\frac{1}{2}}.
\end{align}
\end{definition}

\subsection{Elliptic Estimates and the Spaces $Y_i$}  \label{subsection.sing}

We will first utilize elliptic theory in order to obtain basic $H^k$ estimates for our solution, which hold generically for Stokes-type equations. These elliptic estimates are meant to supplement the energy estimates that we will perform in Section \ref{section.NSR.Linear} which are meant to control the energy norms, $X_1 \cap X_2 \cap X_3$. In particular, the estimates from this section cannot \textit{replace} the energy estimates for two reasons: 
\begin{itemize}
\item[(1)] They scale poorly in $\epsilon$, and 
\item[(2)] One cannot extract sharp enough global-in-$x$ information using this procedure.  
\end{itemize}

Their purpose is to provide more controls near the boundary $x = 1$, which is reflected in the norms $Y_2, Y_3$. To see this one should compare the support of the cut-off functions in our energy norms, $\rho_k$, with the support of $\zeta_3$ in $Y_3$.  For the set of calculations in this subsection, there are many different cutoff functions which arise in addition to the important cutoffs, $\rho_k, \zeta_k$ which have already been defined. 

\begin{remark}[Notational Convention]
To simplify notations, given a cut-off function $\chi$, we introduce the notation $o(\chi)$ to mean either $\chi$ or any of its derivatives, or any positive powers of $\chi$ or any of its derivatives. Roughly speaking, any ``variant" of $\chi$ is denoted by $o(\chi)$, the essential feature being $o(\chi)$ is supported in the same (or similar) region. 
\end{remark}

The first step will be to provide some controls of $||u, v||_{\dot{H}^2}$.

\begin{lemma}[$H^2$ Regularity] \label{LemmaBC} Let $[u,v], [\bar{u}, \bar{v}] \in X_1$ be solutions to (\ref{EQ.NSR.1}) - (\ref{EQ.NSR.3}), with $f,g$ as in (\ref{bar.f}) - (\ref{bar.g}). For some $M_2$, perhaps large, dependent only on universal constants: 
\begin{align}  \n
\sup_{x \le 2000} ||u,v||_{L^\infty_y} &+ ||u,v||_{\dot{H}^2(x \le 2000)} \\   \label{boundary.controls}
& \lesssim \epsilon^{-M_2} \Big[C + \epsilon^{\frac{n}{2}+\gamma} ||\bar{u}, \bar{v}||_{L^\infty} \Big( ||\bar{u},\bar{v}||_{X_1} + ||u,v||_{X_1} \Big) + ||u,v||_{X_1} \Big].
\end{align}
\end{lemma}
\begin{proof}

Notationally, we will not take care to rename generic constants $M_2$ within this proof. We rescale the system back to Eulerian coordinates via: 
\begin{align} \label{scale}
\tilde{u}(x,Y) := u(x,y), \hspace{3 mm} \tilde{v}(x,Y) := \sqrt{\epsilon} v(x,y), \hspace{3 mm} \tilde{P}(x,Y) = \frac{1}{\epsilon}P(x,y), 
\end{align}

which, by rescaling (\ref{EQ.NSR.1}) - (\ref{EQ.NSR.3}), yields the Stokes-sytem, for some $M_2$ perhaps large:
\begin{align}\nonumber
&-\Delta \tilde{u} + \tilde{P}_x = \epsilon^{-M_2} \Big[ \tilde{f} + \tilde{S}_u \Big], \hspace{3 mm} -\Delta \tilde{v} + \tilde{P}_Y = \epsilon^{-M_2} \Big[ \tilde{g} + \tilde{S}_v \Big], \\  \label{stokes.1.1}
&\tilde{u}_x + \tilde{v}_Y = 0.  
\end{align}

Above, $\tilde{f}, \tilde{g}, \tilde{S}_u, \tilde{S}_v$ have been also rescaled to Eulerian coordinates. By considering the stream function $\psi = \int_0^Y \tilde{u}$, we obtain the following biharmonic problem: 
\begin{align} \label{F}
\Delta^2 \psi &= F := \epsilon^{-M_2} \{ \partial_Y \Big[ \tilde{f} + \tilde{S}_u \Big] - \partial_x \Big[ \tilde{g} + \tilde{S}_v \Big] \}, \\ 
\psi(x,0) &= \psi_Y(x,0) = 0, \hspace{3 mm} \psi(1,Y) = \psi_x(1,Y) = 0. 
\end{align}

We now introduce the partition of unity, $\{ \chi_m \}_{m \ge 1}$ of the region $[1,2000] \times [0,\infty)$. The specifications of these cut-off functions are as follows. First, define: 
\begin{align}
 \chi^{(a)}(x) &= \begin{cases}
1 \text{ for } 1 \le x \le 2000,\\
0 \text{ for } x \ge 4000, \hspace{10 mm}
\end{cases}
\chi^{(b)}_m(Y) = \begin{cases}
1 \text{ for } m-1 \le Y \le m+1,\\
0 \text{ for } Y \ge m+2, \\
0 \text{ for } Y \le m - 2. 
\end{cases} \label{chimb}
\end{align}

for all $m \ge 1$. Then define:  
\begin{equation} \label{chi}
\chi_m(x,Y) = \chi^{(a)}(x) \chi_m^{(b)}(Y). 
\end{equation}

The purpose of selecting such a partition is to localize near $x = 1$, in such a way that the region between $1$ and the support of $\rho_2$ is captured (see the definition in (\ref{rho})). Denote by $\bar{\psi}_m = \chi_m \psi$ Then: 
\begin{align}
\Delta^2 \bar{\psi}_m = \chi_m F + [\Delta^2, \chi_m] \psi, \hspace{3 mm} \bar{\psi}_m(x,0) = \partial_Y \bar{\psi}_m(x,0) = \bar{\psi}_m(1,Y) = \partial_x \bar{\psi}_m(1,Y) = 0. 
\end{align} 

Here the commutator is defined as: 
\begin{align} \nonumber
[\Delta^2, \chi]\psi:= &\chi_{YYYY} \psi + 6\chi_{YY} \psi_{YY} + 4 \chi_{YYY} \psi_Y + 4\chi_Y \psi_{YYY} + 6 \chi_{xx} \psi_{xx}  + 4\chi_{xxx} \psi_x \\ \nonumber
& + 4\chi_x \psi_{xxx} + \chi_{xxxx} \psi + \chi_{xxYY} \psi + 2\chi_{YY} \psi_{xx} + 2\chi_{YYx} \psi_x + 2\chi_x \psi_{xYY} \\ \label{BiLCOM} & + 2\chi_{xxY}\psi_y  + 2\chi_y \psi_{xxY} + 4\chi_{xY} \psi_{xY}.
\end{align}

According to \cite{Biharmonic}, Theorems 1 and 2, with $k=1$, coupled with Figure 2, P. 562 in \cite{Biharmonic} with ``$C/C$" boundary conditions, we have the following $H^3$ estimate for $\psi$: 
\begin{align} \label{BiHinH}
||\bar{\psi}_m||_{H^3} \lesssim ||\chi_m F||_{H^{-1}} +|| [\Delta^2, \chi_m] \psi ||_{H^{-1}} . 
\end{align}

We will first address the commutator terms in (\ref{BiHinH}). First, consider the terms in (\ref{BiLCOM}) which have three derivatives on $\psi$, which are denoted by $o(\partial^3 \psi)$. Using the definition of $H^{-1}$, for any compactly supported test function $\alpha(x,Y) \in H^1_0$, 
\begin{align} \nonumber
\Big| \langle o(\chi_m) o(\partial^3 \psi),& \alpha(x,y) \rangle_{H^{-1}, H^1_0} \Big| = \Big| \langle o(\partial \chi_m) o(\partial^2 \psi), \alpha \rangle \Big| + \Big| \langle o(\chi_m) o(\partial^2 \psi), \partial \alpha \rangle \Big| \\
& \le || o( \chi_m) o(\partial^2 \psi) ||_{L^2} ||\alpha||_{H^1_0}.
\end{align}

Then taking the $\sup$ over all $||\alpha||_{H^1_0} = 1$, we obtain: 
\begin{align}
||o(\chi_m) o(\partial^3 \psi) ||_{H^{-1}} \lesssim ||o(\chi_m) o(\partial^2 \psi)||_{L^2} \lesssim ||o(\chi_m) o(\partial) [\tilde{u}, \tilde{v}]||_{L^2} \lesssim \epsilon^{-M_2} ||o(\chi_m) ]\{u,v\}||_{X_1}. 
\end{align}

Next, turn to the terms in (\ref{BiLCOM}) which has two derivatives on $\psi$: 
\begin{align}
||o(\chi_m) o(\p^2 \psi)||_{H^{-1}} \le ||o(\chi_m) o(\p^2 \psi)||_{L^2} \le \eps^{-M_2}||o(\chi_m) ]\{u,v\}||_{X_1}.
\end{align}

Next, we must address those terms in (\ref{BiLCOM}) which has one derivatives on $\psi$. For this, we use the Poincare inequality in $x$ direction: 
\begin{align} \n
||o(\chi_m) \p \psi||_{H^{-1}} \le ||o(\chi_m) \p \psi||_{L^2} &\le ||o(\chi_m) \{u,v\}||_{L^2} \le ||u_x, v_x||_{L^2(x \le 2000)} \\
& \le ||o(\chi_m) \{ u_x, v_x \}||_{L^2} \le \eps^{-M_2}||o(\chi_m) \{u, v \}||_{X_1}.
\end{align}

For the zeroeth order terms in $\psi$, we must argue as follows: 
\begin{align} \nonumber
||o(\chi_m) \psi||_{H^{-1}} &\lesssim ||o(\chi_m) \psi||_{L^2} = ||o(\chi_m) \int_0^x v ||_{L^2} \lesssim || o(\chi_m) \frac{1}{x} \int_0^x v ||_{L^2} \\ 
&\lesssim  \eps^{-1} ||o(\chi_m) \{u,v\}||_{X_1}. 
\end{align}

Summarizing, then,  
\begin{align} \label{chya}
||[\Delta^2 , \chi_m] \psi||_{H^{-1}} \lesssim \epsilon^{-M_2} ||o(\chi_m) \{u,v\}||_{X_1},
\end{align}

It remains, then, to analyze the majorizing terms in $F$ in estimate (\ref{BiHinH}). Up to redefining $M_2$ in (\ref{F}), we may scale back to Prandtl coordinates $(x,y)$. By integrating by parts against compactly supported test functions, we have: 
\begin{align}
|| \chi_m \cdot [ \partial_y (f + S_u) - \partial_x (g + S_v) ] ||_{H^{-1}} \lesssim ||o(\chi_m) \Big[ f + S_u + g + S_v \Big]||_{L^2}.
\end{align}

First, we'll start with: 
\begin{align} \nonumber
||o(\chi_m) \cdot S_u||_{L^2} &= ||o(\chi_m) \Big[ u_R u_x + u_{Rx} u + v_R u_y + u_{Ry} v \Big] ||_{L^2} \\ \nonumber
&\le ||u_R, xu_{Rx}, v_R, u^P_{Ry} y, u^E_{RY}x^{\frac{3}{2}}||_{L^\infty} ||o(\chi_m) \{u_x, u_y, v_y, v_x\}||_{L^2} \\ 
&\lesssim ||o(\chi_m)\{u,v\}||_{X_1}.  
\end{align}

Similarly, 
\begin{align} 
||o(\chi_m) \cdot S_v||_{L^2} &= ||o(\chi_m) \Big[ u_R v_x + v_{Rx}u + v_R v_y + v_{Ry}v\Big]||_{L^2} \\ \nonumber
& \le ||u_R, xv_{Rx}, v_R, v_{Ry}y||_{L^\infty} ||o(\chi_m) \{ v_x, u_x, v_y\}||_{L^2} \lesssim ||o(\chi_m) \{u,v\}||_{X_1}. 
\end{align}

Next, we come to the nonlinear terms: 
\begin{align} \nonumber
\epsilon^{\frac{n}{2}+\gamma}&||o(\chi_m) \cdot \Big[\bar{u} \bar{u}_x + \bar{v} u_y + \bar{u} \bar{v}_x + \bar{v} \bar{v}_y\Big]||_{L^2} \\ \n
&\le \epsilon^{\frac{n}{2}+\gamma}||o(\chi_m) \{\bar{u}, \bar{v}\}||_{L^\infty} ||o(\chi_m) \{ \bar{u}_x, \bar{v}_x, u_y\}||_{L^2} \\
&\lesssim \epsilon^{\frac{n}{2}+\gamma} ||o(\chi_m)\{\bar{u}, \bar{v}\}||_{L^\infty} \Big( ||o(\chi_m) \{\bar{u}, \bar{v}\}||_{X_1} + ||o(\chi_m) \{u, v\}||_{X_1} \Big). 
\end{align}

Finally, we have the forcing terms in $f,g$ for which we cite Lemma \ref{Lemma FE}, 
\begin{align}
\sum_{m = 1}^\infty ||o(\chi_m) \{ R^{u,n}, R^{v,n}\}||_{L^2} \le C. 
\end{align}

Upon taking summation in $m$: 
\begin{align} \n
||\psi||_{H^3(x \le 2000)} &\le ||\sum_{m} \chi_m \psi ||_{H^3} \le \sum_m ||\chi_m \psi||_{H^3} \le \sum_m ||\chi_m F||_{H^{-1}} + \sum_m ||[\Delta^2, \chi_m] \psi||_{H^{-1}} \\ \n
& \lesssim  \eps^{-M_2} \Big[ \sum_{m} ||o(\chi_m) \{ R^{u,n}, R^{v,n}\}||_{L^2}  +  \sum_m ||o(\chi_m) \{u,v \}||_{X_1} \\ \n
& \hspace{10 mm} + \eps^{\frac{n}{2}+\gamma} ||o(\chi_m) \{||\bar{u}, \bar{v} \} ||_{L^\infty} \Big( ||o(\chi_m) \{\bar{u}, \bar{v}\}||_{X_1} + ||o(\chi_m) \{u, v\}||_{X_1} \Big) \Big] \\ \n
& \lesssim  \eps^{-M_2}\Big[ C + ||\{u, v \}||_{X_1} + \eps^{\frac{n}{2}+\gamma} ||\bar{u}, \bar{v}||_{L^\infty}\Big( || \{\bar{u}, \bar{v}\}||_{X_1} + ||\{u, v\}||_{X_1} \Big) \Big].
\end{align}

For the $L^\infty$ component of our desired claim, we simply use the standard $H^2$ embedding. 

\end{proof}

\begin{corollary} Suppose $||\bar{u}, \bar{v}||_Z \le 1$. There exists a universal constant $M_2$ such that for any selection of $N_2, N_4, n, \gamma$, the following estimate holds: 
\begin{align}\n
\epsilon^{N_2} ||u,v||_{Y_2}& \lesssim \epsilon^{N_2-M_2} + \epsilon^{ \frac{n}{2}+\gamma + N_2 -M_2 - N_4} ||\bar{u}, \bar{v}||_Z^2  \\  \label{Y1em}
& + \Big(\epsilon^{N_2-M_2} + \eps^{\frac{n}{2}+\gamma + N_2 - M_2 - N_4} \Big) ||u,v||_{X_1} + \epsilon^{N_2} ||u,v||_{X_2}.  
\end{align}
\end{corollary}

We will now bootstrap the above elliptic regularity, away from the boundary $x = 1$ (thereby avoiding the corners of our domain). 

\begin{lemma}[Third-Order Elliptic Regularity] \label{LemmaBC2} For some $M_3 \ge 0$ perhaps large, but independent of $n$, and supposing $||\bar{u}, \bar{v}||_Z \le 1$, we have: 
\begin{align} \n
|| \{u_x, \sqrt{\epsilon}v_x \}\cdot \zeta_3  ||_{\dot{H}^2(x \le 1000)} &\lesssim \epsilon^{-M_3} \Big[ C + ||u,v||_{X_1 \cap Y_2} + \eps^{\frac{n}{2}+\gamma - N_4 - N_2} ||u,v||_{X_1 \cap Y_2} \Big] \\ 
& + \epsilon^{-M_3 - 2N_2 - 2N_4 + \frac{n}{2}+\gamma} ||\bar{u}, \bar{v}||_Z^2.
\end{align}
\end{lemma}
\begin{proof}

Start with the system in (\ref{F}), and differentiate once in $x$:
\begin{align} \n
\Delta^2 \psi_x &= F_x := \eps^{-M_2} \{ \p_Y \Big[ \tilde{f}_x + \p_x\tilde{S}_u \Big] - \p_x \Big[ \tilde{g}_x + \p_x \tilde{S}_v \Big] \}, \\
\p_x \psi(x,0) &= \p_x \psi_Y(x,0) = 0. 
\end{align}

We shall now define a new cut-off function, via: 
\begin{align}
 \chi^{(2,a)}(x) &= \begin{cases}
0 \text{ for } 1 \le x \le \frac{3}{2},\\
1 \text{ for } 2 \le x \le 1000, \\
0 \text{ for } 2000 \le x. 
\end{cases}
\end{align}

Then, referring back to (\ref{chimb}),
\begin{align}  \label{chi2}
\chi^{(2)}_m(x,Y) := \chi^{(2,a)}(x) \chi_m^{(b)}(Y)
\end{align}

The collection $\{ \chi^{(2)}_m\}$ is meant to fill the gap between $\zeta_3$ and $\rho_3$ (see the definitions in (\ref{zeta}), (\ref{rho})).  We will consider the unknown $ \chi^{(2)}_m \psi_x$, which satisfies the system: 
\begin{align}
\Delta^2 \Big(  \chi^{(2)}_m \psi_x \Big) = \chi^{(2)}_m F_x + [\Delta^2,  \chi^{(2)}_m] \psi_x,
\end{align}

where the commutator expression is given by (\ref{BiLCOM}). Then via the standard, local in $x$, $H^3$ estimate, one has: 
\begin{align}
||\chi^{(2)}_m \psi_x ||_{H^3} \lesssim ||\chi^{(2)}_m F_x||_{H^{-1}} + ||[\Delta^2,\chi^{(2)}_m] \psi_x||_{H^{-1}}. 
\end{align} 

Again, as $[\Delta^2, \chi^{(2)}_m]$ is localized in $x$ and contains three-derivatives of $\psi_x$, we easily obtain: 
\begin{align}
||[\Delta^2, \chi^{(2)}_m] \psi_x||_{H^{-1}} \lesssim \eps^{-M_3} ||o(\chi^{(2)}_m) \{u,v\}||_{Y_2 \cap X_1}. 
\end{align}

We will now evaluate the $F_x$ term above. By selecting the exponent $M_3$ large enough, one can rescale $\tilde{f}_x, \tilde{g}_x$ to $f_x, g_x$ and $\partial_x \tilde{S}_u, \partial_x \tilde{S}_v$ to $S_u, S_v$, which we automatically do. Then, 
\begin{align} \n
||\chi^{(2)}_m F_x||_{H^{-1}} &\le \eps^{-M_3} \Big[ ||\chi^{(2)}_m f_{xy} ||_{H^{-1}} + ||\chi^{(2)}_m g_{xx}||_{H^{-1}} \Big] \\
& \le \eps^{-M_3} \Big[ ||o(\chi^{(2)}_m) f_{x} ||_{L^{2}} + ||o(\chi^{(2)}_m) g_{x}||_{L^2} \Big].
\end{align}

Then, according to the definition (\ref{bar.f}) - (\ref{bar.g}), one has for the nonlinear terms: 
\begin{align} \n
||o(\chi^{(2)}_m) \cdot \mathcal{N}^u_x||_{L^2} &= \epsilon^{\frac{n}{2}+\gamma} ||o(\chi^{(2)}_m)[\bar{u} \bar{u}_{xx} + \bar{u}_x^2 + \bar{v}_x u_y + \bar{v} u_{xy}]||_{L^2} \\ \n
& \le  \epsilon^{\frac{n}{2}+\gamma} \Big[ ||o(\chi^{(2)}_m) \bar{u}||_{L^\infty} ||o(\chi^{(2)}_m)\bar{u}_{xx}||_{L^2} + ||o(\chi^{(2)}_m) \bar{u}_x||_{L^4}^2 \\ \n
 &\hspace{10 mm} + ||o(\chi^{(2)}_m)\bar{v}_x||_{L^4} ||o(\chi^{(2)}_m) u_y||_{L^4} + ||o(\chi^{(2)}_m) \bar{v}||_{L^\infty} ||o(\chi^{(2)}_m) u_{xy}||_{L^2} \Big] \\ \n
 &\le \epsilon^{\frac{n}{2}+\gamma} \Big[ ||o(\chi^{(2)}_m) \bar{u}||_{L^\infty} ||o(\chi^{(2)}_m)\bar{u}_{xx}||_{L^2} + ||o(\chi^{(2)}_m) \bar{u}_x||_{L^2}||o(\chi^{(2)}_m) \bar{u}_x||_{\dot{H}^1} \\ \n
 &\hspace{10 mm} + ||o(\chi^{(2)}_m)\bar{v}_x||_{L^2}^{\frac{1}{2}} ||o(\chi^{(2)}_m)\bar{v}_x||_{\dot{H}^1}^{\frac{1}{2}} ||o(\chi^{(2)}_m) u_y||_{L^2}^{\frac{1}{2}} ||o(\chi^{(2)}_m) u_y||_{\dot{H}^1}^{\frac{1}{2}} \\ \n
 & \hspace{10 mm} + ||o(\chi^{(2)}_m) \bar{v}||_{L^\infty} ||o(\chi^{(2)}_m) u_{xy}||_{L^2} \Big] \\ \label{acc.3.1}
& \lesssim \epsilon^{\frac{n}{2}+\gamma} \Big[ \eps^{-2N_2 - 2N_4} ||o(\chi^{(2)}_m) \{\bar{u}, \bar{v}\}||_Z^2 + \eps^{-2N_2 - 2N_4} ||o(\chi^{(2)}_m) \{u,v\}||_{X_1 \cap Y_2} \Big].
\end{align}

Above, we have used the assumption that $||\bar{u}, \bar{v}||_Z \le 1$ in the term: 
\begin{align} \n
 ||o(\chi^{(2)}_m)\bar{v}_x||_{L^2}^{\frac{1}{2}} ||o(\chi^{(2)}_m)\bar{v}_x||_{\dot{H}^1}^{\frac{1}{2}} & ||o(\chi^{(2)}_m) u_y||_{L^2}^{\frac{1}{2}} ||o(\chi^{(2)}_m) u_y||_{\dot{H}^1}^{\frac{1}{2}} \\ \n
 & \le \eps^{-N_2}||\bar{v}||_Z ||o(\chi^{(2)}_m) u_y||_{L^2}^{\frac{1}{2}} ||o(\chi^{(2)}_m) u_y||_{\dot{H}^1}^{\frac{1}{2}} \\  \n
 & \le \eps^{-N_2} ||o(\chi^{(2)}_m) u_y||_{L^2}^{\frac{1}{2}} ||o(\chi^{(2)}_m) u_y||_{\dot{H}^1}^{\frac{1}{2}} \\
 & \le \eps^{-N_2} ||o(\chi_m^{(2)}) u_y||_{X_1}^{\frac{1}{2}} \eps^{-N_2} ||o(\chi_m^{(2)}) u||_{Y_2}^{\frac{1}{2}}.
\end{align}

In the final estimate, we have used that $\chi^{(2)}$ is supported on a strictly smaller region in $x$ than the region over which we have controlled $u_{yy}$ (compare (\ref{chi}) to the $u_{yy}$ estimate in (\ref{boundary.controls})). Similarly, 
\begin{align}\n
|| o(\chi^{(2)}_m) \mathcal{N}^v_x||_{L^2} &=  \epsilon^{\frac{n}{2}+\gamma} || o(\chi^{(2)}_m) [ \bar{u} \bar{v}_{xx} +  \bar{u}_x \bar{v}_x + \bar{v}_x \bar{v}_y + \bar{v} \bar{v}_{xy}] ||_{L^2} \\ \n
& \le \epsilon^{\frac{n}{2}+\gamma} \Big[ ||o(\chi^{(2)}_m) \bar{u}||_{L^\infty} ||o(\chi^{(2)}_m) \bar{v}_{xx}||_{L^2} + ||o(\chi^{(2)}_m) \bar{u}_x||_{L^4} ||o(\chi^{(2)}_m) \bar{v}_x||_{L^4} \\ \n
&+ ||o(\chi^{(2)}_m) \bar{v}_x||_{L^4} ||o(\chi^{(2)}_m) \bar{v}_y||_{L^4} + ||o(\chi^{(2)}_m) \bar{v}||_{L^\infty} ||o(\chi^{(2)}_m) \bar{v}_{xy}||_{L^2}  \Big] \\ \label{acc.2}
& \le \epsilon^{\frac{n}{2}+\gamma} \Big[ \eps^{-2N_2 - 2N_4} ||o(\chi^{(2)}_m) \{ \bar{u}, \bar{v}\}||_Z^2 \Big]. 
\end{align}

Next, according to estimates (\ref{Rnl2}), one has: 
\begin{align} \label{acc.1}
\sum_{m \ge 1} ||o(\chi^{(2)}_m) R^{u,n}_x, o(\chi^{(2)}_m) R^{v,n}_x||_{L^2} \le C. 
\end{align}

Taking summation in $m$ and scaling back to Prandtl coordinates gives the desired result upon comparing the supports of $\chi_2$ and $\rho_2$:
\begin{align} \n
|| \chi^{(2)} \psi||_{H^3} &\le \sum_{m} ||\chi^{(2)}_m \psi||_{H^3} \\ \n 
& \lesssim \eps^{-M_3}  \sum_{m} ||o(\chi^{(2)}_m) u,v||_{Y_2 \cap X_1} +  \epsilon^{\frac{n}{2}+\gamma}  \sum_{m} \eps^{-2N_2 - 2N_4} ||o(\chi^{(2)}_m) \{\bar{u}, \bar{v}\}||_Z^2 \\ \n
& +\epsilon^{\frac{n}{2}+\gamma} \eps^{-N_2 - N_4} \sum_m ||o(\chi^{(2)}_m) \{u,v\}||_{X_1 \cap Y_2} + \sum_m ||o(\chi^{(2)}_m) R^{u,n}_x, o(\chi^{(2)}_m) R^{v,n}_x||_{L^2} \\
& \lesssim \eps^{-M_3} \Big[  C + ||u,v||_{Y_2 \cap X_1} + \eps^{\frac{n}{2}+\gamma - 2N_2 - 2N_4} ||\bar{u}, \bar{v}||_Z^2 + \eps^{\frac{n}{2}+\gamma - N_2 - N_4} ||u,v||_{X_1 \cap Y_2} \Big]. 
\end{align}

\end{proof}

\begin{corollary} There exists a universal constant $M_3$ such that for any selection of $N_2, N_3, N_4, n, \gamma$, so long as $||u, v||_Z \le 1$, we have: 
\begin{align} \n
\epsilon^{N_3} ||u,v||_{Y_3}& \lesssim \epsilon^{N_3-M_3} \Big[C + ||u,v||_{X_1 \cap Y_2}   \Big] + \eps^{N_3} ||u,v||_{X_3} \\ \label{Y3em}
&+ \eps^{\frac{n}{2}+\gamma - N_2 - N_4 - M_3 + N_3}||u,v||_{X_1 \cap Y_2} + \epsilon^{N_3 -M_3 - 2N_2 - 2N_4 + \frac{n}{2}+\gamma} ||\bar{u}, \bar{v}||_Z^2.
\end{align}
\end{corollary}

We now upgrade the previous estimate to a fourth order, weighted estimate. This fourth order estimate is not included in our norm, (\ref{norm.Z}). This will simply be used in order to justify rigorously one of the integrations by parts in the energy estimates (in particular, calculation (\ref{justify.HO})). This is the reason our right-hand side below in (\ref{rhsb}) need not be depicted explicitly; we only need the qualitative information that this quantity is finite.  

\begin{lemma}[Fourth-Order Elliptic Regularity] Suppose $||\bar{u}, \bar{v}||_Z \le 1$. Solutions $[u,v,P] \in Z$ to the system (\ref{EQ.NSR.1}) - (\ref{EQ.NSR.3}), with forcing terms as in (\ref{bar.f}) - (\ref{bar.g}) satisfy the following fourth-order estimate: 
\begin{align} \label{rhsb}
||\{u_{xx}, \sqrt{\eps}v_{xx} \} \cdot x ||_{\dot{H}^2(x \ge 20)}  < \infty.  
\end{align}
\end{lemma}
\begin{proof}

Taking two derivatives of (\ref{F}), one has $\Delta^2 (\psi_{xx}) = F_{xx}$. Define a partition of unity of $\{x \ge 20\}$, $\{ \chi_m \}_{m \ge 0}$, such that $\chi_0$ is supported on $Y \ge 1$, and all of the $\chi_m, m \ge 1$ are localized near the boundary $Y = 0$ in the region $Y \in [0,2]$ and near $x = m$. Then: 
\begin{align} \label{CHIM}
\Delta^2 (\chi_m x \psi_{xx} ) = [ \Delta^2, \chi_m x ] \psi_{xx} + x \chi_m F_{xx}, \\ \label{chibulk}
\Delta^2 (\chi_0 x \psi_{xx} ) = [ \Delta^2, \chi_0 x ] \psi_{xx} + x \chi_0 F_{xx},
\end{align}

where we refer the reader to the commutator expression in (\ref{BiLCOM}). For equation (\ref{chibulk}), one uses the standard, interior $\dot{H}^3$ estimate for the Bi-Laplacian:
\begin{align} \label{four.four.1}
||\chi_0 x \psi_{xx} ||_{\dot{H}^3} \lesssim ||[\Delta^2, \chi_0 x] \psi_{xx}||_{H^{-1}} + ||x \chi_0 F_{xx} ||_{H^{-1}}. 
\end{align}

For equation (\ref{CHIM}), one uses the boundary $H^3$ estimate which gives: 
\begin{align} \label{four.four.2}
||\chi_m x \psi_{xx} ||_{H^3} \lesssim || [\Delta^2, \chi_m x] \psi_{xx}||_{H^{-1}} + ||x \chi_m F_{xx} ||_{H^{-1}}.
\end{align}

Let us write the expression for $F_{xx}$ that we will read from, referring to (\ref{F}) (we will rename the power of $\eps$)
\begin{align}
F_{xx} = \eps^{-M_4} \Big[ \p_{xxY} \tilde{f} + \p_{xxY} \tilde{S}_u - \p_{xxx} \tilde{g} - \p_{xxx} \tilde{S}_v \Big].
\end{align}

As usual, by sacrificing powers of $\eps$, it suffices to evaluate the $H^{-1}$ norm of the above expression in Prandtl coordinates. We will start with $S_u$: 
\begin{align}
||\{\chi_m, \chi_0 \}x \p_{xxY} \tilde{S}_u||_{H^{-1}} \lesssim ||\{\chi_m, \chi_0 \}x \p_{xx} \tilde{S}_u||_{L^2} + ||\{o(\p_y \chi_m, \p_y \chi_0) x \p_{xx} \tilde{S}_u \}||_{L^2}. 
\end{align}

Similarly for $S_v$, one obtains: 
\begin{align} \n
||o(\chi_m, \chi_0)\p_{xxx} xS_v||_{H^{-1}} &\le ||o(\chi_m, \chi_0)\p_{xx} S_v||_{L^2} + ||o(\p_x \chi_m, \p_x \chi_0)\p_{xx} xS_v||_{L^2} \\ 
& + ||o(\chi_m, \chi_0)x \p_{xx} S_v||_{L^2} .
\end{align}

Computing two derivatives of $S_u, S_v$, one obtains: 
\begin{align} \n
\p_{xx} S_u &= u_{Rxxx}u + 2u_{Rxx}u_x + u_{Rx} u_{xx} + u_{Rxx} u_x + 2u_{Rx}u_{xx} \\ 
& + u_R u_{xxx} + u_{Ryxx} v + 2u_{Ryx}v_x + u_{Ry}v_{xx} + v_{Rxx}u_y + 2v_{Rx}u_{xy} + v_R u_{xxy}, \\ \n
\p_{xx}S_v &= u_{Rxx} v_x + 2 u_{Rx} v_{xx} + u_R v_{xxx} + v_{Rxxx} u + 2v_{Rxx} u_x \\ 
& + v_{Rx} u_{xx} + v_{Rxx} v_y + 2v_{Rx} v_{xy} + v_R v_{xxy} + v_{Rxxy} v + 2v_{Rxy}v_x + v_{Ry} v_{xx}.
\end{align}

From here, given $\chi_0, \chi_m$ are supported on $x \ge 20$, it is easy to see that: 
\begin{align} \label{four.2.n}
||o(\chi_m, \chi_0)\Big[ x\p_{xx} S_u,  x \p_{xx} S_v \Big]||_{L^2(x \ge 20)}  \lesssim ||o(\chi_m, \chi_0)\{u,v\}||_{Z}.
\end{align}

Next, turning to the expressions in $f, g$: 
\begin{align}
f_{xx} &=  \eps^{-\frac{n}{2}-\gamma}R^{u,n}_{xx} + \bar{u}\bar{u}_{xxx} + 3\bar{u}_x \bar{u}_{xx} + \bar{v}_{xx}u_y + 2 \bar{v}_x u_{xy} + \bar{v}u_{xxy}, \\
g_{xx} &= \eps^{-\frac{n}{2}-\gamma}R^{v,n}_{xx} + \bar{u}_{xx} \bar{v}_x + \bar{u}\bar{v}_{xxx} + 2\bar{u}_x \bar{v}_{xx} + \bar{v}_{xx}\bar{v}_y + \bar{v} \bar{v}_{xxy} + 2\bar{v}_x \bar{v}_{xy}.
\end{align}

From here, using (\ref{Rnl2}), and the definition the norm $Z$ in (\ref{norm.Z}), one observes: 
\begin{align} \label{four.2}
||o(\chi_m, \chi_0)\Big[ x\p_{xx}f, x\p_{xx} g \Big]||_{L^2(x \ge 20)}  \lesssim ||o(\chi_m, \chi_0)\{u,v\}||_{Z}^2 + ||o(\chi_m, \chi_0)\{\bar{u},\bar{v}\}||_{Z}^2.
\end{align}

Summarizing this: 
\begin{align} \n
\sum_{m \ge 0} ||x \chi_m F_{xx}||_{H^{-1}} &\lesssim \sum_{m \ge 0} ||o(\chi_m) \{u,v \}||_Z + \sum_{m \ge 0} ||o(\chi_m) \{u,v \}||_Z^2 \\ \label{four.BED}
& + \sum_{m \ge 0} ||o(\chi_m) \{\bar{u},\bar{v} \}||_Z  < \infty. 
\end{align}

We now move to the commutator terms from (\ref{four.four.1}) - (\ref{four.four.2}), which we write out, denoting by $\chi$ generically either $\chi_m$ or $\chi_0$: 
\begin{align} \n
[\Delta^2, x \chi] v_x &= \p_{xyy}(x \chi) v_{xx} + \p_{xxy}(x \chi) v_{xy} + \p_x(x \chi) \p_{xyy} v_x + x \p_{yy} \chi v_{xxx} \\ \n
& + \p_y (x \chi) \p_{xxy} v_x + \p_{xx} (x \chi) v_{xyy} + \p_{xy} (x \chi) v_{xxy} \\ \label{COM2}
& + \sum_{k = 1}^4 \p_x^k (x \chi) \p_x^{4-k} v_x + \sum_{k =1}^4 x \p_y^k \chi \p_y^{4-k} v_x. 
\end{align}

Examining the commutator terms, one observes that the worst term is when all derivatives fall on the cut-off, and none on either $x$ or $v_x$. Such a term arises, for instance, when $k = 4$ in the final summation above in (\ref{COM2}). In this case, one uses that the partition of unity is selected such that $\p_y \chi_0, \chi_m$ is a bounded region in $y$: 
\begin{align}
||\p_y^4 \{\chi_m, \chi_0\} \cdot xv_x||_{H^{-1}} \le ||\p_y^4 \{\chi_m, \chi_0\} \cdot xv_x||_{L^2} \lesssim ||o(\chi_m), o(\p \chi_0) xv_{xy}||_{L^2}.
\end{align}

It is straightforward to check that all terms in the commutator can be estimated either in this manner or directly using the definitions of $Z$ in (\ref{norm.Z}). Thus, taking summation over the partition of unity again gives: 
\begin{align}\label{four.BED.2}
\sum_{m \ge 0} ||[\Delta^2, \chi_m x] \psi_{xx} ||_{H^{-1}} \lesssim ||u,v||_Z < \infty. 
\end{align}

Combining (\ref{four.four.1}), (\ref{four.four.2}), (\ref{four.BED}), and (\ref{four.BED.2}), the lemma is proven. 

\end{proof}

\subsection{Embedding Theorems for the Space $Z$} \label{subsection.EM}

The next task is to pinpoint the interplay between the uniform quantities in the norm $Z$ and the norms $Y_i$. Let us start with: 

\begin{lemma} \label{L.Evol} For $\sigma > 0$ arbitrarily small,  
\begin{align} \label{evo.low}
&\sup_{x \ge 1} \Big[ || \sqrt{\eps} \psi x^{-1-\sigma}||_{L^2_y}^2 + ||ux^{-\sigma}||_{L^2_y}^2 + ||\sqrt{\eps} v x^{-\sigma}||_{L^2_y}^2 \Big] \lesssim C(\sigma) ||u,v||_{X_1}^2, \\ \label{evo.mid}
&\sup_{x \ge 1} \Big[ ||u_y x^{\frac{1}{2}}||_{L^2_y}^2 + ||u_x x||_{L^2_y}^2 \Big] + \sup_{x \ge 20} ||\sqrt{\epsilon} v_x x||_{L^2_y}^2 \lesssim ||u,v||_{X_1 \cap Y_2}^2, \\ \label{evo.high}
&\sup_{x \ge 20} \Big[ ||u_{xy} x^{\frac{3}{2}}||_{L^2_y}^2 + ||\{v_{xy}, \sqrt{\eps}v_{xx} \} x^2||_{L^2_y}^2 \Big] \lesssim ||u,v||_{Y_2 \cap Y_3}^2.
\end{align}
The constant $C(\sigma) \uparrow \infty$ as $\sigma \downarrow 0$. Finally, for $[u,v] \in Z$, we have the following property: 
\begin{align} \label{three}
\sup_{x \ge 20} ||\{v_{xxx}, u_{xxy} \} x||_{L^2_y} < \infty, 
\end{align}
\end{lemma}
\begin{remark} The estimate (\ref{three}) is required in order to rigorously justify one integration by parts in our energy estimates, in particular calculation (\ref{formal.three}), and is not required as part of the norm $Z$, which is why we do not characterize the right-hand side. The most important parts of this lemma are the first three estimates, (\ref{evo.low}) - (\ref{evo.high}). 
\end{remark}
\begin{proof}

First, take a differentiation of: 
\begin{align}
\partial_x \int u^2 x^{-2\sigma} \ud y = 2\int uu_x x^{-2\sigma} \ud y -2\sigma \int u^2 x^{-1-2\sigma} \ud y,
\end{align}

which upon an integration in $x$, and recalling $u(1,y) = 0$, yields: 
\begin{align}
\int u^2 x_1^{-2\sigma} \ud y = 2\int_1^{x_1} \int uu_x x^{-2\sigma} \ud x \ud y - 2\sigma \int_1^{x_1} \int u^2 x^{-1-2\sigma} \ud x \ud y.
\end{align}

Taking absolute values, applying Holder's inequality, and taking the supremum in $x_1$:
\begin{align} \label{crit.H.1}
\sup_{x \ge 1} \int u^2 x^{-2\sigma} \lesssim ||ux^{-\sigma - \frac{1}{2}}||_{L^2} ||u_x x^{\frac{1}{2}}||_{L^2} + ||ux^{-\frac{1}{2}-\sigma}||_{L^2}^2 \lesssim ||u_x x^{\frac{1}{2}}||_{L^2}^2 \lesssim ||u,v||_{X_1}^2. 
\end{align}

Here the factor of $\sigma > 0$ is required to avert the critical Hardy inequality, which occurs with weight $x^{-\frac{1}{2}}$ in $L^2$. The stream function estimate follows similarly: 
\begin{align} \nonumber
\partial_x \int \psi^2 x^{-2-\sigma} &= 2\int \psi v x^{-2-\sigma} -(2-\sigma) \int \psi^2 x^{-3-\sigma},
\end{align}

and so integration gives: 
\begin{align} \nonumber
\sup_{x \ge 1} \int \psi^2 x^{-2-\sigma} &\le ||\psi v x^{-2-\sigma}||_{L^1} + ||\psi x^{-\frac{3}{2}-\frac{\sigma}{2}}||_{L^2}^2 \lesssim ||vx^{-\frac{1}{2}-\frac{\sigma}{2}}||_{L^2}^2 \lesssim ||v_x x^{\frac{1}{2}-\frac{\sigma}{2}}||_{L^2}^2 \\
& \lesssim ||u,v||_{X_1}^2. 
\end{align}

The estimate for $v$ in (\ref{evo.low}) works in a similar fashion. We will now move to first-order estimates in (\ref{evo.mid}). The $u_y$ estimate follows easily after a differentiation: 
\begin{align}
\partial_x \int u_y^2 x = 2\int u_y u_{xy} x + \int u_y^2. 
\end{align}

Taking an integration in $x$, and recalling that $u_y(1,y) = 0$, then gives: 
\begin{align}
\sup_{x \ge 1} \int u_y^2 x \lesssim ||u_y||_{L^2} ||u_{xy} x||_{L^2} + ||u_y||_{L^2}^2 \lesssim ||u,v||_{X_1 \cap Y_2}^2.  
\end{align}

Next, 
\begin{align}
\partial_x \int u_x^2 x^2 = 2\int u_x u_{xx} x^2 + 2\int u_x^2 x, 
\end{align}
so taking an $x$-integration and using that $u_x(1,y) = 0$, we have: 
\begin{align} \nonumber
\sup_{x \ge 1} \int u_x^2 x^2 &\lesssim ||u_x x^{\frac{1}{2}}||_{L^2} ||u_{xx} x^{\frac{3}{2}}||_{L^2} + ||u_x x^{\frac{1}{2}}||_{L^2}^2 \lesssim  ||u||_{X_1 \cap Y_2}^2.  
\end{align}

Let us now introduce new cut-off functions, for $k = 2,3$: 
\begin{align}
 \eta_k(x) &= \begin{cases}
0 \text{ for } 1 \le x \le 3 + 6(k-2),\\
1 \text{ for } x \ge 6 + 6(k-2).
\end{cases} \label{eta}
\end{align}

Then $\eta_{2}, \eta_3$ are ``in-between" $\zeta_3$ and $\rho_{2}, \rho_3$. Consider now the quantity $\int \epsilon v_x^2 x^2 \eta_2(x)$:
\begin{align}
\partial_x \int \epsilon v_x^2 x^2 \eta_2(x) = 2 \int \epsilon v_x^2 x \eta_2(x) + 2\int \eps v_x  v_{xx} x^2 \eta_2 + \int \eps v_x^2 x^2 \eta_2'(x). 
\end{align}

As $\eta_2$ vanishes on $[1,3]$, we can take the integration up from $x = 1$: 
\begin{align}
\sup_{x \ge 20} \int \eps v_x^2 x^2 \le \sup_{x \ge 1} \int \epsilon v_x^2 x^2 \eta_2(x) \lesssim ||\sqrt{\eps} v_x x^{\frac{1}{2}}||_{L^2}^2 + ||\sqrt{\eps} v_{xx} x^{\frac{3}{2}}||_{L^2}^2, 
\end{align}

where we have used that $|\eta_2'(x) x^2| \lesssim 1$. We now move to the second-order estimates in (\ref{evo.high}), starting with: 
\begin{align}
\partial_x \int u_{xy}^2 x^3 \eta_3  = 2\int u_{xy} u_{xxy} x^3 \eta_3 + \int u_{xy}^2 3x^2 \eta_3 + \int u_{xy}^2 x^3 \eta_3'. 
\end{align}

Using that $\eta_3 u_{xy}(1,y) = 0$, we have: 
\begin{align} 
\sup_{x \ge 1} \int u_{xy}^2 x^3 \eta_3 &\lesssim ||u_{xy} x||_{L^2}^2 + ||u_{xy} x||_{L^2} ||u_{xxy} x^2 \eta_3||_{L^2} \lesssim ||u,v||_{Y_2 \cap Y_3}^2.
\end{align}

Above, we have used that $\eta_3 \le \zeta_3$. The final calculation is: 
\begin{align} \label{d1}
\partial_x \int v_{xy}^2 x^{4} \eta_3 = C \int v_{xy} v_{xxy} x^{4} \eta_3 + C \int v_{xy}^2 x^3 \eta_3 + C \int v_{xy}^2 x^2 \eta_3'.
\end{align}

Integrating up from $x = 1$, 
\begin{align}
\sup_{x \ge 1} \int v_{xy}^2 x^4 \eta_3 \lesssim ||v_{xy} x^{\frac{3}{2}}||_{L^2}^2 +  ||v_{xxy} x^{\frac{5}{2}} \eta_3||_{L^2}^2 \lesssim ||u,v||_{Y_2 \cap Y_3}^2. 
\end{align}

It is clear that $\sqrt{\eps}v_{xx}$ works in an identical manner to (\ref{d1}), and also in an identical manner, one obtains (\ref{three}) by pairing with (\ref{rhsb}).

\end{proof}

We will now record the following about the $x \rightarrow \infty$ behavior of elements in $(X_1 \cap X_2 \cap X_3)(\Omega)$:
\begin{lemma} \label{L.XiBC} Suppose $[u,v] \in X_1 \cap X_2 \cap X_3(\Omega)$. Then the following boundary conditions are automatically enforced: 
\begin{align}
[u,v]|_{x = 1} = [u,v]|_{y = 0} = \lim_{x \rightarrow \infty} [u,v] = \lim_{y \rightarrow \infty} [u,v] = 0
\end{align}
\end{lemma}
\begin{proof}
All follow as in Lemma \ref{Lemma.BCZ} aside from the condition at $x \rightarrow \infty$. For this, we use:
\begin{align} \label{too.weak.1}
||ux^{\frac{1}{4}-\sigma}||_{L^\infty_y} \le ||u x^{-2\sigma} ||_{L^2_y}^{\frac{1}{2}}||u_y x^{\frac{1}{2}}||_{L^2_y}^{\frac{1}{2}} < \infty, \\ \label{too.weak.2}
||vx^{\frac{1}{2}-\sigma}||_{L^\infty_y} \le ||v x^{-2\sigma} ||_{L^2_y}^{\frac{1}{2}}||v_y x ||_{L^2_y}^{\frac{1}{2}} < \infty,
\end{align}

according to the estimates in (\ref{evo.low}). Thus, we can conclude that $[u,v] \rightarrow 0$ as $x \rightarrow \infty$. 

\end{proof}

The estimates (\ref{too.weak.1}) - (\ref{too.weak.2}) yield uniform decay of $[u,v]$ as $x \rightarrow \infty$. However, the factor of $x^{-\sigma}$ is too weak to close our nonlinear analysis (see estimate (\ref{exact.NL})). This in turn is caused by Hardy's inequality being critical (see (\ref{crit.H.1})). In the next lemma, we use higher-order decay estimates to avoid this criticality: 

\begin{lemma}[Uniform Embeddings] \label{Lemma.UIMP} For $[u,v] \in X_1 \cap X_2 \cap X_3(\Omega)$, we have: 
\begin{align} \label{u0}
&\sup_{x \ge 1} || u x^{\frac{1}{4}}, \sqrt{\eps} v x^{\frac{1}{2}}||_{L^\infty_y} \lesssim ||u,v||_{X_1 \cap Y_2 \cap Y_3} , \\ \label{Uf}
& \sup_{x \ge 20} ||u_x x^{\frac{5}{4}}, \sqrt{\eps} v_x x^{\frac{3}{2}}||_{L^\infty_y} \lesssim ||u,v||_{X_1 \cap Y_2 \cap Y_3}.  
\end{align}
\end{lemma}
\begin{proof}

We will first turn to the first-order estimates in (\ref{Uf}). These follow from the evolution estimates in the previous lemma via standard Sobolev interpolation: 
\begin{align}
||u_x x^{\frac{5}{4}} ||_{L^\infty_y} \le ||u_x x||_{L^2_y}^{\frac{1}{2}} ||u_{xy} x^{\frac{3}{2}} ||_{L^2_y}^{\frac{1}{2}}.
\end{align}

Similarly, 
\begin{align}
||v_x x^{\frac{3}{2}}||_{L^\infty_y} \le ||v_x x||_{L^2_y}^{\frac{1}{2}} ||v_{xy} x^{2}||_{L^2_y}^{\frac{1}{2}}. 
\end{align}

The result now follows after taking the supremum in $x \ge 20$ and appealing to the evolution estimates above. Next, we address (\ref{u0}). For $x \le 20$, we may simply appeal to the uniform estimates in estimate (\ref{boundary.controls}). For $v$, the result follows immediately from: 
\begin{align}
\Big| v(x,y)\Big| = \Big|\int_x^\infty v_x(x',y) dx' \Big| \le ||v_x x^{\frac{3}{2}}||_{L^\infty_y} \int_x^\infty (x')^{-\frac{3}{2}} dx' = \Big[ \sup_{x \ge 20} ||v_x x^{\frac{3}{2}}||_{L^\infty_y} \Big] x^{-\frac{1}{2}}. 
\end{align}

We have used the qualitative boundary condition that $v \rightarrow 0$ as $x \rightarrow \infty$, which is available due to Lemma \ref{L.XiBC}. For the $u$ estimate follows in the same manner: 
\begin{align}
\Big| u(x,y) \Big| \le ||u_x x^{\frac{5}{4}}||_{L^\infty_y} \int_x^\infty (x')^{-\frac{5}{4}}dx' \lesssim  \Big[ \sup_{x \ge 20} ||u_x x^{\frac{5}{4}}||_{L^\infty_y} \Big] x^{-\frac{1}{4}}.
\end{align}

Now taking the sup in $x,y$ of both of the above inequalities yields the desired result. 

\end{proof}

We now turn to controlling the top-order uniform-type norm:
\begin{lemma}[Top Order Uniform Embedding] For $[u,v] \in X_1 \cap Y_2 \cap Y_3$, one has the following mixed-norm estimate:
\begin{align} \label{unitop}
 \int_{20}^\infty x^4 ||\sqrt{\eps} v_{xx}||_{L^\infty_y}^2 dx + \int_{20}^\infty x^{\frac{7}{2}} ||u_{xx}||_{L^\infty_y}^2 dx \lesssim ||u,v||_{X_1 \cap Y_2 \cap Y_3}^2. 
\end{align}
\end{lemma}
\begin{proof}
For $x \ge 20$, we start with:
\begin{align} \nonumber
x^2 ||\sqrt{\eps}v_{xx}||_{L^\infty_y} &\le ||x^{\frac{3}{2}}\sqrt{\eps} v_{xx}||_{L^2_y}^{\frac{1}{2}} ||x^{\frac{5}{2}} \sqrt{\eps} v_{xxy}||_{L^2_y}^{\frac{1}{2}}. 
\end{align}

Taking square on both sides: 
\begin{align}
x^4 ||\sqrt{\eps} v_x||_{L^\infty_y}^2 \le ||x^{\frac{3}{2}}\sqrt{\eps} v_{xx}||_{L^2_y}^2 + ||x^{\frac{5}{2}} \sqrt{\eps} v_{xxy}||_{L^2_y}^2,
\end{align}

so integrating
\begin{align}
 \int_{20}^\infty x^4 ||\sqrt{\eps} v_{xx}||_{L^\infty_y}^2 dx \le ||u,v||_{X_1 \cap Y_2 \cap Y_3}^2. 
\end{align}

For the $u_{xx}$ estimate, we have: 
\begin{align}
x^{\frac{7}{4}} ||u_{xx}||_{L^\infty_y} \le ||x^{\frac{3}{2}} u_{xx}||_{L^2_y}^{\frac{1}{2}} ||x^{2}u_{xxy}||_{L^2_y}^{\frac{1}{2}},
\end{align}

Squaring both sides and taking an $x$-integration gives: 
\begin{align} \nonumber
\int_{20}^\infty x^{\frac{7}{2}}||u_{xx}||_{L^\infty_y}^2 &\le \int_{20}^\infty ||x^{\frac{3}{2}}u_{xx}||_{L^2_y}^2 + \int_{20}^\infty ||x^2 u_{xxy}||_{L^2_y}^2 \\ 
& \lesssim ||u,v||_{X_1 \cap Y_2 \cap Y_3}^2 .
\end{align}

\end{proof}

We shall make the following selections, given arbitrary constants $M_2, M_3 \ge 0$:
\begin{align} \label{sel.N2}
&N_2 \text{ is selected such that } N_2 - M_2 = 100; \\  \label{sel.N3}
&N_3 \text{ is selected such that } N_3 - M_3 = 2N_2; \\  \label{sel.N4}
& N_k = N_3 \text{ for } k = 4,5,6,7, \\ \label{sel.n}
& n \text{ sufficiently large relative to } N_i, M_i .
\end{align}

We now consolidate the above series of lemmas, coupled with (\ref{Y1em}) and (\ref{Y3em}). 
\begin{corollary} With $N_4,...,N_7$ as in (\ref{sel.N4}), we have:
\begin{align} \label{Uemb}
||u,v||_{\mathcal{U}} \lesssim \eps^{\max\{N_2,N_3\}} ||u,v||_{X_1 \cap Y_2 \cap Y_3}.
\end{align}
\end{corollary}

\begin{theorem}[$Z$ embedding] \label{thm.z} Suppose $||\bar{u}, \bar{v}||_Z \le 1$. For appropriate choices of $N_2, ... N_7$, based only on universal constants, there exists a universal constant $\omega(N_i)$ such that: 
\begin{align} \label{Z.driver}
||u,v||_Z &\lesssim  \epsilon^{100} + ||u,v||_{X_1 \cap X_2 \cap X_3} + \eps^{\frac{n}{2}+\gamma - \omega(N_i)}||\bar{u}, \bar{v}||_Z^2. 
\end{align}
\end{theorem}
\begin{proof}

According to (\ref{Uemb}), it suffices to treat the $Y_2, Y_3$ terms in $||\cdot||_Z$. For this we simply use the selections in (\ref{sel.N2}) - (\ref{sel.n}) to rewrite (\ref{Y3em})
\begin{align}
\eps^{N_3} ||u,v||_{Y_3} \lesssim \eps^{100} + \eps^{100} ||u,v||_{X_1 \cap X_3} + \eps^{N_2}||u,v||_{Y_2} + \eps^{\frac{n}{2}+\gamma - \omega(N_i)} ||\bar{u}, \bar{v}||_Z^2. 
\end{align} 

The condition (\ref{sel.n}) allows us to control, in estimate (\ref{Y3em}):  
\begin{equation}
\eps^{\frac{n}{2}+\gamma - N_2 - N_4 - M_3 + N_3} ||u,v||_{X_1 \cap Y_2} \le \eps^{N_2 + 100} ||u,v||_{X_1 \cap Y_2},
\end{equation}

for instance, so long as: $\frac{n}{2}+\gamma - N_2 - N_4 - M_3 + N_3 > N_2 + 100$. There are many such criteria which arise over the course of our analysis, and so we retain the generality as stated in (\ref{sel.n}). Next, rewriting (\ref{Y1em}) in a similar fashion, one has: 
\begin{align}
\eps^{N_2} ||u,v||_{Y_2} \lesssim \eps^{100} + \eps^{100} ||u,v||_{X_1 \cap X_2} + \eps^{\frac{n}{2}+\gamma - \omega(N_i)} ||\bar{u}, \bar{v}||_Z^2. 
\end{align}

Summing these two estimates yields the desired result. 

\end{proof}

\subsection{Function Space, $Z(\Omega^N)$} \label{appendix}

We will have occasion to consider, $Z(\Omega^N)$, where $\Omega^N$ is defined in (\ref{omegaN}), and  $N$ is some large but finite, fixed number. Due to the boundedness in the $y$-direction of $\Omega^N$: 

\begin{lemma} For $[u,v] \in Z(\Omega^N)$, 
\begin{align} \label{ZN1}
||\psi||_{L^2(\Omega^N)} + ||\{\sqrt{\epsilon}v, v_y \} x||_{L^2(\Omega^N)} \le C(N) ||u, v||_{Z(\Omega^N)},
\end{align}
where $C(N)$ depends poorly on large $N$.
\end{lemma}
\begin{proof}
This follows from the Poincare inequality, as both $\psi = v = v_y = 0$ on $y = 0$:
\begin{align}
||vx||_{L^2} \le C(N) ||v_y x||_{L^2} \le C(N) ||v_{yy} x||_{L^2} = ||u_{xy}x||_{L^2_y} \lesssim C(N) ||u,v||_{Z(\Omega^N)}.
\end{align}

Similarly, 
\begin{align}
||\psi||_{L^2} \le C(N) ||\psi_y||_{L^2} = ||u||_{L^2} \le C(N) ||u_y||_{L^2} \le C(N) ||u||_{Z(\Omega^N)}.
\end{align}
\end{proof}

\begin{lemma} For $[u,v] \in Z(\Omega^N)$, 
\begin{align} \label{ZN}
\sup_{x \ge 20} \Big[ x^{\frac{3}{2}} ||v, v_y||_{L^2_y} + x^{\frac{1}{2}}||\psi||_{L^2_y} + x^{\frac{1}{2}}||u||_{L^2_y} +  x^2 ||v_x||_{L^2_y}  \Big] \le C(N) ||u,v||_{Z(\Omega^N)},
\end{align}
where $C(N)$ depends poorly on large $N$.
\end{lemma}
\begin{proof}

By applying the Poincare inequality twice in the $y$-direction, 
\begin{align}
||vx^{\frac{3}{2}}||_{L^2_y} \le C(N) ||v_y x^{\frac{3}{2}}||_{L^2_y} \le C(N) ||u_{xy} x^{\frac{3}{2}}||_{L^2_y} \lesssim ||u,v||_Z. 
\end{align}

The final inequality following from (\ref{evo.high}). Similarly, 
\begin{align}
||v_x x^2||_{L^2_y} \le C(N) ||v_{xy} x^2 ||_{L^2_y} \lesssim ||u,v||_Z.
\end{align}

\end{proof}

Finally, by repeating all of the calculations which culminated in Theorem \ref{thm.z} on the domain $\Omega^N$, the following estimate holds: 
\begin{theorem}[$Z(\Omega^N)$ embedding] \label{thm.zN} Fix any $N > 0$, large. Suppose $||\bar{u}, \bar{v}||_{Z(\Omega^N)} \le 1$. For appropriate choices of $N_2, ... N_7$, based only on universal constants, there exists a universal constant $\omega(N_i)$ such that we have: 
\begin{align} \label{Z.driver.N}
||u,v||_{Z(\Omega^N)} &\lesssim  \epsilon^{100} + ||u,v||_{X_1 \cap X_2 \cap X_3(\Omega^N)} + \eps^{\frac{n}{2}+\gamma - \omega(N_i)}||\bar{u}, \bar{v}||_{Z(\Omega^N)}^2. 
\end{align}
The constants in the above estimate are independent of $N$.
\end{theorem}

\section{Navier-Stokes Remainders: Energy Estimates}  \label{section.NSR.Linear}

In this section, we shall obtain a family of energy and positivity estimates for the system in (\ref{EQ.NSR.1}) - (\ref{EQ.NSR.3}). As mentioned in the prior section, we seek a solution $[u,v] \in Z(\Omega)$, where $Z(\Omega)$ is defined precisely in equation (\ref{defn.Z.space}). Our point of view for this section, then, is to obtain \textit{a-priori} estimates under the assumption that $[u,v] \in Z$. Such a solution necessarily encodes decay rates of the solutions and their derivatives (see, for instance, (\ref{evo.low}) - (\ref{evo.high}), (\ref{u0}), and (\ref{Uf})). We shall give \textit{a-priori} estimates on the domain $\Omega^N$, as defined in (\ref{omegaN}), which are independent of $N$, allowing us to take $N \rightarrow \infty$. For this purpose, we take the boundary conditions shown in (\ref{BCN}). 

\begin{remark}[Notational Convention] For Section \ref{section.NSR.Linear}, all integrations $\int \int$ and norms, $||\cdot||$, without further specification of domains are over $\Omega^N$. 
\end{remark}

Going to the vorticity formulation of (\ref{EQ.NSR.1}) - (\ref{EQ.NSR.3}): 
\begin{align} \n
\partial_y \Big( -\Delta_\epsilon u + P_x + S_u \Big) - \epsilon \partial_x\Big( - \Delta_\epsilon v + \frac{P_y}{\epsilon} + S_v  \Big) = \\  \label{stream.1}
\partial_y \Big( -\Delta_\epsilon u +  S_u \Big) - \epsilon \partial_x\Big( - \Delta_\epsilon v + S_v  \Big) = f_y - \epsilon g_x.
\end{align}

Define the stream function through
\begin{equation}
 \psi(x,y) = - \int_0^y u(x,y') dy', \hspace{3 mm} \psi_x = v, \hspace{3 mm} \psi_y = -u,
\end{equation}

and note via the boundary conditions (\ref{nsr.bc.1}) and (\ref{BCN}), 
\begin{equation} \label{psibc}
\psi|_{y = 0, y = N}  = \psi|_{x = 1} = \psi_x|_{x = 1} = \psi_y|_{y = 0, y = N} = 0. 
\end{equation}

To see that $\psi|_{y = N} = 0$, we can write: 
\begin{align}
\partial_x \psi = - \int_0^N u_x(x,y') dy' = \int_0^N v_y(x,y') dy' = v(x,N) - v(x,0) = 0. 
\end{align}

Next, the boundary condition $\psi(1,y) = 0$ enables us to evaluate $\psi$ the corner: $\psi(1,N) = 0$. Thus, coupling these two facts yields $\psi(x,N) = 0$. Next, we record the observation:
\begin{equation}
\psi(x,y) = \psi(x,y) - \psi(1,y) = \int_1^x \partial_x \psi(x',y) dx' = \int_1^x v(x',y) dx', 
\end{equation}

and so taking absolute values, and supremum in $y$ yields: 
\begin{equation} \label{psi.x.rate}
||\psi(x')||_{L^\infty_y} \le \int_1^x ||v(x')||_{L^\infty_y} dx' \le \int_1^x (x')^{-\frac{1}{2}} dx' \lesssim x^{\frac{1}{2}}. 
\end{equation}

We also will have occassional need for the auxiliary domain: 
\begin{align}
\Omega^N_M := \{0 < x <M, 0 < y < N \}. 
\end{align}

\subsection{Energy Estimates}

We now give the energy estimates on $[u,v]$. Let us first introduce the notation: 
\begin{align} \label{calw1.E}
&\mathcal{W}_{1,E} = |\int \int  f \cdot u|  +  \int \int \eps |g| |v|, \\ \label{calw1.P}
&\mathcal{W}_{1,P} = \int \int |f| |v_y| x + \int \int \eps |g| |v_x| x, \\ \label{calw1}
&\mathcal{W}_1 = \mathcal{W}_{1,E} + \mathcal{W}_{1,P}.
\end{align}

\begin{proposition} \label{thm.energy} Let $\epsilon << \delta$ and $\delta, \eps$ be sufficiently small relative to universal constants. Let $[u,v] \in Z(\Omega^N)$ be solutions to the system (\ref{EQ.NSR.1}) - (\ref{EQ.NSR.3}) on the domain $\Omega^N$. Then these solutions satisfy the \textit{a-priori} energy estimate: 
\begin{equation} \label{NSR.Energy}
 || \sqrt{\epsilon}u_x, u_y||_{L^2}^2 \le \mathcal{O}(\delta) ||  \sqrt{\epsilon}v_x x^{\frac{1}{4}}, v_y x^{\frac{1}{2}} ||_{L^2}^2 + \mathcal{W}_{1,E}.
\end{equation}

The constant in the above estimate is independent of $N$. 
\end{proposition}

\begin{remark} Note carefully in (\ref{calw1.E}) that we do not bring the absolute value inside of the integration for the $\int \int f u$ term, which is important to treat term (\ref{order.NL}). For the remaining terms above in $\mathcal{W}_1$, we place the absolute values inside the integration for convenience (in terms of comparison with other terms that arise, see for instance (\ref{gsec}) - (\ref{gsec.2})). 
\end{remark}

\begin{proof}
The vorticity equation (\ref{stream.1}) is multiplied by the stream function $\psi$ and we proceed to integrate by parts. According to the definition of the norm $Z$ in (\ref{norm.Z}) and the estimate (\ref{ZN1}), all integrands appearing in this estimate will be $L^1(\Omega^N)$, and so all applications of Fubini are justified. First, we will treat the highest order terms: 
\begin{align} \label{stream.2.0}
&\int \int \partial_y \Big( - \Delta_\epsilon u \Big) \psi + \epsilon \partial_x \Delta_\epsilon v \psi = \int \int \Delta_\epsilon u \psi_y -\epsilon \Delta_\epsilon v \psi_x  \\ \label{stream.2.1}
& \hspace{10 mm} = -\int \int \Delta_\epsilon u u - \epsilon \int \int \Delta_\epsilon v v \\ 
& \hspace{10 mm} =  \int \int |\nabla_\epsilon u|^2   + \epsilon \int \int |\nabla_\epsilon v|^2 - \lim_{M \rightarrow \infty} \Big[ \int_{x = M} \eps u_x u + \eps^2 v_x v \Big] \\
& \hspace{10 mm} =  \int \int |\nabla_\epsilon u|^2   + \epsilon \int \int |\nabla_\epsilon v|^2.
\end{align}

For the limiting integrals over $x = M$ above, we have used the bounds from (\ref{ZN1}) to conclude: 
\begin{align}
|\int_{x = M} \eps uu_x | \le \eps ||ux^{\frac{1}{2}}||_{L^2_y} ||u_x x^{\frac{3}{2}}||_{L^2_y} M^{-2} \xrightarrow{M \rightarrow \infty} 0, \\
|\int_{x = M} \eps^2 vv_x | \le \eps^2 ||vx^{\frac{3}{2}}||_{L^2_y} ||v_x x^{2}||_{L^2_y} M^{-\frac{7}{2}} \xrightarrow{M \rightarrow \infty} 0.
\end{align}

Let us justify rigorously the integration by parts found in line (\ref{stream.2.0}). To isolate the corners, define $C_r^{1,2}$ to be solid balls of radius $r$ centered at the two corners, $(1,0)$, and $(1,N)$. Then, 
\begin{align}
\int \int \partial_y &\Big( - \Delta_\epsilon u \Big) \psi = \int \int_{\Omega^N - \cup_{i=1}^2 C_r^i} \partial_y \Big( - \Delta_\epsilon u \Big) \psi + \sum_{i = 1}^2 \int \int_{C_r^i} \partial_y \Big( - \Delta_\epsilon u \Big) \psi \\ \label{corners.1}
& = \int \int_{\Omega^N - \cup_{i=1}^2 C^i_r} -\Delta_\epsilon u u - \int_{\partial C_r^i} \Delta_\epsilon u \psi dS + \int \int_{C_r^i} \partial_y \Big( - \Delta_\epsilon u \Big) \psi. 
\end{align}

First, as we know $-\Delta_\epsilon u u \in L^1(\Omega^N)$, we have: 
\begin{align}
\lim_{r \rightarrow 0} \int \int_{\Omega^N - \cup_{i=1}^2 C^i_r} -\Delta_\epsilon u u = \int \int -\Delta_\epsilon u u. 
\end{align}

Next, we appeal to the classical expansion in the vicinity of a corner point in \cite{Biharmonic}, Page 57, equation (5.5), from which it follows that: 
\begin{align}
|u, v| \lesssim r, \hspace{3 mm} |\psi| \lesssim r,  \hspace{3 mm} |\nabla^2 u, \nabla^2 v| \lesssim r^{-1} \text{ in } C^i_r.
\end{align}

Then the boundary integral from (\ref{corners.1}) can be controlled via: 
\begin{align}
\Big| \int_{\partial C^i_r} \Delta_\epsilon u \psi \Big| \le \int_{\partial C_r^i} \Big| \Delta_\epsilon u \psi \Big| \le \int_{\partial C_r^i} r r^{-1} dS \le r \xrightarrow{r \rightarrow 0} 0.
\end{align}

Finally, we arrive at the interior term from the corners in (\ref{corners.1}). For this, the expansion in \cite{Biharmonic}, Page 57, equation (5.5) implies that:
\begin{align}
|\nabla^3 [u,v]| + |\nabla^3 \psi| \lesssim r^{-2} + \tilde{u}, \text{ where } \tilde{u} \in L^2. 
\end{align}

Using this gives: 
\begin{align} \nonumber
\Big| \int \int_{C^i_r} &\partial_y \Big( - \Delta_\epsilon u \Big) \psi dxdy \Big| \lesssim \int \int_{C^i_r} \Big| r^{-2} r \Big| dxdy + \int \int_{C^i_r} |\tilde{u}| r dx dy \\
& \le \int \int_{C_r^i} r^{-2} r r dr d\omega + ||\tilde{u}||_{L^2(C^i_r)} || r ||_{L^2(C^i_r)} \xrightarrow{r \rightarrow 0} 0.
\end{align}

Summarizing, we have shown the validity of the integration by parts
\begin{align}
\int \int \partial_y(-\Delta_\epsilon u) \psi = \int \int -\Delta_\epsilon u u. 
\end{align}

This calculation works generically at the corners (it was not specific to the particular  derivatives involved, just the order of them), and so we will avoid repeating it each time. We now turn to the $v$-terms in (\ref{stream.2.0}), for which we write:
\begin{align} \nonumber
\int \int_{\Omega^N - C^i_r} \partial_x \Delta_\epsilon v \psi &= \lim_{M \rightarrow \infty} \int \int_{\Omega^N_M - C^i_r}  \partial_x \Delta_\epsilon v \psi \\ \nonumber
& = - \lim_{M \rightarrow \infty} \Big[ \int \int_{\Omega^N_M - C^i_r} \Delta_\epsilon v v + \int_{x = M} \Delta_\epsilon v \psi dy - \int_{\partial C^i_r} \Delta_\epsilon v \psi dy \Big] \\ \label{bbad.1}
& = - \int \int_{\Omega^N - C^i_r} \Delta_\epsilon v v + \lim_{M \rightarrow \infty}  \int_{x = M}  \Delta_\epsilon v \psi dy - \int_{\partial C^i_r} \Delta_\epsilon v \psi dy.
\end{align}

For the $x = M$ boundary term in (\ref{bbad.1}), we have: 
\begin{align}
\Big| \int_{x = M} \Delta_\epsilon v \psi \Big| \le ||\Delta_\epsilon v||_{L^2_y} ||\psi ||_{L^2_y} \lesssim ||\psi x^{\frac{1}{2}}||_{L^2_y} ||\Delta_\eps v x^{\frac{3}{2}}||_{L^2_y} M^{-2} \xrightarrow{M \rightarrow \infty} 0,
\end{align}

according to (\ref{evo.low}), (\ref{evo.high}), and (\ref{ZN}). Subsequently, sending $r \rightarrow 0$ as above gives the desired identity: 
\begin{align}
\int \int \partial_x \Delta_\epsilon v \psi = - \int \int \Delta_\epsilon v v. 
\end{align}

Next, we come to the profile terms, $S_u$, from the equation (\ref{EQ.NSR.1}). We refer the reader to the definition of $S_u$, which is in (\ref{defn.Su}).
\begin{align} \n
\int \int \partial_y \Big( S_u \Big) \psi &= - \int \int S_u \psi_y =  \int \int S_u u \\ \label{su}
& = \int \int \Big[ u_R u_x + u_{Rx}u + v_R u_y + u_{Ry}v \Big] u. 
\end{align}

The first three of these terms in $\int \int S_u u$ are handled through an integration by parts: 
\begin{align} \label{nsr.e.1}
\int \int u_R u_x u + u_{Rx}u^2 + v_R uu_y = \int \int u_{Rx}u^2 + \lim_{M \rightarrow \infty} \int_{x = M} \frac{u_R}{2}u^2 = -\int \int v_{Ry} u^2.
\end{align}

Above, we have used the estimate for $||ux^{\frac{1}{2}}||_{L^2_y}$ in (\ref{ZN}). The term on the right-hand side of (\ref{nsr.e.1}) is handled via: 
\begin{align} \nonumber
\int \int v_{Ry}u^2 &= \int \int \Big(v_{Ry}^P + \sqrt{\epsilon} v_{RY}^E \Big) u^2 \\ \nonumber
&\le ||v_{Ry}^P y^2||_{L^\infty} ||\frac{u}{y}||_{L^2}^2 + \sqrt{\epsilon} ||v^E_{RY} x^{\frac{3}{2}}||_{L^\infty} ||\frac{u}{x^{\frac{3}{4}}}||_{L^2}^2 \\ \label{e.p.1}
&\le \mathcal{O}(\delta) ||u_y||_{L^2}^2 + \mathcal{O}(\delta) ||u_x x^{\frac{1}{4}}||_{L^2}^2. 
\end{align}

In (\ref{e.p.1}), we have first used the profile estimates in (\ref{PE1}) and (\ref{PE4.new.2}), and subsequently the Hardy inequality which is available (for exponents of $x$ which are not equal to $\frac{1}{2}$) as $u(1,y) = u(x,0) = 0$. The large convective term in $\int \int S_u u$, (\ref{su}), is given by:
\begin{align} 
\int \int u_{Ry} uv &= \int \int \{ u_{Ry}^{P,n-1} + \epsilon^{\frac{n}{2}} u^n_p + \sqrt{\epsilon} u_{RY}^E \} uv.
\end{align} 

First, by estimate (\ref{PE0.5}), with $j  = 1, m = 1$, we have:
\begin{align} \label{NSR.conv.0}
\Big| \int \int u_{Ry}^{P,n-1} uv \Big| \le  ||y^2x^{-\frac{1}{2}} u^{P,n-1}_{Ry}||_{L^\infty}  ||\frac{u}{y}||_{L^2} ||\frac{v x^{\frac{1}{2}}}{y}||_{L^2} \le \mathcal{O}(\delta) ||u_y||_{L^2} ||v_y x^{\frac{1}{2}}||_{L^2}. 
\end{align}

Second, according to (\ref{PE3}) with $j = 1$:
\begin{align} \nonumber
\Big| \int \int \epsilon^{\frac{n}{2}} u^n_{py} uv \Big| &\le \epsilon^{\frac{n}{2}} ||u^n_{py} yx^{\frac{1}{2}-\sigma_n} ||_{L^\infty} ||ux^{-1+{\sigma_n}}||_{L^2} ||\frac{vx^{\frac{1}{2}}}{y}||_{L^2} \\  \label{NSR.conv.1}
& \le C(n) \epsilon^{\frac{n}{2}} ||u_x x^{\sigma_n}||_{L^2} ||v_y x^{\frac{1}{2}}||_{L^2} \lesssim C(n) \epsilon^{\frac{n}{2}} ||u_x x^{\frac{1}{2}}||_{L^2} ||v_y x^{\frac{1}{2}}||_{L^2}.
\end{align}

Third, according to (\ref{PE5}):
\begin{align} \nonumber
\Big| \int \int \sqrt{\epsilon} u^E_{RY} uv \Big| &\le  || u_{RY}^E x^{\frac{3}{2}}||_{L^\infty} ||\frac{u}{x^{\frac{3}{4}}}||_{L^2} ||\sqrt{\epsilon} \frac{v}{x^{\frac{3}{4}}}||_{L^2} \\ \label{NSR.conv.2}
& \le  \sqrt{\epsilon} ||u_x x^{\frac{1}{4}} ||_{L^2} ||\sqrt{\epsilon}v_x x^{\frac{1}{4}}||_{L^2}.
\end{align}

In (\ref{NSR.conv.0}), we have used the Hardy inequality in the $y$ direction: 
\begin{equation}
||\frac{v}{y}x^{\frac{1}{2}}||_{L^2} = || ||\frac{v}{y}x^{\frac{1}{2}}||_{L^2_y} ||_{L^2_x} \le ||v_y x^{\frac{1}{2}}||_{L^2},
\end{equation}

which is available as $v|_{y=0} = 0$. We have also used the Hardy inequality in $x$ direction, which is available as $v|_{x=1} = 0$. Rigorously, turning to (\ref{NSR.conv.1}): 
\begin{align}
||\frac{u}{x^{1-\sigma}}||_{L^2} &=|| ||\frac{u}{x^{1-\sigma}}||_{L^2_x} ||_{L^2_y} \le || ||\frac{u}{(x-1)^{1-\sigma}}||_{L^2_x} ||_{L^2_y} \lesssim || ||u_x (x-1)^{\sigma}||_{L^2_x} ||_{L^2_y} \\
& \lesssim || ||u_x x^{\sigma}||_{L^2_x} ||_{L^2_y} = ||u_x x^{\sigma} ||_{L^2}. 
\end{align}

Summarizing, 
\begin{equation} \label{NSR.E.Su}
 \Big|  \int \int S_u u \Big| \le \mathcal{O}(\delta) ||u_y||_{L^2}^2 + \mathcal{O}(\delta)||\sqrt{\epsilon}v_x x^{\frac{1}{4}}, v_y x^{\frac{1}{2}}||_{L^2}^2. 
\end{equation}

The important mechanism in controlling (\ref{NSR.conv.0}) is the ability to trade a factor of $y^2 x^{-\frac{1}{2}}$ which is absorbed by the Prandtl profiles, $u^{P,n-1}_{Ry}$, according to (\ref{PE3}) with $j = 1, m= 2$. This creates two $y$ derivatives, $u_y$ and $v_y x^{\frac{1}{2}}$, both of which are order 1. The next step is to control the profile terms $S_v$: 
\begin{equation} \label{stream.prof.v}
-\epsilon \int \int \partial_x S_v \psi = \epsilon \int \int S_v v - \eps \lim_{M \rightarrow \infty} \int_{x = M} S_v \psi = \epsilon \int \int S_v v.
\end{equation}

By inspecting (\ref{ZN}), one sees easily that the above limit vanishes. We now treat the interior terms from (\ref{stream.prof.v}), which we expand for convenience: 
\begin{align} \label{NSR.E.Sv}
\int \int S_v \epsilon v = \int \int \Big(u_R v_x + v_{Rx}u + v_R v_y + v_{Ry}v \Big) \epsilon v.
\end{align}

The first, third, and fourth terms above in (\ref{NSR.E.Sv}) are given via the following calculations: 
\begin{align} \nonumber
\int \int \epsilon u_R v_x v &+ \epsilon v_R vv_y + \epsilon v_{Ry}v^2  =  \epsilon \int \int v_{Ry} v^2 + \lim_{M \rightarrow \infty} \int_{x = M} \frac{\eps}{2} u_R v^2 \\
 & = \epsilon \int \int v_{Ry} v^2   \\ \label{Svv}
&\le  \epsilon ||v_{Ry}^P y^2||_{L^\infty} ||v_y||_{L^2}^2 + \sqrt{\epsilon} ||v_{RY}^E x^{\frac{3}{2}}||_{L^\infty} ||\sqrt{\epsilon}v_x x^{\frac{1}{4}}||_{L^2}^2 \\
& \lesssim \eps ||v_y||_{L^2}^2 + \sqrt{\eps} ||\sqrt{\eps} v_x x^{\frac{1}{4}}||_{L^2}^2. 
\end{align}

The $M-$limit vanishes by the estimate for $||v x^{\frac{3}{2}}||_{L^2_y}$ in (\ref{ZN}). We have also used estimates (\ref{PE1}) and (\ref{PE4}) for the profiles. Next, 
\begin{align}
&\Big| \int \int \epsilon v_{Rx} uv \Big| \le \sqrt{\epsilon} ||x^{\frac{3}{2}}v_{Rx}||_{L^\infty} ||u_x x^{\frac{1}{4}}||_{L^2}||\sqrt{\epsilon}v_x x^{\frac{1}{4}}||_{L^2}.  
\end{align}

Again, we have used (\ref{PE0.1}) and (\ref{PE4}) for the profiles. Thus, 
\begin{equation} \label{stream.prof.v.sum}
\Big|\int \int S_v \epsilon v \Big| \lesssim \sqrt{\epsilon} ||\sqrt{\eps} v_x x^{\frac{1}{4}}, v_y x^{\frac{1}{4}}||_{L^2}^2. 
\end{equation}

On the right-hand side of equation (\ref{stream.1}), we have: 
\begin{align}   \label{stream.RHS}
\int \int \Big( f_y - \epsilon g_x \Big) \psi &= - \int \int f \psi_y + \int \int \epsilon g \psi_x = \int \int fu + \epsilon g v  . 
\end{align}

First, we will note that each term in the above integration by parts is in $L^1(\Omega^N)$, which follows from the definitions in (\ref{defn.SU.SV}), and a consultation with the definition of $Z$ in (\ref{norm.Z}):
\begin{align}
f_y = \epsilon^{-\frac{n}{2}-\gamma}R^{u,n}_y + \epsilon^{\frac{n}{2}+\gamma} \Big( u_y u_x + uu_{xy} + v_y u_y + vu_{yy} \Big), \\
g_x = \epsilon^{-\frac{n}{2}-\gamma}R^{v,n}_x + \epsilon^{\frac{n}{2}+\gamma} \Big( u_x v_x + uv_{xx} + v_x v_y  + vv_{xy} \Big).  
\end{align}

It remains to justify the boundary terms resulting from the $x$-integration by parts, at $x \rightarrow \infty$, in (\ref{stream.RHS}). This, however, follows just as in (\ref{stream.prof.v}) by inspecting the decay rates in (\ref{ZN}). A comparison with (\ref{calw1.E}) then gives the desired result. Combining all of the previous estimates proves (\ref{NSR.Energy}). 
\end{proof}

\subsection{Positivity Estimate}

We now give the following Positivity estimate: 

\begin{proposition} \label{prop.pos} Let $\epsilon << \delta$ and $\delta, \eps$ be sufficiently small relative to universal constants. Then $[u,v] \in Z(\Omega^N)$ solutions to the system (\ref{EQ.NSR.1}) - (\ref{EQ.NSR.3}) satisfy the following estimate:
\begin{equation} \label{NSR.Positivity}
\epsilon^2 \int_{x=1} v_x^2 \ud y + \lim_{M \rightarrow \infty} \int_{x = M} u_y^2 x \ud y +  || \Big(\sqrt{\epsilon} v_x, v_y \Big) x^{\frac{1}{2}} ||_{L^2}^2 \lesssim ||u_y||_{L^2}^2 + \mathcal{W}_1.
\end{equation}

\end{proposition}

\begin{remark}[Selection of Multiplier] There is a distinction between the Positivity estimate from \cite{GN}, Page 31, and the present case. In the case of \cite{GN}, the profiles, $u_R$, were not assumed small, and so they required the normalized multiplier $\frac{v_y}{u_R} - \eps \frac{v_x}{u_R}$. In our case as the profiles are assumed size $\delta$, we need not normalize by a factor of $u_R$. On the other hand, we need to capture precise behavior at $x \rightarrow \infty$, which is the reason our multiplier is $(v_y - \eps v_x) \cdot x$, or equivalently in the vorticity formulation, $v\cdot x$. 
\end{remark}

\begin{proof}

We apply the multiplier $x \psi_x = xv$ to the equation (\ref{stream.1}), and we will subsequently take the following integration: 
\begin{align} \label{Ord.Of.Int}
\lim_{M \rightarrow \infty} \int \int_{\Omega^N_M} \text{ Equation }(\ref{stream.1}) \cdot xv \ud y \ud x. 
\end{align}

The purpose of specifying the order of integration is for the terms in (\ref{fubini.2}), which are not automatically in $L^1(\Omega^N)$ prior to integrating by parts in $y$. All other terms are in $L^1(\Omega^N)$ according to the norm $Z$, (\ref{norm.Z}), and the estimate (\ref{ZN1}). Thus, with the exception of the term in (\ref{fubini.2}), the limiting procedure above can (and will) be omitted, and Fubini can be justified. 

\subsubsection*{Second Order Terms}

First, we will treat the highest order terms. Let us begin with: 
\begin{align} \label{fubini.1}
-\int \int  \partial_y \Delta_\epsilon u v x  = \int \int (\epsilon u_{xx} + u_{yy}) v_y x.
\end{align}

We will note that the $u_{xx}$ term on the right-hand side above is in $L^1(\Omega^N)$, away from the corners. This follows from the definition of $Z$, and (\ref{ZN1}). Thus, the forthcoming applications of Fubini are justified:
\begin{align} \n
\int \int \epsilon u_{xx} v_y x &= - \int \int \epsilon v_{xy} v_y x  = -\frac{\epsilon}{2} \int \int \partial_x \Big( v_y^2 \Big) x \\ \label{fubini.1.5} & = \frac{\epsilon}{2} \int \int v_y^2  - \lim_{M \rightarrow \infty} \frac{\epsilon}{2} \int_{x = M} v_y^2 x =  \frac{\epsilon}{2} \int \int v_y^2.
\end{align}

The above limit vanishes according to (\ref{ZN}). For the $u_{yy}$ term in (\ref{fubini.1}), we can integrate by parts again in $y$, due to (\ref{Ord.Of.Int}), to obtain: 
\begin{align} \label{fubini.2}
\int \int_{\Omega^N_M} u_{yy} v_y x \ud y \ud x = - \int \int_{\Omega^N_M} u_y v_{yy} x \ud y \ud x. 
\end{align}

\textit{Now,} the right-hand side of (\ref{fubini.2}) is in $L^1(\Omega^N)$ according to our norm $Z$.  Interchanging the order of integration: 
\begin{align} \label{fubini.3}
\int \int_{\Omega^N_M} u_{yy} v_y x \ud y \ud x = \int \int_{\Omega^N_M} \partial_x \Big( \frac{u_y^2}{2} \Big)x \ud x \ud y = - \int \int_{\Omega^N_M} \frac{u_y^2}{2} \ud x \ud y + \int_{x = M} \frac{u_y^2}{2} x. 
\end{align}

As the solid integrals on the right-hand side of (\ref{fubini.3}) is known to be in $L^1(\Omega^N)$, we can pass to the limit $M \rightarrow \infty$ (and also drop the notation $\ud x \ud y$ when the order no longer matters):
\begin{align}
\lim_{M \rightarrow \infty} &-\int \int_{\Omega^N_M} \partial_yu_{yy} \cdot v x \ud y \ud x = - \int \int  \frac{u_y^2}{2} + \lim_{M \rightarrow \infty} \int_{x = M} u_y^2 x. 
\end{align}

The limit above appears with a good sign, and therefore contributes to the left-hand side of the desired estimate in (\ref{NSR.Positivity}).  We have omitted the delicate limiting process near the corners in this calculation, as this process is identical to that of (\ref{corners.1}). We shall now examine the term: 
\begin{align} \label{hotwo}
\int \int \epsilon \partial_x \Delta_\epsilon v vx \ud y \ud x = \int \int \epsilon \Big( \epsilon v_{xxx} + v_{xyy} \Big) vx \ud y \ud x
\end{align}

A direct computation then gives: 
\begin{align} \label{limit.1}
\int \int \epsilon^2 v_{xxx} v x &= - \int \int \epsilon^2 v_{xx} \partial_x (vx) + \epsilon^2 \lim_{M \rightarrow \infty} \int_{x = M} v_{xx} v x dy \\
& =  - \int \int \epsilon^2 v_{xx} \partial_x (vx) 
\end{align}

The $x = M$ term in (\ref{limit.1}) is controlled by using (\ref{evo.high}) and (\ref{ZN}):
\begin{align}
\Big| \int_{x = M} v_{xx} v x \Big| \le ||v_{xx}x^{\frac{3}{2}}||_{L^2_y} ||v x^{\frac{3}{2}}||_{L^2_y} M^{-3} \rightarrow 0. 
\end{align}

Again, we omit displaying the limiting process which handles the corners of the domain, because this is identical to (\ref{corners.1}). We now turn to the next integration by parts for the interior term in (\ref{limit.1}). The first observation is that both $v_{xx} v_x x$ and $v_{xx}v$ are in $L^1(\Omega^N)$ by inspection of the norm $Z$. Therefore, we are justified in the integration by parts: 
\begin{align} \nonumber
-\epsilon^2 \int \int v_{xx} v_x x &- \epsilon^2 \int \int v_{xx} v \\ \n
& =  \frac{3\epsilon^2}{2} \int \int v_x^2 + \frac{\epsilon^2}{2} \int_{x=1} v_x^2 - \lim_{M \rightarrow \infty} \Big[\frac{\epsilon^2}{2}  \int_{x = M} v_x^2 x dy + \frac{\epsilon^2}{2} \int_{x=M} vv_x dy \Big]  \\ 
& = \frac{3\epsilon^2}{2} \int \int v_x^2 + \frac{\epsilon^2}{2} \int_{x=1} v_x^2,
\end{align}

where we have appealed to estimates (\ref{ZN}) to show the limits above vanish. We must now treat the $v_{xyy}$ term from (\ref{hotwo}):
\begin{align} \nonumber
\int \int \epsilon v_{xyy} vx &= - \epsilon \int \int v_{yx} v_y x = - \epsilon \int \int \partial_x \Big( \frac{v_y^2}{2} \Big) x \\ 
& = \frac{\epsilon}{2} \int \int v_y^2 - \lim_{M \rightarrow \infty} \frac{\epsilon}{2} \int_{x = M} v_y^2 x = \frac{\epsilon}{2} \int \int v_y^2.
\end{align}

Again, we appeal to (\ref{ZN}) to show the limit vanishes. Summarizing, then: 
\begin{align} \label{sumlap.1}
\int \int \epsilon \partial_x \Delta_\epsilon vx  = \frac{3\epsilon^2}{2} \int \int v_x^2  + \int \int \frac{\epsilon}{2}v_y^2+ \frac{\epsilon^2}{2} \int_{x=1} v_x^2.
\end{align}

\subsubsection*{Profile Terms, $S_u$:} Next, we treat the $S_u$ profile terms, for which we refer the reader to the expressions in (\ref{defn.Su}). Integrating by parts in $y$:
\begin{align} \label{Supos.1}
\int \int \partial_y S_u v x = - \int \int S_u v_y x = - \int \int u_R u_x v_y x + \mathcal{R}_3 = \int \int u_R v_y^2 x + \mathcal{R}_3. 
\end{align}

We give estimates on $\mathcal{R}_3$, starting with: 
\begin{align}\nonumber
\int \int u_{Rx} u x v_y &= \int \int \{ u^P_{Rx} +  u^E_{Rx} \} u x v_y \\ \nonumber
&\le ||yx^{\frac{1}{2}} u^P_{Rx}||_{L^\infty} ||u_y||_{L^2} ||v_y x^{\frac{1}{2}}||_{L^2} +  ||u^E_{Rx} x^{\frac{3}{2}} ||_{L^\infty} ||u_x||_{L^2}||v_y x^{\frac{1}{2}}||_{L^2} \\
&\le \mathcal{O}(\delta) ||u_y||_{L^2}^2 +\sqrt{\eps} ||v_y x^{\frac{1}{2}}||_{L^2}^2. 
\end{align}

Above, we have used estimate (\ref{PE0.4}) with $m = 1$, and (\ref{PE5}). Next, by using (\ref{PE1}), we have:
\begin{align} \label{imp.vR}
\Big| \int \int v_R u_y v_y x \Big| \le ||v_R x^{\frac{1}{2}} ||_{L^\infty} ||u_y||_{L^2} ||v_y x^{\frac{1}{2}}||_{L^2} \le \mathcal{O}(\delta) ||u_y||_{L^2}||v_y x^{\frac{1}{2}}||_{L^2}. 
\end{align}

The estimate (\ref{imp.vR}) is significant in that it essentially determines the rate of decay, $x^{\frac{1}{2}}$, that must be satisfied exactly by the profiles, $v_R$. The next profile term, according to (\ref{PE0.5}) with $j = 1, m = 0$, estimates (\ref{PE3}), and (\ref{PE5}), is: 
\begin{align} \nonumber
\int \int u_{Ry} v v_y x &= \int \int \{ u_{Ry}^P + \sqrt{\epsilon}u_{RY}^E \} vv_y x \\ \nonumber
&\le ||y u_{Ry}^P||_{L^\infty} ||v_y x^{\frac{1}{2}}||_{L^2}^2 + ||u^E_{RY} x^{\frac{3}{2}}||_{L^\infty} ||v_y x^{\frac{1}{2}}||_{L^2}||\sqrt{\epsilon} v_x||_{L^2} \\
&\le \mathcal{O}(\delta) ||v_y x^{\frac{1}{2}}||_{L^2}^2 + \sqrt{\epsilon} ||\sqrt{\epsilon}v_x||_{L^2}^2.   
\end{align}

Summarizing, we have: 
\begin{align} \label{Supos.2}
\int \int \partial_y S_u vx = \int \int u_R v_y^2 x + \mathcal{R}_3, \hspace{3 mm} \Big| \mathcal{R}_3 \Big| \le \mathcal{O}(\delta) ||u_y||_{L^2}^2 + \mathcal{O}(\delta) ||v_y x^{\frac{1}{2}}||_{L^2}^2 + \sqrt{\epsilon}||\sqrt{\epsilon}v_x||_{L^2}^2 . 
\end{align}

\subsubsection*{Profile Terms, $S_v$:}

Next, we treat the $S_v$ profile terms, for which we refer the reader to (\ref{defn.Su}). The first step is to integrate by parts in $x$:
\begin{align} \n
-\int \int \epsilon \partial_x S_v vx  &= \epsilon \int \int S_v x v_x + \epsilon \int \int S_v v - \lim_{M \rightarrow \infty} \int_{x = M} \eps S_v vx \\ \label{str.p.pr.v}
& = \epsilon \int \int S_v x v_x + \epsilon \int \int S_v v.
\end{align}

The limit above is easily seen to vanish using (\ref{ZN}).  We now treat the first term on the right-hand side of (\ref{str.p.pr.v}). We start with the following term which enables control over $\sqrt{\epsilon} v_x$:
\begin{align}
 \int \int \epsilon u_R v_x^2 x  \ge \min u_R \int \int \epsilon v_x^2 x,
\end{align}

The second term from $S_v$ can be controlled according to estimate (\ref{PE0.1}) and (\ref{PE4}), via: 
\begin{align} \n
\epsilon \int \int v_{Rx} u v_x x &\le \sqrt{\epsilon} ||x^{\frac{3}{2}} v_{Rx}||_{L^\infty} ||\sqrt{\epsilon}v_x x^{\frac{1}{2}}||_{L^2} ||u_x||_{L^2} \\  \label{Supos.3}
& \le \sqrt{\epsilon}  ||\sqrt{\epsilon}v_x x^{\frac{1}{2}}||_{L^2} ||u_x||_{L^2} 
\end{align}

Next, we come to the third term in $S_v$, where again we use (\ref{PE1}) and (\ref{PE4}): 
\begin{align} \n
\epsilon \Big| \int \int v_R v_y v_x x \Big| &\le \sqrt{\epsilon} ||v_R||_{L^\infty} ||v_y x^{\frac{1}{2}}||_{L^2} || \sqrt{\epsilon}v_x x^{\frac{1}{2}}||_{L^2} \\ \label{Supos.4}
& \le  \sqrt{\epsilon} \mathcal{O}(\delta) ||v_y x^{\frac{1}{2}}||_{L^2} || \sqrt{\epsilon}v_x x^{\frac{1}{2}}||_{L^2} 
\end{align}

The fourth profile term in $S_v$ is controlled according to estimates (\ref{PE1}) and (\ref{PE4.new.2}) by: 
\begin{align}\n
\epsilon \Big| \int \int v_{Ry} vv_x x \Big| & \le \sqrt{\epsilon} ||v_{Ry}^P y||_{L^\infty} ||v_y x^{\frac{1}{2}}||_{L^2} ||\sqrt{\epsilon} v_x x^{\frac{1}{2}}||_{L^2} + \sqrt{\epsilon}||v_{RY}^E x^{\frac{3}{2}}||_{L^\infty} ||\sqrt{\epsilon}v_x x^{\frac{1}{2}}||_{L^2}^2 \\ \label{Supos.5}
& \le  \sqrt{\epsilon} \mathcal{O}(\delta) ||v_y x^{\frac{1}{2}}||_{L^2}^2 + \sqrt{\eps} \mathcal{O}(\delta) ||\sqrt{\epsilon}v_x x^{\frac{1}{2}}||_{L^2}^2
\end{align}

Next, we note that the second interior term on the right-hand side of (\ref{str.p.pr.v}) is exactly that contained in (\ref{stream.prof.v.sum}), which yields: 
\begin{align} \nonumber
&-\int \int \epsilon \partial_x S_v vx = \int \int \epsilon u_R v_x^2 x + \mathcal{R}_4, \\ \label{Supos.6}
&  \Big| \mathcal{R}_4 \Big| \le \mathcal{O}(\delta) ||\{\sqrt{\epsilon}v_x, v_y \} x^{\frac{1}{2}}||_{L^2}^2 + \mathcal{O}(\delta) ||u_y||_{L^2}^2.  
\end{align}

\subsubsection*{Right-Hand Side}

On the right-hand side, we have
\begin{align}\label{stream.p.rhs}
\int \int \Big( f_y - \epsilon g_x \Big) v x &= - \int \int f v_y x + \epsilon \int \int g \Big( v_x x + v \Big) .
\end{align}

We now justify both the integration by parts above. First, let us turn to the $f_y$ term, which, according to (\ref{defn.SU.SV}) contains the forcing terms $R^{u,n}$ and the nonlinearity $\mathcal{N}^u$: 
\begin{align} \nonumber
f_y v x &= \Big( \epsilon^{-\frac{n}{2}-\gamma} \partial_y R^{u,n} + \partial_y \{uu_x + vu_y \} \Big) vx \\
& = \Big( \epsilon^{-\frac{n}{2}-\gamma} \partial_y R^{u,n}  + u_y u_x + uu_{xy} + v_y u_y + vu_{yy} \Big) vx.
\end{align}

From here it is easy to see that $[f_y v x, \epsilon g_x v x] \in L^1(\Omega^N)$. Therefore, it remains to treat the $x$-integration by parts boundary terms at $x = \infty$, for which we simply appeal to (\ref{ZN}) in an identical fashion to (\ref{str.p.pr.v}). Combining the previous estimates proves (\ref{NSR.Positivity}). 

\end{proof}

\subsection{Second Order Bounds}

In this part, we obtain second order control of the solution to the system (\ref{EQ.NSR.1}) - (\ref{EQ.NSR.3}). To do so, we consider the differentiated system in vorticity form:
\begin{align} \nonumber
\partial_{xy} \Big( -\Delta_\epsilon u &+ P_x + S_u \Big) - \epsilon \partial_{xx}\Big( - \Delta_\epsilon v + \frac{P_y}{\epsilon} + S_v  \Big) = \\  \label{stream.2}
&\partial_{xy} \Big( -\Delta_\epsilon u + S_u \Big) - \epsilon \partial_{xx}\Big( - \Delta_\epsilon v +  S_v  \Big) = f_{xy} - \epsilon g_{xx}.
\end{align}

We will now repeat the Energy and Positivity estimates from the previous section, with higher order multipliers. One should briefly recall the definition of the cut-off function $\rho_2$ from (\ref{rho}). Define our weight via: 
\begin{equation} \label{weho.1}
w_2 = \rho_2 x. 
\end{equation}

The essential property of this weight is that: 
\begin{lemma}[Almost Linear Property] 
\begin{align} \label{ALProp}
|\partial_x^k w_2| \le \partial_x^k x \text{ for } k = 0,1, \text{ and }
|\partial_x^k w_2| \lesssim x^{-M}, \text{ for } k \ge 2, \text{ for any } M. 
\end{align}
\end{lemma}

\begin{remark}
This property of the weight distinguishes it from a generic weight approximating the function $x$ in that all of the nonlinear fluctuations are in an order-1 region around $x = 1$. This structure is needed in (\ref{cruc.}), and distinguishes the second order energy estimates from the third-order estimates.  
\end{remark}

Define:
\begin{align}\label{calw2.E}
&\mathcal{W}_{2,E} = \int \int |f_x| |u_x | |\rho_2^2 x^2 | + \epsilon | g_x |  |v_x| | \rho_2^2 x^2|. \\ \label{calw2.P}
&\mathcal{W}_{2,P} = \int \int |f_x| |v_{xy}|| \rho_2^3 x^3| +  \epsilon | g_x | |v_{xx}|| \rho_2^3 x^3| \}, \\ \label{calw2}
&\mathcal{W}_2 = \mathcal{W}_{2,E} + \mathcal{W}_{2,P}.
\end{align}

\begin{proposition}[Second-Order Energy Estimate] \label{prop.ho.1} Let $\epsilon << \delta$, and $\delta, \eps$ be sufficiently small relative to universal constants. Then solutions $[u,v] \in Z(\Omega^N)$ to the system (\ref{EQ.NSR.1}) - (\ref{EQ.NSR.3}) satisfy the following energy estimate: 
\begin{align} \label{dfig} 
||u_{xy}w_2||_{L^2}^2 + \epsilon|| \{ v_{xy}, \sqrt{\epsilon}v_{xx} \} w_2||_{L^2}^2 & \le \mathcal{O}(\delta) ||\{\sqrt{\epsilon} v_{xx}, v_{xy} \} w_2^{\frac{3}{2}}||_{L^2}^2 + ||u,v||_{X_1}^2+ \mathcal{W}_1 + \mathcal{W}_{2,E}.
\end{align}

\end{proposition}

\begin{remark} Note the presence of absolute values inside the integration in the definition of $\mathcal{W}_2$ for the $f$ term, unlike in $\mathcal{W}_1$. This will be important for calculation (\ref{order.NL}).
\end{remark}

\begin{remark}[Degenerate Weights near $x = 1$] Due to our weight, $w_2$, degenerating near the boundary $x = 1$, we cannot say that $w_2^2 \lesssim w_2^3$. It is imperative that we retain control of the non-degenerate weight of $w$ on the left-hand side of (\ref{dfig}) for the terms $\{v_{xy}, \sqrt{\epsilon}v_{xx}\}$. 
\end{remark}

\begin{remark} We will continue to justify rigorously each integration by parts, as we have not cut-off as $x \rightarrow \infty$. Starting with the second-order positivity estimate (see (\ref{POS.2.2})), it becomes possible to work with cut-offs in $x$, and so the rigorous justifications of Fubini and vanishing boundary contributions at $x = \infty$ become automatic. However, for the present calculation, as in the first-order positivity estimate, we take integrations in the order: 
\begin{align} \label{order.of.int}
\lim_{M \rightarrow \infty} \int \int_{\Omega^N_M} \cdot \ud y \ud x. 
\end{align}
We also refer the reader to the results on the auxiliary space $Z(\Omega^N)$ in Subsection \ref{appendix}, which will be cited in the forthcoming calculation, for some formal justifications. 
\end{remark}

\begin{proof}[Proof of Proposition]
We apply the multiplier $vw_2^2$ to the system (\ref{stream.2}). We shall drop the subscript-$2$ from $w_2$, for the proof, with the understanding that $w = w_2$ for this calculation. Integration by parts several times gives the highest order terms: 
\begin{align} \nonumber
&\int \int \{ \partial_{xy}\Big(-\Delta_\epsilon u\Big) -\epsilon \partial_{xx}\Big(- \Delta_\epsilon v \Big) \} \cdot vw^2 \\ 
&= \int \int u_{xy}^2 w^2 + \int \int \epsilon u_{xx}^2 w^2 + \int \int \epsilon v_{xy}^2w^2 + \int \int \epsilon^2 v_{xx}^2 w^2 + J_0,
\end{align}

where $|J_0| \lesssim \epsilon ||u,v||_{X_1}^2$. To see this, let us first start with the $\Delta_\epsilon u$ terms: 
\begin{align} \nonumber
-\int \int \partial_{xy} (\Delta_\epsilon u) v w^2 &= -\int \int \Delta_\epsilon u_x u_x w^2 \\ \label{weight.1.-2}
& =  \int \int \epsilon u_{xx} \partial_x \Big(u_x w^2 \Big) + \int \int u_{xy}^2 w^2 \\ \label{weight.1.-1}
& = \int \int ( \epsilon u_{xx}^2  + u_{yx}^2 ) w^2 + 2\int \int \epsilon u_{xx} u_x w \partial_x w \\ \label{weight.1}
& = \int \int ( \epsilon u_{xx}^2  + u_{yx}^2 ) w^2  - \int \int \epsilon u_x^2 \partial_x^2 w^2. 
\end{align}

We will now examine the weight in (\ref{weight.1}). Indeed, 
\begin{align}
|\partial_x^2 w^2| = |2w \partial_x^2 w + 2(\partial_x w)^2| \lesssim 1,  
\end{align}

according to (\ref{ALProp}). Therefore, 
\begin{align}
\Big| \int \int \epsilon u_x^2 \partial_x^2 w^2 \Big| \lesssim \int \int \epsilon u_x^2 \le \epsilon ||u||_{X_1}^2. 
\end{align}

Due to the cutoff function, $\rho_2$, in $w$, there are no boundary terms at $x = 1$ when integrating by parts in the $x$ direction. We shall now provide some formalities. First, due to the order of integration in (\ref{order.of.int}), one can integrate by parts twice in $y$ for the following term: 
\begin{align} \n
-\int \int_{\Omega^M_N} \partial_x \partial_y u_{yy} \cdot v w^2 \ud y \ud x &= \int \int_{\Omega^M_N} \partial_x u_{yy} v_y w^2 \ud y \ud x = - \int \int_{\Omega^M_N} u_{xy} v_{yy} w^2 \\
& = \int \int_{\Omega^M_N} u_{xy}^2 w^2.
\end{align}

The final quantity above is known to be in $L^1(\Omega^N)$ according to our norm $Z$, and therefore we can take the limit: 
\begin{align}
\lim_{M \rightarrow \infty} \int \int_{\Omega^M_N} u_{xy}^2 w^2 = \int \int u_{xy}^2 w^2. 
\end{align}

We next turn to the integration in $x$ in (\ref{weight.1.-2}). One easily checks that the integrand $\eps u_{xxx} u_x w^2 \in L^1(\Omega^N)$ for $u \in Z$, and so the following calculation is justified:
\begin{align} \label{stand.1}
\int \int \eps u_{xxx} u_x w^2 = - \int \int \eps u_{xx} \partial_x (u_x w^2) + \lim_{M \rightarrow \infty} \int_{x = M} \eps u_{xx} u_x w^2. 
\end{align}

For the limit, we use (\ref{evo.mid}) - (\ref{evo.high}):
\begin{align} \label{stand.2}
|\int_{x = M} u_{xx}u_x w^2| \lesssim ||u_{xx} x||_{L^2_y} ||u_{x} x||_{L^2_y} \lesssim M^{-1} \xrightarrow{M \rightarrow \infty} 0. 
\end{align} 

The $x$-integration in (\ref{weight.1.-1}) works in an identical manner. Let us turn to the $\Delta_\epsilon v$ term: 
\begin{align} \label{justify.HO}
\int \int \epsilon \partial_{xx} (\Delta_\epsilon v) vw^2 &= - \int \int \epsilon \partial_x (\Delta_\epsilon v) \partial_x (vw^2) \\ \nonumber
& =  \int \int (\epsilon^2 v_{xx}^2 + \epsilon v_{xy}^2) w^2 +  \int \int  \epsilon^2 v^2 \partial_x^4 w^2 - \epsilon^2 v_x^2 \partial_x^2 w^2 - \epsilon v_y^2 \partial_x^2 w^2. 
\end{align}

The final three integrations above are estimated as: 
\begin{align}
& \Big| \int \int \epsilon^2 v_x^2 \partial_x^2 w^2 \Big| \le ||\partial_x^2 w^2||_{L^\infty} \int \int \epsilon^2 v_x^2 \lesssim \epsilon|| v||_{X_1}^2, \\
& \Big| \int \int \epsilon^2 v^2 \partial_x^4 w^2 \Big| \le ||x^2 \partial_x^4 w^2 ||_{L^\infty} \int \int \epsilon^2 \frac{v^2}{x^2} \lesssim \int \int \epsilon^2  v_x^2 \lesssim \epsilon ||v||_{X_1}^2, \\
& \Big| \int \int \epsilon v_y^2 \partial_x^2 w^2 \Big| \lesssim ||\partial_x^2 w^2||_{L^\infty} \int \int \epsilon v_y^2 \lesssim \epsilon ||v||_{X_1}^2. 
\end{align}

Let us now give formal justifications for (\ref{justify.HO}). First, we note the following bound: 
\begin{align} \n
|\int \int \eps \partial_{xx} (\eps v_{xx} + v_{yy}) vw^2| &\lesssim || (v_{xxxx} + v_{xxyy}) w||_{L^2} ||vw||_{L^2} \\
& \lesssim ||\{u_{xx}, v_{xx}\} \zeta_4 x ||_{\dot{H}^2} ||vw||_{L^2} < \infty.
\end{align}

For the final estimate, we have used the elliptic regularity established in (\ref{rhsb}) and the estimate found in (\ref{ZN1}). This then justifies: 
\begin{align} \label{formal.three}
\int \int \eps \partial_{xx} (\Delta_\eps v) vw^2 &= - \int \int \eps \partial_x (\Delta_\eps v) \partial_x (vw^2) + \lim_{M \rightarrow \infty} \int \eps \partial_x \Delta_\eps v vw^2 \\
& = \int \int \eps \partial_x (\Delta_\eps v) \partial_x (vw^2).
\end{align}

For the above limit, one must use (\ref{ZN}) together with the estimate (\ref{three}). The remaining integrations in (\ref{justify.HO}) are justified in the standard way, as in (\ref{stand.1}) - (\ref{stand.2}) with the aid of (\ref{ZN1}) - (\ref{ZN}). In so doing, one computes limits of the following types: 
\begin{align}
&|\int_{x = M} v_{xx} v \partial_x w^2 | \le ||v_{xx} x^2||_{L^2_y} ||v x^{\frac{3}{2}}||_{L^2_y} M^{-\frac{7}{2}} M^2, \\
& |\int_{x = M} v_{xx} v_x w^2| \le ||v_{xx} x^2||_{L^2_y} ||v_x x^2||_{L^2_y} M^{-4}M^2, \\
& | \int_{x = M} v_x^2 w \le ||v_x x^2||_{L^2_y}^2 M^{-4} M, \\
& | \int_{x = M} v_x v \partial_x^2 w^2 | \le ||v_x x^2||_{L^2_y} ||v x^{\frac{3}{2}}||_{L^2_y} M^{-\frac{7}{2}}, \\
& | \int_{x = M} v^2 \partial_x^3 w^2| \le ||v x^{\frac{3}{2}}||_{L^2_y}^2 M^{-3}.
\end{align}

All of these terms vanish upon taking $M \rightarrow \infty$. This completes the formal justification of all of the integrations thus far. For the profile terms, calculations which are analogous to the lowest-order case yield:
\begin{align}\nonumber
\Big| \int \int \{\partial_{yx} S_u  &- \epsilon \partial_{xx} S_v \} \cdot vw^2\Big| \lesssim ||u, v||_{X_1}^2 + \mathcal{O}(\delta) ||\{\sqrt{\epsilon}v_{xx}, v_{xy}\} w^{\frac{3}{2}}||_{L^2}^2. 
\end{align}

For completeness, we include all details here. We will now be working through the following set of terms, referring to the definition in (\ref{defn.Su}):
\begin{align} \n
\int \int \partial_{yx} S_u \cdot vw^2 &= \int \int \partial_x S_u \cdot u_x w^2 \\ \label{read.0.0}
& = \int \int \partial_x \Big[ u_R u_x + u_{Rx}u +  u_{Ry}v + v_R u_y \Big] \cdot u_x w^2.
\end{align}

We begin with:
\begin{align} \nonumber
|\int \int \partial_{x}(u_R u_x ) u_x w^2| &=| \int \int u_{Rx} u_x^2 w^2 + \int \int u_R u_{xx} u_x w^2| \\ \nonumber
& = | \int \int \frac{u_{Rx}}{2} u_x^2 w^2 - \int \int \frac{u_R}{2} u_x^2 \partial_x w^2 + \lim_{M \rightarrow \infty} \int_{x = M} \frac{u_R}{2} u_x^2 w^2| \\
& \le ||u_{Rx} x, u_R||_{L^\infty} ||\partial_x w||_{L^\infty} ||u_x x^{\frac{1}{2}}||_{L^2}^2 \lesssim ||u||_{X_1}^2.  
\end{align}

Above, we have used the estimate (\ref{PE0.4}) and (\ref{PE5}) for the $u^E_{Rx}$ term. The limit above has vanished according to (\ref{ZN}). Next, 
\begin{align}
\Big| \int \int \partial_x (u_{Rx} u) u_x w^2 \Big| = \Big| \int \int u_{Rxx} uu_x w^2 + u_{Rx} u_x^2 w^2 \Big|
\end{align}

We must break up the profile term, $u_{Rxx}$, into Euler and Prandtl. For the $u^P_{Rxx}$ term we use (\ref{PE0.3}) with $k = 2, j = 0, m = 1$, and subsequently the Hardy inequality in $y$:
\begin{align} \n
|\int \int u^P_{Rxx}uu_x w^2|  &\le ||u^P_{Rxx} yx^{\frac{3}{2}}||_{L^\infty} ||\frac{u}{y}||_{L^2} ||u_x x^{\frac{1}{2}}||_{L^2} \\
& \lesssim ||u_y||_{L^2} ||u_x x^{\frac{1}{2}}||_{L^2} \lesssim ||u||_{X_1}^2.
\end{align}

For the $u^E_{Rxx}$ term, we use (\ref{PE5}), followed by the Hardy inequality in $x$:
\begin{align}
&|\int \int u^E_{Rxx}uu_x w^2| |\le ||u^E_{Rxx} x^{\frac{5}{2}}||_{L^\infty} ||\frac{u}{x}||_{L^2} ||u_x x^{\frac{1}{2}}||_{L^2} \lesssim \sqrt{\eps} ||u||_{X_1}^2, \\
&|\int \int u_{Rx} u_x^2 w^2| \le ||u_{Rx} x||_{L^\infty} ||u_x x^{\frac{1}{2}}||_{L^2}^2 \le \mathcal{O}(\delta) ||u||_{X_1}^2. 
\end{align}

The next profile term from (\ref{read.0.0}) is the most delicate convective term: 
\begin{align} \label{conv.ho.1}
\int \int \partial_x (u_{Ry}v) \cdot u_x w^2 = \int \int u_{Rxy}v u_x w^2 + \int \int u_{Ry} v_x u_x w^2 
\end{align}

For the first term in (\ref{conv.ho.1}), we bound, according to (\ref{PE0.3}) for the Prandtl contribution, coupled with the Hardy inequality in $y$: 
\begin{align} \n
|\int \int u^P_{Rxy} vu_x w^2 |& \le ||u^P_{Rxy} xy||_{L^\infty} ||\frac{v}{y}x^{\frac{1}{2}}||_{L^2} ||u_x x^{\frac{1}{2}}||_{L^2} \\ 
& \le ||u^P_{Rxy} xy||_{L^\infty} ||v_y x^{\frac{1}{2}}||_{L^2} ||u_x x^{\frac{1}{2}}||_{L^2} \lesssim ||u,v||_{X_1}^2, 
\end{align}

and (\ref{PE5}) for the Euler contribution, followed by the Hardy inequality in $x$ direction:
\begin{align} \n
|\int \int \sqrt{\epsilon} u^E_{RxY} vu_x w^2| &\le \sqrt{\epsilon} ||u^E_{RxY} x^{\frac{5}{2}}||_{L^\infty} ||\frac{v}{x}||_{L^2} ||u_x x^{\frac{1}{2}}||_{L^2} \\ 
& \lesssim  \sqrt{\epsilon} ||\sqrt{\epsilon}v_x||_{L^2} ||u_x x^{\frac{1}{2}}||_{L^2} \lesssim \sqrt{\epsilon}  ||u,v||_{X_1}^2. 
\end{align}

For the second term in (\ref{conv.ho.1}), we bound by using profile estimates (\ref{PE0.5}), (\ref{PE3}), (\ref{PE5}):
\begin{align} \nonumber
\Big| \int \int u_{Ry} v_x u_x w^2 \Big|&=  \Big| \int \int  \Big( u^P_{Ry} + \sqrt{\epsilon} u^E_{RY} \Big) v_x u_x w^2 \Big| \\ \nonumber
& \le ||u^P_{Ry} y||_{L^\infty} ||\frac{v_x}{y} w^{\frac{3}{2}}||_{L^2} ||u_x x^{\frac{1}{2}}||_{L^2} + ||u^E_{RY} x^{\frac{3}{2}}||_{L^\infty} ||\sqrt{\epsilon} v_x ||_{L^2} ||u_x x^{\frac{1}{2}}||_{L^2} \\
& \le \mathcal{O}(\delta) ||v_{xy} w^{\frac{3}{2}}||_{L^2} ||u||_{X_1} + \sqrt{\eps} ||u,v||_{X_1}^2. 
\end{align}

Implicit in the above calculation is the fact that both the profile terms $u_{Ry}$ and the terms from $X_1$ are controlled on the full domain, $\Omega^N$, and therefore do not demand any of the ``nondegeneracy" of the cut-off weight $w$ near $x = 1$. The next profile term from $S_u$ is: 
\begin{align} \label{2.Su.4}
\int \int \partial_x \Big( v_R u_y \Big) \cdot u_x w^2= \int \int v_{Rx} u_y u_x w^2 + \int \int v_R u_{xy} u_x w^2
\end{align}

We estimate by using (\ref{PE0.1}) - (\ref{PE1}), and the Eulerian estimates in (\ref{PE4}) and (\ref{PE4.new.2}):
\begin{align}
&\Big| \int \int v_{Rx} u_y u_x w^2 \Big| \le ||v_{Rx} x^{\frac{3}{2}}||_{L^\infty} ||u_y||_{L^2} ||u_x x^{\frac{1}{2}}||_{L^2} \lesssim ||u,v||_{X_1}^2, \\ \n
&\Big| \int \int v_R u_{xy} u_x w^2  \Big| \le ||v_R x^{\frac{1}{2}}||_{L^\infty} ||u_{xy}w||_{L^2} ||u_x x^{\frac{1}{2}}||_{L^2} \\
& \hspace{30 mm} \lesssim ||u,v||_{X_1}^2 + \mathcal{O}(\delta) ||u_{xy}w||_{L^2}^2.
\end{align}

Let us now summarize the $S_u$ contribution: 
\begin{align}
\Big| \int \int \partial_x \Big(S_u \Big) \cdot u_x w^2 \Big| \lesssim ||u,v||_{X_1}^2 + \mathcal{O}(\delta) ||v_{xy}w^{\frac{3}{2}}||_{L^2}^2 + \mathcal{O}(\delta) ||u_{xy}w||_{L^2}^2.
\end{align}

The $u_{xy}$ term in the above estimate is absorbed to the left-hand side of (\ref{weight.1}), by taking $\delta$ sufficiently small. The next task is to move to the four profile terms in $S_v$, which we now do, and recall the terms for convenience:
\begin{align} \label{dr.1}
\int \int -\epsilon \partial_{xx} \{u_R v_x + v_{Rx}u + v_R v_y + v_{Ry}v \} vw^2.
\end{align}

Let us begin by giving some formal justification to the initial $x$-integration by parts which will be required to treat the above set of terms. First, one observes using (\ref{ZN1}) and the definition of $Z$ in (\ref{norm.Z}) that all terms are in $L^1(\Omega^N)$. Therefore, an integration by parts in $x$ would contribute the following integration in the limit: 
\begin{align} \n
&\lim_{M \rightarrow \infty} \int -\epsilon \partial_x \Big[ u_R v_x + v_{Rx}u + v_R v_y + v_{Ry}u \Big] vw^2 \ud y \\ \n
& = \lim_{M \rightarrow \infty} \int -\epsilon \Big[ u_{Rx}v_x + u_R v_{xx} + v_{Rxx} u + v_{Rx}u_x + v_{Rx}v_y \\ \label{limit.above}
& \hspace{30 mm} + v_R v_{xy} + v_{Rxy} u + v_{Ry}u_x \Big] vw^2 \ud y.
\end{align}

We shall estimate each term above, with the aid of the profile estimates in (\ref{PE1}) - (\ref{PE4}), and also the $Z(\Omega^N)$ estimates in Subsection \ref{appendix}, (\ref{ZN}). 
\begin{align}
&|\int u_{Rx} v_x v| \le ||u_{Rx}x||_{L^\infty} ||v_x x^2||_{L^2_y} ||v x^{\frac{3}{2}}||_{L^2_y} M^{-\frac{9}{2}}, \\
&| \int u_R v_{xx} v| \le ||v_{xx} x^2||_{L^2_y} ||v x^{\frac{3}{2}}||_{L^2_y} M^{-\frac{7}{2}}, \\
&| \int v_{Rxx}u v| \le ||v_{Rxx} x^{\frac{5}{2}}||_{L^\infty} ||ux^{\frac{1}{2}}||_{L^2_y} ||vx^{\frac{3}{2}}||_{L^2_y} M^{-\frac{9}{2}}, \\
& | \int v_{Rx}u_x v| \le ||v_{Rx} x^{\frac{3}{2}}||_{L^\infty} ||u_x x^{\frac{3}{2}}||_{L^2_y} ||vx^{\frac{3}{2}}||_{L^2} M^{-\frac{9}{2}}, \\
& | \int v_R v_{xy} v| \le ||v_R x^{\frac{1}{2}}||_{L^\infty} ||v_{xy} x^2||_{L^2_y} ||v x^{\frac{3}{2}}||_{L^2_y} M^{-4}, \\
& | \int v_{Rxy}u v| \le ||v_{Rxy} x^2||_{L^\infty} ||ux^{\frac{1}{2}}||_{L^2_y} ||v x^{\frac{3}{2}}||_{L^2_y} M^{-4}, \\
& | \int v_{Ry} u_x v| \le ||v_{Ry}x||_{L^\infty} ||u_x x^{\frac{3}{2}}||_{L^2_y} ||v x^{\frac{3}{2}}||_{L^2_y} M^{-4}. 
\end{align}

From here, it is clear that the limit above in (\ref{limit.above}) is zero. With this formal justification in hand, we continue with the \textit{a-priori} estimate. For the first term from (\ref{dr.1}),
\begin{align} \nonumber
\Big| \int \int -\epsilon \partial_{xx} (u_R v_x)& \cdot v w^2 \Big| = \Big| \int \int \epsilon \partial_x (u_R v_x) \cdot \{v_x w^2 +2 v w w' \} \Big| \\ \label{st.1}
& = \Big| \int \int \epsilon \{u_{Rx}v_x + u_R v_{xx} \} \cdot \{v_x w^2 + 2v w w' \} \Big|.
\end{align}

Let us individually treat each term in (\ref{st.1}). First, by combining (\ref{PE0.4}) and (\ref{PE5}):
\begin{align} \label{I1.it}
\Big| \int \int \epsilon u_{Rx} v_x^2 w^2 \Big| \lesssim ||u_{Rx} x||_{L^\infty} ||\sqrt{\epsilon} v_x x^{\frac{1}{2}}||_{L^2}^2 \le \mathcal{O}(\delta) ||v||_{X_1}^2. 
\end{align}

Next, 
\begin{align} \nonumber
\Big| \int \int \epsilon u_R v_{xx} v_x w^2 \Big| &= \Big| \int \int \frac{\epsilon}{2} v_x^2 \partial_x (u_R w^2) + \lim_{M \rightarrow \infty} \int_{x = M} \frac{\eps}{2} u_R v_x^2 w^2 \Big| \\ \nonumber
& = \Big| \int \int \frac{\epsilon}{2} v_x^2 u_{Rx} w^2 \Big| +  \Big| \int \int \epsilon v_x^2 u_R ww' \Big| \\ \label{I1.it.2}
& \lesssim ||u_{Rx}w, w'||_{L^\infty} ||\sqrt{\epsilon} v_x x^{\frac{1}{2}}||_{L^2}^2 \lesssim ||v||_{X_1}^2. 
\end{align}

The limit above vanishes according to (\ref{ZN}). Next, we shall split $u_{Rx} = u^P_{Rx} + u^E_{Rx}$:
\begin{align} \nonumber
\Big| \int \int \epsilon u_{Rx} v_x v ww' \Big| &\le \Big| \int \int \epsilon u^P_{Rx} v_x v ww' \Big| +  \Big| \int \int \epsilon u^E_{Rx} v_x v ww' \Big| \\ \nonumber
& \le \sqrt{\epsilon}|| u^P_{Rx} yx^{\frac{1}{2}}||_{L^\infty} ||\sqrt{\epsilon} v_x x^{\frac{1}{2}}||_{L^2} ||\frac{v}{y}||_{L^2} \\ \nonumber
& \hspace{20 mm} + ||u^E_{Rx} x^{\frac{3}{2}}||_{L^\infty} ||\sqrt{\epsilon}v_x x^{\frac{1}{2}}||_{L^2} ||\sqrt{\epsilon} \frac{v}{x}||_{L^2} \\ \n
&\le \sqrt{\epsilon}|| u^P_{Rx} yx^{\frac{1}{2}}||_{L^\infty} ||\sqrt{\epsilon} v_x x^{\frac{1}{2}}||_{L^2} ||v_y||_{L^2} \\ \nonumber
& \hspace{20 mm} + ||u^E_{Rx} x^{\frac{3}{2}}||_{L^\infty} ||\sqrt{\epsilon}v_x x^{\frac{1}{2}}||_{L^2} ||\sqrt{\epsilon} v_x ||_{L^2}  \\ \label{I1.it.3}
& \le \mathcal{O}(\delta) ||v||_{X_1}^2. 
\end{align}

Above, we have used (\ref{PE0.4}) (with $m = 1$), coupled with Hardy inequalities in $y$ and $x$, and (\ref{PE5}) for the Euler term. The fourth term in (\ref{st.1}) is by far the most delicate: 
\begin{align} \nonumber
\int \int \epsilon u_R v_{xx} v \partial_x w^2 &= - \int \int \epsilon v_x \partial_x \Big( u_R v \partial_x w^2 \Big)  + \lim_{M \rightarrow \infty} \int_{x = M} \eps u_R v_x v \p_x w^2  \\ \nonumber
& = - \int \int \epsilon v_x \{  u_{Rx}v  \partial_x w^2 + u_R v_x  \partial_x w^2 + u_R v \partial_x^2 w^2 \} \\
& = I_1 + I_2 + I_3. 
\end{align}

We justify the above integration. It is clear, according to (\ref{ZN1}), that $v_{xx}v \partial_x w^2 \in L^1(\Omega^N)$, and so the boundary contribution at $x = \infty$ is:
\begin{align}
|\int_{x = M} \eps u_R v_x v \partial_x w^2| \le ||v_x x^2||_{L^2_y} ||v x^{\frac{3}{2}}||_{L^2_y} M^{-\frac{7}{2}}M \xrightarrow{M \rightarrow \infty} 0. 
\end{align}

$I_1$ follows similarly to (\ref{I1.it.3}). For $I_2$, we have:  
\begin{align}
|I_2| = |\int \int \eps u_R v_x^2 \p_x w^2| &\lesssim ||\p_x w||_{L^\infty} ||u_R||_{L^\infty} ||\sqrt{\eps} v_x x^{\frac{1}{2}}||_{L^2}^2  \lesssim ||v||_{X_1}^2. 
\end{align}

We will estimate $I_3$. For this, we note that due to the almost-linear structure of our weight, 
\begin{align} \label{crucAL}
|\partial_x^3 w^2| \lesssim x^{-K}, \text{ for any $K$,} 
\end{align}

and so (again with the aid of (\ref{ZN})): 
\begin{align} \n
\Big| \int \int \epsilon u_R v_x v \partial_x^2 w^2 \Big| &= \Big| \int \int \frac{\epsilon}{2} v^2 \Big( u_{Rx} \partial_x^2 w^2 + u_R \partial_x^3 w^2 \Big) + \lim_{M \rightarrow \infty} \int_{x = M} \eps u_R v^2 \partial_x^2 w^2 \Big| \\ \label{cruc.}
& = \Big| \int \int \frac{\epsilon}{2} v^2 \Big( u_{Rx} \partial_x^2 w^2 + u_R \partial_x^3 w^2 \Big) \Big|.
\end{align}

For the $u_{Rx}$ term above in (\ref{cruc.}), we shall spit into Euler and Prandtl components, and use estimates (\ref{PE0.4}) (with $m = 2$), (\ref{PE5}):
\begin{align}
 |\int \int \epsilon v^2 u^P_{Rx} \partial_x^2 w^2 | \lesssim \epsilon ||u^P_{Rx}y^2||_{L^\infty} ||\frac{v}{y} ||_{L^2}^2 \le \epsilon ||u^P_{Rx}y^2||_{L^\infty} ||v_y||_{L^2}^2 \lesssim \epsilon \mathcal{O}(\delta) ||v||_{X_1}^2, \\
 | \int \int \epsilon v^2 u^E_{Rx} \partial_x^2 w^2| \lesssim ||u^E_{Rx} x^{\frac{3}{2}}||_{L^\infty} ||\sqrt{\epsilon}\frac{v}{x^{\frac{3}{4}}}||_{L^2}^2 \lesssim \sqrt{\eps} ||\sqrt{\epsilon}v_x x^{\frac{1}{4}}||_{L^2}^2 \lesssim \mathcal{O}(\delta) ||v||_{X_1}^2. 
\end{align}

For the $u_R$ term in (\ref{cruc.}) above, the structure of our weight, (\ref{crucAL}) is important: 
\begin{align}
|\int \int \epsilon v^2 \partial_x^3 w^2| \lesssim ||\sqrt{\epsilon}v_x||_{L^2}^2 \lesssim ||v||_{X_1}^2. 
\end{align}

For the second term from (\ref{dr.1}), we integrate by parts again to arrive at: 
\begin{align} \nonumber
\Big| \int \int &\epsilon \partial_x (v_{Rx} u) \cdot \{v_x w^2 + 2v w w' \} \Big| \\
& = \Big| \int \int \epsilon \{v_{Rxx}u + v_{Rx}u_x \}\cdot \{v_x w^2 + 2vw w' \} \Big| 
\end{align}

Of these, we estimate, according to (\ref{PE0.1}) - (\ref{PE1}) and (\ref{PE4}): 
\begin{align}
&\epsilon \Big| \int \int v_{Rxx} uv ww' \Big| \le \epsilon ||v_{Rxx} x^{\frac{5}{2}}||_{L^\infty} ||\frac{u}{x^{\frac{3}{4}}}||_{L^2} ||\frac{v}{x^{\frac{3}{4}}}||_{L^2} \lesssim \sqrt{\epsilon} ||u_xx^{\frac{1}{4}}||_{L^2} ||\sqrt{\epsilon}v_x x^{\frac{1}{4}}||_{L^2}, \\
& \epsilon \Big| \int \int v_{Rxx}u v_x w^2 \Big| \le \sqrt{\epsilon} ||v_{Rxx}x^{\frac{5}{2}}||_{L^\infty} ||\frac{u}{x}||_{L^2} ||\sqrt{\epsilon}v_x x^{\frac{1}{2}}||_{L^2} \lesssim \sqrt{\eps} ||u,v||_{X_1}^2,\\
& \epsilon \Big| \int \int v_{Rx}u_x v_x w^2 \Big| \le \sqrt{\eps} ||v_{Rx} x^{\frac{3}{2}}||_{L^\infty} ||u_x||_{L^2} || \sqrt{\eps} v_x x^{\frac{1}{2}}||_{L^2} \lesssim \sqrt{\eps} ||u,v||_{X_1}^2,\\
& \epsilon \Big| \int \int v_{Rx}u_x v ww' \Big| \le \epsilon ||v_{Rx} x^{\frac{3}{2}}||_{L^\infty} ||u_x x^{\frac{1}{2}}||_{L^2} ||\frac{v}{x}||_{L^2} \lesssim \sqrt{\epsilon} ||u_x x^{\frac{1}{2}}||_{L^2} ||\sqrt{\epsilon}v_x||_{L^2}.
\end{align}

The third term from (\ref{dr.1}) is given by: 
\begin{align} \nonumber
-\epsilon \int \int &\partial_{xx} (v_R v_y) vw^2 = \int \int \epsilon \partial_x (v_R v_y) \{v_x w^2 + 2w w' v \}\\
&= \int \int \epsilon \{v_{Rx} v_y + v_R v_{xy} \} \{v_x w^2 + 2w w' v \}.
\end{align}

We estimate term by term, appealing to estimate (\ref{PE0.1}) and (\ref{PE4}):
\begin{align}
\Big| \int \int \epsilon v_{Rx} v_y v_x w^2 \Big| &\lesssim \sqrt{\epsilon}||v_{Rx} x^{\frac{3}{2}}||_{L^\infty} ||v_y||_{L^2} ||\sqrt{\epsilon} v_xx^{\frac{1}{2}}||_{L^2} \lesssim \sqrt{\epsilon} ||v||_{X_1}^2, \\ \nonumber
 \Big|  \int \int \epsilon v_{Rx} v_y v w w' \Big| &\lesssim \epsilon ||v_{Rx} x^{\frac{3}{2}}||_{L^\infty} ||v_y x^{\frac{1}{2}}||_{L^2} ||\frac{v}{x}||_{L^2} \lesssim \sqrt{\epsilon}|| v_y x^{\frac{1}{2}}||_{L^2}  ||\sqrt{\epsilon}v_x||_{L^2} \\
&\lesssim \sqrt{\epsilon} ||v||_{X_1}^2. 
\end{align}

Next, appealing to estimate (\ref{PE1}) (with $k = 0, j=0,1, m= 0,1$),  and the Euler estimate in (\ref{PE4.new.2}), coupled with the Hardy inequality in $x$ and in $y$ directions:
\begin{align} \nonumber
 \Big|\int \int \epsilon v_R v_{xy} vw w' \Big| &= \Big|\int \int \epsilon v_x \{v_R v_y + v_{Ry}v \}ww' \Big| \\ \nonumber
 &\lesssim ||v_Rx^{\frac{1}{2}}||_{L^\infty} ||\sqrt{\epsilon}v_x x^{\frac{1}{2}}||_{L^2}||v_y ||_{L^2}  + \sqrt{\epsilon} ||v^P_{Ry} yx^{\frac{1}{2}}||_{L^\infty} ||\sqrt{\epsilon}v_x x^{\frac{1}{2}}||_{L^2} ||\frac{v}{y}||_{L^2} \\  \nonumber
 & \hspace{20 mm} + \sqrt{\eps} ||v^E_{RY} x^{\frac{3}{2}}||_{L^\infty} ||\sqrt{\epsilon}v_x x^{\frac{1}{2}}||_{L^2} ||\sqrt{\epsilon} \frac{v}{x} ||_{L^2}  \\ \n
 & \lesssim ||v_Rx^{\frac{1}{2}}||_{L^\infty} ||\sqrt{\epsilon}v_x x^{\frac{1}{2}}||_{L^2}||v_y ||_{L^2}  + \sqrt{\epsilon} ||v^P_{Ry} yx^{\frac{1}{2}}||_{L^\infty} ||\sqrt{\epsilon}v_x x^{\frac{1}{2}}||_{L^2} ||v_y||_{L^2} \\  \nonumber
 & \hspace{20 mm} + \sqrt{\eps} ||v^E_{RY} x^{\frac{3}{2}}||_{L^\infty} ||\sqrt{\epsilon}v_x x^{\frac{1}{2}}||_{L^2} ||\sqrt{\epsilon} v_x||_{L^2} \\
 &\lesssim \mathcal{O}(\delta) ||u, v||_{X_1}^2. 
\end{align}

Upon integrating by parts  and appealing to (\ref{PE1}) (with $k = 0, j = 1, m = 0$) and (\ref{PE4.new.2}):
\begin{align}
& \Big| \int \int \epsilon v_R v_{xy} v_x w^2  \Big| =  \Big| \int \int \frac{\epsilon}{2} v_{Ry}v_x^2 w^2  \Big| \lesssim \sqrt{\epsilon} ||xv_{Ry}||_{L^\infty} ||\sqrt{\epsilon}v_x x^{\frac{1}{2}}||_{L^2}^2 \le \mathcal{O}(\delta) ||v||_{X_1}^2. 
\end{align}

The final contribution from $S_v$ in (\ref{dr.1}) is: 
\begin{align} \label{finalfour}
\epsilon \int \int \partial_{xx} (v_{Ry}v) \cdot vw^2 = \epsilon \int \int \{v_{Rxxy} v+ 2v_{Rxy} v_x + v_{Ry}v_{xx} \} vw^2
\end{align}

For the first term, we split:
\begin{align} \label{c11}
&\Big| \int \int \epsilon v^P_{Rxxy} v^2 w^2 \Big| \le ||v^P_{Rxxy} x^2 y^2||_{L^\infty} ||\frac{v}{y}||_{L^2}^2 \lesssim ||v_y||_{L^2}^2 \lesssim  ||u, v||_{X_1}^2, \\ \label{c12}
&\Big| \int \int \epsilon v^E_{Rxxy} v^2 w^2 \Big| \le \epsilon^{\frac{3}{2}} ||v^E_{RxxY} x^{\frac{7}{2}}||_{L^\infty} ||\frac{v}{x^{\frac{3}{4}}}||_{L^2}^2 \lesssim \sqrt{\epsilon} ||\sqrt{\eps} v_x x^{\frac{1}{4}}||_{L^2}^2 \lesssim \sqrt{\epsilon} ||u, v||_{X_1}^2. 
\end{align}

For (\ref{c11}), we appeal to estimate (\ref{PE0.1}) (with $k = 2, j = 1, m = 2$), and for (\ref{c12}), we appeal to estimate (\ref{PE4}).  For the next term in (\ref{finalfour}), again with the splitting $v_R = v^P_R + v^E_R$, by using estimate (\ref{PE0.1}), we have:
\begin{align} \n
|\int \int \eps v^P_{Rxy} v_x v w^2 | &\le \sqrt{\eps} ||v^P_{Rxy} yx^{\frac{3}{2}}||_{L^\infty} || \frac{v}{y}x^{\frac{1}{2}}||_{L^2} ||\sqrt{\eps} v_x x^{\frac{1}{2}}||_{L^2} \\ \n
& \le \sqrt{\eps} ||v^P_{Rxy} yx^{\frac{3}{2}}||_{L^\infty} || v_y x^{\frac{1}{2}}||_{L^2} ||\sqrt{\eps} v_x x^{\frac{1}{2}}||_{L^2} \\
& \le \sqrt{\eps}  ||u, v||_{X_1}^2.
\end{align}

For the Euler contribution, by (\ref{PE4}), we have: 
\begin{align} \n
| \int \int \eps^{\frac{3}{2}} v^E_{RxY} vv_x w^2 | &\le \sqrt{\eps} ||v^E_{RxY} x^{\frac{5}{2}}||_{L^\infty} || \sqrt{\eps} \frac{v}{x}||_{L^2} ||\sqrt{\eps} v_x x^{\frac{1}{2}}||_{L^2} \\ \n
& \le \sqrt{\eps} ||v^E_{RxY} x^{\frac{5}{2}}||_{L^\infty} || \sqrt{\eps} v_x ||_{L^2} ||\sqrt{\eps} v_x x^{\frac{1}{2}}||_{L^2} \\
& \le \sqrt{\eps} ||v||_{X_1}^2. 
\end{align}

The third term in (\ref{finalfour}) we split into Euler and Prandtl components, and use (\ref{PE1}) (with $j= 1, m = 1$), and (\ref{PE4}):
\begin{align} \nonumber
\Big| \int \int \epsilon v_{Ry} v_{xx} v w^2 \Big| &\le \Big| \int \int \epsilon v_{Ry}^P v_{xx} v w^2 \Big| + \Big| \int \int \epsilon v_{Ry}^E v_{xx} v w^2 \Big| \\ \nonumber
& \le \sqrt{\epsilon} ||v^P_{Ry} yx^{\frac{1}{2}}||_{L^\infty} ||\frac{v}{y}x^{\frac{1}{2}}||_{L^2} ||\sqrt{\epsilon}v_{xx} w^{\frac{3}{2}}||_{L^2} \\ \nonumber
& \hspace{30 mm} +\sqrt{\eps} ||v^E_{RY} x^{\frac{3}{2}}||_{L^\infty} ||\sqrt{\epsilon}v_{xx} w^{\frac{3}{2}}||_{L^2} ||\sqrt{\epsilon}\frac{v}{x}||_{L^2} \\ \nonumber
& \le \sqrt{\eps} ||\{v_y, \sqrt{\epsilon}v_x \} x^{\frac{1}{2}}||_{L^2} ||\sqrt{\epsilon}v_{xx} w^{\frac{3}{2}}||_{L^2} \\
&\le \sqrt{\eps} ||v||_{X_1} ||\sqrt{\epsilon}v_{xx} w^{\frac{3}{2}}||_{L^2}.
\end{align}

Summarizing the contributions from $S_v$:
\begin{align}
\Big| \int \int -\epsilon \partial_{xx} \{S_v \} vw^2\Big| \lesssim ||u,v||_{X_1}^2 + \mathcal{O}(\delta) ||\sqrt{\epsilon}v_{xx}w^{\frac{3}{2}}||_{L^2}^2
\end{align}

On the right-hand side, we have: 
\begin{align}
\int \int \partial_{xy} f \cdot vw^2 &= - \int \int f_x \cdot v_y w^2, \\ \nonumber
- \epsilon \int \int \partial_{xx} g \cdot vw^2 &= \epsilon \int \int g_x \cdot \{v_x w^2 + 2vw w' \} \\ \label{gsec}
& = \epsilon \int \int g_x v_x w^2 - \epsilon \int \int g \{ v_x w w' + v \partial_{xx}(w^2)\}. 
\end{align}

We estimate the $g$ term: 
\begin{align} \n
| \int \int \eps g v_x ww' + g v \p_{xx} (w^2) | &\le \eps \int \int |g| |v_x| |ww'| + \eps \int \int |g| |v| |\p_{xx}(w^2)| \\ \label{gsec.2}
& \le \int \int \eps |g| |v_x| |w| + \eps |g| |v| \le \mathcal{W}_1. 
\end{align}

A formal point: one can easily check that $\partial_{xx} g \cdot vw^2$ and $ g_x \{v_x w^2 + v \partial_x w^2 \} \in L^1(\Omega^N)$, and so the term at $x = \infty$ which is contributed as a result of integrating by parts in $x$ is: 
\begin{align} \n
& \lim_{M \rightarrow \infty} |\int_{x = M} \eps \partial_x g v w^2| \\
& = \lim_{M \rightarrow \infty} |\int_{x = M} \eps \Big( \eps^{-\frac{n}{2}-\gamma} R^{v,n}_x + \eps^{\frac{n}{2}+\gamma} (u_x v_x + uv_{xx} + v_x v_y + v v_{xy}) \Big) vw^2 |.
\end{align}   

We can estimate each term, using (\ref{Rnl2}), (\ref{ZN}), for arbitrarily small constants $\kappa > 0$, and $\sigma_n$ as in (\ref{sigma.i}) 
\begin{align}
&|\int_{x = M} R^{v,n}_x \cdot vw^2 | \lesssim ||R^{v,n}_x x^{\frac{9}{4}-2\sigma_n - \kappa} ||_{L^2_y} ||vx^{\frac{3}{2}}||_{L^2_y} M^{-\frac{15}{4}+2\sigma_n + \kappa} M^2, \\
& |\int_{x = M} u_x v_x vw^2| \lesssim ||v_x x^{\frac{3}{2}}||_{L^\infty} ||u_x x^{\frac{3}{2}}||_{L^2_y} ||vx^{\frac{3}{2}}||_{L^2_y} \lesssim M^{-\frac{9}{2}} M^2, \\
& |\int_{x = M} uv_{xx} vw^2| \lesssim ||u||_{L^\infty} ||v_{xx} x^2||_{L^2_y} ||vx^{\frac{3}{2}}||_{L^2_y} M^{-\frac{7}{2}}M^2, \\
& |\int_{x = M} v_x v_y v w^2 | \le ||v_x x^{\frac{3}{2}}||_{L^\infty} ||v_y x^{\frac{3}{2}}||_{L^2_y} ||v x^{\frac{3}{2}}||_{L^2_y} M^{-\frac{9}{2}} M^2, \\
& |\int_{x= M} vv_{xy} v w^2| \le ||vx^{\frac{1}{2}}||_{L^\infty} ||v_{xy} x^2||_{L^2_y} ||vx^{\frac{3}{2}}||_{L^2_y} \le M^{-4} M^2. 
\end{align}

It is clear that all of these vanish as $M \rightarrow \infty$. This justifies the integration by parts in (\ref{gsec}). The second terms in (\ref{gsec}) are part of (\ref{calw1}). The desired estimate is established. 

\end{proof}

Next, we come to the second-order positivity estimate. It is necessary to apply approximations to the actual weights we would like to control in order to avoid boundary contributions from $x = \infty$. To this end, let us define $\phi(x)$ to be the standard mollifier, with the usual properties of unit mass, positivity, and support in $B(0,1)$. Next, let $\chi(x)$ be a standard cutoff function, equal to $1$ inside $[1,2]$, and equal to zero on the interval $[3, \infty)$. Also, recall the definition of $\rho_2$ provided in equation (\ref{rho}). Then let us define
\begin{align} \label{Nweight}
a_L(x) &:= \text{min} \{x, L\}, \hspace{3 mm} \phi_L(x) := \frac{1}{(L/2)} \phi \Big( \frac{x}{(L/2)} \Big), \\ w_{2,L}(x) &:= \Big( a_L \ast \phi_{L} \Big) \chi \Big( \frac{x}{10 L} \Big) \rho_2(x). 
\end{align}

The relevant properties of this weight are summarized: 
\begin{lemma} The weight $w_L$ satisfies: 
\begin{align}
&w_{2,L}(x) = x, \text{ for } 60 \le x \le \frac{L}{2}, \hspace{3 mm} w_{2,L}(x) = L \text{ for }  \frac{3L}{2} \le x \le 10 L,  \\  \label{weight.m}
& w_{2,L}(x) = 0 \text{ for } x \ge 30L \text{ and } 1 \le x \le 50, \hspace{3 mm} \Big| x^{k-1} \partial_x^k w_{2,L} \Big| \le C \text{ independent of $L$}.
\end{align}
\end{lemma}
\begin{proof}
All are clear by basic properties of convolutions. 
\end{proof}

Heuristically, the property (\ref{weight.m}) ensures that the weight $w_{2,L}(x)$ behaves like $x$ in the sense that each additional derivative eliminates one factor of the weight. We record: 
\begin{align}
\lim_{L \rightarrow \infty} w_{2,L}(x) = x \rho_2(x) = w_2(x), \text{ for all } x \ge 1. 
\end{align}

\begin{proposition}[Second-Order Positivity] \label{prop.ho.2} For $\delta, \epsilon$ sufficiently small relative to universal constants and $\eps << \delta$, solutions $[u,v] \in Z(\Omega^N)$ to the system (\ref{EQ.NSR.1}) - (\ref{EQ.NSR.3}) satisfy: 
\begin{align}  \label{POS.2.2}
||\{ v_{xy}, \sqrt{\epsilon} v_{xx} \} w_2^{\frac{3}{2}}||_{L^2}^2 \lesssim ||u_{xy}w_2||_{L^2}^2 &+  \epsilon|| \{ v_{xy}, \sqrt{\epsilon}v_{xx} \} w_2||_{L^2}^2 + ||u,v||_{X_1}^2+ \mathcal{W}_1 +  \mathcal{W}_2 .
\end{align}
\end{proposition}

\begin{proof}

We apply the multiplier $v_x w_{2,L}^3$. We will drop the subscript-2 from $w_{2,L}$ and simply use $w_L$ for this calculation. Upon integrating by parts, the highest-order terms are: 
\begin{align}  \label{pos.lap.1}
&\int \int \partial_{xy}\Big(-u_{yy} \Big) \cdot v_x w_L^3 = -\frac{3}{2} \int \int u_{xy}^2 w_L^2 w_L', \\ \label{pos.lap.2}
&\int \int \partial_{xy} \Big( -\epsilon u_{xx} \Big) \cdot v_x w_L^3 = \frac{3\epsilon}{2} \int \int u_{xx}^2 w_L^2 w_L' , \\ \label{pos.lap.3}
&\int \int \epsilon^2 \partial_x^4v \cdot v_x w_L^3 = C \int \int \epsilon^2 v_{xx}^2 w_L^2 w_L' + \int \int \epsilon v_x^2 \partial_x^3\{ w_L^3 \}, \\ \label{pos.lap.4}
& \int \int \epsilon \partial_{xxyy} v \cdot v_x w_L^3 = C\int \int \epsilon v_{xy}^2 w_L^2 w_L'. 
\end{align}

Above, we have used the property (\ref{weight.m}): 
\begin{align}
\Big| \int \int \epsilon v_x^2 \partial_x^3 \{ w_L^3 \} \Big| \le ||\partial_x^3 w_L^3||_{L^\infty} \int \int \epsilon v_x^2 \lesssim ||v||_{X_1}^2. 
\end{align}

Similarly, we have: 
\begin{align}
\Big| \int \int \epsilon^2 v_{xx}^2 w_L^2 w_L' + \epsilon v_{xy}^2 w_L^2 w_L' \Big| \le ||w_L'||_{L^\infty} \epsilon ||\{\sqrt{\epsilon} v_{xx}, v_{xy} \} w_L||_{L^2}^2. 
\end{align}

Thus, all of the terms on the right-hand sides of (\ref{pos.lap.1}) - (\ref{pos.lap.4}) appear on the right-hand side of (\ref{POS.2.2}), which in turn are controlled by the Energy Estimate, see (\ref{dfig}). We will now work through the profile terms, contained in $S_u$, which we write below upon using definition (\ref{defn.Su}): 
\begin{align}
\Big| \int \int \partial_{xy} \{u_R u_x + u_{Rx} u + u_{Ry}v + v_R u_y  \} \cdot v_x w_L^3 \Big| 
\end{align}

First, we have the main profile terms:
\begin{align} \nonumber
\int \int \partial_{xy} \Big(u_R u_x \Big) \cdot v_x w_L^3 &= - \int \int \partial_x(u_R u_x) v_{xy} w_L^3 \\ \label{urm}
& = \int \int u_{Rx} u_x v_{xy} w_L^3 + u_R u_{xx}^2 w_L^3. 
\end{align}

As usual, this term retains control of the main term: 
\begin{align} \label{ABSP.1}
\int \int u_R u_{xx}^2 w_L^3 \gtrsim \min |u_R| \int \int u_{xx}^2 w_L^3. 
\end{align}

The other term in (\ref{urm}) may be estimated by recalling (\ref{PE0.4}) and (\ref{PE5}), via: 
\begin{align} \nonumber
\Big| \int \int u_{Rx} u_x v_{xy} w_L^3 \Big| &\le ||u_{Rx} x||_{L^\infty} ||u_x x^{\frac{1}{2}}||_{L^2} ||v_{xy} w_L^\frac{3}{2}||_{L^2} \\ \label{ABSP}
&\le \mathcal{O}(\delta) ||u, v||_{X_1}^2 + \mathcal{O}(\delta) ||v_{xy} w_L^{\frac{3}{2}}||_{L^2}^2. 
\end{align}

The second term on the right-hand side above, in (\ref{ABSP}), can be absorbed by the main positive term, (\ref{ABSP.1}) by taking $\delta$ small enough. Next, we have the remaining profile terms from $S_u$:
\begin{align} \label{remsu}
\Big| \int \int \partial_{xy} \{u_{Rx} u + u_{Ry}v + v_R u_y  \} \cdot v_x w_L^3 \Big| 
\end{align}

For the first term in (\ref{remsu}), we will integrate by parts in $y$, expand the product, and use Young's inequality:
\begin{align} \nonumber
\Big| \int \int \partial_{xy} &\{ u_{Rx}u \}\cdot v_x w_L^3 \Big| = \Big|- \int \int \{ u_{Rxx}u + u_{Rx}u_x \} v_{xy}w_L^3 \Big| \\ \nonumber
&\le ||u^P_{Rxx} yx^{\frac{3}{2}}||_{L^\infty} ||u_y||_{L^2} ||v_{xy} w_L^{\frac{3}{2}}||_{L^2}+ ||x^{\frac{5}{2}} u^E_{Rxx} ||_{L^\infty} ||u_x||_{L^2} ||v_{xy}w_L^{\frac{3}{2}}||_{L^2} \\ \nonumber
& \hspace{40 mm} + ||u_{Rx}x ||_{L^\infty} ||u_x x^{\frac{1}{2}}||_{L^2} ||v_{xy} w_L^{\frac{3}{2}}||_{L^2} \\ 
&\lesssim  ||u, v||_{X_1} ||v_{xy} w_L^{\frac{3}{2}}||_{L^2} \le \frac{1}{100,000} ||v_{xy} w_L^{\frac{3}{2}}||_{L^2}^2 + C||u,v||_{X_1}^2. 
\end{align}

We have used (\ref{PE0.3}), (\ref{PE0.4}), and the Euler estimates from (\ref{PE5}). The next term from (\ref{remsu}) is the convection term, which we start by integrating by parts: 
\begin{align} \nonumber
\int \int \partial_{xy}(u_{Ry}v) v_x w_L^3 &= - \int \int \partial_x (u_{Ry}v) v_{xy} w_L^3 \\ \nonumber
& = - \int \int u_{Rxy} v v_{xy} w_L^3 - \int \int u_{Ry} v_x v_{xy} w_L^3 \\ \label{ccr.1}
& = - \int \int u_{Rxy} vv_{xy} w_L^3 + \frac{1}{2} \int \int u_{Ryy} v_x^2 w_L^3 \\ \nonumber
& = (\ref{ccr.1}.1) + (\ref{ccr.1}.2).
\end{align}

First, by applying (\ref{PE0.3}) (with $k = j = m = 1$) and the Euler estimate in (\ref{PE5}), and subsequently the Hardy inequality in both the $y$ and $x$ directions,
\begin{align} \nonumber
\Big| (\ref{ccr.1}.1) \Big| &\le ||u^P_{Rxy} xy||_{L^\infty} ||\frac{v}{y} x^{\frac{1}{2}}||_{L^2} ||v_{xy} w_L^{\frac{3}{2}}||_{L^2} + ||u^E_{RxY} x^{\frac{5}{2}}||_{L^\infty} ||\sqrt{\epsilon}\frac{v}{x}||_{L^2} ||v_{xy}w_L^{\frac{3}{2}}||_{L^2} \\ \nonumber
&\le  ||u^P_{Rxy} xy||_{L^\infty} ||v_y x^{\frac{1}{2}}||_{L^2} ||v_{xy} w_L^{\frac{3}{2}}||_{L^2} + ||u^E_{RxY} x^{\frac{5}{2}}||_{L^\infty} ||\sqrt{\epsilon}v_x||_{L^2} ||v_{xy}w_L^{\frac{3}{2}}||_{L^2} \\
&\lesssim  ||v||_{X_1} ||v_{xy} w_L^{\frac{3}{2}}||_{L^2} \le \frac{1}{100,000} ||v_{xy} w_L^{\frac{3}{2}}||_{L^2}^2 + C ||v||_{X_1}^2. 
\end{align}

Next, by applying (\ref{PE0.5}), and (\ref{PE3}) with $j =2$ and (\ref{PE5}), we have:
\begin{align} \nonumber
\Big| (\ref{ccr.1}.2) \Big| &\le ||y^2 u^P_{Ryy} ||_{L^\infty} ||\frac{v_{x}}{y} w_L^{\frac{3}{2}}||_{L^2}^2 + ||u^E_{RYY} x^{\frac{5}{2}}||_{L^\infty} ||\sqrt{\epsilon}v_x x^{\frac{1}{2}}||_{L^2}^2 \\ \n
&\le ||y^2 u^P_{Ryy} ||_{L^\infty} ||v_{xy} w_L^{\frac{3}{2}}||_{L^2}^2 + ||u^E_{RYY} x^{\frac{5}{2}}||_{L^\infty} ||\sqrt{\epsilon}v_x x^{\frac{1}{2}}||_{L^2}^2 \\
& \le \mathcal{O}(\delta) ||v_{xy} w_L^{\frac{3}{2}}||_{L^2}^2 +\sqrt{\eps} ||u,v||_{X_1}^2.
\end{align}

The final term from (\ref{remsu}), upon integrating by parts once in $y$, is: 
\begin{align} \nonumber
\Big| \int \int \partial_x \Big( v_R u_y \Big)& \cdot v_{xy}w_L^3 \Big| \le \Big|\int \int v_{Rx} u_y v_{xy}w_L^3 \Big| + \Big| \int \int v_R u_{xy} v_{xy}w_L^3  \Big| \\ \n
& \lesssim ||v_R x^{\frac{1}{2}}||_{L^\infty} \Big( ||u_{xy} w_L||_{L^2}^2 + ||v_{xy} w_L^{\frac{3}{2}}||_{L^2}^2 \Big) + ||v_{Rx} x^{\frac{3}{2}}||_{L^\infty}||u_y||_{L^2} ||v_{xy} w_L^{\frac{3}{2}}||_{L^2} \\ 
& \lesssim \mathcal{O}(\delta) ||u_{xy} w_L||_{L^2}^2 + \Big(\mathcal{O}(\delta) + \frac{1}{100,000} \Big) ||v_{xy} w_L^{\frac{3}{2}}||_{L^2}^2  + C||u, v||_{X_1}.
\end{align}

We have used (\ref{PE1}) and (\ref{PE4.new.2}) to estimate $v_R$, in which we crucially retain the smallness from $\mathcal{O}(\delta)$. We have used (\ref{PE0.1}), (\ref{PE4}), followed by Young's inequality for the $v_{Rx}$ term. Summarizing: 
\begin{align} 
\Big| \int \int \partial_{xy} &\{u_{Rx} u + u_{Ry}v + v_R u_y  \} \cdot v_x w_L^3 \Big| \\ \n
& \lesssim \mathcal{O}(\delta) ||u,v||_{X_1}^2 + \Big(\mathcal{O}(\delta) + \frac{1}{100,000} \Big) \Big( ||v_{xy} w_L^{\frac{3}{2}}||_{L^2}^2   \Big) + \mathcal{O}(\delta) ||u_{xy}w_L||_{L^2}^2.
\end{align}

The middle term on the right-hand side above gets absorbed into (\ref{ABSP.1}). We will now come to the profile terms from the normal equation, $S_v$. For convenience, we display the terms we will be reading from here, according to the definition in (\ref{defn.Su}):  
\begin{align} \label{Read}
-\epsilon \int \int \partial_{xx} \Big[ u_R v_x + v_{Rx}u + v_{Ry}v +  v_R v_y  \Big] \cdot v_x w_L^3.
\end{align}

The main term upon integrating by parts once in $x$ and expanding the product is: 
\begin{align} \label{suma}
-\int \int \epsilon \partial_{xx} (u_R v_x ) \cdot v_x w_L^3 = \int \int \epsilon \{u_R v_{xx} + u_{Rx} v_x \} \cdot \{v_{xx}w_L^3 + 3v_x w_L^2 w_L' \}.
\end{align}

The first term above yields the desired positivity, namely: 
\begin{align}
\int \int \epsilon u_R v_{xx}^2 w_L^3 \gtrsim \min |u_R| \int \int \epsilon v_{xx}^2 w_L^3. 
\end{align}

Let us now turn to the remaining terms above in (\ref{suma}). First an integration by parts gives: 
\begin{align} \nonumber
\Big| \int \int \epsilon u_R v_{xx} v_x \partial_x w_L^3\Big| &= \Big|- \int \int \epsilon v_x^2 \partial_x \{ u_R \partial_x w_L^3 \} \Big| \\ \nonumber
& = \Big| - \int \int \epsilon v_x^2 \{u_{Rx} \partial_x w_L^3 + u_R \partial_x^2 w_L^3 \} \Big| \\
& \lesssim ||u_{Rx}x||_{L^\infty} ||\sqrt{\epsilon}v_x x^{\frac{1}{2}}||_{L^2}^2 \lesssim \mathcal{O}(\delta) ||v||_{X_1}^2,
\end{align}

where we have used the estimate $|\partial_x^2 w_L^3 | \lesssim w_L \lesssim x$, and also (\ref{PE0.4}) and (\ref{PE5}), both of which guarantee the smallness of $\mathcal{O}(\delta)$. Still from (\ref{suma}), by using (\ref{PE0.3}) - (\ref{PE0.4}) and (\ref{PE5}), we have: 
\begin{align} \nonumber
\Big| \int \int \epsilon u_{Rx} v_x v_{xx} w_L^3\Big| &= \Big| - \int \int \frac{\epsilon}{2}v_x^2 \partial_x \{ u_{Rx} w_L^3 \}\Big| \\ \nonumber
& =\Big| -\int \int \frac{\epsilon}{2} v_x^2 u_{Rxx} w_L^3 - \int \int \frac{\epsilon}{2}  v_x^2 u_{Rx} \partial_x w_L^3 \Big| \\
& \lesssim ||u_{Rxx} x^2, u_{Rx} x||_{L^\infty} ||\sqrt{\epsilon} v_x x^{\frac{1}{2}}||_{L^2}^2 \lesssim ||v||_{X_1}^2. 
\end{align}

Last from (\ref{suma}), 
\begin{align}
\Big| \int \int \epsilon v_x^2 u_{Rx} w_L^2 w_L' \Big| \lesssim ||u_{Rx} x||_{L^\infty} ||\sqrt{\epsilon} v_x x^{\frac{1}{2}}||_{L^2}^2 \lesssim \mathcal{O}(\delta) ||v||_{X_1}^2. 
\end{align}

Above, we have used (\ref{PE0.4}) and (\ref{PE5}). We now move to the second term in $S_v$, for which we integrate by parts once in $x$ and expand the resulting product: 
\begin{align} \nonumber
\int \int \epsilon \partial_{xx}& (v_{Rx}u) \cdot v_x w_L^3  = - \int \int \epsilon \partial_x(v_{Rx}u) \partial_x \{ v_x w_L^3 \} \\ \nonumber
& = \int \int -\epsilon v_{Rxx} u v_{xx } w_L^3 - \epsilon v_{Rx} u_x v_{xx} w_L^3 - \epsilon v_{Rxx} uv_x \partial_x w_L^3 - \epsilon v_{Rx} u_x v_x \partial_x w_L^3 \\ \label{h.1}
& = (\ref{h.1}.1) + ... + (\ref{h.1}.4).
\end{align}

First, by (\ref{PE0.1}) and (\ref{PE4}):
\begin{align} \nonumber
\Big| (\ref{h.1}.1)   \Big| &\le \sqrt{\epsilon} ||v_{Rxx} x^{\frac{5}{2}}||_{L^\infty} ||\sqrt{\epsilon}v_{xx} w_L^{\frac{3}{2}}||_{L^2} ||\frac{u}{x}||_{L^2} \\ \nonumber
& \lesssim \sqrt{\epsilon} ||v_{Rxx} x^{\frac{5}{2}}||_{L^\infty} ||\sqrt{\epsilon}v_{xx} w_L^{\frac{3}{2}}||_{L^2} ||u_x||_{L^2} \\
& \lesssim \sqrt{\epsilon}  ||\sqrt{\epsilon}v_{xx} w_L^{\frac{3}{2}}||_{L^2} ||u||_{X_1}. 
\end{align}

Second, 
\begin{align} 
\Big|  (\ref{h.1}.2) \Big| &\lesssim \sqrt{\epsilon} ||v_{Rx} x^{\frac{3}{2}}||_{L^\infty} ||u_x||_{L^2} ||\sqrt{\epsilon}v_{xx} w_L^{\frac{3}{2}}||_{L^2} \lesssim \sqrt{\epsilon} ||u||_{X_1} ||\sqrt{\epsilon}v_{xx} w_L^{\frac{3}{2}}||_{L^2}. 
\end{align}

Third, 
\begin{align}
\Big|  (\ref{h.1}.3) \Big| \lesssim \sqrt{\epsilon} ||v_{Rxx} x^{\frac{5}{2}}||_{L^\infty} ||\frac{u}{x}||_{L^2} ||\sqrt{\epsilon}v_x x^{\frac{1}{2}}||_{L^2} \lesssim \sqrt{\epsilon}||u,v||_{X_1}^2. 
\end{align}

Fourth, 
\begin{align}
\Big|   (\ref{h.1}.4)  \Big| \lesssim \sqrt{\epsilon} ||v_{Rx} x^{\frac{3}{2}}||_{L^\infty} ||u_x x^{\frac{1}{2}}||_{L^2} ||\sqrt{\epsilon}v_x ||_{L^2} \lesssim \sqrt{\epsilon}||u,v||_{X_1}^2. 
\end{align}

We have used estimate (\ref{PE0.1}) and (\ref{PE4}) in the above calculations. Let us now move to the third term in $S_v$ for which we integrate by parts once in $x$ and expand the product: 
\begin{align} \nonumber
\int \int \epsilon \partial_{xx} ( v_{Ry} v ) v_x w_L^3 &= - \epsilon \int \int \partial_x(v_{Ry}v) \partial_x(v_x w_L^3) \\ \nonumber
& = \int \int - \epsilon v_{Rxy} v v_{xx} w_L^3 - \epsilon v_{Ry} v_x v_{xx} w_L^3  - 3\epsilon v_{Rxy} vv_x w_L^2 w_L' \\ \label{Sv.2.3}
& \hspace{40 mm} - 3\epsilon v_{Ry} v_x^2 w_L^2 w_L' \\ \nonumber
& = (\ref{Sv.2.3}).1 + ... + (\ref{Sv.2.3}).4. 
\end{align}

We shall treat term by term above, starting with: 
\begin{align} \nonumber
|(\ref{Sv.2.3}).1| &= |\int \int - \epsilon v^P_{Rxy} vv_{xx}w_L^3 - \epsilon^{\frac{3}{2}} v^E_{RxY} vv_{xx}w_L^3| \\ \nonumber
& \le \sqrt{\epsilon} ||v^P_{Rxy} yx^{\frac{3}{2}}||_{L^\infty} ||\frac{v}{y}x^{\frac{1}{2}}||_{L^2} ||\sqrt{\epsilon}v_{xx} w_L^{\frac{3}{2}}||_{L^2} \\ \nonumber & \hspace{40 mm} + \sqrt{\epsilon} ||v^E_{RxY} x^{\frac{5}{2}}||_{L^\infty} ||\sqrt{\epsilon}v_{xx} w_L^{\frac{3}{2}}||_{L^2} ||\sqrt{\epsilon} \frac{v}{x}||_{L^2} \\ \nonumber
&\le  \sqrt{\epsilon} ||v^P_{Rxy} yx^{\frac{3}{2}}||_{L^\infty} ||v_yx^{\frac{1}{2}}||_{L^2} ||\sqrt{\epsilon}v_{xx} w_L^{\frac{3}{2}}||_{L^2} \\ \nonumber  & \hspace{40 mm} + \sqrt{\epsilon} ||v^E_{RxY} x^{\frac{5}{2}}||_{L^\infty} ||\sqrt{\epsilon}v_{xx} w_L^{\frac{3}{2}}||_{L^2} ||\sqrt{\epsilon} v_x||_{L^2}\\
& \lesssim \sqrt{\epsilon} ||v||_{X_1} ||\sqrt{\epsilon}v_{xx} w_L^{\frac{3}{2}}||_{L^2} .
\end{align}

Above, we have used (\ref{PE0.1}) with $k = 1, j =1, m = 1$, and the Euler bounds in (\ref{PE4}). Next, according to estimate (\ref{PE1}), (\ref{PE4}):
\begin{align}
|(\ref{Sv.2.3}).2| \le ||v_{Ry} x||_{L^\infty} ||\sqrt{\epsilon}v_x x^{\frac{1}{2}}||_{L^2} ||\sqrt{\epsilon}v_{xx} w_L^{\frac{3}{2}}||_{L^2} \lesssim \mathcal{O}(\delta) ||v||_{X_1} ||\sqrt{\epsilon}v_{xx}w_L^{\frac{3}{2}}||_{L^2}. 
\end{align}

Third, also using the estimates from (\ref{PE0.1}) and (\ref{PE4}),
\begin{align} \nonumber
|(\ref{Sv.2.3}).3| &\le \sqrt{\epsilon} ||v^P_{Rxy} yx^{\frac{3}{2}}||_{L^\infty} ||\frac{v}{y}x^{\frac{1}{2}}||_{L^2} ||\sqrt{\epsilon}v_x x^{\frac{1}{2}}||_{L^2} + \sqrt{\epsilon} ||v^E_{RxY} x^{\frac{5}{2}}||_{L^\infty} ||\sqrt{\epsilon} \frac{v}{x}||_{L^2} ||\sqrt{\epsilon}v_x x^{\frac{1}{2}}||_{L^2} \\
&\le \sqrt{\eps} ||v_y x^{\frac{1}{2}}||_{L^2} ||\sqrt{\epsilon}v_x x^{\frac{1}{2}}||_{L^2} + \sqrt{\epsilon} ||\sqrt{\epsilon}v_x x^{\frac{1}{2}}||_{L^2}^2 \lesssim \sqrt{\eps} ||v||_{X_1}^2. 
\end{align}

Fourth, again by estimate (\ref{PE1}) and (\ref{PE4}):
\begin{align}
|(\ref{Sv.2.3}).4| \lesssim ||v_{Ry} x||_{L^\infty} ||\sqrt{\epsilon} v_x x^{\frac{1}{2}}||_{L^2}^2 \le \mathcal{O}(\delta) ||v||_{X_1}^2. 
\end{align}

We now move to the fourth term in (\ref{Read}), for which we integrate by parts once and distribute the product: 
\begin{align} \nonumber
\int \int \epsilon \partial_{xx} (v_R v_y) v_x w_L^3 &=- \int \int \epsilon \partial_x(v_R v_y) \partial_x(v_x w_L^3) \\ \nonumber
& = \int \int -\epsilon v_{Rx} v_y v_{xx}w_L^3 - \epsilon v_R v_{xy}v_{xx} w_L^3 - 3\epsilon v_{Rx} v_y v_x w_L^2 w_L' \\ \label{Sv.2.4}
& \hspace{40 mm} - 3\epsilon v_R v_{xy} v_x w_L^2 w_L' \\ \nonumber
& = (\ref{Sv.2.4}).1 + ... + (\ref{Sv.2.4}).4.
\end{align}

First, by (\ref{PE0.1}), (\ref{PE1}) and (\ref{PE4}):
\begin{align}
| (\ref{Sv.2.4}).1 | \le \sqrt{\epsilon} ||v_{Rx} x^{\frac{3}{2}}||_{L^\infty} ||v_y x^{\frac{1}{2}}||_{L^2} ||\sqrt{\epsilon}v_{xx} w_L^{\frac{3}{2}}||_{L^2}, \\
| (\ref{Sv.2.4}).2 | \le \sqrt{\epsilon} ||v_R x^{\frac{1}{2}}||_{L^\infty} ||\sqrt{\epsilon}v_{xx} w_L^{\frac{3}{2}}||_{L^2} ||v_{xy} w_L^{\frac{3}{2}}||_{L^2}, \\
|(\ref{Sv.2.4}).3 | \le \sqrt{\epsilon} ||v_{Rx} x^{\frac{3}{2}}||_{L^\infty} ||v_y x^{\frac{1}{2}}||_{L^2} ||\sqrt{\epsilon} v_x x^{\frac{1}{2}}||_{L^2}. 
\end{align}

For the fourth term, we integrate by parts in $y$ and again appeal to (\ref{PE1}) and (\ref{PE4}):
\begin{align}
|(\ref{Sv.2.4}).4| = | \int \int \frac{3\epsilon}{2} v_{Ry} v_x^2 w_L^2 w_L' | \lesssim ||v_{Ry}x||_{L^\infty} ||\sqrt{\epsilon}v_x x^{\frac{1}{2}}||_{L^2}^2.  
\end{align}

Summarizing the last three terms from the $S_v$ contribution: 
\begin{align} \nonumber
\Big| \int \int -\epsilon \partial_{xx}& \{v_{Rx}u + v_R v_y + v_{Ry} v \} \cdot v_{x}w_L^3 \Big| \\
&\le \mathcal{O}(\delta) ||\sqrt{\epsilon}v_{xx} w_L^{\frac{3}{2}}||_{L^2}^2 + \mathcal{O}(\delta) ||v_{xy} w_L^{\frac{3}{2}}||_{L^2}^2 + \mathcal{O}(\delta) ||u,v||_{X_1}^2. 
\end{align}

Finally, on the right-hand side, we have: 
\begin{align}
&\int \int \partial_{xy} f \cdot v_x w_L^3 = \int \int f_{x} u_{xx} w_L^3, \\ \nonumber
&\int \int \epsilon \partial_{xx} g \cdot v_x w_L^3 = \int \int \epsilon g_{x} \{v_{xx} w_L^3 +3 v_x w_L^2w_L' \}.
\end{align}

We estimate the term: 
\begin{align}
| \int \int \eps g_x v_x w_L^2 w_L' | \le \int \int \eps |g_x | |v_x| (\rho_2 x)^2 \le \mathcal{W}_2. 
\end{align}

Taking the limit as $L \rightarrow \infty$, and appealing to the Monotone Convergence Theorem then completes the calculation. 

\end{proof}

\subsection{Third Order Bounds}

In this step, we obtain third-order bounds for our solution. We must repeat the above calculations to the twice-differentiated system:
\begin{align} \nonumber
&\partial_{y} \partial_x^{2} \Big( -\Delta_\epsilon u + P_x + S_u \Big) - \epsilon \partial_{x} \partial_x^{2}\Big( - \Delta_\epsilon v + \frac{P_y}{\epsilon} + S_v  \Big) = \\ \label{stream.3x}
&\partial_{y} \partial_x^{2} \Big( -\Delta_\epsilon u + S_u \Big) - \epsilon \partial_{x} \partial_x^{2}\Big( - \Delta_\epsilon v + S_v  \Big) = \partial_y \partial_x^{2}f - \epsilon \partial_x \partial_x^{2}g,
\end{align}

To state our energy estimate, we will recall the definition of $\rho_3$ from (\ref{rho}) and the definitions of $a_L, \phi_L$ given in (\ref{Nweight}). We will then define the weights: 
\begin{align} \label{wk}
w_3 = \rho_3(x) x, \hspace{3 mm} w_{3,L}(x) &:= \Big( a_L \ast \phi_{L} \Big) \chi \Big( \frac{x}{10 L} \Big) \rho_3(x).  
\end{align}

\begin{remark}[Selection of cut-offs, $\rho_k$] As $k$ increases from $2$ to $3$, the supports of $\rho_k$ shift away from $x = 1$. The purpose is so that when obtaining the third-order estimate, the second-order terms should have a non-degenerate estimate in the support of $\rho_3$, which is achieved so long as $\rho_3$ is supported sufficiently far to the right of $x = 1$ as compared to $\rho_{2}$. 
\end{remark}

\begin{remark} It is worth emphasizing again the distinguishing feature of the second-order estimate, which vanishes for the third-order estimate. This is estimate (\ref{cruc.}) where factors of $v$ (no derivative) appear. This subsequently forces the weight $w$ to be ``almost-linear" in the sense of (\ref{crucAL}). As soon as the order is upgraded to third-order, this problem vanishes, enabling us to apply the weights $w_{3,L}$ which vanish at $x = \infty$. See (\ref{CRUC.10}) in the forthcoming estimate to contrast with (\ref{cruc.}). 
\end{remark}

Let us now define the third-order forcing term: 
\begin{align} \label{calw3.E}
&\mathcal{W}_{3,E} =  \int \int | f_{xx}|  |u_{xx}|w_3^4 +  \epsilon |g_{xx}|  |v_{xx}|w_3^4 , \\  \label{calw3.P}
&\mathcal{W}_{3,P} =  \int \int | f_{xx}| | u_{xxx}|w_3^5 +  \epsilon |g_{xx}|  |v_{xxx}| w_3^5 , \\ \label{calw3}
&\mathcal{W}_3 =  \mathcal{W}_{3,E} + \mathcal{W}_{E,P}.
\end{align}

We are now ready to state our energy estimate: 
\begin{proposition}[Third-Order Energy Estimate] \label{prop.ho.3} Let $\delta, \epsilon$ be sufficiently small relative to universal constants, and $\eps << \delta$. Then solutions $[u,v] \in Z(\Omega^N)$ to the system (\ref{EQ.NSR.1}) - (\ref{EQ.NSR.3}) satisfy:  
\begin{align} \n
||\partial_x^{2}u_y w_3^{2}||_{L^2}^2 &+ \epsilon ||\partial_x^{2} \{\sqrt{\epsilon}v_x, v_y\} w_3^{2}||_{L^2}^2  \\  \label{Energy.K}
& \lesssim \mathcal{O}(\delta) || \partial_x^{2} \{ \sqrt{\epsilon}v_x, v_y \} w_3^{\frac{5}{2}}||_{L^2}^2 + ||u,v||_{X_1 \cap X_2}^2 + \sum_{j=1}^2 \mathcal{W}_j + \mathcal{W}_{3,E}.
\end{align}

\end{proposition}
\begin{proof}

We shall apply the multiplier $v_x w_{3,L}^{4}$ to the system (\ref{stream.3x}). Notationally, we drop the subscript-$3$ from the weight, and simply call $w_L$ the weight appearing in (\ref{wk}) for this calculation. We start with the highest-order terms, first from $\Delta_\epsilon u$:
\begin{align} \nonumber
\int \int \partial_y \partial_x^{2} &\Big(-u_{yy} - \epsilon u_{xx} \Big) \cdot v_x w_L^4 = - \int \int \Delta_\epsilon u_{xx} \cdot u_{xx} w_L^4 \\ \label{HO.1}
& = \int \int \Big( u_{xxy}^2 + \epsilon u_{xxx}^2 \Big) w_L^4 - \frac{1}{2}\int \int \epsilon u_{xx}^2 \partial_x^2 w_L^4. 
\end{align}

Next, the terms from $\Delta_\epsilon v$:
\begin{align} \nonumber
\int \int \epsilon^2 \partial_x \partial_x^2 v_{xx} &v_x w_L^4 \\ \nonumber
& = - \int \int \epsilon^2 \partial_x^2 v_{xx}  v_{xx} w_L^4 - \int \int \epsilon^2 \partial_x^2 v_{xx} v_x \partial_x w_L^4 \\ \label{HO.2}
& = \int \int \epsilon^2 v_{xxx}^2 w_L^4 - \frac{1}{2}\epsilon^2 v_{xx}^2 \partial_x^2 w_L^4 + C \epsilon^2 v_x^2 \partial_x^4 w_L^4. 
\end{align}

Finally, 
\begin{align} \nonumber
\int \int \epsilon \partial_x \p_x^2v_{yy} \cdot &v_x w_L^4 \\ \nonumber
& = - \int \int \epsilon \partial_x^2 v_{yy}  v_{xx} w_L^4 - c \epsilon \partial_x^2 v_{yy} v_x \p_x w_L^{4} \\ \label{HO.3}
& = \int \int \epsilon v_{xxy}^2 w_L^4 - \epsilon v_{xy}^2 \partial_x^2 w_L^4. 
\end{align}

For the final integrations from (\ref{HO.1}) - (\ref{HO.3}), we use the inequalities
\begin{align}
|\partial_x^2 w_L^4| \lesssim w_L^{2}, \hspace{3 mm} |\partial_x^4 w_L^4| \lesssim C,
\end{align}

to estimate: 
\begin{align}
&| \int \int \epsilon u_{xx}^2 \partial_x^2 w_L^4 |\lesssim ||\sqrt{\epsilon} u_{xx} w_L||_{L^2}^2 \lesssim \epsilon ||u||_{X_{2}}^2, \\
 &| \int \int \epsilon^2 v_{xx}^2 \partial_x^2 w_L^4 | \lesssim \epsilon||\sqrt{\epsilon}v_{xx} w_L||_{L^2}^2 \lesssim \epsilon ||v||_{X_{2}}^2, \\
 & | \int \int \epsilon^2 v_x^2 \partial_x^4 w_L^4 | \lesssim \epsilon ||\sqrt{\epsilon} v_x ||_{L^2}^2 \lesssim \epsilon ||v||_{X_{1}}^2, \\
 & |\int \int \epsilon v_{xy}^2 \partial_x^2 w_L^4| \lesssim \epsilon ||u_{xx} w_L||_{L^2}^2 \le \epsilon ||u||_{X_{2}}^2. 
\end{align}

This then leaves from (\ref{HO.1}) - (\ref{HO.3}) the three terms on the left-hand side of (\ref{Energy.K}). We now move to the profile terms contained in $S_u$, which we will display here for convenience upon integrating by parts in $y$ and using the divergence-free condition:
\begin{align} \nonumber
\int \int \partial_y \partial_{xx}& \Big[ u_R u_x + u_{Rx} u + v_R u_y + u_{Ry}v \Big] \cdot  v_x w_L^4 \\ \label{read.2}
&= \int \int  \partial_{xx} \Big[ u_R u_x + u_{Rx} u + v_R u_y + u_{Ry}v \Big] \cdot u_{xx} w_L^4. 
\end{align}

First, we will expand via the product rule:  
\begin{align} 
 \int \int \p_{xx} (u_R u_x) u_{xx} w_L^4 = \sum_{k=0}^{2} c_k \int \int \partial_x^k u_R \p_x^{2-k}u_x \cdot u_{xx} w_L^4.
\end{align}

For $k = 0$, we integrate by parts and appeal to (\ref{PE0.4}), (\ref{PE4}):
\begin{align} \nonumber
|\int \int u_R \partial_x \Big( |u_{xx} |^2 \Big) w_L^4 &|=| - \int \int |u_{xx}|^2 \partial_x (u_R w_L^4) |\\ \n
& = |\int \int |u_{xx}|^2 \Big( u_{Rx} w_L^4 + c u_R \partial_x w_L^4\Big) |\\
& \lesssim ||u_{Rx} x, u_R||_{L^\infty} ||u_{xx} w_L^{\frac{3}{2}}||_{L^2}^2 \lesssim ||u||_{X_2}^2. 
\end{align}

For $k > 0$, we directly estimate using (\ref{PE0.3}) and (\ref{PE4}):
\begin{align} \nonumber
| \int \int \partial_x^k u_R \partial_x^{2-k} u_x &\cdot  u_{xx}  w_L^4 | \\ \n
& \lesssim ||\partial_x^k u_R x^k||_{L^\infty} ||\p_x^{2-k}u_x w_L^{\frac{5}{2}- k} ||_{L^2} ||u_{xx} x^{\frac{3}{2}}||_{L^2} \\
& \lesssim  ||u||_{X_1 \cap X_{2}}^2. 
\end{align}

Let us now move to the second term from (\ref{read.2}):
\begin{align} \label{read.2.2}
\int \int \partial_x^2(u_{Rx}u) \cdot u_{xx} w_L^4 = \sum_{k=0}^{2} c_k \int \int \partial_x^k u_{Rx} \p_x^{2-k}u \cdot u_{xx} w_L^4
\end{align}

For the $k = 2$ case, that is when all derivatives avoid the $u$ term, we have to employ the splitting $\partial_x^{3} u_R = \partial_x^3 u^P_R + \partial_x^3 u^E_R$, and treat each piece separately. First, by (\ref{PE0.3}):
\begin{align} \nonumber
|\int \int \partial_x^3 u^P_R  u \cdot  u_{xx} w_L^4| &\le ||\partial_x^3 u^P_R yx^{\frac{5}{2}}||_{L^\infty} ||\frac{u}{y}||_{L^2} ||u_{xx} w_L^{\frac{3}{2}}||_{L^2} \\ 
& \lesssim ||u_y||_{L^2} ||u_{xx} w_L^{\frac{3}{2}}||_{L^2} \lesssim  ||u||_{X_1} ||u||_{X_2}.
\end{align}

Next, for the $u^E_R$ contribution, by (\ref{PE5}):
\begin{align} \n
|\int \int \partial_x^3 u^E_R  u \cdot  u_{xx} w_L^4| &\le ||\partial_x^3 u^E_R x^{\frac{7}{2}}||_{L^2} ||\frac{u}{x}||_{L^2} ||u_{xx} w_L^{\frac{3}{2}}||_{L^2} \\ 
& \lesssim  ||u_x||_{L^2} ||u_{xx} w_L^{\frac{3}{2}}||_{L^2} \lesssim ||u||_{X_1} ||u||_{X_{2}}. 
\end{align}

For (\ref{read.2.2}) when $k < 2$, we have by (\ref{PE0.3}), (\ref{PE5}):
\begin{align} \n
|\int \int \partial_x^k u_{Rx} \p_x^{2-k}u &\cdot  u_{xx} w_L^4| \\ \n
& \le ||\partial_x^k u_{Rx} x^{1+k}||_{L^\infty} ||\p_x^{2-k}u x^{\frac{3}{2}-k}||_{L^2} ||u_{xx} x^{\frac{3}{2}}||_{L^2} \\
& \le ||u||_{X_{2-k}} ||u||_{X_{2}}. 
\end{align}

We will now move to the third term from (\ref{read.2}):
\begin{align}\n
\int \int \p_{xx} &(v_R u_y) \cdot u_{xx} w_L^4 = \sum_{k=0}^{2} c_k \int \int \partial_x^k v_R \p_x^{2-k}u_y u_{xx} w_L^4 \\ \n
& \le \sum_{k = 0}^{2} ||\partial_x^k v_R x^{k + \frac{1}{2}}||_{L^\infty} ||\partial_x^{2-k} u_y w_L^{2-k}||_{L^2} ||u_{xx} x^{\frac{3}{2}}||_{L^2} \\
 & \lesssim \mathcal{O}(\delta) ||\p_x^2 u_y w_L^2||_{L^2}||u,v||_{X_2} +  \sum_{k=1}^2  ||u||_{X_{3-k}} ||u||_{X_2}. 
\end{align}

Above we have used estimate (\ref{PE1}) and (\ref{PE4.new.2}) to obtain the smallness of $\mathcal{O}(\delta)$ for the $v_R$ term, and estimate (\ref{PE0.1}) for the remaining terms. We must take $\delta$ small enough to absorb the top order term above into (\ref{HO.1}). Finally, the fourth convective term from (\ref{read.2}) is expanded into, 
\begin{align}
\int \int \p_x^2 (u_{Ry}v) \cdot u_{xx} w_L^4 = \sum_{k=0}^{2} c_k \int \int \p_x^k u_{Ry} \p_x^{2-k} v \cdot u_{xx} w_L^4
\end{align}

We now split $u_R = u^P_R + u^E_R$. First, according to (\ref{PE0.3}), (\ref{PE2}), (\ref{PE0.5}), and (\ref{PE3}):
\begin{align} \nonumber
|\int \int \p_x^k u^P_{Ry} \p_x^{2-k} v &\cdot u_{xx} w_L^4| \\ \n
& \le ||\partial_x^k u^P_{Ry} y x^k ||_{L^\infty} || \frac{\partial_x^{2-k}v}{y} x^{\frac{5}{2}-k}||_{L^2} ||u_{xx} x^{\frac{3}{2}}||_{L^2} \\ \n
& \lesssim  ||\partial_x^k u^P_{Ry} y x^k ||_{L^\infty} ||\partial_x^{2-k}v_y x^{\frac{5}{2}-k}||_{L^2} ||u_{xx} x^{\frac{3}{2}}||_{L^2} \\ \label{only}
& \lesssim \mathcal{O}(\delta) ||\p_x^2 v_y x^{\frac{5}{2}}||_{L^2} ||u||_{X_{2}} + \sum_{k=1}^2 ||u,v||_{X_{3-k}}||u||_{X_2}. 
\end{align}

Let us remark that term (\ref{only}), when $k = 0$, requires the top order norm, $||v||_{X_3}$ in order to control, and so the smallness of $\mathcal{O}(\delta)$ is essential above. This smallness is obtained via (\ref{PE0.5}) and (\ref{PE3}), the latter of which is accompanied by $\eps^{\frac{n}{2}}$. Next, we treat the Eulerian contribution via estimate (\ref{PE5}):
\begin{align} \n
|\sqrt{\epsilon} \int \int \p_x^k u^E_{RY} \p_x^{2-k} v&\cdot u_{xx} w_L^4| \\ \n
& \le ||\partial_x^k u^E_{RY} x^{k + \frac{3}{2}} ||_{L^\infty} || \sqrt{\epsilon} \p_x^{2-k}v x^{1-k}||_{L^2} ||u_{xx} x^{\frac{3}{2}}||_{L^2} \\ 
& \lesssim  \sqrt{\eps} ||u,v||_{X_{2-k}} ||u||_{X_2}.
\end{align}

For the previous calculation, in the event that $k = 2$, we have used the Hardy inequality: 
\begin{align}
||\sqrt{\epsilon} \frac{v}{x}||_{L^2} \lesssim ||\sqrt{\epsilon}v_x ||_{L^2} \lesssim ||v||_{X_1}. 
\end{align}

We are now ready to move to the terms from $S_v$, which we depict here for convenience: 
\begin{align} \label{read.h.v}
\int \int - \epsilon \partial_x \partial_x^2 \Big[u_R v_x + v_{Rx}u + v_R v_y + v_{Ry}v \Big] \cdot  v_x w_L^4 
\end{align} 

We will now work through the first term: 
\begin{align} \n
\int \int -\epsilon \partial_x \partial_{xx} (u_R v_x)& v_x w_L^4 \\ \n
& = \int \int \epsilon \p_{xx}(u_R v_x) \Big(v_{xx} w_L^4 + v_x \partial_x w_L^{4} \Big) \\ \label{HHH}
& = \sum_{k=0}^{2} \sum_{i=1}^2 c_{k,i} \int \int \epsilon \partial_x^k u_R \p_x^{2-k}v_x \cdot \partial_x^i v \partial_x^{2-i} w_L^4. 
\end{align}

First, we will treat the $k = 0$ terms from (\ref{HHH}), using estimates (\ref{PE0.4}) and (\ref{PE5}):
\begin{align} \n
\int \int \epsilon u_R v_{xxx} \cdot v_{xx} w_L^4 &= - \frac{1}{2} \int \int \epsilon v_{xx}^2 \partial_x(u_R w_L^4)  = C \int \int \epsilon v_{xx}^2 \Big(u_{Rx} w_L^4 + u_R \partial_x w_L^4 \Big) \\
&  \lesssim ||u_R, u_{Rx}x||_{L^\infty} ||\sqrt{\epsilon} v_{xx} w_L^{\frac{3}{2}}||_{L^2}^2 \lesssim ||v||_{X_2}^2. 
\end{align}

The second $k = 0$ term, again using (\ref{PE0.4}) and (\ref{PE5}):
\begin{align} \nonumber
\int \int \epsilon u_R v_{xxx}  v_x \partial_x w_L^{4}& = - \int \int \epsilon  v_{xx} \cdot \partial_x \{ u_R v_x  \partial_x w_L^{4} \} \\ \label{CRUC.10}
& =  - \int \int \epsilon v_{xx} u_{Rx} v_x \partial_x w_L^{4} - \epsilon v_{xx} ^2 u_R \partial_x w_L^4 - \epsilon v_{xx} u_R v_x \partial_x^2 w_L^{4} \\ 
& \lesssim ||u_R, u_{Rx}x ||_{L^\infty} \Big( || \sqrt{\epsilon}  v_{xx} w_L^{\frac{3}{2}}||_{L^2}^2 +  ||\sqrt{\epsilon}v_x  x^{\frac{1}{2}}||_{L^2}^2 \Big) \\ \
& \lesssim ||v||_{X_1 \cap X_2}^2. 
\end{align}

\begin{remark} The third term in line (\ref{CRUC.10}) is the term which has improved for the higher-order energy estimates from the second order energy estimate, allowing us to use the cut-off weight $w_L$. Indeed, reading this term with $k = 2$ shows that a factor of $v$ appears, which cannot be controlled in $L^2$.
\end{remark}

Next, we move to the $k > 0$ terms from (\ref{HHH}), for which (\ref{PE0.3}), (\ref{PE0.4}), and (\ref{PE5}) are in constant use: 
\begin{align} \nonumber
|\int \int \epsilon \partial_x^k u_R \p_x^{2-k}v_x &\cdot v_{xx} w_L^4 | \\ \n
& \lesssim ||\partial_x^k u_R x^k||_{L^\infty} ||\sqrt{\epsilon} \partial_x^{3-k}v x^{\frac{5}{2}-k}||_{L^2} ||\sqrt{\epsilon} v_{xx} x^{\frac{3}{2}}||_{L^2} \\
& \lesssim  ||v||_{X_{3-k}} ||v||_{X_2} , \\ \n
|\int \int \epsilon \partial_x^k u_R \p_x^{2-k}v_x &\cdot v_x \partial_x w_L^{4} | \\ \n
& \lesssim ||\partial_x^j u_R x^j||_{L^\infty} ||\sqrt{\epsilon} \partial_x^{3-k}v x^{\frac{5}{2}-k}||_{L^2} ||\sqrt{\epsilon} v_x x^{\frac{1}{2}}||_{L^2} \\
& \lesssim ||v||_{X_{3-k}} ||v||_{X_1}.
\end{align}

We will now approach the second term from (\ref{read.h.v}), for which we use (\ref{PE0.1}) and (\ref{PE4}):
\begin{align} \n
\int \int -\epsilon \partial_x \partial_x^2& (v_{Rx}u) \p_xv w_L^4 \\ \n
& = \sum_{k=0}^{2} \int \int \epsilon \partial_x^k v_{Rx} \p_x^{2-k} u \cdot v_{xx} w_L^4 - \epsilon \p_x^k v_{Rx} \p_x^{2-k} u \cdot v_x \p_x w_L^{4}  \\ \n
& \le \sum_{k= 0}^{2} \sqrt{\epsilon} ||\partial_x^k v_{Rx} x^{k+\frac{3}{2}}||_{L^\infty} ||\partial_x^{2-k}u x^{1-k }||_{L^2} \Big( || \sqrt{\epsilon}  v_{xx} w_L^{\frac{3}{2}}||_{L^2} \\ \n
& \hspace{40 mm} + ||\sqrt{\epsilon} v_x w_L^{\frac{1}{2}} ||_{L^2} \Big) \\
& \le  ||u,v||_{X_1 \cap X_2}^2. 
\end{align}

Above, in the event when $k = 2$, we have used the Hardy inequality to obtain control $||\frac{u}{x}||_{L^2} \lesssim ||u_x||_{L^2}$.  We will now approach the third term from (\ref{read.h.v}), where again we use (\ref{PE0.1}), (\ref{PE1}), and (\ref{PE4}):
\begin{align} \n
|\int \int &\epsilon \partial_x \partial_x^2(v_R v_y) \p_x v w_L^4| \\  \n
& = |\int \int -\epsilon \p_x^2 (v_R v_y) v_{xx} w_L^4 + \p_x^2 (v_R v_y) v_x \partial_x w_L^{4}| \\ \n
 & =| \sum_{k=0}^{2} c_k \int \int \epsilon \partial_x^k v_R \p_x^{2-k}v_y \Big(  v_{xx} w_L^4 +v_x \partial_x w_L^{4} \Big) | \\ \n
 & \lesssim \sum_{k=0}^2 \sqrt{\epsilon} ||\partial_x^k v_R x^{\frac{1}{2}+k} ||_{L^\infty} ||\p_x^{2-k}v_y w_L^{\frac{5}{2}-k} ||_{L^2} \Big( || \sqrt{\epsilon}  v_{xx} w_L^{\frac{3}{2}}||_{L^2} + ||\sqrt{\epsilon} v_x x^{\frac{1}{2}}||_{L^2} \Big)\\ 
 & \lesssim ||v||_{X_1 \cap X_2 }^2 + \eps ||\p_x^2 v_y w_L^{\frac{5}{2}}||_{L^2}^2. 
\end{align}

We move to the fourth term from (\ref{read.h.v}): 
\begin{align} \n
\int \int \epsilon \partial_x \partial_x^2 &(v_{Ry}v) v_x  w_L^4 \\ \n
& = \int \int - \epsilon \partial_x^2 (v_{Ry}v) \Big( v_{xx} w_L^4 + C v_x \partial_x w_L^{4} \Big) \\  
& = \sum_{k=0}^{2} c_k \int \int \epsilon \partial_x^k v_{Ry} \p_x^{2-k}v \Big(v_{xx} w_L^4 + C v_x \partial_x w_L^{4} \Big).
\end{align}

We must split $v_R = v^P_R + v^E_R$. First, we treat the Prandtl component by using estimates (\ref{PE0.1}) - (\ref{PE1}):
\begin{align} \n
|\int \int \epsilon \partial_x^k v^P_{Ry} &\p_x^{2-k}v \Big(  v_{xx} w_L^4 + C v_x \partial_x w_L^{4} \Big)| \\ \n
& \le \sqrt{\epsilon} ||\partial_x^k v^P_{Ry} yx^{\frac{1}{2}+k}||_{L^\infty} ||\p_x^{2-k}\frac{v}{y} w_L^{\frac{5}{2}-k}||_{L^2} ||\sqrt{\epsilon}  v_{xx} w_L^{\frac{3}{2}}, \sqrt{\epsilon} v_x x^{\frac{3}{2}}||_{L^2} \\ \n
& \le \sqrt{\epsilon} ||\partial_x^k v^P_{Ry} yx^{\frac{1}{2}+k}||_{L^\infty} ||\p_x^{2-k}v_y w_L^{\frac{5}{2}-k}||_{L^2} ||\sqrt{\epsilon}  v_{xx} w_L^{\frac{3}{2}}, \sqrt{\epsilon} v_x x^{\frac{3}{2}}||_{L^2} \\ 
& \le \sqrt{\epsilon}  ||u,v||_{X_1 \cap X_2}^2+ \sqrt{\eps} ||\p_x^2 v_y w_L^{\frac{5}{2}}||_{L^2} ||u,v||_{X_1 \cap X_2}. 
\end{align}

Now, fix $k < 2$. The Eulerian contribution is controlled via (\ref{PE4}):
\begin{align} \n
\sqrt{\epsilon} \int \int \epsilon \partial_x^k v^E_{RY} &\p_x^{2-k}v \Big(v_{xx} w_L^4 + C v_x \p_x w_L^{4} \Big) \\ \n
& \le \sqrt{\epsilon} ||\partial_x^k v^E_{RY} x^{\frac{3}{2}+k} ||_{L^\infty} ||\sqrt{\epsilon} \p_x^{2-k}v x^{1-k}||_{L^2} ||\sqrt{\epsilon}  v_{xx} x^{\frac{3}{2}}, \sqrt{\epsilon} v_x x^{\frac{1}{2}}||_{L^2} \\
& \lesssim \sqrt{\epsilon} ||v||_{X_{2-k}} ||v||_{X_1 \cap X_2}. 
\end{align}

For $k = 2$, we must employ the Hardy inequality in addition to (\ref{PE4}) to conclude: 
\begin{align} \n
\sqrt{\epsilon} \int \int \epsilon \partial_x^2 v^E_{RY} v &\Big(v_{xx} w_L^4 + C v_x \p_x w_L^{4} \Big) \\ \n
& \le \sqrt{\epsilon} ||\partial_x^2 v^E_{RY} x^{\frac{3}{2}+2} ||_{L^\infty} ||\sqrt{\epsilon} v x^{-1}||_{L^2} ||\sqrt{\epsilon}  v_{xx} x^{\frac{3}{2}}, \sqrt{\epsilon} v_x x^{\frac{1}{2}}||_{L^2} \\ \n
& \le \sqrt{\epsilon} ||\partial_x^2 v^E_{RY} x^{\frac{3}{2}+2} ||_{L^\infty} ||\sqrt{\epsilon} v_x||_{L^2} ||\sqrt{\epsilon}  v_{xx} x^{\frac{3}{2}}, \sqrt{\epsilon} v_x x^{\frac{1}{2}}||_{L^2} \\
& \le \sqrt{\epsilon} ||\partial_x^2 v^E_{RY} x^{\frac{3}{2}+2} ||_{L^\infty} ||v||_{X_1} ||v||_{X_1 \cap X_2}. 
\end{align}

The final task is to address the right-hand side: 
\begin{align}
\int \int \partial_{xxy} f \cdot v_{x} w_L^4 &= \int \int f_{xx} u_{xx} w_L^4, \\ \nonumber
 -\epsilon \int \int \partial_{xxx} g \cdot v_{x} w_L^4 &= \int \int \epsilon g_{xx} \{v_{xx}w_L^4 + v_x 4w_L^3 w_L' \} \\
 & = \int \int \epsilon g_{xx} v_{xx} w_L^4 + C \epsilon g_x \{ v_{xx} w_L^3w_L' + v_{x} \partial_{xx}\{w_L^4 \}  \}.
\end{align} 

We estimate the terms: 
\begin{align} \n
| \int \int \eps g_x \{v_{xx}w_L^3 w_L' &+ v_x \p_{xx}\{ w_L^4 \} \} | \lesssim \int \int \eps |g_x | \{ |v_{xx}| w_L^3 + |v_x| w_L^2 \} \\ \n
& \le \int \int \eps |g_x| \{ |v_{xx}| w_3^3 + |v_x| w_3^2 \}  \\
& \le \int \int \eps |g_x| \{ |v_{xx}| w_2^3 + |v_x| w_2^2 \} \le \mathcal{W}_2.
\end{align}

A comparison now with the definitions in (\ref{calw3}) gives the desired result upon taking $L \rightarrow \infty$. 

\end{proof}

We now give the third-order positivity estimate, which is the final estimate for our linear analysis:

\begin{proposition}[Third Order Positivity Estimate] \label{prop.ho.4} Let $\delta, \epsilon$ be sufficiently small relative to universal constants. Then solutions $[u,v] \in Z(\Omega^N)$ to the system (\ref{EQ.NSR.1}) - (\ref{EQ.NSR.3}) satisfy: 
\begin{align} \nonumber
||\{\sqrt{\epsilon}v_{xxx}, v_{xxy} \} w_3^{\frac{5}{2}}||_{L^2}^2 &\lesssim ||\{u_{xxy}, \sqrt{\epsilon}u_{xxx}, \epsilon v_{xxx} \} w_3^2||_{L^2}^2  + ||u,v||_{X_1 \cap X_2}^2 \\ 
& \hspace{20 mm} + \mathcal{W}_1 + \mathcal{W}_2 + \mathcal{W}_3. 
\end{align}
\end{proposition}

\begin{proof}
We apply the multiplier $v_{xx}w_{3,L}^5$. As usual, we shall drop the subscript-3 for this calculation, with the understanding that $w_{3,L} = w_L$.The highest-order terms are:
\begin{align} \nonumber
\int \int \{\partial_{yxx} &(-\Delta_\epsilon u) +\epsilon \partial_{xxx}(\Delta_\epsilon v) \} \cdot v_{xx}w_L^5 \\ \nonumber
& =- \int \int u_{yxx} \p_x w_L^5 + C\epsilon \int \int u_{xxx}^2 \p_x w_L^5 + C\epsilon^2 v_{xxx}^2 \p_x w_L^5 \\
&+ C\epsilon^2 \int \int v_{xx}^2 \partial_{x}^3 \{w_L^5 \}.
\end{align}

Next, we have the main profile term from $S_u$, which we will list here for convenience: 
\begin{align} \n
\int \int \partial_y \partial_{xx}& \Big[ u_R u_x + u_{Rx}u + v_R u_y + u_{Ry}v \Big] \cdot v_{xx} w_L^5  \\ \label{read.p.1}
& = \int \int \partial_{xx} \Big[ u_R u_x + u_{Rx}u + v_R u_y + u_{Ry}v \Big] \cdot u_{xxx} w_L^5.
\end{align}

We will now treat the first term from (\ref{read.p.1}):
\begin{align} \n
\int \int \partial_{xx} &(u_R u_x) \cdot u_{xxx} w_L^5 = \int \int \Big( u_R u_{xxx} + 2 u_{Rx} u_{xx} + u_{Rxx} u_x \Big) \cdot u_{xxx} w_L^5 \\ \label{r.p.1}
& \gtrsim \min u_R \int \int u_{xxx}^2 w_L^5 + \int \int u_{Rx}u_{xx}u_{xxx} w_L^5 + u_{Rxx} u_x u_{xxx} w_L^5.
\end{align}

Bounds on the final two integrals on the right-hand side above follow from (\ref{PE0.3}), (\ref{PE5}): 
\begin{align} \n
&| \int \int u_{Rx} u_{xx} u_{xxx} w_L^5| \le ||u_{Rx} x||_{L^\infty} ||u_{xx} w_L^{\frac{3}{2}}||_{L^2} ||u_{xxx} w_L^{\frac{5}{2}}||_{L^2} \\
& \hspace{30 mm} \le C||u,v||_{X_2}^2 + \frac{1}{100,000} ||u_{xxx} w_L^{\frac{5}{2}}||_{L^2}, \\ \n
&| \int \int u_{Rxx} u_x u_{xxx} w_L^5 | \le ||u_{Rxx} x^2||_{L^\infty} ||u_x x^{\frac{1}{2}}||_{L^2}||u_{xxx} w_L^{\frac{5}{2}}||_{L^2} \\
& \hspace{30 mm} \le C||u,v||_{X_1}^2 + \frac{1}{100,000} ||u_{xxx} w_L^{\frac{5}{2}}||_{L^2}^2. 
\end{align}

The $u_{xxx}$ terms on the right-hand side above can be absorbed by the first term in (\ref{r.p.1}). Here, we use that in the support of $\rho_3$, $C_1 x \le \rho_2(x) \le C_2 x$ so that $u_{xx}$ from (\ref{r.p.1}), for instance, does not require any of the degeneration from $w_L$. The next term from (\ref{read.p.1}) is:
\begin{align} \nonumber
\int \int \partial_{xx}&(u_{Rx}u) \cdot u_{xxx} w_L^5| =| \int \int \Big( u_{Rxxx}u + 2u_{Rxx}u_x + u_{Rx}u_{xx} \Big) \cdot u_{xxx}w_L^5.
\end{align}

We shall now split $u_R = u^P_R + u^E_R$. First, the $u^P_R$ contribution, via estimate (\ref{PE0.3}):
\begin{align} \n
| \int \int \partial_x^k u^P_R \partial_x^{i} \cdot u_{xxx} w_L^5| &\le ||\partial_x^k u^P_R x^{k-\frac{1}{2}}y||_{L^\infty} || \frac{\partial_x^{3-k} u}{y} x^{3-k}||_{L^2} ||u_{xxx} w_L^{\frac{5}{2}}||_{L^2} \\ \n
& \le  ||\partial_x^k u^P_R x^{k-\frac{1}{2}}y||_{L^\infty} || \partial_x^{3-k} u_y x^{3-k}||_{L^2} ||u_{xxx} w_L^{\frac{5}{2}}||_{L^2} \\ 
& \lesssim  ||u||_{X_{4-k}} ||u_{xxx} w_L^{\frac{5}{2}}||_{L^2} \le C||u,v||_{X_1 \cap X_2}^2 + \frac{1}{100,000}||u_{xxx} w_L^{\frac{5}{2}}||_{L^2}^2,
\end{align}

for $k = 3,2$. It is worth distinguishing the $k = 1$ case, although the calculation is identical, by appealing to estimate (\ref{PE0.4}):
\begin{align} \n
| \int \int u^P_{Rx} u_{xx} u_{xxx} w_L^5| &\le ||u^P_{Rx} yx^{\frac{1}{2}}||_{L^\infty} ||u_{xxy} w_L^2||_{L^2} ||u_{xxx} w_L^{\frac{5}{2}}||_{L^2} \\
& \le \mathcal{O}(\delta) ||u_{xxy} w_L^2||_{L^2} ||u_{xxx} w_L^{\frac{5}{2}}||_{L^2},
\end{align}

because now both majorizers are part of the $X_3$ norm: $u_{xxy}$ has been estimated in the energy estimate, and $u_{xxx}$ appears in (\ref{r.p.1}). We may then absorb the $u_{xxx}$ term from above into  (\ref{r.p.1}) by taking $\delta$ sufficiently small. Next, the Eulerian contribution, for which we use (\ref{PE5}), first with $k = 1,2$:
\begin{align} \n
| \int \int \partial_x^k u^E_R \partial_x^{3-k}u \cdot u_{xxx} w_L^5 | &\le ||\p_x^k u^E_R x^{k+\frac{1}{2}}||_{L^\infty} ||\p_x^{3-k}u x^{3-k-\frac{1}{2}}||_{L^2} ||u_{xxx} w_L^{\frac{5}{2}}||_{L^2} \\
& \le \sqrt{\eps} ||u||_{X_{3-k}} ||u_{xxx} w_L^{\frac{5}{2}}||_{L^2}. 
\end{align}

For the $k = 3$ case, we must add an extra step via Hardy's inequality: 
\begin{align} \n
| \int \int u^E_{Rxxx} u u_{xxx} w_L^5 | &\le ||u^E_{Rxxx} x^{\frac{7}{2}}||_{L^\infty} ||\frac{u}{x}||_{L^2} ||u_{xxx} w_L^{\frac{5}{2}}||_{L^2} \\
& \lesssim \sqrt{\eps}  ||u_x ||_{L^2}  ||u_{xxx} w_L^{\frac{5}{2}}||_{L^2} \lesssim \sqrt{\eps}  ||u||_{X_1}  ||u_{xxx} w_L^{\frac{5}{2}}||_{L^2}. 
\end{align}

Next, we have the convection term (we set $k = 0,1,2$ below): 
\begin{align} \nonumber
\int \int \partial_{xx} (u_{Ry}v) \cdot u_{xxx} w_L^5 &= \sum_{k=0}^2 \int \int \partial_x^k u_{Ry} \partial_x^{2-k}v \cdot u_{xxx} w_L^5 \\
& = \sum_{k=0}^2 \int \int \partial_x^k \Big(u^P_{Ry} + \sqrt{\epsilon}u^E_{RY}\Big) \partial_x^{2-k}v \cdot u_{xxx} w_L^5.
\end{align}

First, the Prandtl contributions, using estimates (\ref{PE2}):
\begin{align} \n
|\int \int \partial_x^k u^P_{Ry} \partial_x^{2-k} v \cdot u_{xxx} w_L^5 | &\lesssim || \partial_x^k u^P_{Ry} yx^k||_{L^\infty} ||\frac{\partial_x^{2-k}v}{y} w_L^{3-k-\frac{1}{2}}||_{L^2} ||u_{xxx} w_L^{\frac{5}{2}}||_{L^2} \\ \n
& \lesssim || \partial_x^k u^P_{Ry} yx^k||_{L^\infty} ||\partial_x^{2-k}v_y w_L^{3-k-\frac{1}{2}}||_{L^2} ||u_{xxx} w_L^{\frac{5}{2}}||_{L^2}.
\end{align}

Again, we distinguish the $k = 0$ case above, both majorizing terms are in $X_3$, and so we must use the smallness of $\mathcal{O}(\delta)$ to absorb into the (\ref{r.p.1}) positive term. In particular, appealing to estimate (\ref{PE0.5}), (\ref{PE3}), we have:
\begin{align}
||u^P_{Ry} y||_{L^\infty} ||\p_x^2 v_y w_L^{\frac{5}{2}}||_{L^2} ||u_{xxx} w_L^{\frac{5}{2}}||_{L^2} \le \mathcal{O}(\delta)||\p_x^2 v_y w_L^{\frac{5}{2}}||_{L^2} ||u_{xxx} w_L^{\frac{5}{2}}||_{L^2}.
\end{align}

For $k > 0$, we have by using (\ref{PE0.3}) - (\ref{PE2}) and then Young's inequality:
\begin{align} \n
||\p_x^k u^P_{Ry} yx^k||_{L^\infty} ||\p_x^{2-k} v_y w_L^{3-k-\frac{1}{2}}||_{L^2}&||u_{xxx} w_L^{\frac{5}{2}}||_{L^2} \lesssim ||u,v||_{X_1 \cap X_2} ||u_{xxx} w_L^{\frac{5}{2}}||_{L^2} \\
& \le C ||u,v||_{X_1 \cap X_2}^2 + \frac{1}{100,000}||u_{xxx} w_L^{\frac{5}{2}}||_{L^2}^2. 
\end{align}

For the Euler contributions $u^E_R$, we estimate for the $k = 0,1$ cases, using (\ref{PE5}): 
\begin{align} \n
| \int \int \sqrt{\epsilon} \partial_x^k u^E_{RY} \partial_x^{2-k} v u_{xxx} w_L^5| &\lesssim ||\partial_x^k u^E_{RY} x^{k+1+ \frac{1}{2}}||_{L^\infty} ||\sqrt{\epsilon}\partial_x^{2-k}v w_L^{2-k-\frac{1}{2}}||_{L^2} ||u_{xxx}w_L^{\frac{5}{2}}||_{L^2} \\
& \lesssim \sqrt{\eps}  ||v||_{X_{2-k}} ||u_{xxx} w_L^{\frac{5}{2}}||_{L^2}.  
\end{align}

For the $k = 2$ case, we must additionally use the Hardy inequality: 
\begin{align}\n
| \int \int \sqrt{\epsilon} u^E_{RYxx}  v u_{xxx} w_L^5| &\lesssim ||u^E_{RYxx} x^{\frac{7}{2}}||_{L^\infty} ||\sqrt{\epsilon} \frac{v}{x}||_{L^2} ||u_{xxx} w_L^{\frac{5}{2}}||_{L^2} \\ \n
& \lesssim ||u^E_{RYxx} x^{\frac{7}{2}}||_{L^\infty} ||\sqrt{\epsilon} v_x||_{L^2} ||u_{xxx} w_L^{\frac{5}{2}}||_{L^2}  \\
& \lesssim \sqrt{\eps} ||v||_{X_1} ||u_{xxx} w_L^{\frac{5}{2}}||_{L^2}.
\end{align}

The final term in $S_u$ is easily estimated directly, where $k = 0,1,2$, by applying estimates (\ref{PE0.1}) - (\ref{PE1}), and (\ref{PE4}), (\ref{PE4.new.2}), crucially obtaining the factor $\mathcal{O}(\delta)$ when $k = 0$:
\begin{align} 
\Big| \int \int \partial_{xx} &(v_R u_y) \cdot u_{xxx}w_L^5 \Big| \lesssim ||\partial_x^k v_R x^{k+\frac{1}{2}}||_{L^\infty} ||\partial_x^{2-k} u_y w_L^{2-k}||_{L^2} ||u_{xxx} w_L^{\frac{5}{2}}||_{L^2} \\ \nonumber
& \lesssim \mathcal{O}(\delta) ||\p_x^2 u_y w_L^2||_{L^2} ||u_{xxx} w_L^{\frac{5}{2}}||_{L^2} + \frac{1}{100,000} ||u_{xxx} w_L^{\frac{5}{2}}||_{L^2}^2 + C||u,v||_{X_1 \cap X_2}^2.
\end{align}

We now address the profile terms from $S_v$. We shall record these below for convenience: 
\begin{align} \n
\int \int -\epsilon \partial_x \partial_{xx}& \Big[ u_R v_x + v_{Rx}u + v_R v_y + v_{Ry}v \Big] \cdot v_{xx}w_L^5 \\ \label{read.p.2}
& = \int \int \epsilon \partial_{xx} \Big[ u_R v_x + v_{Rx}u + v_R v_y + v_{Ry}v \Big] \cdot \Big[v_{xxx}w_L^5 + v_{xx} \partial_x w_L^5 \Big].
\end{align}

Let us start with the first term above, which yields the desired positivity: 
\begin{align} \n
\int \int \epsilon \partial_{xx} (u_R v_x ) &\cdot (v_{xxx}  w_L^5 + v_{xx} \partial_x w_L^5) \\ 
& = \int \int \epsilon u_R v_{xxx}^2 w_L^5 + \sum_{k=1}^2 \sum_{i=2}^3 \epsilon \partial_x^k u_R \n\partial_x^{3-k}v \cdot \partial_x^i v \partial_x^{3-i} w_L^5 \\ \label{r.p.2}
&\ge \min u_R \int \int \epsilon v_{xxx}^2 w_L^5 + \sum_{k=1}^2 \sum_{i=2}^3 \epsilon \partial_x^k u_R \partial_x^{3-k}v  \cdot \partial_x^i v \partial_x^{3-i} w_L^5.
\end{align}

We estimate the summation on the right-hand side above, by using (\ref{PE0.3}) - (\ref{PE0.4}), (\ref{PE5})
\begin{align} \n
|\int \int \epsilon \partial_x^k u_R \partial_x^{3-k}v  &\cdot \partial_x^i v \partial_x^{3-i} w_L^5|| \\ \n
& \le ||\partial_x^k u_R x^k||_{L^\infty} ||\sqrt{\epsilon} \partial_x^{3-k} v x^{3-k-\frac{1}{2}}||_{L^2} || \sqrt{\epsilon} \partial_x^i v w_L^{i- \frac{1}{2}}||_{L^2} \\
& \lesssim ||v||_{X_{3-k}} ||v||_{X_i} \le C ||u,v||_{X_1 \cap X_2}^2 + \frac{1}{100,000}||\sqrt{\eps}v_{xxx}w_L^{\frac{5}{2}}||_{L^2}^2
\end{align}

When $i = 3$, one obtains the top norm $X_3$ above, and must be absorbed into the positive term from (\ref{r.p.2}), which is done via Young's inequality. The next profile term follows by using the bounds in (\ref{PE0.1}) - (\ref{PE1}) for $v^P_R$ and (\ref{PE4}) for $v^E_R$: 
\begin{align} \nonumber
| \int \int \epsilon \partial_{xx} (v_{Rx}u) &\cdot ( v_{xxx}w_L^5 + v_{xx} \partial_x w_L^5 ) | \\ \n
& \le  \sum_{k=0}^2 \sum_{i=2}^3 |\int \int \epsilon \partial_x^{k+1} v_R \partial_x^{2-k} u \partial_x^i v \partial_x^{3-i} w_L^5 | \\ \n
& \lesssim \sqrt{\epsilon} \sum_{k=0}^2 \sum_{i=2}^3 || \partial_x^{k+1}v_R x^{k+\frac{3}{2}}||_{L^\infty} ||\partial_x^{2-k}u x^{1-k}||_{L^2} ||\sqrt{\epsilon} \partial_x^iv \cdot w_L^{i-\frac{1}{2}}||_{L^2} \\
& \lesssim \sqrt{\epsilon} ||u,v||_{X_1 \cap X_2}^2 + \sqrt{\eps} ||\sqrt{\eps} \p_x^3 v w_L^{\frac{5}{2}}||_{L^2}^2. 
\end{align}

Above, we have used the Hardy inequality in the case when $k = 2$, for the term:
\begin{align}
||\frac{u}{x}||_{L^2} \lesssim ||u_x||_{L^2} \lesssim ||u||_{X_1}. 
\end{align}

The third profile term requires a splitting into Euler and Prandtl components:
\begin{align} \nonumber
\int \int \epsilon \partial_{xx} (v_{Ry}v) &\cdot (v_{xxx}w_L^5 + v_{xx} \partial_x w_L^5) \\ \n
& = \sum_{k=0}^2 \sum_{i=2}^3 \int \int \epsilon \partial_x^k v_{Ry} \partial_x^{2-k}v \cdot \partial_x^i v \partial_x^{3-i} w_L^5 \\
& =  \sum_{k=0}^2 \sum_{i=2}^3 \int \int \epsilon \partial_x^k \Big( v^P_{Ry} + \sqrt{\epsilon} v^E_{RY} \Big) \partial_x^{2-k}v \cdot \partial_x^i v \partial_x^{3-i} w_L^5 
\end{align}

First, according to (\ref{PE0.1}) - (\ref{PE1}), 
\begin{align} \n
| \int \int \epsilon \partial_x^k v^P_{Ry} \partial_x^{2-k} v &\partial_x^i v \partial_x^{3-i}w_L^5 | \\ \n
& \le \sqrt{\epsilon}  ||\partial_x^k v^P_{Ry} y x^{k+ \frac{1}{2}}||_{L^\infty} ||\frac{\partial_x^{2-k} v}{y} w_L^{3-k-\frac{1}{2}}||_{L^2} ||\sqrt{\epsilon} \partial_x^i v w_L^{i-\frac{1}{2}}||_{L^2} \\ \n
& \lesssim \sqrt{\epsilon} ||\partial_x^k v^P_{Ry} y x^{k+ \frac{1}{2}}||_{L^\infty} || \partial_x^{2-k} v_y w_L^{3-k-\frac{1}{2}}||_{L^2} ||\sqrt{\epsilon} \partial_x^i v w_L^{i-\frac{1}{2}}||_{L^2} \\
& \lesssim \sqrt{\epsilon} ||u,v||_{X_1 \cap X_2}^2 + \sqrt{\eps} ||\sqrt{\eps} \p_x^3 v w_L^{\frac{5}{2}}||_{L^2}^2 + \sqrt{\eps} ||\p_x^2 v_y w_L^{\frac{5}{2}}||_{L^2}^2.
\end{align}

Next, according to (\ref{PE4}), 
\begin{align} \n
| \int \int \epsilon^{\frac{3}{2}} \partial_x^k v^E_{RY} &\partial_x^{2-k} v \partial_x^i v \partial_x^{3-i}w_L^5 | \\ \n
& \le \sqrt{\epsilon} ||\partial_x^k v^E_{RY} x^{k+\frac{3}{2}}||_{L^\infty} ||\sqrt{\epsilon} \partial_x^{2-k} v w_L^{1-k}||_{L^2} ||\sqrt{\epsilon} \partial_x^i v w_L^{i-\frac{1}{2}}||_{L^2} \\
& \lesssim \sqrt{\epsilon}  ||u,v||_{X_1 \cap X_2}^2 + \sqrt{\eps} ||\sqrt{\eps} \p_x^3 v w_L^{\frac{5}{2}}||_{L^2}^2. 
\end{align}

Above, we have used the Hardy inequality in the case when $k = 2$ for the term: 
\begin{align}
||\sqrt{\epsilon}\frac{v}{x}||_{L^2} \lesssim ||\sqrt{\epsilon}v_x||_{L^2} \lesssim ||v||_{X_1}. 
\end{align}

The final profile term is handled by appealing to estimates (\ref{PE0.1}) - (\ref{PE1}) and (\ref{PE4}):
\begin{align} \nonumber
|\int \int \epsilon \partial_{xx} (v_R v_y) &\cdot \Big( v_{xxx} w_L^5 + v_{xx} \partial_x w_L^5 \Big)| \\ \n
& \le \eps \sum_{k=0}^2 \sum_{i=2}^3 | \int \int \epsilon \partial_x^k v_R \partial_x^{2-k} v_y \partial_x^i v \partial_x^{3-i}w_L |\\ \n
& \le \eps \sum_{k=0}^2 \sum_{i=2}^3 ||\partial_x^k v_R x^{k+\frac{1}{2}}||_{L^\infty} ||\partial_x^{2-k}v_y w_L^{3-k-\frac{1}{2}} ||_{L^2} ||\partial_x^i v w_L^{i-\frac{1}{2}}||_{L^2} \\
& \lesssim ||u,v||_{X_1 \cap X_2}^2 + \sqrt{\eps} ||\p_x^2 v_y w_L^{\frac{5}{2}}||_{L^2}^2 + \sqrt{\eps} ||\sqrt{\eps}\p_x^3v w_L^{\frac{5}{2}}||_{L^2}^2.
\end{align}

Finally, we have the right-hand side: 
\begin{align}
&\int \int \partial_{xxy} f \cdot v_{xx}w_L^5 = -\int \int f_{xx} \cdot v_{xxy}w_L^5, \\
&- \int \int \epsilon \partial_{xxx}g \cdot v_{xx} w_L^5 = \int \int \epsilon g_{xx} \{ v_{xxx} w_L^5 + 5v_{xx}w_L^4 w_L' \}. 
\end{align}

We estimate: 
\begin{align}
| \int \int \eps g_{xx} v_{xx} w_L^4 w_L' | \le \int \int \eps |g_{xx}| |v_{xx}| w_L^4 \le  \int \int \eps |g_{xx}| |v_{xx}| w_3^4 \le \mathcal{W}_3. 
\end{align}

We end by taking $L \rightarrow \infty$. A comparison with (\ref{calw3}) shows that our claim is proven. 

\end{proof}

Summarizing, then, the conclusions of the linear analysis: 
\begin{theorem}[Linear Estimates] Let $\eps, \delta$ be sufficiently small relative to universal constants, and let $\eps << \delta$. Then $[u,v] \in Z$, solutions to the system (\ref{EQ.NSR.1}) - (\ref{EQ.NSR.3}), (\ref{defn.SU.SV}) - (\ref{bar.g}), with boundary conditions (\ref{BCN}), satisfy the following \textit{a-priori} estimate: 
\begin{align} \label{linsum}
||u,v||_{X_1 \cap X_2 \cap X_3}^2 \lesssim \mathcal{W}_1 + \mathcal{W}_2 + \mathcal{W}_3.
\end{align}
\end{theorem}

\section{Nonlinear Analysis} \label{Sec.NL}

\subsection{a-priori Estimate of Nonlinearities}

In this subsection, we exhibit control of the right-hand side of (\ref{linsum}). We continue to consider the system (in a manner independent of $N$):
\begin{equation} \label{end.ns.0}
 \left.\begin{aligned}  
	-\Delta_\eps u + S_u + P_x &= \eps^{-\frac{n}{2}-\gamma} R^{u,n} + \eps^{\frac{n}{2}+\gamma} \Big[ \bar{u} \bar{u}_x + \bar{v} u_y \Big]   \\ 
	-\Delta_\eps v + S_v + \frac{P_y}{\eps} &= \eps^{-\frac{n}{2}-\gamma} R^{v,n} + \eps^{\frac{n}{2}+\gamma} \Big[ \bar{u} \bar{v}_x + \bar{v} \bar{v}_y \Big], \\ 
u_x + v_y &= 0. 
       \end{aligned}
 \right\}
  \qquad \text{in $\Omega^N$}
\end{equation}

That is, $f = f(u, \bar{u}, \bar{v})$ and $g = g(\bar{u}, \bar{v})$, as in (\ref{bar.f}), (\ref{bar.g}). Notationally, we continue to depict integration over $\Omega^N$ by $\int \int$ and similarly, norms without further specification are taken over $\Omega^N$. For the forthcoming calculation, we refer the reader to the definitions of $\mathcal{W}^i$, in equations (\ref{calw1}), (\ref{calw2}), (\ref{calw3}). 

\begin{lemma} \label{LemmaW} Suppose $||\bar{u}, \bar{v}||_Z \le 1$. For $0 \le \gamma < \frac{1}{4}$, and fixed parameters $\kappa > 0$ arbitrarily small, for $\delta, \epsilon$ sufficiently small, $n$ sufficiently large: 
\begin{align}\nonumber
\Big| \mathcal{W}_1 + \mathcal{W}_2 + \mathcal{W}_3 \Big| &\lesssim \epsilon^{\frac{1}{4}-\gamma - \kappa} + \epsilon^{\frac{1}{4}-\gamma - \kappa} ||u,v||_{X_1 \cap X_2 \cap X_3}^2 \\ \label{NLO.1}
&+ \epsilon^{\frac{n}{2}-\omega(N_i)} ||u,v||_{Z}^2+ \epsilon^{\frac{n}{2}-\omega(N_i)} ||\bar{u},\bar{v}||_{Z}^4 ,
\end{align}
where $\omega(N_i)$ is a function which depends only on universal constants and $N_i$ in the definition of norm $Z$. 
\end{lemma}
\begin{proof}

For clarity of exposition, the order in which we treat the terms are as follows: we will first treat the nonlinear terms, $\mathcal{N}^u$, arising from $f$ in $\mathcal{W}_1, \mathcal{W}_2, \mathcal{W}_3$, then those nonlinear terms, $\mathcal{N}^v$, arising from $g$, and finally the forcing terms, $R^{u,n}, R^{v,n}$, in both $f$ and $g$. Turning first to $\mathcal{W}_1$, equation (\ref{calw1}), we start with the term: 
\begin{align} \label{w.f.1}
\int \int \mathcal{N}^u \cdot u = \int \int \eps^{\frac{n}{2}+\gamma} \Big(\bar{u} \bar{u}_x + \bar{v}u_y \Big) \cdot u. 
\end{align}

First, 
\begin{align} 
\Big| \int \int \eps^{\frac{n}{2}+\gamma} \bar{u} \bar{u}_x \cdot u \Big| &\lesssim \eps^{\frac{n}{2}+\gamma} ||\bar{u} x^{\frac{1}{4}}||_{L^\infty} ||\bar{u}_x x^{\frac{1}{2}}||_{L^2} ||\frac{u}{x^{\frac{3}{4}}}||_{L^2} \\ \n
&  \lesssim \eps^{\frac{n}{2}+\gamma} ||\bar{u} x^{\frac{1}{4}}||_{L^\infty} ||\bar{u}_x x^{\frac{1}{2}}||_{L^2} || u_x x^{\frac{1}{4}}||_{L^2}  \lesssim \eps^{\frac{n}{2}+\gamma-\omega(N_i)} ||\bar{u}, \bar{v}||_Z^2 ||u, v||_Z. 
\end{align}

For the next term in (\ref{w.f.1}), we must use the structure of the nonlinearity by integrating by parts once in $y$: 
\begin{align} \label{order.NL} 
\Big| \int \int \eps^{\frac{n}{2}+\gamma} \bar{v} u_y \cdot u \Big| &= \Big| \int \int \eps^{\frac{n}{2}+\gamma} \bar{v}_y \frac{u^2}{2} \Big| \lesssim \eps^{\frac{n}{2}+\gamma} ||ux^{\frac{1}{4}}||_{L^\infty} ||\bar{v}_y x^{\frac{1}{2}}||_{L^2} ||ux^{-\frac{3}{4}}||_{L^2} \\ \n &\lesssim \eps^{\frac{n}{2}+\gamma}  ||ux^{\frac{1}{4}}||_{L^\infty} ||u_x x^{\frac{1}{4}}||_{L^2} ||\bar{v}_y x^{\frac{1}{2}}||_{L^2} \lesssim \eps^{\frac{n}{2}+\gamma - \omega(N_i)} ||u, v||_Z^2 ||\bar{u}, \bar{v}||_Z. 
\end{align}

Note carefully that no absolute values were included on this term in the definition of $\mathcal{W}_1$, equation (\ref{calw1}). We then move to the next term from (\ref{calw1}), which we display below:
\begin{align}
\int \int |\mathcal{N}^u| |v_y| x \le \int \int \eps^{\frac{n}{2}+\gamma} |\bar{u} \bar{u}_x + \bar{v} u_y| |v_y| x.
\end{align}

The nonlinear terms here are treated via: 
\begin{align} \n
\int \int \eps^{\frac{n}{2}+\gamma} |\bar{u} \bar{u}_x v_y x| &\le \eps^{\frac{n}{2}+\gamma} ||\bar{u}||_{L^\infty} ||\bar{u}_x x^{\frac{1}{2}}||_{L^2} ||v_y x^{\frac{1}{2}}||_{L^2} , \\
&\lesssim \eps^{\frac{n}{2}+\gamma - \omega(N_i)} ||\bar{u}, \bar{v}||_Z^2 ||u, v||_Z, \\ \n
\int \int \eps^{\frac{n}{2}+\gamma} | \bar{v} u_y v_y x | &\le \eps^{\frac{n}{2}+\gamma} ||\bar{v}x^{\frac{1}{2}}||_{L^\infty} ||u_y||_{L^2} ||v_y x^{\frac{1}{2}}||_{L^2} \\ \label{exact.NL}
& \lesssim \eps^{\frac{n}{2}+\gamma - \omega(N_i)} ||\bar{u}, \bar{v}||_Z ||u,v||_Z^2. 
\end{align}

Staying with $\mathcal{N}^u$, we will move to $\mathcal{W}_2$ in (\ref{calw2}):
\begin{align} \n
\int \int |\p_x \mathcal{N}^u| \Big[ |u_x| \rho_2^2 x^2 &+ |u_{xx}| \rho_2^3 x^3 \Big] \\ \label{w.f.2}
& = \int \int \eps^{\frac{n}{2}+\gamma} |\p_x (\bar{u} \bar{u}_x + \bar{v} u_y ) | \Big[ |u_x| \rho_2^2 x^2 + |u_{xx}| \rho_2^3 x^3  \Big].
\end{align}

First, we will expand: 
\begin{align} \nonumber
\int \int \eps^{\frac{n}{2}+\gamma} |\partial_x &\Big(\bar{u} \bar{u}_x \Big) | |u_x| \rho_2^2 x^2 \\ \n
& = \int \int  \eps^{\frac{n}{2}+\gamma} |\bar{u}_x^2| |u_x| \rho_2^2 x^2 + \int \int \eps^{\frac{n}{2}+\gamma} | \bar{u} \bar{u}_{xx}| |u_x| \rho_2^2 x^2 \\ \n
& \le  \eps^{\frac{n}{2}+\gamma} ||u_x  x^{\frac{5}{4}}||_{L^\infty(x \ge 20)} ||\bar{u}_x x^{\frac{1}{2}}||_{L^2}^2 +  \eps^{\frac{n}{2}+\gamma} ||\bar{u}||_{L^\infty} ||u_x x^{\frac{1}{2}}||_{L^2} ||\bar{u}_{xx} x^{\frac{3}{2}}||_{L^2} \\
& \le \eps^{\frac{n}{2}+\gamma - \omega(N_i)} ||u, v||_Z ||\bar{u}, \bar{v}||_Z^2. 
\end{align}

Above, we have used that the support of $\rho_2$ is when $x \ge 50$. Next, 
\begin{align} \nonumber
\int \int \eps^{\frac{n}{2}+\gamma} &|\partial_x \Big( \bar{u} \bar{u}_x \Big)|  |u_{xx} | \rho_2^3 x^3 = \int \int \eps^{\frac{n}{2}+\gamma} | \{\bar{u}_x^2 + \bar{u} \bar{u}_{xx} \} | \cdot | u_{xx} | \rho_2^3 x^3 \\ \nonumber
& \le ||\bar{u}_x x^{\frac{5}{4}}||_{L^\infty(x \ge 20)} ||\bar{u}_x x^{\frac{1}{2}}||_{L^2} ||u_{xx} x^{\frac{3}{2}}||_{L^2} + ||\bar{u} x^{\frac{1}{4}} ||_{L^\infty} ||\bar{u}_{xx} x^{\frac{3}{2}}||_{L^2} ||u_{xx} x^{\frac{3}{2}}||_{L^2} \\
& \lesssim \eps^{\frac{n}{2}+\gamma - \omega(N_i)} ||u, v||_Z ||\bar{u}, \bar{v}||_Z^2. 
\end{align}

Next, the second nonlinearity in (\ref{w.f.2}): 
\begin{align} \nonumber
\int \int \eps^{\frac{n}{2}+\gamma} |\partial_x &\Big( \bar{v} u_y \Big)| |u_x | \rho_2^2 x^2 = \int \int \eps^{\frac{n}{2}+\gamma} |\{\bar{v}_x u_y + \bar{v} u_{xy} \}|  |u_x| \rho_2^2 x^2 \\ \n
& \le \eps^{\frac{n}{2}+\gamma} ||\bar{v}_x x^{\frac{3}{2}}||_{L^\infty(x \ge 20)} ||u_y ||_{L^2} ||u_x x^{\frac{1}{2}}||_{L^2}  + \eps^{\frac{n}{2}+\gamma} ||\bar{v} x^{\frac{1}{2}}||_{L^\infty} ||u_{xy} x||_{L^2} ||u_x x^{\frac{1}{2}}||_{L^2} \\
& \lesssim \eps^{\frac{n}{2}+\gamma - \omega(N_i)} ||\bar{u}, \bar{v}||_Z ||u,v||_Z^2. 
\end{align}

For this same nonlinearity: 
\begin{align} \nonumber
\int \int \eps^{\frac{n}{2}+\gamma}& |\{\bar{v} u_{xy} + \bar{v}_x u_y \}|  |u_{xx} | \rho_2^3 x^3 \\ \n
& \le \eps^{\frac{n}{2}+\gamma} ||\bar{v} x^{\frac{1}{2}}||_{L^\infty} ||u_{xy} x||_{L^2} ||u_{xx} x^{\frac{3}{2}}||_{L^2} + \eps^{\frac{n}{2}+\gamma} ||\bar{v}_x x^{\frac{3}{2}} ||_{L^\infty(x \ge 20)} ||u_y ||_{L^2} ||u_{xx} x^{\frac{3}{2}}||_{L^2} \\ \label{danger.zone}
&  \lesssim \eps^{\frac{n}{2}+\gamma - \omega(N_i)} ||\bar{u}, \bar{v}||_Z ||u,v||_Z^2. 
\end{align}

We'll now move to the $\mathcal{N}^u$ terms in $\mathcal{W}_3$ (see (\ref{calw3})), which are the most delicate. These terms are: 
\begin{align} \n
\int \int \p_{xx} \mathcal{N}^u &\Big[ |u_{xx} \rho_3^4 x^4 + |u_{xxx}|^2 \rho_3^5 x^5 \Big] \\ \label{w.f.4}
& = \int \int \eps^{\frac{n}{2}+\gamma} \p_{xx} \Big( \bar{u} \bar{u}_x + \bar{v} u_y \Big) \cdot \Big[ |u_{xx}| \rho_3^4 x^4 + |u_{xxx}|^2 \rho_3^5 x^5 \Big].
\end{align}

First, 
\begin{align} \nonumber
\int \int \eps^{\frac{n}{2}+\gamma} | \partial_{xx} &\Big( \bar{u} \bar{u}_x \Big) |  |u_{xx}| \rho_3^4 x^4 = \int \int \eps^{\frac{n}{2}+\gamma} | \{\bar{u} \bar{u}_{xxx} +3 \bar{u}_x \bar{u}_{xx} \} |  |u_{xx}| \rho_3^4 x^4 \\ \nonumber
& \le \eps^{\frac{n}{2}+\gamma}  ||\bar{u}||_{L^\infty} ||\bar{u}_{xxx} x^{\frac{5}{2}}||_{L^2(x \ge 20)} ||u_{xx} x^{\frac{3}{2}}||_{L^2} \\  \n
& \hspace{20 mm} + \eps^{\frac{n}{2}+\gamma} ||\bar{u}_x x^{\frac{5}{4}}||_{L^\infty(x \ge 20)} ||\bar{u}_{xx} x^{\frac{3}{2}}||_{L^2}  ||u_{xx} x^{\frac{3}{2}}||_{L^2} \\
& \lesssim \eps^{\frac{n}{2}+\gamma - \omega(N_i)} ||\bar{u}, \bar{v}||_Z^2 ||u,v||_Z.
\end{align}

Above, we have used that $\rho_3$ is supported in a strict subset of $\{x \ge 20\}$, which in turn is supported in a strict subset of $\zeta_3$, which appears in the norm $Z$ (see (\ref{zeta}) - (\ref{rho})). Referring back to (\ref{w.f.4}), for this same nonlinear term: 
\begin{align} \nonumber
\int \int \eps^{\frac{n}{2}+\gamma} &| \{ \bar{u} \bar{u}_{xxx} + 3\bar{u}_x \bar{u}_{xx} \} |  |u_{xxx}| \rho_3^5 x^5 \\  \n
& \le  \eps^{\frac{n}{2}+\gamma} ||\bar{u}||_{L^\infty} ||\bar{u}_{xxx} x^{\frac{5}{2}}||_{L^2(x \ge 20)} ||u_{xxx} x^{\frac{5}{2}}||_{L^2( x \ge 20)}  \\ \n
& \hspace{20 mm} + \eps^{\frac{n}{2}+\gamma}  ||\bar{u}_x x^{\frac{5}{4}}||_{L^\infty(x \ge 20)} ||\bar{u}_{xx} x^{\frac{3}{2}}||_{L^2}  ||u_{xxx} x^{\frac{5}{2}}||_{L^2(x \ge 20)} \\
& \le  \eps^{\frac{n}{2}+\gamma -\omega(N_i)}  ||\bar{u}, \bar{v}||_Z^2 ||u,v||_Z. 
\end{align}

We now move to the final nonlinearity contributed by $\mathcal{N}^u$, which is the $\partial_{xx} \{vu_y\}$ term from (\ref{w.f.4}). This term is the most delicate to control. First, expanding yields: 
\begin{align}
\int \int  \eps^{\frac{n}{2}+\gamma} | \partial_{xx} \Big( \bar{v} u_y \Big)| | u_{xx}| \rho_3^4 x^4 = \int \int   \eps^{\frac{n}{2}+\gamma} | \{ \bar{v} u_{xxy} + 2 \bar{v}_x u_{xy} + \bar{v}_{xx} u_y \} |  |u_{xx}| \rho_3^4 x^4 
\end{align}

First, 
\begin{align} \n
 \int \int  \eps^{\frac{n}{2}+\gamma} | \bar{v} u_{xxy}||u_{xx}| \rho_3^4 x^4 &\le  \eps^{\frac{n}{2}+\gamma} ||\bar{v} x^{\frac{1}{2}}||_{L^\infty} ||u_{xxy} x^2||_{L^2(x \ge 20)} ||u_{xx} x^{\frac{3}{2}}||_{L^2} \\
 & \le  \eps^{\frac{n}{2}+\gamma - \omega(N_i)} ||\bar{u}, \bar{v}||_Z ||u,v||_Z^2. 
\end{align}

Second, 
\begin{align} \n
\int \int  \eps^{\frac{n}{2}+\gamma} |\bar{v}_x u_{xy} | |u_{xx}| \rho_3^4 x^4 &\le  \eps^{\frac{n}{2}+\gamma} ||\bar{v}_x x^{\frac{3}{2}}||_{L^\infty(x \ge 20)} ||u_{xy} x||_{L^2} ||u_{xx} x^{\frac{3}{2}}||_{L^2} \\
& \lesssim   \eps^{\frac{n}{2}+\gamma -\omega(N_i)}  ||u,v||_Z^2 ||\bar{u}, \bar{v}||_Z. 
\end{align}

Finally, 
\begin{align} \n
\int \int  \eps^{\frac{n}{2}+\gamma} | \bar{v}_{xx} u_y || u_{xx}| \rho_3^4 x^4 \le & \eps^{\frac{n}{2}+\gamma}\Big[ \sup_{x \ge 20} ||u_y x^{\frac{1}{2}}||_{L^2_y}\Big] ||u_{xx} x^{\frac{3}{2}}||_{L^2} \Big[ \int_{x= 20} x^4 ||\bar{v}_{xx}||_{L^\infty_y}^2 \Big]^{\frac{1}{2}} \\ \label{TON.1}
& \lesssim  \eps^{\frac{n}{2}+\gamma - \omega(N_i)} ||\bar{u}, \bar{v}||_Z ||u,v||_Z^2. 
\end{align}

Turning back to (\ref{w.f.4}), for this same nonlinearity, it remains to treat: 
\begin{align}
\int \int  \eps^{\frac{n}{2}+\gamma} | \{\bar{v} u_{xxy} + 2 \bar{v}_x u_{xy} + \bar{v}_{xx} u_y \}|| u_{xxx} | \rho_3^5 x^5.
\end{align}

First, 
\begin{align} \n
\int \int  \eps^{\frac{n}{2}+\gamma} |\bar{v} u_{xxy}| |u_{xxx}| \rho_3^5 x^5  &\le \eps^{\frac{n}{2}+\gamma} ||\bar{v}x^{\frac{1}{2}}||_{L^\infty} ||u_{xxy} x^{2}||_{L^2(x \ge 20)} ||u_{xxx}x^{\frac{5}{2}}||_{L^2(x \ge 20)} \\
& \lesssim   \eps^{\frac{n}{2}+\gamma -\omega(N_i)}  ||\bar{u}, \bar{v}||_Z ||u,v||_Z^2.  
\end{align}

Second, 
\begin{align} \n
 \int \int  \eps^{\frac{n}{2}+\gamma} | \bar{v}_x u_{xy}| | u_{xxx} | \rho_3^5 x^5  &\le  \eps^{\frac{n}{2}+\gamma} ||\bar{v}_x x^{\frac{3}{2}}||_{L^\infty(x \ge 20)} ||u_{xy} x||_{L^2} ||u_{xxx} x^{\frac{5}{2}}||_{L^2(x \ge 20)} \\
 & \lesssim  \eps^{\frac{n}{2}+\gamma - \omega(N_i)} ||\bar{u}, \bar{v}||_Z ||u,v||_Z^2. 
\end{align}

Last, 
\begin{align}  \n
 \int \int \eps^{\frac{n}{2}+\gamma} &| \bar{v}_{xx} u_y| | u_{xxx}| \rho_3^5 x^5  \\ \n
 & \le  \eps^{\frac{n}{2}+\gamma} \sup_{x \ge 20} ||u_y x^{\frac{1}{2}}||_{L^2_y} ||u_{xxx} x^{\frac{5}{2}}||_{L^2(x \ge 20)}   \Big[\int_{x = 20}^\infty x^4 ||\bar{v}_{xx}||_{L^\infty}^2 \ud x \Big]^{\frac{1}{2}} \\ \label{TON.2}
 & \lesssim  \eps^{\frac{n}{2}+\gamma - \omega(N_i)} ||\bar{u}, \bar{v}||_Z ||u,v||_Z^2. 
\end{align}

We now move to the nonlinear terms, $\mathcal{N}^v$, from $g$, defined in (\ref{bar.g}), starting with $\mathcal{W}_1$, equation (\ref{calw1}): 
\begin{align}
\int \int \eps |\mathcal{N}^v| \Big[ |v| + |v_x| x \Big] = \int \int \eps \Big[ \eps^{\frac{n}{2}+\gamma} |\bar{u} \bar{v}_x| + |\bar{v} \bar{v}_y| \Big] \Big[ |v| + |v_x| x \Big].
\end{align}

First, 
\begin{align} \n
&\int \int \eps^{\frac{n}{2}+\gamma + 1} | \bar{u} \bar{v}_x |\cdot |v| \lesssim \eps^{\frac{n}{2}+\gamma + 1} ||\bar{u} x^{\frac{1}{4}}||_{L^\infty} ||\bar{v}_x x^{\frac{1}{2}}||_{L^2} ||x^{-\frac{3}{4}}v||_{L^2} \\ \n
& \hspace{30 mm} \lesssim \eps^{\frac{n}{2}+\gamma } ||\bar{u} x^{\frac{1}{4}}||_{L^\infty} ||\sqrt{\eps}\bar{v}_x x^{\frac{1}{2}}||_{L^2} ||\sqrt{\eps} v_x x^{\frac{1}{4}}||_{L^2}, \\ 
& \hspace{30 mm} \lesssim \eps^{\frac{n}{2}+\gamma - \omega(N_i)} ||\bar{u}, \bar{v}||_Z^2 ||u,v||_Z, \\ \n
& \int \int \eps^{\frac{n}{2}+\gamma + 1} | \bar{v} \bar{v}_y | \cdot |v | \lesssim \eps^{\frac{n}{2}+\gamma + 1} ||\bar{v} x^{\frac{1}{2}}||_{L^\infty} ||\bar{v}_y x^{\frac{1}{2}}||_{L^2} ||v_x||_{L^2} \\
& \hspace{30 mm} \lesssim \eps^{\frac{n}{2}+\gamma -\omega(N_i)} ||\bar{u}, \bar{v}||_Z^2 ||u,v||_Z. 
\end{align}

Next, 
\begin{align} \n
& \int \int \eps^{\frac{n}{2}+\gamma + 1} | \bar{u} \bar{v}_x | |v_x| x \le \eps^{\frac{n}{2}+\gamma + 1} ||\bar{u}||_{L^\infty} ||\bar{v}_x x^{\frac{1}{2}}||_{L^2} ||v_x x^{\frac{1}{2}}||_{L^2}, \\ 
& \hspace{30 mm} \lesssim \eps^{\frac{n}{2}+\gamma -\omega(N_i)} ||\bar{u}, \bar{v}||_Z^2 ||u,v||_Z,  \\ \n
& \int \int \eps^{\frac{n}{2}+\gamma + 1} |\bar{v} \bar{v}_y| | v_x | x  \le \eps^{\frac{n}{2}+\gamma+1} ||\bar{v} x^{\frac{1}{2}}||_{L^\infty} ||\bar{v}_y x^{\frac{1}{2}}||_{L^2} ||v_x x^{\frac{1}{2}}||_{L^2}, \\
&\hspace{30 mm} \lesssim \eps^{\frac{n}{2}+\gamma -\omega(N_i)} ||u,v||_Z ||\bar{u}, \bar{v}||_Z^2. 
\end{align}

We will now move to the nonlinear terms, $\mathcal{N}^v$, arising from $\mathcal{W}_2$, (see equation (\ref{calw2})), which are summarized here:  
\begin{align} \n
\int \int \eps |\p_x \mathcal{N}^v|& \Big[ |v_x| \rho_2^2 x^2 + |v_{xx}| \rho_2^3 x^3 \Big] \\ \label{w.f.5}
& = \int \int \eps^{\frac{n}{2}+\gamma + 1} |\p_x \Big[\bar{u} \bar{v}_x + \bar{v} \bar{v}_y  \Big]| \Big[ |v_x| \rho_2^2 x^2 + |v_{xx}| \rho_2^3 x^3 \Big] 
\end{align}

We will go through (\ref{w.f.5}) term by term, starting with: 
\begin{align} \nonumber
\int \int \eps^{\frac{n}{2}+\gamma + 1} |\partial_x &\Big( \bar{u} \bar{v}_x \Big)| | v_x | \rho_2^2 x^2 = \int \int \eps^{\frac{n}{2}+\gamma + 1} | \{ \bar{u}_x \bar{v}_x + \bar{u} \bar{v}_{xx} \} | \cdot | v_x | \rho_2^2 x^2 \\ \nonumber
& \le \eps^{\frac{n}{2}+\gamma + 1} ||\bar{u}_x  x^{\frac{5}{4}}||_{L^\infty(x \ge 20)} ||\bar{v}_x x^{\frac{1}{2}}||_{L^2} ||v_x x^{\frac{1}{2}}||_{L^2}  \\ \n
& \hspace{20 mm} + \eps^{\frac{n}{2}+\gamma + 1} ||\bar{u}||_{L^\infty} ||\bar{v}_{xx} x^{\frac{3}{2}}||_{L^2} ||v_x x^{\frac{1}{2}}||_{L^2}, \\
& \le \eps^{\frac{n}{2} + \gamma - \omega(N_i)} ||\bar{u}, \bar{v}||_Z^2 ||u,v||_Z. 
\end{align}

Staying with this nonlinearity, 
\begin{align} \nonumber
\int \int \eps^{\frac{n}{2}+\gamma + 1} &| \{ \bar{u}_x \bar{v}_x + \bar{u} \bar{v}_{xx} \} | \cdot | v_{xx} | \rho_2^3 x^3 \\ \n
&\le \eps^{\frac{n}{2}+\gamma + 1} ||\bar{u}_x x^{\frac{5}{4}}||_{L^\infty(x \ge 20)} ||\bar{v}_x x^{\frac{1}{2}}||_{L^2} ||v_{xx} x^{\frac{3}{2}}||_{L^2}  \\ \n
& \hspace{20 mm} + \eps^{\frac{n}{2}+\gamma + 1} ||\bar{u}||_{L^\infty} ||\bar{v}_{xx} x^{\frac{3}{2}}||_{L^2} ||v_{xx} x^{\frac{3}{2}}||_{L^2} \\
& \lesssim \eps^{\frac{n}{2}+\gamma - \omega(N_i)} ||\bar{u}, \bar{v}||_Z^2 ||u,v||_Z. 
\end{align}

We now move to the $vv_y$ nonlinearity, still in term (\ref{w.f.5}), which we expand:
\begin{align} \nonumber
 \int \int \eps^{\frac{n}{2}+\gamma + 1} | \partial_x &\Big( \bar{v} \bar{v}_y \Big)| | v_x | \rho_2^2 x^2  = \int \int \eps^{\frac{n}{2}+\gamma + 1} | \{\bar{v}_x \bar{v}_y + \bar{v} \bar{v}_{xy} \} | | v_x | \rho_2^2 x^2  \\ \n
& \le \eps^{\frac{n}{2}+\gamma + 1} ||\bar{v}_x x^{\frac{3}{2}}||_{L^\infty(x \ge 20)} ||\bar{v}_y x^{\frac{1}{2}} ||v_x x^{\frac{1}{2}}||_{L^2} \\ \n
& \hspace{20 mm} + \eps^{\frac{n}{2}+\gamma + 1} ||\bar{v} x^{\frac{1}{2}}||_{L^\infty} ||v_x x^{\frac{1}{2}}||_{L^2} ||\bar{v}_{xy} x^{\frac{3}{2}}||_{L^2} \\
& \lesssim \eps^{\frac{n}{2}+\gamma -\omega(N_i)} ||\bar{u}, \bar{v}||_Z^2 ||u,v||_Z. 
\end{align}

For this same nonlinear term: 
\begin{align} \nonumber
\int \int \eps^{\frac{n}{2}+\gamma + 1} &|\{ \bar{v}_x \bar{v}_y + \bar{v} \bar{v}_{xy} \}| \cdot |v_{xx}| \rho_2^3 x^3  \\ \n
& \le \eps^{\frac{n}{2}+\gamma + 1} ||\bar{v}_x x^{\frac{3}{2}}||_{L^\infty(x \ge 20)} ||\bar{v}_y x^{\frac{1}{2}}||_{L^2} ||v_{xx} x^{\frac{3}{2}}||_{L^2} \\ \n
&\hspace{20 mm} + \eps^{\frac{n}{2}+\gamma + 1} ||\bar{v} x^{\frac{1}{2}}||_{L^\infty} ||\bar{v}_{xy} x^{\frac{3}{2}}||_{L^2} ||v_{xx} x^{\frac{3}{2}}||_{L^2} \\
& \lesssim \eps^{\frac{n}{2}+\gamma -\omega(N_i)} ||u,v||_Z ||\bar{u}, \bar{v}||_Z^2. 
\end{align}

We now move to the highest-order terms, which we read from (\ref{calw3}), and summarize here: 
\begin{align} \n
\int \int \eps^{\frac{n}{2}+\gamma + 1}& |\p_{xx} \mathcal{N}^v| \Big[ |v_{xx}| \rho_3^4 x^4 + |v_{xxx}| \rho_3^5 x^5 \Big] \\ \label{w.f.6}
& = \int \int \eps^{\frac{n}{2}+\gamma + 1} | \p_{xx}\Big[ \bar{u} \bar{v}_x + \bar{v} \bar{v}_y \Big] | \Big[ |v_{xx}| \rho_3^4 x^4 + |v_{xxx}| \rho_3^5 x^5 \Big].
\end{align}

We shall expand: 
\begin{align} \n
 \int \int \eps^{\frac{n}{2}+\gamma + 1} &| \partial_{xx} \Big( \bar{u} \bar{v}_x \Big)| \cdot | v_{xx} | \rho_3^4 x^4  \\ 
 & =  \int \int \eps^{\frac{n}{2}+\gamma + 1} | \{ \bar{u} \bar{v}_{xxx} + 2 \bar{u}_x \bar{v}_{xx} + \bar{u}_{xx} \bar{v}_x \} | \cdot | v_{xx} | \rho_3^4 x^4. 
\end{align}

First, 
\begin{align} \n
\int \int \eps^{\frac{n}{2}+\gamma + 1} | \bar{u} \bar{v}_{xxx} | \cdot |v_{xx}| \rho_3^4 x^4  &\le \eps^{\frac{n}{2}+\gamma + 1} ||\bar{u}||_{L^\infty} ||\bar{v}_{xxx} x^{\frac{5}{2}}||_{L^2(x \ge 20)} ||v_{xx} x^{\frac{3}{2}}||_{L^2} \\
& \lesssim \eps^{\frac{n}{2}+\gamma -\omega(N_i)} ||\bar{u}, \bar{v}||_Z^2 ||u,v||_Z. 
\end{align}

Next, 
\begin{align} \n
\int \int \eps^{\frac{n}{2}+\gamma + 1} | \bar{u}_x \bar{v}_{xx} | \cdot |v_{xx}| \rho_3^4 x^4  &\le \eps^{\frac{n}{2}+\gamma + 1} ||\bar{u}_x  x^{\frac{5}{4}}||_{L^\infty(x \ge 20)} ||\bar{v}_{xx} x^{\frac{3}{2}}||_{L^2} ||v_{xx} x^{\frac{3}{2}}||_{L^2} \\
& \le \eps^{\frac{n}{2}+\gamma -\omega(N_i)} ||\bar{u}, \bar{v}||_Z^2 ||u,v||_Z. 
\end{align}

Third, 
\begin{align} \n
\int \int \eps^{\frac{n}{2}+\gamma + 1} | \bar{v}_x \bar{u}_{xx} | | v_{xx} | \rho_3^4 x^4  &\le \eps^{\frac{n}{2}+\gamma + 1} ||\bar{v}_x x^{\frac{3}{2}} ||_{L^\infty(x \ge 20)} ||\bar{u}_{xx}x^{\frac{3}{2}} ||_{L^2} ||\bar{v}_{xx}x^{\frac{3}{2}}||_{L^2} \\
& \lesssim \eps^{\frac{n}{2}+ \gamma - \omega(N_i)} ||\bar{u}, \bar{v}||_Z^2 ||u, v||_Z. 
\end{align}

Turning back to (\ref{w.f.6}), for this same nonlinearity, we will now treat:
\begin{align}
\int \int \eps^{\frac{n}{2}+\gamma + 1} |\{ \bar{u} \bar{v}_{xxx} + 2 \bar{u}_x \bar{v}_{xx} + \bar{u}_{xx} \bar{v}_x \}| |v_{xxx}| \rho_3^5 x^5
\end{align}

First, 
\begin{align} \n
\int \int \eps^{\frac{n}{2}+\gamma + 1} |\bar{u} \bar{v}_{xxx}| |v_{xxx}| \rho_3^5 x^5 &\le \eps^{\frac{n}{2}+\gamma + 1} ||\bar{u}||_{L^\infty} ||\bar{v}_{xxx} x^{\frac{5}{2}}||_{L^2(x \ge 20)} ||v_{xxx} x^{\frac{5}{2}}||_{L^2(x \ge 20)} \\
& \lesssim \eps^{\frac{n}{2}+\gamma - \omega(N_i)} ||\bar{u}, \bar{v}||_Z^2 ||u,v||_Z. 
\end{align}

Second, 
\begin{align} \n
\int \int \eps^{\frac{n}{2}+\gamma + 1} | \bar{u}_x \bar{v}_{xx}| |v_{xxx}| \rho_3^5 x^5  &\le \eps^{\frac{n}{2}+\gamma + 1} ||\bar{u}_x  x^{\frac{5}{4}}||_{L^\infty(x \ge 20)} ||\bar{v}_{xx} x^{\frac{3}{2}}||_{L^2} ||v_{xxx} x^{\frac{5}{2}}||_{L^2(x \ge 20)} \\
& \lesssim \eps^{\frac{n}{2}+\gamma - \omega(N_i)} ||\bar{u}, \bar{v}||_Z^2 ||u,v||_Z.
\end{align}

Third, 
\begin{align} \n
\int \int \eps^{\frac{n}{2}+\gamma + 1} | \bar{u}_{xx} \bar{v}_x| | v_{xxx} | \rho_3^5 x^5  &\le \eps^{\frac{n}{2}+\gamma + 1} ||\bar{v}_x  x^{\frac{3}{2}}||_{L^\infty(x \ge 20)} ||\bar{u}_{xx} x^{\frac{3}{2}}||_{L^2} ||v_{xxx} x^{\frac{5}{2}}||_{L^2(x \ge 20)} \\
& \lesssim \eps^{\frac{n}{2}+\gamma - \omega(N_i)} ||\bar{u}, \bar{v}||_Z^2 ||u,v||_Z.
\end{align}

Turning to (\ref{w.f.6}), we now approach the final nonlinear term in $g$ (see definition (\ref{bar.g})), which is the $\bar{v} \bar{v}_y$ term. First, we will expand: 
\begin{align} \n
\int \int \eps^{\frac{n}{2}+\gamma + 1} &|\partial_{xx} \Big( \bar{v} \bar{v}_y \Big)| \cdot |v_{xx}| \rho_3^4 x^4 \\
& = \int \int \eps^{\frac{n}{2}+\gamma + 1} |\{\bar{v}_{xx} \bar{v}_y + \bar{v}_x \bar{v}_{xy} + \bar{v} \bar{v}_{xxy} \} |\cdot |v_{xx}| \rho_3^4 x^4.
\end{align}

First, 
\begin{align} \n
 \int \int \eps^{\frac{n}{2}+\gamma + 1} | \bar{v}_{xx} \bar{v}_y| \cdot |v_{xx}| \rho_3^4 x^4 &\le \eps^{\frac{n}{2}+\gamma + 1} ||\bar{u}_x x^{\frac{5}{4}}||_{L^\infty(x \ge 20)} || \bar{v}_{xx} x^{\frac{3}{2}}||_{L^2} ||v_{xx} x^{\frac{3}{2}}||_{L^2} \\
 & \lesssim \eps^{\frac{n}{2}+\gamma - \omega(N_i)} ||\bar{u}, \bar{v}||_Z^2 ||u,v||_Z. 
\end{align}

Second, 
\begin{align} \n
\int \int \eps^{\frac{n}{2}+\gamma + 1} | \bar{v}_x \bar{v}_{xy}| |v_{xx} | \rho_3^4 x^4  &\le \eps^{\frac{n}{2}+\gamma + 1} ||\bar{v}_x x^{\frac{3}{2}}||_{L^\infty(x \ge 20)} ||\bar{v}_{xy} x^{\frac{3}{2}}||_{L^2} || \bar{v}_{xx} x^{\frac{3}{2}}||_{L^2} \\
& \lesssim \eps^{\frac{n}{2}+\gamma - \omega(N_i)} ||\bar{u}, \bar{v}||_Z^2 ||u,v||_Z. 
\end{align}

Third, 
\begin{align} \n
 \int \int \eps^{\frac{n}{2}+\gamma + 1} | \bar{v} \bar{v}_{xxy}| |v_{xx}| \rho_3^4 x^4 &\le \eps^{\frac{n}{2}+\gamma + 1} ||\bar{v}x^{\frac{1}{2}}||_{L^\infty} ||\bar{v}_{xxy} x^{\frac{5}{2}}||_{L^2(x \ge 20)} ||v_{xx} x^{\frac{3}{2}}||_{L^2} \\
 & \lesssim \eps^{\frac{n}{2}+\gamma - \omega(N_i)} ||\bar{u}, \bar{v}||_Z^2 ||u,v||_Z. 
\end{align}

Again turning to (\ref{w.f.6}), for this same nonlinearity, we must also treat: 
\begin{align}
\int \int \eps^{\frac{n}{2}+\gamma + 1}  |\{\bar{v}_{xx} \bar{v}_y + \bar{v}_x \bar{v}_{xy} + \bar{v} \bar{v}_{xxy} \}| \cdot |v_{xxx}| \rho_3^5 x^5.
\end{align}

First, 
\begin{align} \n
 \int \int \eps^{\frac{n}{2}+\gamma + 1} &|\bar{v}_{xx} \bar{v}_y| |v_{xxx}| \rho_3^5 x^5 \\ \n
 &  \le \eps^{\frac{n}{2}+\gamma + 1} ||\bar{u}_x  x^{\frac{5}{4}}||_{L^\infty(x \ge 20)} ||\bar{v}_{xx} x^{\frac{3}{2}}||_{L^2} ||v_{xxx} x^{\frac{5}{2}}||_{L^2(x \ge 20)} \\
 & \le \eps^{\frac{n}{2}+\gamma - \omega(N_i)} ||\bar{u}, \bar{v}||_Z^2 ||u,v||_Z.  
\end{align}

Second, 
\begin{align} \n
 \int \int  \eps^{\frac{n}{2}+\gamma + 1} &|\bar{v}_x \bar{v}_{xy}| |v_{xxx}| \rho_3^5 x^5 \\ \n
 & \le  \eps^{\frac{n}{2}+\gamma + 1}  ||\bar{v}_x x^{\frac{3}{2}}||_{L^\infty(x \ge 20)} ||\bar{v}_{xy} x^{\frac{3}{2}}||_{L^2} ||v_{xxx} x^{\frac{5}{2}}||_{L^2(x \ge 20)} \\
 & \lesssim  \eps^{\frac{n}{2}+\gamma -\omega(N_i)}||\bar{u}, \bar{v}||_Z^2 ||u,v||_Z. 
\end{align}

Finally, 
\begin{align} \n
\int \int  \eps^{\frac{n}{2}+\gamma + 1} &|\bar{v} \bar{v}_{xxy}| |v_{xxx}| \rho_3^5 x^5  \\ \n
& \le  \eps^{\frac{n}{2}+\gamma + 1} ||\bar{v} x^{\frac{1}{2}}||_{L^\infty} ||\bar{v}_{xxy} x^{\frac{5}{2}}||_{L^2(x \ge 20)} ||v_{xxx} x^{\frac{5}{2}}||_{L^2( x \ge 20)} \\
& \lesssim \eps^{\frac{n}{2}+\gamma -\omega(N_i)} ||\bar{u}, \bar{v}||_Z^2 ||u,v||_Z. 
\end{align}

This now concludes all of the nonlinear terms in $\mathcal{W}_i$. The final task is to control the $R^{u,n}, R^{v,n}$ terms in $f,g$. Via Lemma \ref{Lemma FE}, for $\kappa$ arbitrarily small and $\sigma_n$ as in (\ref{sigma.i}), we choose now $\sigma' = \frac{1}{10,000}$. Then, we have: 
\begin{align} \nonumber
\Big| \int \int \epsilon^{-\frac{n}{2}-\gamma}  R^{u,n}u &+ \int \int \epsilon^{-\frac{n}{2}-\gamma}  R^{u,n} v_y x \Big| \lesssim \epsilon^{-\frac{n}{2}-\gamma}|| R^{u,n} x^{\frac{1}{2}+\sigma'}||_{L^2} ||x^{-\frac{1}{2}-\sigma'} u, x^{\frac{1}{2}}v_y||_{L^2} \\ \label{sighardy.1}
&\le C(n) \epsilon^{-\frac{n}{2}-\gamma} || ||R^{u,n}||_{L^2_y} x^{\frac{1}{2}+\sigma'} ||_{L^2_x} ||u,v||_{X_1} \\  \label{sighardy.2}
&\le C(  n) \epsilon^{\frac{1}{4}-\gamma-\kappa} ||x^{-\frac{3}{4}+2\sigma_n + \sigma' + \kappa}||_{L^2_x} ||u,v||_{X_1} \\
& \le C( n) \epsilon^{\frac{1}{4}-\gamma-\kappa}||u,v||_{X_1}.
\end{align}

In (\ref{sighardy.1}), we have used the Hardy inequality with power $x^{-\frac{1}{2}-\sigma}$, which is admissible as $u(1,y) = 0$. Upon citing (\ref{sigma.i}), one has: 
\begin{align} \label{choice}
|| x^{-\frac{3}{4}+2\sigma_n + \sigma' + \kappa}||_{L^2_x}|| = || x^{-\frac{3}{4}+\frac{2}{10,000} + \frac{1}{10,000} + \kappa}||_{L^2_x} < \infty. 
\end{align}

Next, again via  Lemma \ref{Lemma FE}, we have: 
\begin{align} \nonumber
\int \int \epsilon^{1- \frac{n}{2}-\gamma} &| R^{v,n}| \Big( |v| + |v_x| x\Big) \le \epsilon^{-\frac{n}{2}-\gamma} ||\sqrt{\epsilon} R^{v,n} x^{\frac{1}{2}+\sigma'} ||_{L^2} ||x^{-\frac{1}{2}-\sigma'} \sqrt{\epsilon}v, x^{\frac{1}{2}} \sqrt{\epsilon} v_x||_{L^2} \\
& \lesssim \epsilon^{\frac{1}{4}-\gamma -\kappa} ||u,v||_{X_1}. 
\end{align}

Summarizing these forcing terms: 
\begin{align}\nonumber
\epsilon^{-\frac{n}{2}-\gamma}  \Big| \int \int R^{u,n}u &+ \epsilon R^{v,n} v + R^{u,n} v_y x + \epsilon  R^{v,n} v_x x \Big| \\  \label{forcing.ultimate}
&\le  \epsilon^{\frac{1}{4}-\gamma -\kappa } + \epsilon^{\frac{1}{4}-\gamma -\kappa } ||u,v||_{X_1}^2. 
\end{align}

Similarly, for higher order terms, 
\begin{align} \nonumber
 \int \int \epsilon^{-\frac{n}{2}-\gamma}  &|R^{u,n}_x| \{|u_x| \rho_2^2 x^2 + |u_{xx}| \rho_2^3 x^3  \}  \\ \n
&\lesssim \epsilon^{\frac{1}{4}-\gamma-\kappa} ||x^{-1-\frac{5}{4}+\sigma'+\kappa_0} x^{\frac{3}{2}}||_{L^2_x} ||u_x x^{\frac{1}{2}}, u_{xx} \rho_2^{\frac{3}{2}} x^{\frac{3}{2}}||_{L^2} \\
& \lesssim \eps^{\frac{1}{4}-\gamma - \kappa} ||u,v||_{X_1 \cap X_2}. 
\end{align}

and 
\begin{align} \n
\int \int \epsilon^{-\frac{n}{2}-\gamma} &|R^{u,n}_{xx}| \{|u_{xx}| \rho_3^4 x^4 + |u_{xxx}| \rho_3^5 x^5 \} \\ \n
& \lesssim \epsilon^{\frac{1}{4}-\gamma - \kappa} ||x^{-2-\frac{5}{4}+\sigma' + \kappa} x^{\frac{5}{2}}||_{L^2_x} ||u_{xx} \rho_3^{\frac{3}{2}} x^{\frac{3}{2}}, u_{xxx} \rho_3^{\frac{5}{2}} x^{\frac{5}{2}}||_{L^2} \\
& \lesssim \epsilon^{\frac{1}{4}-\gamma - \kappa} ||u,v||_{X_1 \cap X_2 \cap X_3}. 
\end{align}

For the terms from $g$: 
\begin{align} \nonumber
\epsilon^{-\frac{n}{2}-\gamma}& \int \int \epsilon |R^{v,n}_x| \{ |v_x| \rho_2^2 x^2 + |v_{xx}| \rho_2^3 x^3 \} \\ \n 
\lesssim &\epsilon^{\frac{1}{4}-\gamma - \kappa} || x^{-1-\frac{5}{4}+\sigma' + \kappa} x^{\frac{3}{2}}||_{L^2_x} ||\sqrt{\epsilon} v_x x^{\frac{1}{2}}, \sqrt{\epsilon}v_{xx}\rho_2^{\frac{3}{2}} x^{\frac{3}{2}}||_{L^2} \\
 \lesssim &\eps^{\frac{1}{4}-\gamma - \kappa} ||u,v||_{X_1 \cap X_2}. 
\end{align}

and
\begin{align} \nonumber
\epsilon^{-\frac{n}{2}-\gamma} & \int \int \epsilon |R^{v,n}_{xx}| \{|v_{xx}| \rho_3^4 x^4 + |v_{xxx}| \rho_3^5 x^5 \}  \\ \n
&\lesssim \epsilon^{\frac{1}{4}-\gamma-\kappa} ||x^{-2-\frac{5}{4}+\sigma' + \kappa} x^{\frac{5}{2}}||_{L^2_x} ||\sqrt{\epsilon} v_{xx} \rho_3^{\frac{3}{2}} x^{\frac{3}{2}},\sqrt{\epsilon} v_{xxx} \rho_3^{\frac{5}{2}}  x^{\frac{5}{2}}||_{L^2} \\
& \lesssim \epsilon^{\frac{1}{4}-\gamma-\kappa} ||u,v||_{X_1 \cap X_2 \cap X_3}. 
\end{align}

Combining all of the estimates we have established, we have proven Lemma \ref{LemmaW}. 

\end{proof}

We may now prove the main result, Theorem \ref{thm.m.2}:
\begin{proof}[Proof of Theorem \ref{thm.m.2}]

The starting point is estimate (\ref{Z.driver.N}):
\begin{align}
 ||u,v||_Z^2  &\lesssim \epsilon^{100} + \epsilon^{\frac{n}{2}+\gamma - \omega(N_i)}||\bar{u}, \bar{v}||_Z^4 + ||u,v||_{X_1 \cap X_2 \cap X_3}^2    \\
&\lesssim  \epsilon^{100} + \epsilon^{\frac{n}{2}+\gamma - \omega(N_i)}||\bar{u}, \bar{v}||_Z^4 + \Big| \mathcal{W}_1 + \mathcal{W}_2 + \mathcal{W}_3 \Big| \\ \nonumber
& \lesssim  \epsilon^{100} + \epsilon^{\frac{n}{2}+\gamma - \omega(N_i)}||\bar{u}, \bar{v}||_Z^4 + \epsilon^{\frac{1}{4}-\gamma - \kappa} + \epsilon^{\frac{1}{4}-\gamma - \kappa} ||u,v||_{X_1 \cap X_2 \cap X_3}^2 \\ 
& \hspace{12 mm} + \epsilon^{\frac{n}{2}-\omega(N_i)} ||u,v||_{Z}^2+ \epsilon^{\frac{n}{2}-\omega(N_i)} ||\bar{u},\bar{v}||_{Z}^4 ,
\end{align}

so absorbing the $||u,v||$ terms to the left-hand side gives: 
\begin{align} \label{final.estimate}
||u,v||_{Z}^2 & \lesssim \epsilon^{\frac{1}{4}-\gamma - \kappa} + \epsilon^{\frac{n}{2}-\omega(N_i)} \Big( ||\bar{u}, \bar{v}||_{Z}^4 \Big). 
\end{align}
\end{proof}

\break

\part*{\centerline{Chapter III: Existence and Uniqueness}}
\addcontentsline{toc}{part}{Chapter III: Existence and Uniqueness}

\section{Overview of Results}

With the main estimate, (\ref{final.estimate}), in hand, we will prove existence and uniqueness of solutions to the nonlinear system specified in (\ref{EQ.NSR.1}) - (\ref{EQ.NSR.3}), which is defined on the domain $\Omega$, with boundary conditions given in (\ref{nsr.bc.1}), and $f, g$ as in (\ref{defn.SU.SV}). The main result of this chapter is: 
\begin{theorem} \label{thm.e.u} For $\eps, \delta$ sufficiently small, $\eps << \delta$, $\kappa > 0$ small, and $0 \le \gamma < \frac{1}{4}$,  there exists a unique solution $[u,v] \in Z(\Omega)$ to the system (\ref{EQ.NSR.1}) - (\ref{EQ.NSR.3}), (\ref{nsr.bc.1}), (\ref{defn.SU.SV}) satisfying the bound:
\begin{align} \label{W.Z.2.orig}
||u,v||_{Z(\Omega)} \lesssim C(u_R, v_R) \eps^{\frac{1}{4}-\gamma - \kappa}. 
\end{align}
\end{theorem}

The main result, Theorem \ref{thm.m.1} follows immediately from Theorem \ref{thm.e.u}. The proof of this theorem proceeds in several steps, which we now outline:
\begin{itemize}
\item[(Step 1)] Linear existence of solutions to weighted Stokes system, defined as follows: 
\begin{align} \label{Sa.1}
&\Delta_\epsilon^2 \psi + \alpha A(\psi)  = F_y - \epsilon  G_x \text{ on } \Omega^N, \hspace{3 mm} F, G \in L^2(\Omega^N),\\ \label{psiBC}
&\psi|_{y = 0, N} = \psi_y|_{y = 0, N} = 0, \text{ and } \psi|_{x = 1} = \psi_x|_{x = 1} = 0, \\ \label{psiBC.2}
& \lim_{x \rightarrow \infty} [\psi_{x}, \psi_y] = 0. 
\end{align}

where $\alpha > 0$, and 
\begin{align} \n
A(\psi) = \Big[ \psi x^{2m} - \psi_{yy} x^{2m+2} &- \partial_x(\psi_x x^{2m+2}) + \psi_{yyyy} x^{2m+4} \\ \label{Aofpsi}
& + \partial_x \Big( \psi_{yyx} x^{2m+4} \Big) + \partial_{xx} \Big( \psi_{xx} x^{2m+4} \Big) \Big].
\end{align}

Here, $m > 0$ is sufficiently large, and can remain temporarily unspecified. The scaled Bilaplacian is defined as $\Delta_\epsilon^2 := \partial_y^4 + \epsilon \partial_y^2 \partial_x^2 + \epsilon^2 \partial_x^4$. The right-hand sides, $F, G$, should be thought of as generic elements satisfying $F_y - \eps G_x \in H^{-1}$. Upon introducing appropriate function spaces, we define the weak formulation of (\ref{Sa.1}) - (\ref{psiBC}) in (\ref{weak.1}). Depicting the weak-solution operator to the above system by $S_\alpha^{-1}$ (see \ref{S.alph.defn} for a precise definition), Step 1 amounts to studying the solvability of $S_\alpha \psi = F_y - \eps G_x$. 

The boundary conditions as $x \rightarrow \infty$ in (\ref{psiBC.2}) are selected in order to be consistent with (\ref{nsr.bc.1}). However, due to the terms in $A(\psi)$, the weak solution, $[\psi, u, v]$ exhibits rapid decay as $x \rightarrow \infty$. 

\item[(Step 2)] Linear existence of compact perturbations to $S_\alpha$. Define the maps:
\begin{align} \nonumber
T[\psi] &:= \partial_y \Big[ -u_R \psi_{xy} - u_{Rx}\psi_y - (v_R + \eps^{\frac{n}{2}+\gamma} \bar{v}) \psi_{yy} + u_{Ry} \psi_x  \Big] \\  \label{mapT}
& \hspace{42 mm} - \eps \partial_x \Big[ u_R \psi_{xx} - v_{Ry} \psi_y + v_R \psi_{xy} + v_{Ry} \psi_x \Big], \\ \label{mapT0}
T_0[\psi] &:= T[\psi] + \eps^{\frac{n}{2}+\gamma} \p_y [ \bar{v} \psi_{yy}], \\ \label{mapTa}
T_a[\psi] &:= \Big[ -u_R \psi_{xy} - u_{Rx}\psi_y - v_R \psi_{yy} + u_{Ry} \psi_x  \Big], \\ \label{mapTb}
T_b[\psi] &:= \Big[ u_R \psi_{xx} - v_{Ry} \psi_y + v_R \psi_{xy} + v_{Ry} \psi_x \Big].
\end{align}

$T$ has a dependence on $\bar{v}$, so to be precise we will sometimes write $T[\psi; \bar{v}]$. When there is no danger of confusion, we simply write $T[\psi]$. The map $T_0[\psi]$ is defined to match the profile terms, $S_u(u,v), S_v(u,v)$ (see the definition in (\ref{defn.Su})), when they are written in terms of the stream function, $\psi$. We have defined the notation $T_a, T_b$ so that we can write $T_0 = \p_y T_a - \eps \p_x T_b$. In this step, we are interested in establishing solvability of the system:
\begin{align}\label{sysT}
&S_\alpha \psi + T[\psi] = F_y - \eps G_x \text{ on } \Omega^N, \\ 
&[\psi = \psi_{x}]|_{x = 1} = [\psi = \psi_{y}]|_{y = 0} =  [\psi = \psi_{y}]|_{y = N} = \lim_{x \rightarrow \infty} [\psi_x, \psi_y] = 0. 
\end{align}

The essence of the arguments in this step is that upon applying $S_\alpha^{-1}$ to both sides above, $S^{-1}_\alpha T$ is seen as a compact perturbation of the identity. Despite $\Omega^N$ being unbounded in the $x$-direction, the required compactness arises from the weights, $w$, present in $A(\psi)$ above in (\ref{Aofpsi}).  The solution of (\ref{sysT}) is known to decay rapidly as $x \rightarrow \infty$, due to the presence of $A(\psi)$. This is captured in estimate (\ref{welldefined.}).

\item[(Step 3)] Nonlinear existence of auxiliary system: we first invite the reader to refer back to (\ref{defn.SU.SV}) and (\ref{bar.f}) - (\ref{bar.g}) for the definitions of $f$ and $g$. Given this and the definition of $T$ in (\ref{mapT}), we define: 
\begin{align} \label{tildef}
\tilde{f}(\bar{u}, \bar{v}):= \eps^{-\frac{n}{2}-\gamma} R^{u,n} + \eps^{\frac{n}{2}+\gamma} \bar{u} \bar{u}_x, \hspace{3 mm} \text{so that } f(u, \bar{u}, \bar{v}) = \tilde{f}(\bar{u}, \bar{v}) + \eps^{\frac{n}{2}+\gamma} \bar{v}u_y.
\end{align}

The aim of this step is to obtain existence of solutions (which we now index by $\alpha$ and $N$ for clarity) to the nonlinear system:
\begin{align} \label{sysWFD}
S_\alpha \psi^{\alpha, N} + T[\psi^{\alpha, N}; v^{\alpha, N}] = \tilde{f}_y(u^{\alpha, N}, v^{\alpha, N}) + \eps g_x(u^{\alpha, N},v^{\alpha, N}) \text{ on } \Omega^N.
\end{align}

This existence is obtained in the unit ball of $Z(\Omega^N)$ via Schaefer's fixed point theorem. 

\item[(Step 4)] Nonlinear existence of solutions to the system (\ref{EQ.NSR.1}) - (\ref{EQ.NSR.3}), with $f, g$ as in (\ref{defn.SU.SV}): By re-applying the analyses in Sections \ref{Section.Z} - \ref{section.NSR.Linear} and in Lemma \ref{LemmaW}, one obtains the uniform-in-$(\alpha,N)$ estimate: $||u^{\alpha, N}, v^{\alpha, N} ||_{Z(\Omega^N)} \lesssim \mathcal{O}(\delta) \eps^{\frac{1}{4}-\gamma - \kappa}$, which then enables the passage to weak limits in the space $X_1 \cap X_2 \cap X_3$. The weak limit is denoted by $[u,v]$, and is demonstrated to satisfy a weak formulation of system (\ref{EQ.NSR.1}) - (\ref{EQ.NSR.3}), see (\ref{weak.5.1}) for this formulation. Moreover, $[u,v] \in X_1 \cap X_2 \cap X_3$, gives enough regularity to upgrade immediately to a strong solution of (\ref{EQ.NSR.1}) - (\ref{EQ.NSR.3}). 

\begin{remark} To establish existence, we rely on compactness methods as opposed to applying a contraction mapping. The essential reason for this is seen by examining calculation (\ref{order.NL}), in which the structure is not preserved under taking differences.
\end{remark}

\begin{remark} It is important to establish nonlinear existence of the auxiliary system before establishing nonlinear existence of the system (\ref{EQ.NSR.1}) - (\ref{EQ.NSR.3}), as opposed to jumping from linear existence of (\ref{EQ.NSR.1}) - (\ref{EQ.NSR.3}) to nonlinear existence because the compactness methods we rely on require the weights from $\alpha A(\psi)$. 
\end{remark}

\item[(Step 5)] Nonlinear uniqueness for solutions to the system (\ref{EQ.NSR.1}) - (\ref{EQ.NSR.3}), with $f, g$ as in (\ref{defn.SU.SV}): In order to prove uniqueness, we re-apply the estimates in Sections \ref{Section.Z} - \ref{section.NSR.Linear} with weights that are weaker by $x^{-b}$, where $b < 1$, but is arbitrarily close to $1$. This step is necessary (with the weaker weight) due again to the calculation in (\ref{order.NL}), whose structure is destroyed upon considering differences. 

\end{itemize}

\section{Step 1: Invertibility of Weighted Stokes Operator, $S_\alpha$}

In this step, we study the system (\ref{Sa.1}) - (\ref{psiBC}).  We remind the reader that still, all integrations and all norms are taken over $\Omega^N$ unless otherwise specified. There is an abuse of notation here; $\psi$ should be indexed by $\alpha$ and $N$, but this will not cause any confusion for this step, as we view both $\alpha$ and $N$ as fixed. Our intention of this subsection is to exhibit solvability of the system (\ref{Sa.1}) in the space $Z(\Omega^N)$. Denote by $\chi_1(x)$ a cut-off function satisfying (refer to (\ref{zeta}) for the definition of $\zeta_3$):  
\begin{equation} \label{chi1}
\chi_1 = 1 \text{ on } x \ge \frac{12}{10}, \hspace{3 mm} \chi_1 = 0 \text{ for } 1 \le x \le \frac{11}{10}. 
\end{equation}

We define higher-order cut-offs similar to (\ref{chi1}), satisfying the following property: $\text{support } \chi_k \subset \{ \chi_{k-1} = 1 \}$. Define the following auxiliary norms via:  
\begin{align} \label{normHkw}
&||\psi||_{H^2_w}^2 := \int \int \psi^2 x^{2m} + |\nabla \psi|^2 x^{2m+2} + |\nabla^2 \psi|^2 x^{2m+4} \\
&||\psi||_{H^3_w}^2 := ||\psi||_{H^2_w} + \int \int \Big| \nabla^2 \psi_x \Big|^2 x^{2m+4}, \\
&||\psi||_{G^k_{w, B}}^2 := ||\psi||_{H^k_w}^2 + \int \int_B \Big| \partial_y^k \psi \Big|^2, \text{ for any bounded subset $B \subset \Omega^N$,} k = 0,...3, \\
&||\psi||_{H^k_w}^2 := ||\psi||_{H^3_w}^2 + \int \int \chi_k^2 \Big| \nabla^2 \partial_x^{k-2} \psi \Big|^2 x^{2m+4}, \text{ for $k \ge 4$. } 
\end{align}

We will also call $G^k_{w, loc}(\Omega^N)$ the space such that $||\psi||_{G^k_{w,B}} \le C(B)$ for all compact subsets $B$. Define the weak formulation of (\ref{Sa.1}) to be: 
\begin{align} \nonumber
\int \int \nabla_\epsilon^2 \psi : \nabla_\epsilon^2 \phi &+ \alpha \Big[ \int \int \psi \phi x^{2m} +  \int \int \nabla \psi \cdot \nabla \phi x^{2m+2} + \int \int \nabla^2 \psi : \nabla^2 \phi x^{2m+4} \Big]  \\ \label{weak.1}
&=\langle F_y - \epsilon G_x, \phi \rangle_{H^{-1}, H^1}  \text{ for all } \phi \in C_0^{\infty}(\Omega^N), \text{ where } \psi \in H^2_w(\Omega^N). 
\end{align}

Above, $\nabla^2$ is the Hessian matrix, and the inner product between two matrices is given by $A:B = \text{trace}(AB)$. We will need one more norm: 
\begin{align} \label{hmok}
||(F, G)||_{H^{-1}_k} := \sum_{j = 0}^k ||\p_x^j \{F_y - \eps G_x \}||_{H^{-1}}. 
\end{align}

Relevant spaces are defined here: 
\begin{definition} \label{defn.dens}
$H^2_w(\Omega^N)$ is defined to be the closure of $C_0^{\infty}(\Omega^N)$ under the norm $||\cdot||_{H^2_w}$. $H^k_w(\Omega^N)$ for $k \ge 3$ consists of the subspace of $H^2_w(\Omega^N)$ whose $H^k_w(\Omega^N)$ norm is finite. Note that $H^3_w(\Omega^N)$ does not contain all of the third derivatives of $\psi$; it is missing $\partial_y^3 \psi$, which is the reason for the norm, $||\cdot||_{G_{w, B}}$.
\end{definition}

\begin{remark}
There is a distinction between $H^2_w(\Omega^N)$, and $H^k_w(\Omega^N)$ in that:
\begin{align}
H^2_w(\Omega^N) = \overline{C_0^\infty(\Omega^N)}^{||\cdot||_{H^2_w}} \text{ but for $k \ge 3$, } H^3_w(\Omega^N) \xcancel{=} \overline{C^\infty_0(\Omega^N)}^{||\cdot||_{H^3_w}}. 
\end{align}
Due to the weights, there is no ``$H = W$" theorem generically for $H^k_w(\Omega^N)$. 
\end{remark}

\begin{lemma} \label{LBC10} For $\psi \in H^2_w(\Omega^N)$, the following boundary conditions are satisfied: 
\begin{align}
\psi|_{y = 0, N} = \psi_y|_{y = 0, N} = \psi|_{x = 1} = \psi_x|_{x = 1} = 0
\end{align}
\end{lemma}
\begin{proof}
If $\psi \in H^2_w(\Omega^N)$, obtain a sequence $\phi^{(n)}$ such that $||\phi^{(n)} - \psi||_{H^2_w} \rightarrow 0$. The claim now follows by the standard boundedness properties of the trace operator. 
\end{proof}

\begin{lemma} \label{LemmaHBanach} $H^2_w(\Omega^N)$ as defined in Definition \ref{defn.dens} is a Banach space. 
\end{lemma}
\begin{proof}
Consider the auxiliary space: 
\begin{align}
H^2_{0,w}(\Omega^N) = \Big\{ \psi: \nabla \psi, \nabla^2 \psi \text{ exist in the weak sense, and }  ||\psi||_{H^2_w(\Omega^N)} < \infty \Big\}. 
\end{align}

Through standard arguments, $H^2_{0,w}(\Omega^N)$ is a Banach space. Suppose $\{\psi^{(n)}\}$ is a Cauchy sequence in $H^2_w(\Omega^N)$. Then $\{\psi^{(n)}\}$ is Cauchy in $H^2_{0,w}(\Omega^N)$, and so there exists a limit point $\psi$ such that: $||\psi - \psi^{(n)}||_{H^2_w} \xrightarrow{n \rightarrow \infty} 0$. As $\psi^{(n)} \in H^2_w(\Omega^N)$, we may find a sequence $\{\phi^{(n)}_m\}_{m \ge 1}$ such that $||\phi^{(n)}_m - \psi^{(n)}||_{H^2_w} \xrightarrow{m \rightarrow \infty} 0$, where $\phi^{(n)}_m \in C^\infty_0(\Omega^N)$. In particular, define, for each $n$, by selecting $m$ large enough: $||\phi^{(n)} - \psi^{(n)}||_{H^2_w} < 2^{-n}$. Thus, $||\phi^{(n)} - \psi||_{H^2_w} \xrightarrow{n \rightarrow \infty} 0$, proving that $\psi \in \overline{C^\infty_0}^{||\cdot||_{H^2_w}}$. This establishes the desired result.

\end{proof}

\begin{lemma} Endowed with the inner product, 
\begin{align} \label{IPH}
\langle \psi, \varphi \rangle_{H^2_w} := \int \int \psi \varphi x^{2m} + \nabla \psi \cdot \nabla \varphi x^{2m+2} + \nabla^2 \psi : \nabla^2 \varphi x^{2m+4}, 
\end{align}

$H^2_w$ is a Hilbert Space. The inner product in (\ref{IPH}) induces the norm defined in (\ref{normHkw}).
\end{lemma}
\begin{proof}
One easily verifies the standard axioms of an inner-product for (\ref{IPH}). Non-degeneracy of (\ref{IPH}) is obtained via the boundary conditions in (\ref{psiBC}). Completeness is then obtained via Lemma \ref{LemmaHBanach}
\end{proof}

\begin{definition} The $\alpha-$Stokes operator is defined through: 
\begin{align} \label{S.alph.defn}
S_\alpha \psi = F_y - \eps G_x \text{ for $\psi \in H^2_w(\Omega^N), F_y - \eps G_x \in H^{-1}(\Omega^N)$, if and only if (\ref{weak.1}) holds.}  
\end{align}
\end{definition}

It is our aim to study the invertibility of $S_\alpha$:

\begin{lemma} \label{Lemma.weak.1} Given $F_y - \eps G_x \in H^{-1}(\Omega^N)$, there exists a unique weak solution $\psi \in H^2_w(\Omega^N)$ satisfying (\ref{weak.1}). Such a weak solution satisfies the energy inequality: 
\begin{align} \label{EIdentity}
||\psi||_{H^2_w}^2 \lesssim \frac{1}{\alpha} ||F_y - \epsilon G_x||_{H^{-1}}^2 =  \frac{1}{\alpha} ||S_\alpha \psi||_{H^{-1}}^2. 
\end{align}
\end{lemma}
\begin{proof}
Define: 
\begin{align} \nonumber
B[\psi, \phi] := \int \int \nabla_\epsilon^2 \psi : \nabla_\epsilon^2 \phi &+ \alpha \Big[ \int \int \psi \phi x^{2m} \\ \label{Buv} 
& +  \int \int \nabla \psi \cdot \nabla \phi x^{2m+2} + \int \int \nabla^2 \psi : \nabla^2 \phi x^{2m+4} \Big].
\end{align}

It is immediate to see that $B$ is bilinear, bounded, and coercive on $H^2_w(\Omega^N)$. Next, $F_y -  \eps G_x$ act as bounded linear functionals on $H^2_w(\Omega^N)$ through the pairing: $\langle F_y - \epsilon G_x, \phi \rangle_{H^{-2}_w, H^2_w} := \langle F_y - \epsilon G_x, \phi \rangle_{H^{-1}, H^1}$. This follows from: $\Big| \langle F_y - \epsilon G_x, \phi \rangle_{H^{-1}, H^1} \Big| \le ||F_y - \epsilon G_x||_{H^{-1}} ||\phi||_{H^2_w}.$ The existence of $\psi \in H^2_w(\Omega^N)$ a solution to (\ref{weak.1}) is then a standard application of the Lax-Milgram Lemma to the Hilbert Space $H^2_w(\Omega^N)$. The energy identity above follows from density of $C_0^{\infty}(\Omega^N)$ in $H^2_w(\Omega^N)$, which enables us to replace $\phi$ with $\psi$ in (\ref{weak.1}).
\end{proof}

The above lemma then says that $S_\alpha^{-1}: H^{-1}(\Omega^N) \rightarrow H^2_w(\Omega^N)$ is well-defined. Our intention now is to upgrade regularity. 

\begin{lemma} Given $F_y - \eps G_x \in H^{-1}(\Omega)$, the unique weak solution in $H^2_w(\Omega^N)$ guaranteed by Lemma \ref{Lemma.weak.1} is in $H^3_w(\Omega^N)$ and satisfies: 
\begin{align} \label{h3psi}
||\psi||_{H^3_w}^2 \lesssim \frac{1}{\alpha} ||F_y - \epsilon G_x||_{H^{-1}}^2 = \frac{1}{\alpha} ||S_\alpha \psi||_{H^{-1}}^2. 
\end{align}
\end{lemma}
\begin{proof}

As our weak solutions are only in $H^2_w(\Omega^N)$, we must formally use difference quotients within the weak formulation (\ref{weak.1}) to upgrade to $H^3_w(\Omega^N)$. However, we will generate the $H^3_w$ estimate via differentiating (\ref{Sa.1}), with the understanding that everything that is done can be formalized through the use of difference quotients in the standard manner. As such, we take $\partial_x$ of the system (\ref{Sa.1}), which gives: 
\begin{align} \label{diffBILI} 
\Delta_\eps^2 \psi_x + \alpha A(\psi_x) + [\p_x, \alpha A]\psi = \partial_x(F_y - \epsilon G_x), 
\end{align}

where
\begin{align} \nonumber
[\p_x, \alpha A]\psi = \alpha \Big[ 2mx^{2m-1} \psi &- (2m+2)x^{2m+1} \psi_{yy} - (2m+2) \partial_x (\psi_x x^{2m+1} ) \\ \n
& + (2m+4) x^{2m+3} \psi_{yyyy} + (2m+4)\partial_x \Big( \psi_{yyx} x^{2m+3}  \Big) \\ \label{com.A}
& + (2m+4)\partial_{xx} \Big( \psi_{xx} x^{2m+3)} \Big) \Big].
\end{align}

Let $\chi_1$ be as above in (\ref{chi1}). Define the quantities: 
\begin{align} \label{rhoMnew}
\tilde{\chi}_1 = 1 - \chi_1, \hspace{5 mm} \rho_M(x) = \chi_1(x) \tilde{\chi}_1(\frac{x}{M}), \text{ which implies } \Big| x^k \partial_x^k \rho_M(x) \Big| \le 2. 
\end{align}

We now test the above equation, (\ref{diffBILI}), against the multiplier $\rho_M \psi_x$. Doing so first gives from the Bilaplacian: 
\begin{align} \nonumber
\int \int \Delta_\eps^2 \psi_x \cdot \rho_M \psi_x &= \int \int \rho_M \Big[ \eps |\psi_{xxy}|^2 + \eps^2 |\psi_{xxx}|^2 + |\psi_{xyy}|^2 \Big] \\ 
& + c_0 \int \int \rho_M'' \Big[ \eps |\psi_{xy}|^2 + \eps^2 |\psi_{xx}|^2 \Big] + c_1 \int \int \partial_x^4 \rho_M \cdot |\psi_x|^2,
\end{align}

for constants $c_0, c_1$. Next, we have the terms coming from $A$:
\begin{align} \n
\int \int \alpha& A(\psi_x) \cdot  \rho_M \psi_x \\ \n
&\gtrsim \alpha \int \int [\psi_x^2 x^{2m} + \psi_{xy}^2 x^{2m+2} + \psi_{xx}^2 x^{2m+4} + \psi_{xyy}^2 x^{2m+4} \\ \label{com.A.0} &+ \psi_{xxy}^2 x^{2m+4} + \psi_{xxx}^2 x^{2m+4}] \rho_M - ||\psi||_{H^2_w}^2. 
\end{align}

Through a direct integration by parts, the commutator contains lower order terms: 
\begin{align} \label{com.A.1}
\Big| \int \int [\p_x, \alpha A]\psi \cdot \psi_x \rho_M \Big| \lesssim ||\psi||_{H^2_w}^2 \lesssim \frac{1}{\alpha}||F_y - \epsilon G_x||_{H^{-1}}^2. 
\end{align}

For detailed proofs of calculations (\ref{com.A.0}) and (\ref{com.A.1}), we refer the reader to (\ref{SYO.3}) - (\ref{SYO.4}). Finally, on the right-hand side of (\ref{diffBILI}), we have: 
\begin{align}
\Big| \langle \partial_x (F_y - \epsilon G_x), \psi_x \rho_M \rangle_{H^{-2}, H^2} \Big| \le ||F_y - \epsilon G_x ||_{H^{-1}} ||\rho_M \psi_{xx} ||_{H^1}. 
\end{align}

We can send $M \rightarrow \infty$ so that the weight $\rho_M \uparrow \chi_1$, resulting in 
\begin{align} \label{resid.1}
\int \int \chi_1 \Big| \nabla^2 \psi_x \Big|^2 x^{2m+4} \lesssim \frac{1}{\alpha} ||F_y - \epsilon G_x ||_{H^{-1}}^2. 
\end{align}

For the region $1 \le x \le 20$, and $0 \le y \le N$, we apply the standard $\dot{H}^2(\Omega^N)$ estimate for  solutions, $u^\alpha, v^\alpha$ Stokes' equation near corners (see \cite{Biharmonic}, Theorems 1 and 2, and Figure 2, P. 562 also in \cite{Biharmonic} with ``C/C" boundary conditions). Formally, fix another cut-off function, $\chi_2(x,y)$ localized near the corner $(1,0)$ (the identical argument can be given for the other corner, $(1,N)$). First, by calculation, we have: 
\begin{align}
\Delta_\epsilon^2 \Big( \chi_2 \psi \Big) = \chi_2 \Delta_\epsilon^2 \psi + [\Delta_\epsilon^2, \chi] \psi,
\end{align}

where the expression for the commutator is given explicitly: 
\begin{align} \nonumber
 [\Delta_\epsilon, \chi_2] \psi &= 4\partial_y \chi_2 \partial_y^3 \psi + 4 \partial_y^3 \chi_2 \partial_y \psi + 6 \partial_y^2 \chi_2 \partial_y^2 \psi + 2 \epsilon \partial_x^2 \partial_y^2 \chi_2 \psi + 2 \epsilon \partial^2 \chi_2 \partial_x^2 \psi \\ \nonumber
 & + 4 \epsilon \partial_x \partial_y^2 \chi_2 \partial_x \psi + 2\epsilon \partial_x^2 \chi_2 \partial_y^2 \psi + 4 \epsilon \partial_x \chi_2 \partial_x \partial_y^2 \psi + 4\epsilon \partial_x^2 \partial_y \chi_2 \partial_y \psi \\ \nonumber
 & + 4 \epsilon \partial_y \chi_2 \partial_x^2 \partial_y \psi + 8\epsilon  \partial_{xy} \chi_2 \partial_{xy} \psi + 6 \epsilon^2 \partial_x^2 \chi_2 \partial_x^2 \psi + \epsilon^2 \partial_x^4 \chi_2 \psi + 4 \epsilon^2 \partial_x^3 \chi_2 \partial_x \psi \\  \label{commutator}
 & + 4 \epsilon^2 \partial_x \chi_2 \partial_x^3 \psi. 
\end{align}

The salient feature of (\ref{commutator}) will be:
\begin{align}
 [\Delta_\epsilon, \chi_2] \psi = o(\chi_2 \partial^3 \psi),
\end{align}

where this is short-hand notation for containing up to three $\psi$-derivatives, and localized by $\chi_2$ (or any derivative of $\chi_2$ which is also localized). Localizing (\ref{Sa.1}) using $\chi_2$:  
\begin{align} \label{Hm1}
||\chi_2 \psi||_{H^3} &\lesssim ||\chi_2 (F_y - \epsilon G_x) ||_{H^{-1}} + ||o(\chi_2 \partial^3 \psi)||_{H^{-1}}  \lesssim ||\chi_2 (F_y - \epsilon G_x) ||_{H^{-1}}.
\end{align}

Combining (\ref{Hm1}) and (\ref{resid.1}) gives the desired result. 

\end{proof}

\begin{lemma} Fix any bounded set $B \subset \Omega^N$. Then we have: 
\begin{align} \label{AD}
||\psi||_{G^3_{w,B}}^2 \lesssim C(B) \frac{1}{\alpha} ||F_y - \epsilon G_x||_{H^{-1}}^2,
\end{align}
where the constant $C(B)$ depends on $B$.
\end{lemma}
\begin{proof}
This argument proceeds identically to the calculation from the previous lemma which resulted in (\ref{Hm1}) by simply replacing $\chi_2$ with cut-off functions localized to each interval $x \in [M, M+1]$. The dependence on $B$ in the constant in (\ref{AD}) arises from the weights $x^{2m},x^{2m+2}, x^{2m+4}$ appearing in the equation (\ref{Sa.1}) through $A(\psi)$. 
\end{proof}

The above lemmas roughly show that $S_\alpha^{-1}$ gains four derivatives. By repeating this procedure for higher-order $x$-derivatives, we can upgrade to higher-regularity: 
\begin{lemma} 
Given $(F, G) \in H^{-1}_2$, the unique weak solution guaranteed by Lemma \ref{Lemma.weak.1} satisfies:
\begin{align}
||\psi|_{H^5_w}^2 \lesssim \frac{1}{\alpha} ||(F,G)||_{H^{-1}_2}^2
\end{align}
\end{lemma}

For $(F, G) \in H^{-1}_2$, we can upgrade weak solutions to strong solutions:  
\begin{lemma} \label{Lemma.IBPS} Given $(F,G) \in H^{-1}_2$, the unique weak solution guaranteed by Lemma \ref{Lemma.weak.1} is a strong solution of (\ref{Sa.1}). 
\end{lemma}
\begin{proof}
An integration by parts of the weak formulation (\ref{weak.1}), justified according to the previous lemma, is equivalent to the equation (\ref{Sa.1}) being satisfied pointwise on $\Omega^N$. The boundary conditions at $x = 1, y = 0, y = N$ are satisfied by Lemma \ref{LBC10}. The boundary condition at $x \rightarrow \infty$ comes from the norms, (\ref{normHkw}), which when applying with $k = 5$, imply that up to four derivatives of $\psi$ vanish rapidly at $x \rightarrow \infty$. 
\end{proof}

\section{Step 2: Compact Perturbations, $S_\alpha \psi + T[\psi]$}

For this step, we invite the reader to refer back to the specification of $T[\psi]$, given in (\ref{mapT}), and the system that we will focus on, given in (\ref{sysT}).  Note that $T[\psi]$ contains a loss of three-derivatives for $\psi$. Note also the presence of the term $\eps^{\frac{n}{2}+\gamma} \bar{v}$. We will now need some compactness lemmas. 

\begin{lemma} Fix two weights, $w_1 = x^{m_1}$, and $w_2 = x^{m_2}$, where $m_2 > m_1 \ge 0$. Then, one has the following compact embedding:
\begin{align} \label{cpct.em.1}
H^1_{ loc}(\Omega^N) \cap L^2_{w_2}(\Omega^N) \hookrightarrow \hookrightarrow L^2_{w_1}(\Omega^N).
\end{align}
\end{lemma}
\begin{proof}

Consider a family of functions $\{f^n\}$ defined on $\Omega^N$ such that:
\begin{equation}
\sup_n \int \int f_n^2 w_2^2 < \infty,
\end{equation}

and such that $f_n \in H^1_{loc}(\Omega^N)$, uniformly in $n$. By taking Sobolev extensions across $\p \Omega^N$, and subsequently cutting off in the $y$ and negative $x$ directions, we can assume $\{f_n \}$ are defined on $\mathbb{R}^2$, compactly supported in the $y$ direction and negative $x$ direction. Fix any $\delta' > 0$. Since $m_2 > m_1$, there exists a compact set $K = K(\delta')$ such that: 
\begin{align}
\sup_n ||f_n||_{L^2_{w_1}(K^c)} \le \frac{\delta'}{2}.
\end{align}

On $K$, by Rellich compactness, there exists a subsequence (depending on $\delta'$) such that 
\begin{align}
\limsup_{j,k \rightarrow \infty} ||f_{n_j} - f_{n_k} ||_{L^2(K)} \le \frac{\delta'}{2 \times diam(K)^{m_1}}. 
\end{align}

Then, 
\begin{align}
\limsup_{j,k \rightarrow \infty} ||f_{n_j} - f_{n_k} ||_{L^2_{w_1}(K)} \le \frac{\delta'}{2}. 
\end{align}

Combining the above two estimates, 
\begin{align}
\limsup_{j,k \rightarrow \infty} ||f_{n_j} - f_{n_k} ||_{L^2_{w_1}} \le \delta'. 
\end{align}

Taking successively $\delta' = 2^{-n}$ and applying a diagonalization argument gives the result. 

\end{proof}

\begin{lemma}  \label{lemma.cpct} Let the weight, $x^{2m}$, in the expression for $A(\psi)$, equation (\ref{Aofpsi}), be selected for any $m > 0$. Then the map $S_\alpha^{-1}T$ is well-defined and compact $H^2(\Omega^N) \rightarrow H^2(\Omega^N)$.
\end{lemma}
\begin{proof}
According to (\ref{AD}), this follows from the compactness of $G^3_{w, loc}(\Omega^N) \hookrightarrow \hookrightarrow H^2(\Omega^N)$, which in turn follows from (\ref{cpct.em.1}). The lemma is proven.
\end{proof}

We are now ready to study system (\ref{sysT}). The first task is to obtain an energy estimate to the inhomogeneous problem: 
\begin{lemma} \label{lem.wk.e} Suppose $\psi \in H^2(\Omega^N)$ is a solution to (\ref{sysT}), where $(F,G) \in H^{-1}_2$, and $||\bar{u}, \bar{v}||_Z \le 1$. Then $\psi$ obeys the following energy estimate: 
\begin{align} \label{pos.we}
||u_y||_{L^2}^2 + \alpha ||\psi||_{H^2_w}^2 \lesssim \mathcal{O}(\delta) ||\sqrt{\epsilon} v_x x^{\frac{1}{2}}, v_y x^{\frac{1}{2}}||_{L^2}^2 + \int \int F u + \eps |G| |v|. 
\end{align}
\end{lemma}
\begin{proof}

Supposing there existed such a $\psi$, we would have $T[\psi]  \in H^{-1}(\Omega^N)$, and so by (\ref{h3psi}), we know $\psi \in H^3_w(\Omega^N)$. By bootstrapping this regularity, we obtain that: 
\begin{equation} \label{bsreg}
\psi \in H^5_w(\Omega^N). 
\end{equation}

We would like to apply the multiplier $\psi$ to the equation (\ref{sysT}) in order to repeat the energy estimate from Proposition \ref{thm.energy}. Select test functions, $\phi^{(n)} \in C^\infty_0(\Omega^N)$, which satisfy:  
\begin{align} \label{f.density.1}
||\phi^{(n)} - \psi||_{H^2_w} \rightarrow 0.
\end{align}

This is possible according to the density of $C^\infty_0(\Omega^N)$ in $H^2_w$ in Definition \ref{defn.dens}. Multiplying (\ref{sysT}) by $\phi^{(n)}$, then gives on the left-hand side:
\begin{align} \label{bofa.5}
\int \int \Big( \Delta_\eps^2 \psi + T\psi \Big) \cdot \phi^{(n)} + \alpha \int \int A(\psi) \phi^{(n)}.
\end{align}

First, we shall use (\ref{mapT0}) to write: 
\begin{align} \n
\int \int \Big(\Delta_\eps^2 \psi + T[\psi] \Big) \phi^{(n)} &= \int \int  \Big(\Delta_\eps^2 \psi + T_0[\psi] \Big) \phi^{(n)} - \int \int \eps^{\frac{n}{2}+\gamma} \p_y (\bar{v} u_y) \cdot \phi^{(n)} \\ \label{sf.1.2}
& =  \int \int  \Big(\Delta_\eps^2 \psi + T_0[\psi] \Big) \phi^{(n)} + \int \int \eps^{\frac{n}{2}+\gamma} (\bar{v} u_y) \cdot \phi^{(n)}_y.
\end{align}

According to (\ref{f.density.1}), we pass to limits in the following terms:
\begin{align} \n
\int \int (\Delta_\eps^2 \psi + T_0[\psi]) \phi^{(n)} &= \int \int \nabla_\eps^2 \psi : \nabla_\eps^2 \phi^{(n)} + \int \int T_0[\psi] \phi^{(n)} \\ \label{bsreg.2}
& \xrightarrow{n \rightarrow \infty} \int \int |\nabla_\eps^2 \psi|^2 + \int \int T_0[\psi] \psi. 
\end{align}

We have used:
\begin{align} \n
|\int \int T_0[\psi] \phi^{(n)} - T_0[\psi] \psi| &= \int \int ( \p_y T_a[\psi] - \p_x T_b[\psi] ) \cdot (\phi^{(n)} - \psi) \\ \n
& \le ||T_a[\psi], T_b[\psi]||_{L^2} ||\phi^{(n)} - \psi||_{H^1} \\ 
& \le ||\psi||_{H^5_w}||\phi^{(n)} - \psi||_{H^1}  \xrightarrow{n \rightarrow \infty} 0,
\end{align}

according to (\ref{bsreg}) and the definition in equation (\ref{mapT0}). The integration on the right-hand side of (\ref{bsreg.2}) arises exactly from the energy estimates, Proposition \ref{thm.energy}, in particular, terms (\ref{stream.2.0}), (\ref{su}), (\ref{NSR.E.Sv}), and so we may write: 
\begin{align}
| \lim_{n \rightarrow \infty} \int \int \Big(\Delta_\eps^2 \psi + T_0[\psi] \Big) \phi^{(n)} | \gtrsim ||u_y||_{L^2}^2 - \mathcal{O}(\delta) ||\sqrt{\epsilon} v_x x^{\frac{1}{2}}, v_y x^{\frac{1}{2}}||_{L^2}^2.
\end{align}

We may pass to the limit in the final term of (\ref{sf.1.2}) due to the calculation: 
\begin{align} \n
| \int \int \bar{v} u_y (\phi^{(n)}_y - \psi_y) | &\le ||\bar{v}||_{L^\infty} ||u_y||_{L^2} || \phi^{(n)}_y - \psi_y ||_{L^2} \\ 
& \le \eps^{-N_4} ||\bar{v}||_{Z} ||u_y||_{L^2} || \phi^{(n)}_y - \psi_y ||_{L^2}  \xrightarrow{n \rightarrow \infty} 0.
\end{align}

Upon passing to the limit, we integrate by parts: 
\begin{align}
-\int \int \eps^{\frac{n}{2}+\gamma} (\bar{v} u_y) \cdot \phi^{(n)}_y \xrightarrow{n \rightarrow \infty} \int \int \eps^{\frac{n}{2}+\gamma} (\bar{v} u_y) \cdot u = -\int \int \frac{\eps^{\frac{n}{2}+\gamma}}{2}\bar{v}_y  u^2
\end{align}

From here, we estimate identically as in (\ref{order.NL}):
\begin{align}
|\int \int \frac{\eps^{\frac{n}{2}+\gamma}}{2}\bar{v}_y  u^2| \le \eps^{\frac{n}{2}+\gamma - \omega(N_i)} ||\bar{u}, \bar{v}||_Z ||v_y x^{\frac{1}{2}}||_{L^2}^2 \le \eps^{\frac{n}{2}+\gamma - \omega(N_i)} ||v_y x^{\frac{1}{2}}||_{L^2}^2. 
\end{align}

It remains to treat (\ref{bofa.5}), for which we use the compact support of $\phi^{(n)}$ to justify the integration by parts: 
\begin{align}
\int \int \alpha A(\psi) \cdot \phi^{(n)} = \int \int \psi \phi^{(n)} x^{2m} + \nabla \psi \cdot \nabla \phi^{(n)} x^{2m+2} + \nabla^2 \psi : \nabla^2 \phi^{(n)}x^{2m+4}.
\end{align}

Passing to the limit, according to (\ref{f.density.1}):
\begin{align} \label{AC.0}
\lim_{n \rightarrow \infty} \int \int \alpha A(\psi) \phi^{(n)} = \alpha \int \int \psi^2 x^{2m} + |\nabla \psi|^2 x^{2m+2} + |\nabla^2 \psi|^2 x^{2m+4},
\end{align}

On the right-hand side, we have: 
\begin{align}
\int \int F_y \cdot \phi^{(n)} = - \int \int F \phi^{(n)}_y \xrightarrow{n \rightarrow \infty} \int \int F u, \\
\int \int -\eps G_x \cdot \phi^{(n)} =  \int \int \eps G \cdot \phi^{(n)}_x \xrightarrow{n \rightarrow \infty} \int \int \eps G v. 
\end{align}

Consolidating the previous estimates gives the desired estimate, (\ref{pos.we}). 

\end{proof}

The task now is to estimate the right-hand side of (\ref{pos.we}) in terms of the left-hand side using the smallness of $\mathcal{O}(\delta)$. We refer the reader to Proposition \ref{prop.pos}, whose proof we follow closely. We will point out the subtle differences:
\begin{lemma} \label{lem.wk.p}
 Suppose $\psi \in H^2(\Omega^N)$ is a solution to (\ref{sysT}), where $(F,G) \in H^{-1}_2$, and $||\bar{u}, \bar{v}||_Z \le 1$. Suppose the weight $w = x^{2m}$ from equation (\ref{Aofpsi}) is selected such that $m$ is sufficiently large relative to universal constants. Then $\psi$ obeys: 
\begin{align} \label{bl.mct.1}
||\{\sqrt{\epsilon}v_x, v_y \} x^{\frac{1}{2}} ||_{L^2}^2 &\lesssim ||u_y||_{L^2}^2 +  \alpha ||\psi||_{H^2_w}^2 + \int \int |F||u_x| x + \eps |G||v| + \eps |G| |v_x| x.
\end{align}
\end{lemma}

\begin{proof}

We will repeat the positivity estimate of Proposition \ref{prop.pos}. To do so, we apply the multiplier $\psi_x x \chi_{L,\alpha}^2$ to (\ref{sysT}). Here, $\chi$ is a normalized cut-off function equal to $1$ on $[1,2]$ and $0$ on $[3,\infty)$, and 
\begin{equation} \label{factor.alpha}
\chi_{L,\alpha}(x) := \chi(\frac{\alpha}{L}x), \text{ so that } \p_x^k \chi_{L,\alpha} = \frac{\alpha^{k}}{L^k} \chi^{(k)}. 
\end{equation}

Such a cut-off function was not present in Proposition \ref{prop.pos}. The necessity of it is due to the terms arising from $A(\psi)$. The presence of this cut-off function enables us to justify all integrations by parts in the $x$-direction. For our fixed $\alpha > 0$, we will eventually send $L \rightarrow \infty$. Applying the multiplier $\psi_x x \chi_{L,\alpha}^2$ to (\ref{sysT}), gives on the left-hand side:
\begin{align} \label{sf.3.0}
\int \int \Big( T[\psi] &+ \Delta_\eps^2 \psi \Big) \cdot \psi_x x \chi_{L,\alpha}^2 + \alpha \int \int A(\psi) \cdot \psi_x x \chi_{L,\alpha}^2 \\ \n
& = \int \int \Big( T_0[\psi]  + \Delta_\eps^2 \psi + \eps^{\frac{n}{2}+\gamma} \p_y[\bar{v} u_y]\Big) \cdot \psi_x x \chi_{L,\alpha}^2 + \alpha \int \int A(\psi) \cdot \psi_x x \chi_{L,\alpha}^2.
\end{align}

We will first focus on the first two integrands above in (\ref{sf.3.0}), which appeared in Proposition \ref{prop.pos}.  The obstacle to repeating the calculations exactly as in Proposition \ref{prop.pos} is the presence of the cut-off function, $\chi_{L,\alpha}^2$, in the multiplier. The essential idea is this: when no derivative falls on $\chi_{L,\alpha}^2$, the estimate will be the same as the corresponding term in Proposition \ref{prop.pos}. When at least one derivative falls on $\chi_{L,\alpha}^2$, we may use the factors of $\alpha$ obtained from the scaling in (\ref{factor.alpha}) to absorb the new terms into the left-hand side of (\ref{pos.we}).  Let us start with the profile terms from $T_0[\psi]$, for which we refer the reader to estimates (\ref{Supos.1}) - (\ref{Supos.2}). We will transfer all of the terms to velocity formulation so as to remain consistent with estimates (\ref{Supos.1}) - (\ref{Supos.2}).
\begin{align} \n
\int \int \p_y S_u \cdot v x \chi_{L,\alpha}^2 &= \int \int S_u u_x x \chi_{L,\alpha}^2 \\ \n
& = \int \int [u_R u_x + u_{Rx}u + v_R u_y + u_{Ry}v ] u_x x \chi_{L,\alpha}^2 \\ 
& \gtrsim \int \int u_x^2 x \chi_{L,\alpha}^2 - |\int \int [ u_{Rx}u + v_R u_y + u_{Ry}v ] u_x x \chi_{L,\alpha}^2 |. 
\end{align}

We will treat the three terms on the right-hand side above, using (\ref{PE0.4}) and (\ref{PE5}) starting with: 
\begin{align}\nonumber
\int \int u_{Rx} u x u_x \chi_{L,\alpha}^2 &= \int \int \{ u^P_{Rx} +  u^E_{Rx} \} u x u_x \chi_{L,\alpha}^2 \\ \nonumber
&\le ||yx^{\frac{1}{2}} u^P_{Rx}||_{L^\infty} ||u_y||_{L^2} ||v_y x^{\frac{1}{2}}\chi_{L,\alpha}||_{L^2} \\ \n & \hspace{30 mm} + ||u^E_{Rx} x^{\frac{3}{2}} ||_{L^\infty} ||\frac{u}{x}\chi_{L,\alpha}||_{L^2}||u_x x^{\frac{1}{2}}\chi_{L,\alpha}||_{L^2} \\ \n
&\le \mathcal{O}(\delta) ||u_y||_{L^2}^2 + \mathcal{O}(\delta) ||u_x x^{\frac{1}{2}}\chi_{L,\alpha}||_{L^2}^2 + \frac{\alpha}{L} ||u||_{L^2}^2 \\
& \le \mathcal{O}(\delta) ||u_y||_{L^2}^2 + \mathcal{O}(\delta) ||u_x x^{\frac{1}{2}}\chi_{L,\alpha}||_{L^2}^2 + \frac{\alpha}{L} ||\psi||_{H^2_w}^2.
\end{align}

Above, we have used the Hardy inequality: 
\begin{align} \n
|| \frac{u}{x} \chi_{L,\alpha}||_{L^2} &\lesssim || \p_x(u \chi_{L,\alpha})||_{L^2} \le ||u_x \chi_{L,\alpha}||_{L^2} + \frac{\alpha}{L} ||u \chi'_{L,\alpha}||_{L^2} \\
& \lesssim ||u_x \chi_{L,\alpha}||_{L^2} + \frac{\alpha}{L} ||\psi||_{H^2_w}. 
\end{align}

Next, by (\ref{PE1}), (\ref{PE4.new.2}), we have:
\begin{align} \n
\Big| \int \int v_R u_y v_y x \chi_{L,\alpha}^2 \Big| &\le ||v_R x^{\frac{1}{2}} ||_{L^\infty} ||u_y||_{L^2} ||v_y x^{\frac{1}{2}} \chi_{L,\alpha}||_{L^2} \\ 
& \le \mathcal{O}(\delta) ||u_y||_{L^2}||v_y x^{\frac{1}{2}} \chi_{L,\alpha}||_{L^2}. 
\end{align}

Next, by (\ref{PE0.5}), (\ref{PE3}), (\ref{PE5}) we have: 
\begin{align} \nonumber
\int \int u_{Ry} v u_x x \chi_{L,\alpha}^2 &= \int \int \{ u_{Ry}^P + \sqrt{\epsilon}u_{RY}^E \} vv_y x \chi_{L,\alpha}^2 \\ \nonumber
&\le ||y u_{Ry}^P||_{L^\infty} ||v_y x^{\frac{1}{2}} \chi_{L,\alpha}||_{L^2}^2 \\ \n
& \hspace{20 mm} + \sqrt{\epsilon} ||u^E_{RY} x^{\frac{3}{2}}||_{L^\infty} ||v_y x^{\frac{1}{2}} \chi_{L,\alpha}||_{L^2}||\sqrt{\epsilon} \frac{v}{x} \chi_{L,\alpha}||_{L^2}  \\ \n
&\le ||y u_{Ry}^P||_{L^\infty} ||v_y x^{\frac{1}{2}} \chi_{L,\alpha}||_{L^2}^2 \\ \n & \hspace{20 mm} + \sqrt{\epsilon} ||u^E_{RY} x^{\frac{3}{2}}||_{L^\infty} ||v_y x^{\frac{1}{2}} \chi_{L,\alpha}||_{L^2}||\sqrt{\epsilon} v_x \chi_{L,\alpha}||_{L^2} \\ \n & \hspace{20 mm} + \frac{\alpha}{L}\sqrt{\epsilon} ||u^E_{RY} x^{\frac{3}{2}}||_{L^\infty} ||v_y x^{\frac{1}{2}} \chi_{L,\alpha}||_{L^2}||\sqrt{\epsilon} v \chi'_{L,\alpha}||_{L^2} \\
&\le \mathcal{O}(\delta) ||v_y x^{\frac{1}{2}}\chi_{L,\alpha}||_{L^2}^2 + \sqrt{\epsilon}\mathcal{O}(\delta) ||\sqrt{\epsilon}v_x\chi_{L,\alpha}||_{L^2}^2 + \mathcal{O}(\delta) \frac{\alpha}{L}||\psi|_{H^2_w}^2.   
\end{align}

Summarizing the previous four terms: 
\begin{align} \n
\int \int \p_y S_u \cdot vx \chi_{L,\alpha}^2 \gtrsim ||v_y x^{\frac{1}{2}}& \chi_{L,\alpha}||_{L^2}^2 - \frac{\alpha}{L}||\psi||_{H^2_w}^2 \\ \label{SuBST.1}
& - \mathcal{O}(\delta) ||u_y||_{L^2}^2 - \mathcal{O}(\delta) \sqrt{\eps} ||\sqrt{\eps}v_x \chi_{L,\alpha}||_{L^2}^2. 
\end{align}

We will now move to the profile terms from $S_v$, which are located starting from estimate (\ref{str.p.pr.v}). First, 
\begin{align} \label{mstreep}
\int \int -\eps \p_x S_v \cdot v x \chi_{L,\alpha}^2 = \int \int \eps S_v \cdot [v_x x \chi_{L,\alpha}^2 + v \chi_{L,\alpha}^2 + vx \frac{2\alpha}{L}\chi_{L,\alpha} \chi_{L,\alpha}'].
\end{align} 

Referring to definition (\ref{defn.SU.SV}), consider the term $u_R v_x$ in $S_v$, which is the most delicate profile term: 
\begin{align} \label{bofa.1}
\int \int \eps u_R v_x^2 x \chi_{L,\alpha}^2 + \int \int \eps u_R v_x v [\chi_{L,\alpha} + 2x \frac{\alpha}{L}\chi_{L,\alpha} \chi'_{L,\alpha}].
\end{align} 

The first term above in (\ref{bofa.1}) gives positivity: 
\begin{align}
\int \int \eps u_R v_x^2 x \chi_{L,\alpha}^2 \gtrsim \int \int \eps v_x^2 x \chi_{L,\alpha}^2.
\end{align}

We will treat the second term on the right-hand side of (\ref{bofa.1}):
\begin{align} \label{bofa.2}
&\int \int \eps u_R vv_x \chi_{L,\alpha}^2 = - \int \int \eps \frac{v^2}{2} [u_{Rx} \chi_{L,\alpha}^2 + 2u_R \frac{\alpha}{L} \chi_{L,\alpha} \chi'_{L,\alpha}], \\ \n
&\int \int \eps u_R vv_x  x\frac{\alpha}{L} \chi_{L,\alpha} \chi'_{L,\alpha} = - \int \int \eps \frac{v^2}{2} [u_{Rx} x\frac{\alpha}{L} \chi_{L,\alpha} \chi'_{L,\alpha} \\ \label{bofa.3}
& \hspace{50 mm} + u_R\frac{\alpha}{L} \chi_{L,\alpha} \chi'_{L,\alpha} + u_R x\frac{\alpha^2}{L^2} \chi_{L,\alpha} \chi''_{L,\alpha}  ].
\end{align}

The first term on the right-hand side of (\ref{bofa.2}) yields:
\begin{align}
|\int \int \eps v^2 u_{Rx} \chi^2_{L,\alpha} | \lesssim \sqrt{\eps} ||\chi_{L,\alpha} \{\sqrt{\eps}v_x, v_y\} x^{\frac{1}{2}}||_{L^2}^2 + \frac{\alpha}{L} ||\psi||_{H^2_w}^2, 
\end{align}

in nearly an identical manner to estimate (\ref{Svv}). We now estimate: 
\begin{align}
| \int \int \eps \frac{u_R}{2}v^2 \frac{\alpha}{L} \chi'_{L,\alpha}| \lesssim \int \int \eps v^2  \frac{\alpha}{L} \lesssim \frac{\alpha}{L} ||\psi||_{H^2_w}^2. 
\end{align}

This same estimate can be performed for all the terms in (\ref{bofa.3}). Consolidating these bounds: 
\begin{align}
\int \int \eps v_x^2 x \chi_{L,\alpha}^2 \lesssim (\ref{bofa.1}) + \sqrt{\eps} ||\chi_{L,\alpha} \{\sqrt{\eps}v_x, v_y \} x^{\frac{1}{2}}||_{L^2}^2 + \frac{\alpha}{L}||\psi||_{H^2_w}. 
\end{align}

It remains now to treat the remaining three terms in $S_v$. The second, third, and fourth terms from $S_v$ can be controlled in the same manner as in (\ref{Supos.3}) - (\ref{Supos.5}):
\begin{align}  \label{Supos.a.3}
&\epsilon| \int \int v_{Rx} u v_x x \chi_{L,\alpha}^2| \le \sqrt{\epsilon} ||x^{\frac{3}{2}} v_{Rx}||_{L^\infty} ||\sqrt{\epsilon}v_x x^{\frac{1}{2}} \chi_{L,\alpha}||_{L^2} ||u_x \chi_{L,\alpha}||_{L^2} + \frac{\alpha}{L}||\psi||_{H^2_w}^2. \\ \label{Supos.a.4}
&\epsilon \Big| \int \int v_R v_y v_x x \chi_{L,\alpha}^2 \Big| \le \sqrt{\epsilon} ||v_R||_{L^\infty} ||v_y x^{\frac{1}{2}}\chi_{L,\alpha}||_{L^2} || \sqrt{\epsilon}v_x x^{\frac{1}{2}}\chi_{L,\alpha}||_{L^2} + \frac{\alpha}{L} ||\psi||_{H^2_w}^2. \\ \n
&\epsilon \Big| \int \int v_{Ry} vv_x x \chi_{L,\alpha}^2 \Big| \le \sqrt{\epsilon} ||v_{Ry}^P y||_{L^\infty} ||v_y x^{\frac{1}{2}} \chi_{L,\alpha}||_{L^2} ||\sqrt{\epsilon} v_x x^{\frac{1}{2}}\chi_{L,\alpha}||_{L^2} \\ \label{Supos.a.5}
& \hspace{50 mm} + \sqrt{\epsilon}||v_{RY}^E x^{\frac{3}{2}}||_{L^\infty} ||\sqrt{\epsilon}v_x x^{\frac{1}{2}} \chi_{L,\alpha}||_{L^2}^2 +\frac{\alpha}{L} ||\psi||_{H^2_w}^2.
\end{align}

We now turn back to (\ref{mstreep}), addressing the second term in the bracket for the final three profile terms from $S_v$:
\begin{align}
\int \int [v_{Rx}u + v_R v_y + v_{Ry}v] \cdot \eps v \chi_{L,\alpha}^2. 
\end{align} 

First, through the Hardy inequality and (\ref{PE0.1}), (\ref{PE4}): 
\begin{align} \n
| \int \int \eps v_{Rx} uv \chi_{L,\alpha}^2 | &\le \sqrt{\eps}||x^{\frac{3}{2}}v_{Rx}||_{L^\infty} ||\frac{u}{x^{\frac{3}{4}}} \chi_{L,\alpha}||_{L^2} || \sqrt{\eps}\frac{v}{x^{\frac{3}{4}}} \chi_{L,\alpha}||_{L^2} \\ \n
& \le  \sqrt{\eps} \Big[ ||u_x x^{\frac{1}{4}} \chi_{L,\alpha}||_{L^2}^2 + ||\sqrt{\eps}v_x x^{\frac{1}{4}} \chi_{L,\alpha}||_{L^2}^2 + \frac{\alpha}{L}||\{u, \sqrt{\eps} v x^{\frac{1}{4}} \chi_{L,\alpha}' ||_{L^2}^2 \Big] \\
& \le  \sqrt{\eps} \Big[ ||u_x x^{\frac{1}{4}} \chi_{L,\alpha}||_{L^2}^2 + ||\sqrt{\eps}v_x x^{\frac{1}{4}} \chi_{L,\alpha}||_{L^2}^2 + \frac{\alpha}{L}||\psi||_{H^2_w}^2 \Big],
\end{align}

so long as $w = x^m$ is selected larger than $x^{\frac{1}{4}}$, which is true by the assumption of this lemma.  Next, through an integration by parts and (\ref{PE1}), (\ref{PE4.new.2}):
\begin{align} \n
\int \int[ v_R v_y + v_{Ry}v] \eps v \chi_{L,\alpha}^2 &= \frac{1}{2} \int \int v_{Ry} v^2 \eps \chi_{L,\alpha}^2 \le  \epsilon ||v_{Ry}^P y^2||_{L^\infty} ||v_y \chi_{L,\alpha}||_{L^2}^2 \\ 
& + \sqrt{\epsilon} ||v_{RY}^E x^{\frac{3}{2}}||_{L^\infty} ||\sqrt{\epsilon}v_x x^{\frac{1}{4}} \chi_{L,\alpha}||_{L^2}^2 + \frac{\alpha}{L}||\psi||_{H^2_w}^2. 
\end{align}

The final task for the $S_v$ profile contributions is the third term from (\ref{mstreep}):
\begin{align}
|\int \int \eps [v_{Rx}u + v_R v_y + v_{Ry}v ] \cdot vx \frac{2\alpha}{L}\chi_{L,\alpha} \chi_{L,\alpha}'| \le \frac{\alpha}{L}||\psi||_{H^2_w}^2,
\end{align}

so long as $w = x^m$ is selected larger than $x$, which is true by assumption of this lemma. Let us consolidate all of the calculations from $S_v$:
\begin{align} \label{SuBST.2}
\int \int -\eps \p_x S_v \cdot v x \chi_{L,\alpha}^2 \gtrsim \int \int \eps v_x^2 x \chi_{L,\alpha}^2 - \sqrt{\eps} ||\chi_{L,\alpha} \{\sqrt{\eps}v_x, v_y \} x^{\frac{1}{2}}||_{L^2}^2 - \frac{\alpha}{L}||\psi||_{H^2_w}^2.
\end{align}

It now remains to come to those terms contributed by $\Delta_\eps^2 \psi$ into (\ref{sf.3.0}). We will follow closely the calculations from (\ref{fubini.1}) - (\ref{sumlap.1}) in Proposition \ref{prop.pos}. We will again omit the justifications near the corners of our domain as these are identical to Proposition \ref{prop.pos}. Again, we will write these terms in the velocity form, to remain consistent with the calculations in  (\ref{fubini.1}) - (\ref{sumlap.1}). First, 
\begin{align} \label{SuBST.3}
|\int \int -u_{yy}u_x x \chi_{L,\alpha}^2| = |- \int \int \frac{u_y^2}{2} [\chi_{L,\alpha}^2 + 2x \frac{\alpha}{L} \chi_{L,\alpha} \chi_{L,\alpha}']| \lesssim ||u_y||_{L^2}^2. 
\end{align}

Next, 
\begin{align} \label{SuBST.4}
\int \int -\eps u_{xx} u_x x \chi_{L,\alpha}^2| = | \int \int \eps \frac{u_x^2}{2} [\chi_{L,\alpha}^2 + 2x \frac{\alpha}{L} \chi_{L,\alpha} \chi_{L,\alpha}'] \lesssim \eps ||u_x \chi_{L,\alpha} ||_{L^2}^2 + \eps \frac{\alpha}{L}||\psi||_{H^2_w}^2. 
\end{align}

We now move to the terms from $\Delta_\eps v$, starting with: 
\begin{align} \n
| \int \int \eps^2 v_{xxx} \cdot v x \chi_{L,\alpha}^2 |&= | \int \int \eps^2 v_{xx} \cdot [v_x x \chi_{L,\alpha}^2 + v \chi_{L,\alpha}^2 + 2v x \frac{\alpha}{L} \chi_{L,\alpha} \chi_{L,\alpha}'] | \\ \label{SuBST.5}
&\le \eps ||\sqrt{\eps} v_x  \chi_{L,\alpha} ||_{L^2}^2 + \frac{\alpha}{L}||\psi||_{H^2_w}^2. 
\end{align}

Finally, we have: 
\begin{align} \n
|\int \int -\eps v_{xyy} v x \chi_{L,\alpha}^2 | &=| \int \int \eps v_{xy} v_y x \chi_{L,\alpha}^2| \\ \n
& = | \int \int \eps v_y^2 [\chi_{L,\alpha}^2 + 2x \frac{\alpha}{L} \chi_{L,\alpha} \chi_{L,\alpha}] | \\ \label{SuBST.6}
& \lesssim \eps|| v_y \chi_{L,\alpha}||_{L^2}^2 + \frac{\alpha}{L}||\psi||_{H^2_w}^2. 
\end{align}

By combining calculations (\ref{SuBST.1}), (\ref{SuBST.2}), (\ref{SuBST.3}) - (\ref{SuBST.6}), and absorbing relevant terms to the left-hand side below, we have: 
\begin{align} \n
||\{\sqrt{\epsilon}v_x, v_y \} x^{\frac{1}{2}} \chi_{L,\alpha}||_{L^2}^2 &\lesssim ||u_y||_{L^2}^2 + \frac{\alpha}{L}||\psi||_{H^2_w}^2 \\  \n
& +  \alpha \int \int A(\psi) \cdot \psi_x x \chi_{L, \alpha}^2 + \int \int \eps^{\frac{n}{2}+\gamma} \bar{v} u_y u_x x \chi_{L,\alpha}^2 \\ \label{pos.pos.1}
& + \int \int F_y \cdot v x \chi_{L,\alpha}^2 - \eps G_x \cdot v x \chi_{L,\alpha}^2. 
\end{align}

Via direct integration by parts, which is justified due to the presence of the cut-off function in $x$, we compute: 
\begin{align} \label{AC.1}
\Big| \alpha \int \int A(\psi) \psi_x x \chi_{L,\alpha}^2 \Big| \lesssim \alpha ||\psi||_{H^2_w}^2.
\end{align}

Let us compute each term in $A(\psi)$ to verify (\ref{AC.1}), referring to the definition in (\ref{Aofpsi}), starting with: 
\begin{align} \n
|\int \int \alpha \psi x^{2m} v x \chi_{L,\alpha}^2| &= |-\frac{\alpha}{2} \int \int \psi^2 \p_x[x^{2m+1} \chi_{L,\alpha}^2 ]| \\ \label{aboveme}
& = |-\frac{\alpha}{2} \int \int \psi^2 [C x^{2m} \chi_{L,\alpha}^2 + 2x^{2m+1} \frac{\alpha}{L} \chi_{L,\alpha} \chi'_{L,\alpha}  ]| \\
& \le \alpha ||\psi||_{H^2_w}^2. 
\end{align}

For the second term in (\ref{aboveme}), we have used: $|\frac{\alpha}{L} x \chi_{L,\alpha}| \lesssim 1$.  Next, let us turn to: 
\begin{align} \n
\alpha \int \int (-\psi_{yy} x^{2m+2} &+ \psi_{yyyy} x^{2m+4}) v x \chi_{L,\alpha}^2 \\ \n
& = \alpha \int \int u_y v x^{2m+3}  \chi_{L,\alpha}^2 - \alpha u_y u_{xy} x^{2m+5}\chi_{L,\alpha}^2 \\ \n
& = \alpha \int \int uu_x x^{2m+3}  \chi_{L,\alpha}^2 + \alpha u_y^2 \p_x( x^{2m+5}  \chi_{L,\alpha}^2 ) \\
& \lesssim \alpha \int \int u^2 x^{2m+2} + u_y^2 x^{2m+4} \lesssim \alpha ||\psi||_{H^2_w}^2. 
\end{align}

Next, 
\begin{align} \n
\alpha \int \int \p_x(\psi_x x^{2m+2}) &vx \chi_{L,\alpha}^2 = \alpha \int \int vx^{2m+2}  \p_x(v x \chi_{L,\alpha}^2) \\ \n
& = \alpha \int \int vx^{2m+2}[v_x x \chi_{L,\alpha}^2 + v\chi_{L,\alpha}^2 + 2v x \frac{\alpha}{L} \chi_{L,\alpha} \chi_{L,\alpha}' ] \\
& \lesssim \alpha \int \int v^2 x^{2m+2} \lesssim  \alpha ||\psi||_{H^2_w}^2. 
\end{align}

Next, 
\begin{align} \n
\alpha \int \int \p_x (\psi_{yyx} x^{2m+4}) &v x \chi_{L,\alpha}^2 = \alpha \int \int \p_x(\psi_{xy}x^{2m+4}) v_y x \chi_{L,\alpha}^2 \\
& = \alpha \int \int \p_x(u_x x^{2m+4} ) u_x x \chi_{L,\alpha}^2 \lesssim \alpha \int \int u_x^2 x^{2m+4} \lesssim \alpha ||\psi||_{H^2_w}^2. 
\end{align}

The final term in $A(\psi)$ is: 
\begin{align} \n
\alpha \int \int \p_{xx} (\psi_{xx} x^{2m+4}) v x \chi_{L,\alpha}^2 &= \alpha \int \int \psi_{xx} x^{2m+4} \p_{xx}[v x \chi_{L,\alpha}^2] \\ \n
& = \alpha \int \int v_x x^{2m+4} [v_{xx} x \chi_{L,\alpha}^2 + 2 v_x \p_x(x \chi_{L,\alpha}) + v \p_{xx}(x \chi_{L,\alpha}) ] \\
& \lesssim \alpha ||\psi||_{H^2_w}^2. 
\end{align}

This concludes all the terms in $A(\psi)$, according to (\ref{Aofpsi}). Estimating the next term in (\ref{pos.pos.1}) exactly as in (\ref{exact.NL}) yields: 
\begin{align} \n
|\int \int \eps^{\frac{n}{2}+\gamma} \bar{v} u_y u_x x \chi_{L,\alpha}^2| &\le \eps^{\frac{n}{2}+\gamma - \omega(N_i)} ||\bar{u}, \bar{v}||_Z ||u_y||_{L^2} ||u_x x^{\frac{1}{2}} \chi_{L,\alpha}||_{L^2} \\ & \le \eps^{\frac{n}{2}+\gamma - \omega(N_i)} ||u_x x^{\frac{1}{2}} \chi_{L,\alpha}||_{L^2}^2 + \eps^{\frac{n}{2}+\gamma - \omega(N_i)} ||u_y||_{L^2}^2. 
\end{align}

Finally, we come to the right-hand side: 
\begin{align} \n
\int \int [F_y - \eps G_x ] \cdot vx \chi_{L,\alpha} &= \int \int F u_x x \chi_{L,\alpha} + \eps G \p_x[v x \chi_{L,\alpha}] \\ \n
& \le \int \int |F| |u_x| x + |\int \int \eps G [v_x x \chi_{L,\alpha} + v \chi_{L,\alpha} + v x (\frac{\alpha}{L}) \chi'_{L,\alpha} ]| \\ 
& \le \int \int |F| |u_x| x + \eps |G| |v_x| x + \eps |G| |v|.,
\end{align}

Inserting the previous few calculations into estimate (\ref{pos.pos.1}) gives:
\begin{align} \n
||\{\sqrt{\epsilon}v_x, v_y \} x^{\frac{1}{2}} \chi_{L,\alpha}||_{L^2}^2 &\lesssim ||u_y||_{L^2}^2 + \frac{\alpha}{L}||\psi||_{H^2_w}^2 \\ & +  \alpha ||\psi||_{H^2_w}^2 + \int \int |F| |u_x| x + \eps |G| [ |v_x| x + |v|].
\end{align}

We now send $L \rightarrow \infty$, and appeal to Monotone Convergence Theorem, as $\chi_{L,\alpha} \uparrow 1$ to establish the desired result. 

\end{proof}

Having understood the inhomogeneous problem: 

\begin{lemma} \label{lem.wk} For $(F,G) \in H^{-1}_2$, and $||\bar{u}, \bar{v}||_Z \le 1$, there exists a unique weak solution $\psi \in H^2_w(\Omega^N)$ to the system (\ref{sysT}). 
\end{lemma}
\begin{proof}

We apply $S_\alpha^{-1}$ to both sides of (\ref{sysT}), which is valid as the right-hand side and therefore the left-hand side is assumed to be in at least $H^{-1}(\Omega^N)$, thereby yielding: 
\begin{align} \label{ImK}
\psi + S_\alpha^{-1}T \psi = S_\alpha^{-1} \Big(F_y - \eps G_x \Big).
\end{align}

We will study the equation (\ref{ImK}) as an equality in the space $H^2(\Omega^N)$. According to the Fredholm alternative, which is available according to Lemma \ref{lemma.cpct}, there either exists a unique solution $\psi \in H^2(\Omega^N)$ to the system (\ref{ImK}), or a non-trivial solution $\psi \in H^2(\Omega^N)$ to: 
\begin{align} \label{zero}
\psi + S_\alpha^{-1} T \psi = 0 \iff S_\alpha \psi = -T\psi. 
\end{align}

Therefore, coupling (\ref{bl.mct.1}) with (\ref{pos.we}), taking $F = G = 0$, we have:
\begin{align}
||\sqrt{\epsilon}u_x, u_y||_{L^2}^2 + \alpha ||\psi||_{H^2_w}^2 + ||\sqrt{\epsilon} v_x x^{\frac{1}{2}}, v_y x^{\frac{1}{2}}||_{L^2}^2 \le 0,
\end{align}

implying $\psi, u, v = 0$. Thus, by the Fredholm alternative, there exists a unique solution $\psi \in H^2(\Omega^N)$ to (\ref{ImK}). Rearranging (\ref{ImK}): 
\begin{align}
\psi = S_\alpha^{-1} \Big( F_y - \eps G_x - T\psi \Big), 
\end{align}

where $F_y - \eps G_x - T\psi \in H^{-1}(\Omega^N)$, and so an application of (\ref{EIdentity}) shows that $\psi \in H^2_w(\Omega^N)$. This concludes the proof. 

\end{proof}

\begin{lemma} Let $\psi$ be the unique $H^2_w(\Omega^N)$ weak solution from Lemma \ref{lem.wk}. Then for $(F,G) \in H^{-1}_2$, $\psi \in H^5_w(\Omega^N)$. 
\end{lemma}
\begin{proof}
$T \psi \in H^{-1}(\Omega^N)$, and so $S_\alpha \psi = - T\psi + F_y - \eps G_x \in H^{-1}(\Omega^N)$, which implies that $\psi \in H^3_w(\Omega^N)$ according to (\ref{h3psi}). Iterating this regularity then gives $\psi \in H^5_w(\Omega^N)$. 
\end{proof}

We now introduce more notation, which is more suitable for the velocities:
\begin{align}
L_{\alpha, \bar{v}}[u,v] = (\tilde{f}, g) \iff S_\alpha \psi + T[\psi; \bar{v}] = \tilde{f}_y - \eps g_x. 
\end{align} 

Summarizing the established results, we have: 
\begin{corollary} \label{CORBCS} For $\tilde{f}, g \in H^{-1}_2(\Omega^N)$, $\alpha > 0$, and $||\bar{u}, \bar{v}||_{Z(\Omega^N)} \le 1$, the map $L_{\alpha, \bar{v}}[u,v]$ is invertible, where 
\begin{align} \label{welldefined.}
L_{\alpha, \bar{v}}^{-1}: (\tilde{f}, g) \in H^{-1}_2(\Omega^N) \rightarrow [u,v] \in H^4_w(\Omega^N).
\end{align}

Moreover, the boundary conditions (\ref{psiBC}) are satisfied by $[u,v] = L_{\alpha, \bar{v}}^{-1}[\tilde{f},g]$.
\end{corollary}

It is now our intention to repeat the second and third order energy and positivity estimates from Section \ref{section.NSR.Linear}, with our new system (\ref{sysT}). For this, we will  need to understand several calculations. First, we introduce some norms: 
\begin{align} \label{norm.J.1}
&||\psi||_{J^2}^2 := ||\psi||_{H^2_w}^2 ,\\ \label{norm.J.2}
&||\psi||_{J^{k+2}}^2 := \int \int |\p_x^k \psi|^2 (\rho_{k+1})^{2k} x^{2m+2k} + |\nabla \p_x^k \psi|^2 \rho_{k+1}^{2k} x^{2m+2k+2} \\ \label{norm.J.3}
& \hspace{30 mm} + |\nabla^2 \p_x^k \psi|^2 \rho_{k+1}^{2k} x^{2m+2k+4} \text{ for $k \ge 1$.}
\end{align}

The reader is referred to the definitions of $\rho_k$ provided in (\ref{rho}). The essential difference between these $J^k$-norms and the $H^k_w$ norms introduced in (\ref{normHkw}) are the growing weights of $x$ which each application of $\p_x$, which mimics the structure of the energy norms, $X_k$, in (\ref{norm.x0}). 

\begin{lemma}
\begin{align} \label{SYO.2}
\int \int A(\p_x^k \psi) \cdot \p_x^k \psi x^{2k} \chi_{L,\alpha}^2 \rho_{k+1}^{2k} \gtrsim ||\chi_{L,\alpha} \psi||_{J^{k+2}}^2 - \sum_{i=0}^{k-1} ||\psi||_{J^{i+2}}^2 
\end{align}
\end{lemma}
\begin{proof}

Referring to (\ref{Aofpsi}), the first term is: 
\begin{align}
\int \int |\p_x^k \psi|^2 x^{2m} x^{2k} \chi_{L,\alpha}^2 \rho_{k+1}^{2k}.
\end{align}

The next terms, via an integration by parts in $y$: 
\begin{align}
\int \int -\p_x^k \psi_{yy} x^{2m+2} \cdot \p_x^k \psi x^{2k} \rho_{k+1}^{2k} \chi_{L,\alpha}^2 = \int \int |\p_x^k \psi_y|^2 x^{2m + 2k+2} \rho_{k+1}^{2k} \chi_{L,\alpha}^2, \\
\int \int \p_x^k \psi_{yyyy} x^{2m+4} \cdot \p_x^k \psi x^{2k} \rho_{k+1}^{2k} \chi_{L,\alpha}^2 = \int \int |\p_x^k \psi_{yy}|^2 x^{2m + 2k+4} \rho_{k+1}^{2k} \chi_{L,\alpha}^2. 
\end{align}

Next, 
\begin{align} \n
\int \int \p_x[ (\p_x^k &\psi)_{yyx} x^{2m+4} ] \cdot \p_x^k \psi x^{2k} \chi_{L,\alpha}^2 \rho_{k+1}^{2k} \\
& = \int \int \p_x^{k+1}\psi_y x^{2m+4} \cdot \p_x[ \p_x^k \psi_y  x^{2k} \chi^2_{L,\alpha} \rho_{k+1}^{2k} ] \\ & \gtrsim \int \int |\p_x^{k+1} \psi_y|^2 x^{2m+2k + 4} \chi_{L,\alpha}^2 \rho_{k+1}^{2k} - \int \int |\p_x^k \psi_y|^2 x^{2m+2k + 2} \rho_{k}^{2(k-1)} \\
& \gtrsim \int \int |\p_x^{k+1} \psi_y|^2 x^{2m+2k + 4} \chi_{L,\alpha}^2 \rho_{k+1}^{2k} - \sum_{i=0}^{k-1} ||\psi||_{J^{i+2}}. 
\end{align}

Above, we have used the calculation: 
\begin{align}\n
\p_x[x^{2k} \chi^2_{L,\alpha} \rho^{2k}_{k+1}] &= 2k x^{2k-1}\chi^2_{L,\alpha} \rho^{2k}_{k+1} + x^{2k} \frac{2 \alpha}{L} \chi_{L,\alpha}  \chi_{L,\alpha}' \rho_{k+1}^{2k} \\ \label{lirr.2}
& + x^{2k} \chi^2_{L,\alpha} 2k \rho^{2k-1}_{k+1} \rho_{k+1}'.
\end{align}

For the second term on the right-hand side of (\ref{lirr.2}), we estimate: $\frac{2\alpha}{L} x \chi_{L,\alpha} \lesssim 1$. For the third term on the right-hand side, we use that the support of $\rho_{k+1}'$ is localized in $x$. We also use that: $\text{support} (\rho_k ) \subset \{\rho_{k-1} = 1\}$.  Next, we integrate by parts twice in $x$ to obtain: 
\begin{align} \n
\int \int \p_{xx}&[ (\p_x^k\psi)_{xx} x^{2m+4}] \cdot \p_x^k \psi x^{2k} \chi_{L,\alpha}^2 \rho_{k+1}^{2k} \\
& = \int \int \p_x^{k+2} \psi x^{2m+4} \p_{xx}[  \p_x^k \psi x^{2k} \chi^2_{L,\alpha} \rho^{2k}_{k+1} ] \\ \n
& = \int \int |\p_x^{k+2} \psi|^2 x^{2m+2k+4} \chi^2_{L,\alpha} \rho^{2k}_{k+1} + \int \int \p_x^{k+2} \psi x^{2m+4} \p_x^{k+1} \psi \p_x[x^{2k} \chi^2_{L,\alpha} \rho^{2k}_{k+1}] \\ \label{lirr.3}
& + \int \int \p_x^{k+2} \psi x^{2m+4} \p_x^k \psi \p_{xx}[x^{2k} \chi^2_{L,\alpha} \rho^{2k}_{k+1}].
\end{align}

The final two terms on the right-hand side of (\ref{lirr.3}) are estimated through further integrations by parts: 
\begin{align} \n
|\int \int \p_x^{k+2} &\psi x^{2m+4} \p_x^{k+1} \psi \p_x[x^{2k} \chi^2_{L,\alpha} \rho^{2k}_{k+1}] + \int \int \p_x^{k+2} \psi x^{2m+4} \p_x^k \psi \p_{xx}[x^{2k} \chi^2_{L,\alpha} \rho^{2k}_{k+1}]| \\
& \lesssim ||\psi||_{J^{k+1}}^2.
\end{align}

Finally, 
\begin{align} \n
- \int \int \p_x[ (\p_x^k \psi)_x &x^{2m+2}] \cdot \p_x^k \psi x^{2k} \chi_{L,\alpha}^2 \rho_{k+1}^{2k} = \int \int \p_x^{k+1} \psi x^{2m+2} \p_x [\p_x^k \psi x^{2k}  \chi^2_{L,\alpha} \rho_{k+1}^{2k} ] \\
& = \int \int |\p_x^{k+1} \psi |^2 x^{2m+2+2k} \chi^2_{L,\alpha} \rho^{2k}_{k+1} \\
& \hspace{10 mm} + \int \int \p_x^{k+1}\psi x^{2m+2} \p_x^k \psi \p_x[ x^{2k} \chi^2_{L,\alpha}  \rho^{2k}_{k+1} ] \\
& \gtrsim \int \int |\p_x^{k+1} \psi |^2 x^{2m+2+2k} \chi^2_{L,\alpha} \rho^{2k}_{k+1} - ||\psi||_{J^{k+1}}^2. 
\end{align}

Piecing all of the above estimates together yields the desired bound. 

\end{proof}

\begin{lemma}
\begin{align} \label{SYO.2.P}
|\int \int A(\p_x^k \psi) \cdot \p_x^{k+1} \psi x^{2k+1} \chi_{L,\alpha}^2 \rho_{k+1}^{2k+1}| \lesssim \sum_{i=0}^k ||\psi||_{J^{i+2}}^2. 
\end{align}
\end{lemma}
\begin{proof}

Again, referring to definition (\ref{Aofpsi}), we will proceed term by term, starting with the following, for which we integrate by parts once: 
\begin{align} \label{lirr.4}
| \int \int \p_x^k \psi x^{2m} \cdot \p_x^{k+1} \psi x^{2k+1} \chi_{L,\alpha}^2 \rho_{k+1}^{2k+1} | =-\frac{1}{2} \int \int |\p_x^k \psi|^2  \p_x[x^{2m+2k+1} \chi_{L,\alpha}^2 \rho_{k+1}^{2k+1}]
\end{align}

Let us expand the product rule above: 
\begin{align} \n
\p_x[x^{2m+2k+1} \chi_{L,\alpha}^2 \rho_k^{2k+1}] &= C x^{2m+2k} \chi^2_{L,\alpha} \rho_{k+1}^{2k+1} + C x^{2m+2k+1} \frac{\alpha}{L} \chi_{L,\alpha} \chi'_{L,\alpha} \rho^{2k+1}_{k+1} \\
& + x^{2m+2k+1} \chi^2_{L,\alpha} \rho^{2k}_{k+1} \rho'_{k+1}  \lesssim x^{2m+2k} \rho_{k+1}^{2k}.
\end{align}

This, the term (\ref{lirr.4}) can be controlled via: 
\begin{align}
|(\ref{lirr.4})| \lesssim \int \int |\p_x^k \psi|^2 x^{2m+2k} \rho^{2k}_{k+1} \lesssim \sum_{i=0}^k ||\psi||_{J^{i+2}}^2. 
\end{align}

The second term in (\ref{Aofpsi}) is treated via: 
\begin{align} \n
\int \int - &\partial_x(\p_x^{k+1}\psi x^{2m+2})  \cdot \p_x^{k+1} \psi x^{2k+1} \chi_{L,\alpha}^2 \rho_{k+1}^{2k+1} \\ \n
& = \int \int [-\p_x^{k+2}\psi x^{2m+2} - C\p_x^{k+1}\psi x^{2m+1} ] \cdot  \p_x^{k+1} \psi x^{2k+1} \chi_{L,\alpha}^2 \rho_{k+1}^{2k+1} \\ \label{lirr.7}
& = \int \int |\p_x^{k+1}\psi|^2 \p_x[x^{2m+2k+3} \chi^2_{L,\alpha} \rho_{k+1}^{2k+1} ] -C |\p_x^{k+1}\psi|^2 x^{2m+2k+2} \chi^2_{L,\alpha} \rho_{k+1}^{2k+1} \\
& \lesssim \int \int |\p_x^{k+1} \psi|^2 x^{2m+2k+2} \rho^{2k}_{k+1} \lesssim \sum_{i = 0}^k||\psi||_{J^{i+2}}
\end{align}

We have expanded the product in the first term on the right-hand side of (\ref{lirr.7}):
\begin{align} \n
\p_x[x^{2m+2k+3} \chi^2_{L,\alpha} \rho_{k+1}^{2k+1}] &= Cx^{2m+2k+2} \chi^2_{L,\alpha} \rho^{2k+1}_{k+1} + Cx^{2m+2k+3} (\frac{\alpha}{L}) \chi_{L,\alpha} \chi'_{L,\alpha} \rho_{k+1}^{2k+1} \\ \n
& + Cx^{2m+2k+3} \chi^2_{L,\alpha} \chi^{2k+1}_{k+1} \rho_{k+1}^{2k} \rho_{k+1}' \\
& \lesssim x^{2m+2k+2} \rho^{2k}_{k+1}. 
\end{align}

Next, we have: 
\begin{align}  \label{lirr.5}
\int \int -\p_x^k \psi_{yy} x^{2m+2} \cdot \p_x^{k+1} \psi x^{2k+1} \chi_{L,\alpha}^2 \rho_{k+1}^{2k+1} = \int \int |\p_x^k \psi_y|^2 \p_x[ x^{2m+2k+3}  \chi_{L,\alpha}^2 \rho_{k+1}^{2k+1} ]
\end{align}

We will expand the product rule above: 
\begin{align} \n
\p_x[ x^{2m+2k+3}  \chi_{L,\alpha}^2 \rho_{k+1}^{2k+1} ] &= Cx^{2m+2k+2} \chi_{L,\alpha}^2 \rho_{k+1}^{2k+1} + C \frac{\alpha}{L} x^{2m+2k+3} \chi_{L,\alpha} \chi_{L,\alpha}' \rho_{k+1}^{2k+1} \\
& + C x^{2m+2k+3} \chi_{L,\alpha}^2 \rho_{k+1}^{2k} \rho_k'  \lesssim x^{2k+2m+2} \rho_{k+1}^{2k}. 
\end{align}

Inserting this above yields: $| (\ref{lirr.5})| \lesssim \sum_{i=0}^k ||\psi||_{J^{i+2}}^2$. Next, after two integrations by parts in $y$, and one in $x$:
\begin{align} \label{lirr.6}
\int \int \p_x^k \psi_{yyyy} x^{2m+4} \cdot \p_x^{k+1} \psi x^{2k+1} \rho_{k+1}^{2k+1} \chi_{L,\alpha}^2 = \int \int |\p_x^k \psi_{yy} |^2 \p_x[ x^{2m+2k+5}  \rho_{k+1}^{2k+1} \chi^2_{L,\alpha} ]
\end{align}

Expanding the product rule above yields: 
\begin{align} \n
\p_x[ x^{2m+2k+5}  \rho_{k+1}^{2k+1} \chi^2_{L,\alpha} ] &= C x^{2m+2k+4} \rho_{k+1}^{2k+1} \chi^2_{L,\alpha} + Cx^{2m+2k+5} \rho_{k+1}^{2k} \rho_{k+1}' \chi^2_{L,\alpha} \\
& +C x^{2m+2k+5}  \rho_{k+1}^{2k+1} \chi_{L,\alpha} \chi'_{L,\alpha} \frac{\alpha}{L} \lesssim x^{2m+2k+4} \rho_{k+1}^{2k} 
\end{align}

Inserting above yields: $|(\ref{lirr.6})| \lesssim \sum_{i=0}^k ||\psi||_{J^{i+2}}^2$.  The next term from $A(\psi)$ in definition (\ref{Aofpsi}) is: 
\begin{align} \n
\int \int &\p_x[ (\p_x^k \psi)_{yyx} x^{2m+4} ] \cdot \p_x^{k+1} \psi x^{2k+1} \rho_{k+1}^{2k+1} \chi_{L,\alpha}^2 \\ \n
& = - \int \int \p_x^{k+1} \psi_{y} x^{2m+4} \cdot \p_x[\p_x^{k+1} \psi_y x^{2k+1} \rho^{2k+1}_{k+1} \chi^2_{L,\alpha}  ] \\ \n
& = - \int \int  \p_x^{k+1} \psi_{y} x^{2m+4} \cdot \p_x^{k+2}  \psi_y x^{2k+1} \rho^{2k+1}_{k+1} \chi^2_{L,\alpha} \\ \n
& \hspace{30 mm} - \int \int |\p_x^{k+1}\psi_y|^2 x^{2m+4} \p_x[x^{2k+1} \rho^{2k+1}_{k+1} \chi^2_{L,\alpha}] \\ \n
& =  \int \int  |\p_x^{k+1} \psi_{y}|^2 \p_x[ x^{2m+2k+5} \rho^{2k+1}_{k+1} \chi^2_{L,\alpha} ]\\ \n
& \hspace{30 mm} - \int \int |\p_x^{k+1}\psi_y|^2 x^{2m+4} \p_x[x^{2k+1} \rho^{2k+1}_{k+1} \chi^2_{L,\alpha}] \\
& \lesssim \int \int |\p_x^{k+1}\psi_y|^2 x^{2m+2k+4} \rho_{k+1}^{2k} \lesssim \sum_{i=0}^k ||\psi||_{J^{i+2}}^2. 
\end{align}

The final term from $A(\psi)$ in definition (\ref{Aofpsi}) is: 
\begin{align}
\int \int \p_{xx}&(\p_x^{k+2} \psi x^{2m+4}) \cdot  \p_x^{k+1} \psi x^{2k+1} \rho_{k+1}^{2k+1} \chi_{L,\alpha}^2 \\ \n
& = - \int \int \p_x^{k+2} \psi x^{2m+4} \cdot \p_{xx}[\p_x^{k+1} \psi x^{2k+1} \rho_{k+1}^{2k+1} \chi_{L,\alpha}^2] \\ \n
& = - \int \int \p_x^{k+2} \psi x^{2m+4} \cdot \p_x^{k+3} \psi x^{2k+1} \rho_{k+1}^{2k+1} \chi_{L,\alpha}^2 \\ \n
& \hspace{30 mm} - \int \int |\p_x^{k+2}\psi|^2  x^{2m+4} \cdot  \p_x[x^{2k+1} \rho_{k+1}^{2k+1} \chi_{L,\alpha}^2] \\ \n
& \hspace{30 mm} - \int \int |\p_x^{k+1}\psi|^2 \p_{xxx}[x^{2m+2k+5} \rho_{k+1}^{2k+1}\chi^2_{L,\alpha}] \\
& \lesssim \sum_{i=0}^k ||\psi||_{J^{i+2}}^2.
\end{align}

This concludes the proof of the desired estimate, (\ref{SYO.2.P}). 

\end{proof}

\begin{lemma}
\begin{align} \label{SYO.3}
|\int \int [\p_x^k, A] \psi \cdot \p_x^k \psi x^{2k} \chi_{L,\alpha}^2 \rho_{k+1}^{2k}| \lesssim \sum_{i=0}^{k-1} ||\psi||_{J^{i+2}}^2. 
\end{align}
\end{lemma}
\begin{proof}

To keep notations simple, we will prove the $k = 1$ case, with the $k \ge 2$ cases following identically. We will proceed term by term from the commutator expression in (\ref{com.A}). First, 
\begin{align} \n
\int \int \psi x^{2m-1} \cdot \psi_x x^2 \chi^2_{L,\alpha} \rho_{2}^2 &= -\frac{1}{2} \int \int |\psi|^2 \p_x[x^{2m+1}  \chi^2_{L,\alpha} \rho^2_{2} ] \\
& \lesssim \int \int \psi^2 x^{2m} \lesssim ||\psi||_{J^2}^2. 
\end{align}

Next, 
\begin{align} \n
\int \int -\psi_{yy} x^{2m+1} \cdot \psi_x x^2 \chi^2_{L,\alpha} \rho_2^2 &= \int \int \psi_y^2 \p_x[x^{2m+3} \chi^2_{L,\alpha} \rho^2_2 ] \\ 
& \lesssim \int \int \psi_y^2 x^{2m+2} \lesssim ||\psi||_{J^2}^2. 
\end{align}

Next, 
\begin{align}
\int \int -\p_x (\psi_x x^{2m+1}) \cdot \psi_x x^2 \chi^2_{L,\alpha} \rho^2_2 \lesssim \int \int \psi_x^2 x^{2m+2} \lesssim ||\psi||_{J^2}^2. 
\end{align}

We will now move to the high order terms, starting with: 
\begin{align} \n
\int \int \psi_{yyyy} &x^{2m+3} \cdot \psi_x x^2 \chi^2_{L,\alpha} \rho_2^2 = \int \int \psi_{yy} \psi_{yyx} x^{2m+5} \chi^2_{L,\alpha} \rho_2^2 \\
& = -\frac{1}{2}\int \int |\psi_{yy}|^2 \p_x[x^{2m+5} \chi^2_{L,\alpha} \rho^2_2] \lesssim ||\psi||_{J^2}^2. 
\end{align}

Next, again integrating by parts several times: 
\begin{align} \n
\int \int &\p_x[\psi_{yyx} x^{2m+3}] \cdot \psi_x x^2 \chi^2_{L,\alpha} \rho_2^2 \\
& = \int \int \psi_{xy}^2 [x^{2m+3} \p_x[x^2 \chi^2_{L,\alpha} \rho^2_2 ] -\p_x[x^{2m+5} \chi^2_{L,\alpha} \rho^2_2 ] ] \lesssim ||\psi||_{J^2}^2. 
\end{align}

The final term from $A(\psi)$, which after integrating by parts several times in the same way as above, 
\begin{align}
\int \int \p_{xx}[\psi_{xx} x^{2m+3}] \cdot \psi_x x^2 \chi^2_{L,\alpha} \rho^2_2 \lesssim \int \int ||\psi||_{J^2}^2. 
\end{align}

This concludes the proof of (\ref{SYO.3}).

\end{proof}

\begin{lemma}
\begin{align} \label{SYO.4}
|\int \int [\p_x^k, A] \psi \cdot \p_x^{k+1} \psi x^{2k+1} \chi_{L,\alpha}^2 \rho_{k+1}^{2k+1}| \lesssim \sum_{i=0}^k ||\psi||_{J^{i+2}}.
\end{align}
\end{lemma}
\begin{proof}

This estimate proceeds in the same manner as those from (\ref{SYO.3}), with the adjustment that the extra derivative in the multiplier from (\ref{SYO.4}) is accounted for by the increment in order on the right-hand sides of (\ref{SYO.4}) versus (\ref{SYO.3}). Indeed, let us take the highest order term from the commutator, $[\p_x, A]\psi$:
\begin{align}
|\int \int \p_{xx}(\psi_{xx} x^{2m+3}) \cdot \psi_{xx} x^3 \chi^2_{L,\alpha} \rho_2^3| \lesssim \int \int |\psi_{xxx}|^2 x^{2m+6} \chi^2_{L,\alpha} \rho_2^3 + ||\psi||_{J^2}^2. 
\end{align}

The first term on the right-hand side above can be controlled by $||\psi||_{J^3}^2$, as can be seen from a comparison to (\ref{norm.J.3}) with $k = 1$. The remaining terms work identically. 

\end{proof}

Using the above calculations, we may repeat the energy and positivity estimates, for $k \ge 1$: 
\begin{lemma}[$k+1$'th order Auxiliary Energy Estimate] \label{lem.wk.e.2} Let $k = 1,2$. Then,
\begin{align} \n
||\p_x^k u_y \cdot (\rho_{k+1} x)^{k}||_{L^2}^2 + \alpha ||\psi||_{J^{k+2}}^2 &\lesssim \alpha \sum_{i = 0}^{k-1}||\psi||_{J^{i+2}}^2 + \mathcal{O}(\delta)||\p_x^k \{ \sqrt{\eps}v_x, v_y \} x^{k+\frac{1}{2}} \rho_{k+1}^{k+\frac{1}{2}}||_{L^2}^2 \\
& + \mathcal{W}_1 + \sum_{i=1}^k \mathcal{W}_{i+1}.
\end{align}
\end{lemma}
\begin{proof}

We apply the operator $\p_x^k$ to the system (\ref{sysT}):
\begin{align} \label{lirr.1}
\Delta_\eps \p_x^k \psi + \p_x^k T[\psi] +\alpha A(\p_x^k \psi) + \alpha [\p_x^k, A] \psi = \p_x^k \{F_y - \eps G_x\}.
\end{align}

We subsequently apply the multiplier $\p_x^k \psi x^{2k} \rho_{k+1}^{2k} \chi_{L,\alpha}^2$: 
\begin{align} \n
\int \int [\Delta_\eps \p_x^k \psi &+ \p_x^k T[\psi]  +\alpha A(\p_x^k \psi) + \alpha [\p_x^k, A] \psi] \cdot \p_x^k \psi x^{2k} \rho_{k+1}^{2k} \chi_{L,\alpha}^2 \\ 
& = \int \int [\p_x^k \{F_y - \eps G_x\}] \cdot \p_x^k \psi x^{2k} \rho_{k+1}^{2k} \chi_{L,\alpha}^2.
\end{align}

The desired estimate now follows using similar calculations as in Lemma \ref{lem.wk.e}. 

\end{proof}

\begin{lemma}[$k+1$'th order Auxiliary Positivity Estimate] \label{lem.wk.p.2}
\begin{align} 
||\p_x^k \{ \sqrt{\eps}v_x, v_y \} x^{k+\frac{1}{2}} \rho_{k+1}^{k+\frac{1}{2}}||_{L^2}^2 &\lesssim ||\p_x^k u_y \cdot (\rho_{k+1} x)^{k}||_{L^2}^2  + \alpha \sum_{i = 0}^{k} ||\psi||_{J^{i+2}}^2 + \mathcal{W}_1 + \sum_{i=1}^k \mathcal{W}_{i+1}.
\end{align}
\end{lemma}
\begin{proof}

We apply the multiplier $\p_x^{k+1} \psi x^{2k+1} \rho_{k+1}^{2k+1} \chi_{L,\alpha}^2 $ to the system (\ref{lirr.1}):
\begin{align} \n
\int \int [\Delta_\eps \p_x^k \psi &+ \p_x^k T[\psi]  +\alpha A(\p_x^k \psi) + \alpha [\p_x^k, A] \psi] \cdot \p_x^{k+1} \psi x^{2k+1} \rho_{k+1}^{2k+1} \chi_{L,\alpha}^2 \\ 
& = \int \int [\p_x^k \{F_y - \eps G_x\}] \cdot \p_x^{k+1} \psi x^{2k+1} \rho_{k+1}^{2k+1} \chi_{L,\alpha}^2.
\end{align}

The desired estimate now follows using similar calculations as in Lemma \ref{lem.wk.p}. 

\end{proof}

\section{Step 3: Nonlinear Existence of Auxiliary Systems}

For this subsection, it is necessary to be more precise with notation; we will index solutions by $(\alpha, N)$ and also specify domains over which norms are being taken. We shall also transition our right-hand sides from being generic ($F, G$) to being the particular right-hand sides of interest, $(\tilde{f}, g)$ as defined in $(\ref{tildef})$. Our intention now is to study the map, $M^\alpha$:
\begin{align} \n
M^\alpha&[\bar{u}^{\alpha, N}, \bar{v}^{\alpha, N}] = [u^{\alpha. N},v^{\alpha, N}] \\ \n
&\iff L_{\alpha, \bar{v}^{\alpha, N}}[u^{\alpha, N},v^{\alpha, N}] = \tilde{f}_y(\bar{u}^{\alpha, N}, \bar{v}^{\alpha, N}) - \eps g_x(\bar{u}^{\alpha, N}, \bar{v}^{\alpha, N}) \\ \label{Malphapre}
& \iff [u^{\alpha, N},v^{\alpha, N}] = L_{\alpha, \bar{v}^{\alpha, N}}^{-1} \{ \tilde{f}_y(\bar{u}^{\alpha, N}, \bar{v}^{\alpha, N}) - \eps g_x(\bar{u}^{\alpha, N}, \bar{v}^{\alpha, N}) \}.
\end{align}

which corresponds to the system written in vorticity form: 
\begin{align} \label{ns.N}
\Delta_\eps^2 \psi^{\alpha, N} + \alpha A(\psi^{\alpha, N}) + T(\psi^{\alpha, N}; \bar{v}^{\alpha, N}) = \tilde{f}_y(\bar{u}^{\alpha, N}, \bar{v}^{\alpha, N}) - \eps g_x(\bar{u}^{\alpha, N}, \bar{v}^{\alpha, N}) \hspace{3 mm} \text{ on } \Omega^N. 
\end{align}

A fixed point of (\ref{ns.N}) corresponds to the desired solution of (\ref{sysWFD}). By repeating the analysis in Section \ref{Section.Z},the energy and positivity estimates in Section \ref{section.NSR.Linear}, and finally the estimates on $\mathcal{W}_i$ in Lemma \ref{LemmaW} for the system, one obtains
\begin{lemma} Suppose $||\bar{u}^{\alpha, N}, \bar{v}^{\alpha, N}||_{Z(\Omega^N)} \le 1$. Fix any open set $B \subset \Omega^N$. Let $\alpha > 0$ and $N >> 1$. Solutions $\psi^{\alpha, N}$, or equivalently $[u^{\alpha, N},v^{\alpha, N}]$, to the system (\ref{ns.N}) satisfy the following estimates, independent of $N$, where $\omega(N_i)$ is based on universal constants: 
\begin{align} \label{Z.d.2}
\eps^{N_0} C(B) ||\psi^{\alpha, N}||_{H^5(B)} &+ ||u^{\alpha, N},v^{\alpha, N}||_{Z(\Omega^N)} \\ \n
&\lesssim \epsilon^{100} + ||u^{\alpha, N},v^{\alpha, N}||_{X_1 \cap X_2 \cap X_3(\Omega^N)} + \eps^{\frac{n}{2}+\gamma - \omega(N_i)}||\bar{u}^{\alpha, N}, \bar{v}^{\alpha, N}||_{Z(\Omega^N)}^2. 
\end{align}

The following energy and positivity estimates hold:
\begin{align} \label{Z.d.2.b}
\alpha ||\psi^{\alpha, N}||_{H^4_w(\Omega^N)}^2 +  ||u^{\alpha, N},v^{\alpha, N}||_{X_1 \cap X_2 \cap X_3(\Omega^N)}^2 \lesssim \mathcal{W}_1 + \mathcal{W}_2 + \mathcal{W}_3, 
\end{align}

Finally, one has: 
\begin{align} \n
\alpha||\psi^{\alpha, N}||_{H^4_w(\Omega^N)}^2 + \eps^{N_0} C(B) ||\psi^{\alpha, N}||_{H^5(B)}^2 &+ ||u^{\alpha, N},v^{\alpha, N}||_{Z(\Omega^N)}^2 \\  \label{MOT} 
& \lesssim \eps^{\frac{1}{4}-\gamma - \kappa} + \eps^{\frac{n}{2}-\omega(N_i)} ||\bar{u}^{\alpha, N}, \bar{v}^{\alpha, N}||_{Z(\Omega^N)}^4. 
\end{align}

All constants appearing in the above estimates are independent of $(\alpha, N)$. 
\end{lemma}

\begin{proof}[Proof of Estimate (\ref{Z.d.2})]

This follows by repeating the proofs of elliptic regularity in Subsection \ref{subsection.sing}, namely Lemmas \ref{LemmaBC} and \ref{LemmaBC2}, to the new system, (\ref{ns.N}). The only new term in (\ref{ns.N}) as compared to \ref{EQ.NSR.1} - (\ref{EQ.NSR.3}), (\ref{bar.f}) - (\ref{bar.g}) are $\alpha A(\psi)$. The proof of Lemmas \ref{LemmaBC} and \ref{LemmaBC2} then follows identically, as these are local-in-$x$ estimates, which are unaffected by the weights in $A(\psi)$. At this point, one repeats the estimates in Subsection \ref{subsection.EM}, which hold independent of any equation.

\end{proof}

\begin{proof}[Proof of Estimate \ref{Z.d.2.b}] This follows from Lemmas \ref{lem.wk.e.2} - \ref{lem.wk.p.2}, and subsequently comparing $||\cdot||_{J^k}$ with $||\cdot||_{H^k_w}$. 
\end{proof}

\begin{proof}[Proof of Estimate \ref{MOT}] This follows by repeating the proof of Lemma \ref{LemmaW}.

\end{proof}

Motivated by (\ref{MOT}), we define the notation:
\begin{align} \label{normF}
||u^{\alpha, N}, v^{\alpha, N}||_{\mathcal{F}(\Omega^N)} := \alpha||\psi^{\alpha, N}||_{H^4_w(\Omega^N)}^2 + ||u^{\alpha, N},v^{\alpha, N}||_{Z(\Omega^N)}^2.
\end{align} 

\begin{lemma}[Properties of $M^\alpha$] \label{L.weakstar} Fix any $\alpha > 0$ and any $N > 0$, and $\gamma, \kappa > 0$ arbitrarily small. 
\begin{align} \n
&(1) \hspace{5 mm} M^\alpha: B_Z(1) \subset Z(\Omega^N) \rightarrow B_Z(1) \subset Z(\Omega^N), \text{ where  $B_Z(1)$ is the unit ball in $Z(\Omega^N)$;} \\ \n
&(2) \hspace{5 mm} M^\alpha \text{ is continuous and compact as an operator on $B_Z(1)$}.  \\ \n
&(3) \hspace{5 mm} \text{There exists a fixed point, } [u^{\alpha, N},v^{\alpha, N}] = M^\alpha[u^{\alpha, N},v^{\alpha, N}] \text{ in } B_Z(1). \\ \label{weakstar}
&(4) \hspace{5 mm} \text{The fixed point satisfies, $||u^{\alpha, N}, v^{\alpha, N}||_{Z(\Omega^N)} \lesssim \eps^{\frac{1}{4}-\gamma - \kappa}$, independent of $\alpha, N$. }
\end{align}
\end{lemma}
\begin{proof}
The outline of this proof is as follows. The map $M^\alpha$ is shown to be well-defined in the appropriate domains and codomains, according to (1) above. Continuity of $M^\alpha$ is investigated by considering differences, and compactness of $M^\alpha$ is obtained using our compactness lemmas above. One then applies a fixed point argument to prove (3) and (4).

(1) Suppose $[\bar{u}, \bar{v}] \in Z(\Omega^N)$. This implies that $(\tilde{f},g) \in H^{-1}_2$, so by (\ref{welldefined.}), the map $M^\alpha$ is well-defined on $Z(\Omega^N)$. The property (\ref{clos.1}) is verified according to Lemma \ref{Lemma.weak.1}, and the definition of $H^2_w(\Omega^N)$, Definition \ref{defn.dens}, which ensures that $[u^\alpha, v^\alpha]$ are contained in $\overline{C^\infty_{0,D}}^{||\cdot||_{X_1}}$. Supposing the pre-images are contained in the unit ball of $Z(\Omega^N)$, $||\bar{u}^{\alpha, N}, \bar{v}^{\alpha, N}||_{Z(\Omega^N)} \le 1$, one has estimate (\ref{MOT}), which  implies that $M^\alpha(\bar{u}, \bar{v}) \in B_Z(1)$. 

(2) To check continuity of the map $M^\alpha$ on $B_Z(1)$, suppose:  
\begin{align}
M^\alpha[\bar{u}^{\alpha, N}_i, \bar{v}^{\alpha, N}_i] = [u^{\alpha,N}_i, v^{\alpha, N}_i] \text{ for $i = 1,2$,}
\end{align} 

where 
\begin{equation} \label{sf.2}
||\bar{u}^{\alpha, N}_i, \bar{v}^{\alpha, N}_i||_{Z} \le 1. 
\end{equation}

Define the notation for the differences,
\begin{align}
[\hat{\bar{\psi}}, \hat{\bar{u}}, \hat{\bar{v}}] &= [\bar{\psi}^{\alpha, N}_2 - \bar{\psi}^{\alpha, N}_1, \bar{u}^{\alpha, N}_{2} - \bar{u}^{\alpha, N}_{1}, \bar{v}^{\alpha, N}_{2} - \bar{v}^{\alpha, N}_{1}], \\
[\hat{\psi}, \hat{u}, \hat{v}] &= [\psi^{\alpha, N}_2 - \psi^{\alpha, N}_1, u^{\alpha, N}_{2} - u^{\alpha, N}_{1}, v^{\alpha, N}_{2} - v^{\alpha, N}_{1}].
\end{align}

By consulting (\ref{ns.N}), one then obtains the following system satisfied by the differences:
\begin{align} \n
\Delta_\eps^2 \hat{\psi} + \alpha A(\hat{\psi}) + T(\hat{\psi}) &= \tilde{f}_y(u^{\alpha, N}_2,\bar{u}^{\alpha, N}_2, \bar{v}^{\alpha, N}_2) - \tilde{f}_y(u^{\alpha, N}_1,\bar{u}^{\alpha, N}_1, \bar{v}^{\alpha, N}_1) \\ 
& - \eps g_x(\bar{u}^{\alpha, N}_2, \bar{v}^{\alpha, N}_2) + \eps g_x(\bar{u}^{\alpha, N}_1, \bar{v}^{\alpha, N}_1).
\end{align}

We may then repeat the estimates which resulted in (\ref{Z.d.2}) - (\ref{MOT}) to obtain: 
\begin{align} \label{sf.1}
||\hat{u}, \hat{v}||_{Z(\Omega^N)}^2 \lesssim \frac{1}{\alpha^2} ||\hat{\bar{u}}, \hat{\bar{v}}||_{Z(\Omega^N)}^2. 
\end{align}

The only non-trivial calculation when repeating the estimates which resulted in (\ref{Z.d.2}) - (\ref{MOT}) is to handle the nonlinearity (\ref{order.NL}) under taking differences. For this, we first write: 
\begin{align} \n
\bar{v}^{\alpha, N}_2 u^{\alpha, N}_{2y} &- \bar{v}^{\alpha, N}_1 u^{\alpha, N}_{1y} \\ \n
&= \bar{v}^{\alpha, N}_2 u^{\alpha, N}_{2y} - \bar{v}^{\alpha, N}_2 u^{\alpha, N}_{1y} + \bar{v}^{\alpha, N}_2 u^{\alpha, N}_{1y} - \bar{v}^{\alpha, N}_1 u^{\alpha, N}_{1y} \\
& = \bar{v}^{\alpha, N}_2 \hat{u}_y + \hat{\bar{v}} u^{\alpha, N}_{1,y}.
\end{align}

Repeating calculation (\ref{order.NL}) then yields: 
\begin{align} \n
\int \int \eps^{\frac{n}{2}+\gamma} [\bar{v}^{\alpha, N}_2 u^{\alpha, N}_{2y} - \bar{v}^{\alpha, N}_1 u^{\alpha, N}_{1y} ] \cdot \hat{u} &= \int \int  \eps^{\frac{n}{2}+\gamma} [\bar{v}^{\alpha, N}_2 \hat{u}_y + \hat{\bar{v}} u^{\alpha, N}_{1,y}] \cdot \hat{u} \\
& =- \int \int  \frac{\eps^{\frac{n}{2}+\gamma}}{2} \hat{u}^2 \bar{v}^{\alpha, N}_{2y} + \int \int \eps^{\frac{n}{2}+\gamma} \hat{\bar{v}} u^{\alpha, N}_{1,y} \hat{u} 
\end{align}

For the first term in the right-hand side above, we give the same estimate as in (\ref{order.NL}), which shows: 
\begin{align}
| \int \int  \frac{\eps^{\frac{n}{2}+\gamma}}{2} \hat{u}^2 \bar{v}^{\alpha, N}_{2y}| \lesssim \eps^{\frac{n}{2}+\gamma - \omega(N_i)} ||\hat{u}||_{Z(\Omega^N)}^2 ||\bar{u}^{\alpha, N}_2, \bar{v}^{\alpha, N}_2||_{Z(\Omega^N)} \lesssim \eps^{\frac{n}{2}+\gamma - \omega(N_i)} ||\hat{u}||_{Z(\Omega^N)}^2.
\end{align}

This then gets absorbed into the left-hand side of (\ref{sf.1}). For the second term on the right-hand side above, we estimate: 
\begin{align} \n
|\int \int \eps^{\frac{n}{2}+\gamma} \hat{\bar{v}} u^{\alpha, N}_{1,y} \hat{u} | &\le \eps^{\frac{n}{2}+\gamma} ||\hat{\bar{v}} x^{\frac{1}{2}}||_{L^\infty} ||u^{\alpha, N}_{1,y} x^m||_{L^2} || \hat{u} x^{-m-\frac{1}{2}}||_{L^2} \\
 \n & \lesssim  \eps^{\frac{n}{2}+\gamma - \omega(N_i)} ||\hat{\bar{v}}||_{Z(\Omega^N)} ||u^{\alpha, N}_{1}||_{H^2_w(\Omega^N)} ||\hat{u}||_{Z(\Omega^N)} \\ \n
 & \lesssim \eps^{\frac{n}{2}+\gamma - \omega(N_i)} ||\hat{\bar{v}}||_{Z(\Omega^N)} \frac{1}{\alpha}||u^{\alpha, N}_1||_{\mathcal{F}(\Omega^N)} ||\hat{u}||_{Z(\Omega^N)} \\ \n
 & \lesssim  \eps^{\frac{n}{2}+\gamma- \omega(N_i)} ||\hat{\bar{v}}||_{Z(\Omega^N)} \frac{1}{\alpha} ||\hat{u}||_{Z(\Omega^N)} \\ \label{sf.3}
 & \lesssim \eps^{2(\frac{n}{2}+\gamma - \omega(N_i))} ||\hat{u}||_{Z(\Omega^N)}^2 + \frac{1}{\alpha^2} ||\hat{\bar{v}}||_{Z(\Omega^N)}^2,
\end{align}

where we have used (\ref{sf.2}) coupled with (\ref{MOT}) to conclude that: $||u^{\alpha, N}_1||_{\mathcal{F}(\Omega^N)} \lesssim \eps^{\frac{1}{4}-\gamma - \kappa}$. The weight, $x^m$, arises from the definition (\ref{Aofpsi}), and consequently in (\ref{normHkw}). The first term on the right-hand side of (\ref{sf.3}) is absorbed into the left-hand side of (\ref{sf.1}), whereas the second term contributes to the right-hand side of (\ref{sf.1}). All of the remaining calculations which produced (\ref{MOT}) can be repeated in a similar fashion. Estimate (\ref{sf.1}) then implies the continuity of $M^\alpha$ on $B_Z(1)$. The modulus of continuity of $M^\alpha$ is $\frac{1}{\alpha^2}$, which prevents $M^\alpha$ from being a contraction map. Nevertheless, continuity is retained for all $\alpha > 0$. 

We now turn to compactness. According to Lemma \ref{lemma.cpct}, (\ref{MOT}) shows that $M^\alpha(B_Z(1))$ is compactly embedded in $B_Z(1)$ so long as $m$ is sufficiently large. 

(3 and 4) Consider the family of solutions: 
\begin{align}
[u^{\alpha, N}_\lambda, v^{\alpha, N}_\lambda] = \lambda M^{\alpha}[u^{\alpha, N}_\lambda, v^{\alpha, N}_\lambda], \hspace{5 mm} \text{ for } 0 \le \lambda \le 1. 
\end{align}

By (\ref{Malphapre}) and linearity of $L_\alpha^{-1}$, this occurs if and only if 
\begin{align}
[u^{\alpha, N}_\lambda, v^{\alpha, N}_\lambda] = L_\alpha^{-1}\{ \lambda \tilde{f}_y(u^{\alpha, N}_\lambda, v^{\alpha, N}_\lambda) - \eps \lambda g_x (u^{\alpha, N}_\lambda, v^{\alpha, N}_\lambda) \}.
\end{align}

By repeating the estimates which culminated in (\ref{MOT}), one sees the uniform in $\lambda$ bound: 
\begin{align}
||u^{\alpha, N}_\lambda, v^{\alpha, N}_\lambda||_{Z(\Omega^N)}^2 \lesssim \eps^{\frac{1}{4}-\gamma - \kappa}.
\end{align}

Thus, Schaefer's fixed point theorem applied to the convex subset $B_Z(1) \subset Z(\Omega^N)$ produces a fixed point, $[u^{\alpha, N}, v^{\alpha, N}] \in B_Z(1)$. The estimate it obeys follows from (\ref{MOT}). 
\end{proof}

\section{Step 4: Nonlinear Existence}

We now need to pass to the limit as $\alpha \rightarrow 0$ and as $ N \rightarrow \infty$. The fixed point of the system (\ref{ns.N}), from Lemma \ref{L.weakstar} satisfies the following integral identity for any $\phi \in C^\infty_0(\Omega^N)$:
\begin{align} \nonumber
&\int \int_{\Omega^N} \nabla_\epsilon^2 \psi^{N, \alpha} : \nabla_\epsilon^2 \phi + \alpha \Big[ \int \int_{\Omega^N} \psi^{N, \alpha} \phi x^{2m} +  \int \int_{\Omega^N} \nabla \psi^{N, \alpha} \cdot \nabla \phi x^{2m+2} \\ \n
& \hspace{20 mm} + \int \int_{\Omega^N} \nabla^2 \psi^{N, \alpha} : \nabla^2 \phi x^{2m+4} \Big] 
+ \int \int_{\Omega^N} -S_u \cdot \phi_y + \eps S_v \cdot \phi_x  \\ \n
& \hspace{20 mm} + \int \int_{\Omega^N} \eps^{-\frac{n}{2}-\gamma} \Big[ R^{u,n} \cdot \phi_y  - \eps R^{v,n} \cdot \phi_x \Big]   \\ \label{weak.3}
&= \int \int_{\Omega^N} \eps^{\frac{n}{2}+\gamma} \Big[ - u^{N, \alpha}u^{N, \alpha}_x \phi_y - v^{N, \alpha}u^{N, \alpha}_y \phi_y + \eps u^{N, \alpha}v^{N, \alpha}_x \phi_x + \eps v^{N, \alpha}v^{N, \alpha}_y \phi_x \Big]. 
\end{align}

First, we shall pass to the limit as $\alpha \rightarrow 0$, fixing an $N$. To do so, we first use (\ref{weakstar}) to obtain a weak subsequential limit point:
\begin{align}
u^{N,\alpha} \xrightharpoonup{\alpha \rightarrow 0} u^N, \text{ weakly in $(X_1 \cap X_2 \cap X_3)(\Omega^N)$}.
\end{align}

It is now our task to pass to the limit in the equation, (\ref{weak.3}), along the subsequence $\alpha \rightarrow 0$. Given a test-function, denote by $U_\phi$ to be the support of $\phi$. As $U_\phi$ is bounded, we have Poincare inequalities available:
\begin{align} \n
 \alpha |\Big[ \int \int_{\Omega^N}& \psi^{N, \alpha} \phi x^{2m} +  \int \int_{\Omega^N} \nabla \psi^{N, \alpha} \cdot \nabla \phi x^{2m+2} + \int \int_{\Omega^N} \nabla^2 \psi^{N, \alpha} : \nabla^2 \phi x^{2m+4} \Big]| \\  \n
& \le C(\phi) \alpha \Big[ ||\psi^{N, \alpha}||_{L^2(U_\phi)} + ||\nabla \psi^{N, \alpha}||_{L^2(U_\phi)} + ||\nabla^2 \psi^{N, \alpha}||_{L^2(U_\phi)}\Big] \\
& \le C(\phi) \alpha ||\nabla u^{N, \alpha}, \nabla v^{N,\alpha} ||_{L^2(U_\phi)} \le C(\phi) \alpha ||u^{N, \alpha}, v^{N, \alpha}||_{Z(\Omega^N)} \xrightarrow{\alpha \rightarrow 0} 0. 
\end{align}

For all of the linear terms, we use the weak convergence in $(X_1 \cap X_2 \cap X_3)(\Omega^N)$:
\begin{align} \n
 \lim_{\alpha \rightarrow 0} &\int \int_{\Omega^N} \nabla^2_\eps \psi^{\alpha, N} : \nabla_\eps^2 \phi - \int \int_{\Omega^N} S_u(u^{N,\alpha}, v^{N,\alpha}) \phi_y + \eps S_v(u^{N,\alpha}, v^{N,\alpha}) \cdot \phi_x  \\ 
 & = \int \int_{\Omega^N} \nabla^2_\eps \psi^N : \nabla_\eps^2 \phi - \int \int_{\Omega^N} S_u(u^N, v^N) \phi_y + \eps S_v(u^{N}, v^{N}) \cdot \phi_x.
\end{align}

Finally, we turn to the nonlinear terms for which we integrate by parts: 
\begin{align}
\int \int_{\Omega^N} u^{N, \alpha}u^{N, \alpha}_x \phi_y + v^{N, \alpha}u^{N, \alpha}_y \phi_y =  \int \int_{\Omega^N} - |u^{N, \alpha}|^2 \phi_{xy} - u^{N, \alpha}v^{N, \alpha} \phi_{yy}, \\
\int \int_{\Omega^N} u^{N, \alpha}v^{N, \alpha}_x \phi_x + v^{N, \alpha}v^{N, \alpha}_y \phi_x = \int \int_{\Omega^N} - |v^{N, \alpha}|^2 \phi_{xy} - u^{N, \alpha}v^{N, \alpha} \phi_{xx}.
\end{align}

Fixing a compactly supported $\phi$, we can localize the integrations above to $U_\phi$. On this set, the weak convergence of $u^{N,\alpha} \xrightharpoonup{X_1 \cap X_2 \cap X_3} u^N$ implies strong convergence in $L^2$. Thus, 
\begin{align} \n
| \int \int_{U_\phi} \Big[ &|u^{N,\alpha}|^2- u^{N,\alpha} u^N + u^{N,\alpha} u^N - |u^N|^2 \Big] \phi_{xy} \\ 
& \lesssim ||u^{N,\alpha} - u^N||_{L^2(U_\phi)} ||u^{N,\alpha}||_{L^2(U_\phi)} + ||u^N||_{L^2(U_\phi)} ||u^{N,\alpha} - u^N||_{L^2(U_\phi)}. 
\end{align}

The right-hand side converges to zero. The same bound works for all of the other nonlinear terms. Thus, the weak limit $[u^N,v^N]$ or equivalently $\psi^N$ satisfies the weak formulation: 
\begin{align} \nonumber
&\int \int_{\Omega^N} \nabla_\epsilon^2 \psi^N : \nabla_\epsilon^2 \phi - \int \int_{\Omega^N} S_u(u^N, v^N) \cdot \phi_y + \eps S_v(u^N, v^N) \cdot \phi_x  \\ \n
& \hspace{50 mm} + \int \int_{\Omega^N} \eps^{-\frac{n}{2}-\gamma} \Big[ R^{u,n} \cdot \phi_y  - \eps R^{v,n} \cdot \phi_x \Big]   \\ \label{weak.4}
&= \int \int_{\Omega^N} \eps^{\frac{n}{2}+\gamma} \Big[ -u^Nu^N_x \phi_y - v^Nu^N_y \phi_y + \eps u^Nv^N_x \phi_x + \eps v^Nv^N_y \phi_x \Big]. 
\end{align}

The weak limit $[u^N,v^N]$ must satisfy the bound: 
\begin{align} \label{UNIN}
||u^N,v^N||_{(X_1 \cap X_2 \cap X_3)(\Omega^N)} \lesssim C(u_R, v_R) \eps^{\frac{1}{4}-\gamma - \kappa},
\end{align}

independent of $N$. We may now repeat this exact procedure with the subsequential $N$ limit: denote by $[u,v]$ and $\psi$ the subsequential  $(X_1 \cap X_2 \cap X_3)(\Omega)$-weak limit as $N \rightarrow \infty$, guaranteed by (\ref{UNIN}). One then passes to the limit in the equation (\ref{weak.4}) to obtain: 
\begin{align} \nonumber
&\int \int_{\Omega} \nabla_\epsilon^2 \psi : \nabla_\epsilon^2 \phi - \int \int_{\Omega} S_u(u, v) \cdot \phi_y + \eps S_v(u, v) \cdot \phi_x  \\ \n
& \hspace{50 mm} + \int \int_{\Omega} \eps^{-\frac{n}{2}-\gamma} \Big[ R^{u,n} \cdot \phi_y  - \eps R^{v,n} \cdot \phi_x \Big]   \\ \label{weak.5.1}
&= \int \int_{\Omega} \eps^{\frac{n}{2}+\gamma} \Big[ - uu_x \phi_y- vu_y \phi_y + \eps uv_x \phi_x + \eps vv_y \phi_x \Big],
\end{align}

with the limit satisfying: 
\begin{align} \label{formal}
||u,v||_{(X_1 \cap X_2 \cap X_3)(\Omega)} \lesssim \eps^{\frac{1}{4}-\gamma - \kappa}.
\end{align}

We now state the main existence result:
\begin{theorem} \label{thm.existence} For $\eps, \delta$ sufficiently small, $\kappa > 0$ small, and $0 \le \gamma < \frac{1}{4}$,  there exists a solution to the system (\ref{EQ.NSR.1}) - (\ref{EQ.NSR.3}), (\ref{nsr.bc.1}), (\ref{defn.SU.SV}) satisfying: 
\begin{align} \label{W.Z.2}
||u,v||_{Z(\Omega)} \lesssim C(u_R, v_R) \eps^{\frac{1}{4}-\gamma - \kappa}. 
\end{align}
\end{theorem}
\begin{proof}

Estimate (\ref{formal}) implies enough regularity to integrate by parts identity (\ref{weak.3}) to: 
\begin{align}
\int \int_{\Omega} \Big[\Delta_\eps^2 \psi + \p_y S_u - \eps \p_x S_v - \p_y f + \eps \p_x g \Big] \cdot \phi = 0,
\end{align}

which then implies that the PDE is satisfied pointwise in $\Omega$. The boundary conditions (\ref{nsr.bc.1}) are satisfied by elements in $(X_1 \cap X_2 \cap X_3)(\Omega)$, according to Lemma \ref{L.XiBC}. From here, one repeats the embedding theorems in Section \ref{Section.Z} which give estimate (\ref{W.Z.2}). That this is possible for those embeddings in Subsection \ref{subsection.EM} is straightforward to see, as these did not require $[u,v]$ to satisfy any equations. Let us then turn to Subsection \ref{subsection.sing}. We must repeat the proofs of Lemmas \ref{LemmaBC} and \ref{LemmaBC2} to the \textit{nonlinear} system, (\ref{EQ.NSR.1}) - (\ref{EQ.NSR.3}), with $f, g$ as in (\ref{defn.SU.SV}). This amounts to replacing $[\bar{u}, \bar{v}]$ with $[u, v]$ in Lemmas \ref{LemmaBC} and \ref{LemmaBC2}, and foregoing the assumption that $||\bar{u}, \bar{v}||_Z \le 1$. A nearly identical proof to Lemma \ref{LemmaBC} then yields: 
\begin{align}
\sup_{x \le 2000} ||u,v||_{L^\infty_y} &+ ||u,v||_{\dot{H}^2(x \le 2000)} \lesssim \eps^{-M_2}. 
\end{align}

One now bootstraps the estimate in Lemma \ref{LemmaBC2} in the identical manner. This gives estimate (\ref{W.Z.2}). We have verified that $[u,v] \in Z(\Omega)$ satisfies (\ref{EQ.NSR.1}) - (\ref{EQ.NSR.3}), (\ref{nsr.bc.1}), (\ref{defn.SU.SV}). 
\end{proof}

\section{Step 5: Uniqueness} \label{sub.sec.U}

In this final subsection, we prove uniqueness of the solution $[u,v]$ from Theorem \ref{thm.existence}. Suppose there existed two solutions, $[u_1, v_1]$ and $[u_2, v_2]$ to the system in (\ref{EQ.NSR.1}) - (\ref{EQ.NSR.3}), (\ref{nsr.bc.1}), (\ref{defn.SU.SV}). Define: 
\begin{equation}
\hat{u} = u_1 - u_2, \hspace{3 mm} \hat{v} = v_1 - v_2, \hspace{3 mm} \hat{P} = P_1 - P_2.
\end{equation}

Then the new unknowns satisfy: 
\begin{align} \label{hats.1}
-\Delta_\epsilon \hat{u} + S_u(\hat{u}, \hat{v}) + \hat{P}_{x} = \hat{f} := \eps^{\frac{n}{2}+\gamma} \Big[ u_{1}u_{1x} - u_2 u_{2x} + v_1 u_{1y} - v_2 u_{2y} \Big] ,\\ \label{hats.2}
- \Delta_\epsilon \hat{v} + S_v(\hat{u}, \hat{v}) + \frac{\hat{P}_y}{\epsilon} = \hat{g} :=  \eps^{\frac{n}{2}+\gamma} \Big[ u_1 v_{1x} - u_2 v_{2x} + v_1 v_{1y} -  v_2 v_{2y} \Big],
\end{align}

together with the divergence-free condition, $\hat{u}_x + \hat{v}_y = 0$, and also satisfy the boundary conditions:
\begin{align} \label{hats.bc.1}
\{\hat{u}, \hat{v}\}|_{\{y=0\}} =   \{\hat{u}, \hat{v}\}|_{\{x = 1\}} = 0.
\end{align}

Going to vorticity, 
\begin{align} \n
\p_y \Big[ -\Delta_\eps \hat{u} &+ S_u(\hat{u}, \hat{v}) \Big] - \eps \p_x \Big[ -\Delta_\eps v + S_v(\hat{u}, \hat{v}) \Big] = \eps^{\frac{n}{2}} \Big\{ \p_y \Big[ u_{1}u_{1x} \\ - \label{hats.v}
&  u_2 u_{2x} + v_1 u_{1y} - v_2 u_{2y}\Big] - \eps \p_x \Big[ u_1 v_{1x} - u_2 v_{2x} + v_1 v_{1y} -  v_2 v_{2y}\Big] \Big\}.
\end{align}

We shall repeat the basic energy and positivity estimates using a slightly weaker weight. It is convenient to work with the weak formulation, which is given in (\ref{weak.5.1}). Then, $\hat{u}, \hat{v}$ satisfy the following:  
\begin{align}  \label{weak.5}
\int \int \nabla_\epsilon^2 \hat{\psi} : \nabla_\epsilon^2 \phi + \int \int \eps S_v(\hat{u}, \hat{v})  \cdot \phi_x - S_u(\hat{u}, \hat{v}) \cdot \phi_y = \int \int\Big[ -\hat{f} \phi_y + \eps \hat{g} \phi_x \Big],
\end{align}

for all $\phi \in C_0^\infty(\Omega)$. We make the notational convention that 
\begin{align}
\int \int := \int \int_{\Omega}.
\end{align}

\begin{lemma} \label{L1U} There exists a $0 < b < 1$, sufficiently close to $0$, depending only on universal constants, such that for $\delta, \eps$ sufficiently small and $\eps << \delta << b$, the solutions $[\hat{u}, \hat{v}] \in Z$ to the system (\ref{hats.1}) - (\ref{hats.2}) with boundary conditions (\ref{hats.bc.1}) satisfy the following estimate: 
\begin{align} \label{hat.Energy}
b||\{\hat{u}, \sqrt{\eps} \hat{v}\} x^{-b-\frac{1}{2}}||_{L^2}^2 + ||\hat{u}_y x^{-b}||_{L^2}^2 \lesssim \mathcal{O}(\delta) ||\{\sqrt{\eps}\hat{v}_x, \hat{v}_y\} x^{\frac{1}{2}-b}||_{L^2}^2 + \mathcal{W}_{1,E,b},
\end{align}
where
\begin{align} \label{calw1b.E}
&\mathcal{W}_{1,E,b} :=  \int \int \hat{f} \hat{u} x^{-2b} + \eps \hat{g} \hat{v} x^{-2b} - 2b \eps \hat{g} \hat{\psi} x^{-2b-1}, \\ \label{calw1b.P}
&\mathcal{W}_{1,P,b} :=  \int \int \hat{f} \hat{u}_x x^{1-2b} + \eps \hat{g}\hat{v}_x x^{1-2b}, \\ \label{calw1b}
&\mathcal{W}_{1,b} = \mathcal{W}_{1,E,b} +  \mathcal{W}_{1,P,b}. 
\end{align}
\end{lemma}

\begin{proof}

The estimate will follow upon applying the multiplier $\hat{\psi} \cdot x^{-2b}$ to the system in (\ref{hats.v}). To work rigorously, we will apply approximate multipliers, and work with the weak formulation given in (\ref{weak.5}). Fix $[\hat{u}^{{n}}, \hat{v}^{(n)}, \hat{\psi}^{(n)}] \in C^\infty_0(\Omega)$, such that: 
\begin{align} \label{cnrs.1}
[\hat{u}^{(n)}, \hat{v}^{(n)}] \xrightarrow{X_1} [\hat{u}, \hat{v}], 
\end{align}

where $X_1$ is defined in (\ref{norm.x0}). Within the notation of (\ref{weak.5}), $\phi = \hat{\psi}^{(n)}x^{-2b}$. The existence of the sequence specified in (\ref{cnrs.1}) is guaranteed by $[\hat{u}, \hat{v}] \in Z(\Omega)$. That $\phi$ is compactly supported in $(x,y)$ follows from the representations:   
\begin{equation}
\hat{\psi}^{(n)} = -\int_0^y \hat{u}^{(n)} = \int_0^x \hat{v}^{(n)}. 
\end{equation}

Let us first treat the second-order terms: 
\begin{align} \n
\int \int  \nabla_\eps^2 &\hat{\psi} : \nabla_\eps^2 \phi = \int \int \nabla_\eps^2 \hat{\psi} : \nabla_\eps^2 (\hat{\psi}^{(n)} x^{-2b}) \\ \label{BIG2}
& = \int \int \hat{\psi}_{yy} \hat{\psi}^{(n)}_{yy} x^{-2b} + 2\eps \hat{\psi}_{xy} \p_x \Big( \hat{\psi}^{(n)}_y x^{-2b} \Big) + \eps^2 \hat{\psi}_{xx} \p_{xx} \Big( \hat{\psi}^{(n)} x^{-2b} \Big). 
\end{align}

The first two terms from (\ref{BIG2}) above are: 
\begin{align} \n
\int \int \hat{\psi}_{yy} \hat{\psi}^{(n)}_{yy} x^{-2b} &+ 2\eps \hat{\psi}_{xy} \hat{\psi}^{(n)}_{xy}x^{-2b} + 2\eps \hat{\psi}_{xy} \hat{\psi}^{(n)}_y \p_x x^{-2b} \\  \label{h.UL}
&=  \int \int \hat{u}_y \hat{u}^{(n)}_y x^{-2b} + 2\eps \hat{u}_x \hat{u}^{(n)}_x x^{-2b} -2 \eps \hat{u}_x \hat{u}^{(n)} \p_x x^{-2b}.
\end{align}

We shall take the limit as $n \rightarrow \infty$ above. According to the definition (\ref{norm.x0}), the convergence in (\ref{cnrs.1}) implies:
\begin{align} \n
|\int \int \hat{u}_y (\hat{u}^{(n)}_y - \hat{u}_y) x^{-2b}| &+ | \int \int \hat{u}_x (\hat{u}^{(n)}_x - \hat{u}_x) x^{-2b}| \\
& + | \int \int \hat{u}_x (\hat{u}^{(n)} - \hat{u}) \p_x x^{-2b} | \xrightarrow{n \rightarrow \infty} 0.
\end{align}

Expanding the third term from (\ref{BIG2}), 
\begin{align} \label{h.VL}
\int \int \eps^2 \hat{\psi}_{xx} \p_{xx} \Big( \hat{\psi}^{(n)} x^{-2b} \Big) = \int \int \eps^2 \hat{v}_x \cdot \Big[ \hat{v}^{(n)}_x x^{-2b} +2 \hat{v}^{(n)} \p_x x^{-2b} + \hat{\psi}^{(n)} \p_{xx} x^{-2b} \Big].
\end{align}

By referring to the definition of $X_1$ in (\ref{norm.x0}) and (\ref{cnrs.1}), we may pass to the limit: 
\begin{align} \n
\text{Equation }& (\ref{h.UL}) \xrightarrow{n \rightarrow \infty} \int \int[ \hat{u}_y^2 + 2\eps \hat{u}_x^2] x^{-2b} - \int \int \eps \hat{u}_x \hat{u} \p_x x^{-2b} \\ \n
& = \int \int[ \hat{u}_y^2 + 2\eps \hat{u}_x^2] x^{-2b} + b(2b+1) \eps \hat{u}^2 x^{-2-2b} + \eps b \lim_{M \rightarrow \infty} \int_{x = M} \hat{u}^2 x^{-1-2b} \\
& = \int \int[ \hat{u}_y^2 + 2\eps \hat{u}_x^2] x^{-2b} + b(2b+1) \eps \hat{u}^2 x^{-2-2b},
\end{align}

and: 
\begin{align} \label{rigL.1}
\text{ Equation } (\ref{h.VL}) \xrightarrow{n \rightarrow \infty} \int \int \eps^2 \hat{v}_x^2 x^{-2b} - 4b \eps^2 \hat{v}_x \hat{v} x^{-1-2b} + 2b(2b+1) \eps^2 \hat{v}_x \hat{\psi} x^{-2-2b}. 
\end{align}

Integrating by parts the final two terms above in (\ref{rigL.1}), and referring to estimate (\ref{evo.low}),
\begin{align} \n
-4b \int \int  \eps^2 \hat{v}_x \hat{v} x^{-1-2b} &= \int \int 2b \eps^2 \hat{v}^2 \p_x x^{-1-2b} - 2b \lim_{M \rightarrow \infty} \int_{x = M}\eps^2 \hat{v}^2 x^{-1-2b} \\ \label{fall.1}
& =  -2b (1+2b) \int \int \eps^2 \hat{v}^2  x^{-2-2b},
\end{align}

and similarly, to treat the final term in (\ref{rigL.1}), we appeal to the estimates in (\ref{evo.low}):
\begin{align} \n
 \int \int \eps^2 \hat{v}_x \hat{\psi} x^{-2-2b} &= - \int \int  \eps^2 \hat{v}^2 x^{-2-2b} -  \int \int \eps^2 \hat{v} \hat{\psi} \p_x x^{-2-2b} \\ 
& \hspace{40 mm} + \lim_{M \rightarrow \infty} \int_{x = M}  \eps^2 \hat{v} \hat{\psi} x^{-2-2b} \\ \n
& =  - \int \int  \eps^2 \hat{v}^2 x^{-2-2b} + \int \int \frac{(2b+3)(2b+2)}{2} \eps^2 \hat{\psi}^2 x^{-4-2b} \\ 
& \hspace{40 mm} + \lim_{M \rightarrow \infty} \int_{x = M} \frac{2b+2}{2} \eps^2 \hat{\psi}^2 x^{-3-2b} \\ \label{fall.2}
& =  - \int \int \eps^2 \hat{v}^2 x^{-2-2b} + \int \int \frac{(2b+3)(2b+2)}{2}  \eps^2 \hat{\psi}^2 x^{-4-2b}. 
\end{align}

Therefore, summarizing the highest order calculation: 
\begin{align} \n
 \int \int \nabla_\eps^2 \hat{\psi} : \nabla_\eps^2 (\hat{\psi}^{(n)} x^{-2b}) \gtrsim \int \int &[\hat{u}_y^2 x^{-2b} + 2\eps \hat{u}_x^2 + \eps^2 \hat{v}_x^2] x^{-2b} \\ \label{UQ.1}
 & - \int \int [\eps^2 \hat{v}^2 + \eps \hat{u}^2] x^{-2-2b} + \eps^2 \hat{\psi}^2 x^{-4-2b} \\ \label{UQ.1.cnrs}
  \gtrsim \int \int &\hat{u}_y^2 x^{-2b} - C \int \int \eps^2 \hat{v}_x^2 x^{-2b} - C \int \int \eps \hat{u}_x^2 x^{-2b}.
\end{align}

To go from (\ref{UQ.1}) to (\ref{UQ.1.cnrs}), we have used the Hardy inequality in the $x$-direction. We will now address the profile terms arising from $S_u(\hat{u}, \hat{v})$ in the weak formulation (\ref{weak.5}), whose definition has been given in (\ref{defn.SU.SV}): 
\begin{align} \n
- \int \int \Big[u_R \hat{u}_x &+ u_{Rx}\hat{u} + u_{Ry}\hat{v} + v_R \hat{u}_y \Big] \cdot \p_y \phi  \\ \n
= -\int \int  \Big[u_R &\hat{u}_x + u_{Rx}\hat{u} + u_{Ry}\hat{v} + v_R \hat{u}_y \Big] \cdot \p_y \hat{\psi}^{(n)} x^{-2b} \\ \label{Suhat}
& = \int \int \Big[u_R \hat{u}_x + u_{Rx}\hat{u} + u_{Ry}\hat{v} + v_R \hat{u}_y \Big] \cdot \hat{u}^{(n)} x^{-2b}. 
\end{align}

We will first pass to the limit in (\ref{Suhat}), using the definition of $X_1$ in (\ref{norm.x0}), which gives: 
\begin{align} \label{Suhat.1}
(\ref{Suhat}) \xrightarrow{n \rightarrow \infty}  \int \int \Big[u_R \hat{u}_x + u_{Rx}\hat{u} + u_{Ry}\hat{v} + v_R \hat{u}_y \Big] \cdot \hat{u} x^{-2b}.
\end{align}

We proceed to treat each term in (\ref{Suhat.1}), starting with: 
\begin{align} \n
\int \int u_R \hat{u}_x \hat{u} x^{-2b} &= - \int \int \hat{u}^2 \frac{\p_x}{2} \Big(u_R x^{-2b} \Big) + \lim_{M \rightarrow \infty} \int_{x = M} \hat{u}^2 x^{-2b} \\ \n
& = - \int \int \hat{u}^2 \Big( u_{Rx} x^{-2b} - 2b u_R x^{-2b-1} \Big)  \\  \n
& \gtrsim - ||u_{Rx} x||_{L^\infty} \int \int \hat{u}^2 x^{-2b-1} + 2b \min u_R \int \int \hat{u}^2 x^{-2b-1} \\
& \gtrsim b \int \int \hat{u}^2 x^{-2b-1},
\end{align}

according to estimates (\ref{PE0.4}), (\ref{PE4}), so long as $\delta$ is taken small relative to $b$. For the $M$-limit above, we have used estimate (\ref{evo.low}), which is valid so long as $b > 0$. For the second term in (\ref{Suhat.1}), we again appeal to estimates (\ref{PE0.4}), (\ref{PE5}):
\begin{align} \n
| \int u_{Rx} \hat{u}^2 x^{-2b} | &\lesssim ||u_{Rx} x||_{L^\infty} ||\hat{u} x^{-2b-\frac{1}{2}}||_{L^2}^2  \lesssim \mathcal{O}(\delta)  ||\hat{u} x^{-b-\frac{1}{2}}||_{L^2}^2.
\end{align}

For the third term, we shall split $u_R = u^{n-1,p}_R + \eps^{\frac{n}{2}} u^n_{pR} + u^E_R$. First, we apply estimate (\ref{PE0.5}): 
\begin{align} \n
| \int \int u^{P,n-1}_{Ry} \hat{v} \hat{u} x^{-2b} | &\le ||y^2 x^{-\frac{1}{2}} u^{P,n-1}_{Ry}||_{L^\infty} ||\frac{\hat{u}}{y} x^{-b}||_{L^2} ||\frac{\hat{v}}{y} x^{\frac{1}{2}-b}||_{L^2} \\
& \le \mathcal{O}(\delta) ||\hat{u}_y x^{-b}||_{L^2} ||\hat{v}_y x^{\frac{1}{2}-b}||_{L^2}. 
\end{align}

Next, for $\sigma_n$ as in (\ref{sigma.i}), according to estimate (\ref{PE3}),
\begin{align} \n
| \int \int u^{P,n}_{Ry} \hat{v} \hat{u} x^{-2b} | &\le \eps^{\frac{n}{2}} ||u^n_{py} yx^{\frac{1}{2}-\sigma_n} ||_{L^\infty} ||\hat{u}x^{-1+\sigma_n - b} ||_{L^2} ||\frac{\hat{v}}{y} x^{\frac{1}{2}-b}||_{L^2} \\
& \lesssim \eps^{\frac{n}{2}}  ||\hat{u}_x x^{\sigma_n - b} ||_{L^2} ||\hat{v}_y x^{\frac{1}{2}-b}||_{L^2} \lesssim \eps^{\frac{n}{2}} \mathcal{O}(\delta)  ||\hat{v}_y x^{\frac{1}{2}-b}||_{L^2}^2. 
\end{align}

Finally, the Eulerian contribution is handled by an application of (\ref{PE5}):
\begin{align} \n
| \int \int \sqrt{\eps} u^E_{RY} \hat{u} \hat{v} x^{-2b}| &\le \sqrt{\eps} ||u^E_{RY} x^{\frac{3}{2}}||_{L^\infty} ||\frac{\hat{u}}{x^{\frac{3}{4}+b}} ||_{L^2} ||\frac{\hat{v}}{x^{\frac{3}{4}+b}}||_{L^2} \\
& \lesssim \sqrt{\eps} ||\hat{u}_x x^{\frac{1}{4}-b}||_{L^2} ||\sqrt{\eps}\hat{v}_x x^{\frac{1}{4}-b}||_{L^2}.
\end{align}

The fourth term from (\ref{Suhat}), upon using estimate (\ref{PE0.4}) and (\ref{PE5}), reads:
\begin{align}
| \int \int v_R \hat{u}_y \hat{u} x^{-2b}| = | \int \int \frac{v_{Ry}}{2} \hat{u}^2 x^{-2b} | \lesssim ||u_{Rx} x||_{L^\infty} ||\hat{u} x^{-\frac{1}{2}-b}||_{L^2}^2 \lesssim \mathcal{O}(\delta) ||\hat{u}x^{-\frac{1}{2}-b}||_{L^2}^2. 
\end{align}

Summarizing these calculations, 
\begin{align} \n
| (\ref{Suhat.1})| &\gtrsim b ||\hat{u} x^{-\frac{1}{2}-b}||_{L^2}^2 - \mathcal{O}(\delta) ||\hat{u} x^{-\frac{1}{2}-b}||_{L^2}^2 - \mathcal{O}(\delta) ||\hat{u}_y x^{-b}||_{L^2}^2 \\ \n
& \hspace{27 mm} - \mathcal{O}(\delta) ||\{\sqrt{\eps}v_x \hat{v}_y \} x^{\frac{1}{2}-b}||_{L^2}^2 \\ \label{goodB}
& \gtrsim b ||\hat{u} x^{-\frac{1}{2}-b}||_{L^2}^2 - \mathcal{O}(\delta) ||\{\sqrt{\eps}v_x, v_y \} x^{\frac{1}{2}-b}||_{L^2}^2. 
\end{align}

We have absorbed the $\hat{u}_y$ terms into (\ref{UQ.1}), and taken $\delta$ sufficiently small relative to $b$. We shall now address the profile terms from $S_v$: 
\begin{align} \n
\int \int \eps S_v(\hat{u}, \hat{v}) \cdot \hat{\phi}_x   = \int \int \eps \Big[u_R \hat{v}_x + v_{Rx}\hat{u} + v_R \hat{v}_y &+ v_{Ry}\hat{v} \Big] \times \\  \label{Svhat}
&  \Big[ \hat{v}^{(n)} x^{-2b} - 2b \hat{\psi}^{(n)} x^{-2b - 1} \Big].
\end{align}

We may take $n \rightarrow \infty$ above due to the definition of $X_1$ from (\ref{norm.x0}) and (\ref{cnrs.1}):
\begin{align} \label{Svhat.1}
(\ref{Svhat}) \xrightarrow{n \rightarrow \infty}  \int \int \eps \Big[u_R \hat{v}_x + v_{Rx}\hat{u} + v_R \hat{v}_y + v_{Ry}\hat{v} \Big] \cdot \Big[ \hat{v} x^{-2b} - 2b \hat{\psi} x^{-2b - 1} \Big].
\end{align}

We will now proceed to treat each term in (\ref{Svhat.1}). The first profile term, $u_R v_x$ is the most delicate:
\begin{align} \label{stay.stay}
\int \int \eps u_R \hat{v}_x [\hat{v}x^{-2b} - 2b \hat{\psi} x^{-2b-1}].
\end{align}

First, 
\begin{align} \n
\int \int \eps u_R \hat{v}_x \hat{v} x^{-2b} &= - \int \int \eps \hat{v}^2 \frac{\p_x}{2} \Big( u_R x^{-2b} \Big) + \lim_{M \rightarrow \infty} \int_{x = M} \frac{\eps u_R}{2} \hat{v}^2 x^{-2b}  \\ \label{d.phi.LN.2}
& = - \int \int \eps \hat{v}^2 \frac{u_{Rx}}{2} x^{-2b} +\int \int b \eps u_R \hat{v}^2 x^{-2b-1}. 
\end{align}

The $M$-limit above vanishes due to (\ref{evo.low}). Staying with the term (\ref{stay.stay}): 
\begin{align} \n
-2b \int \int &\eps u_R \hat{v}_x \hat{\psi} x^{-2b-1} = 2b \int \int \eps \hat{v} \p_x \Big( u_R \hat{\psi} x^{-2b-1} \Big) \\ \n
& = \int \int 2b \eps u_{Rx}\hat{v} \hat{\psi} x^{-2b-1} + \int \int 2b \eps u_R \hat{v}^2 x^{-2b-1}  \\ \label{d.psi.LN}
& \hspace{42 mm} - \int \int  2b(2b+1) \eps u_R \hat{\psi} \hat{v} x^{-2b-2} \\ \n
& = \int \int 2b \eps u_{Rx}\hat{v} \hat{\psi} x^{-2b-1} + \int \int 2b \eps u_R \hat{v}^2 x^{-2b-1} \\ \n
& \hspace{42 mm} + \int \int b(2b+1) \eps \hat{\psi}^2 u_{Rx} x^{-2b-2} \\ \label{d.psi}
& \hspace{42 mm}  - \int \int  b(2b+1)(2b+2) \eps u_R \hat{\psi}^2 x^{-2b-3}.
\end{align}

Combining the positive terms in (\ref{d.psi}) and (\ref{d.phi.LN.2}), the total positive contribution is $\int \int 3b\eps u_R \hat{v}^2 x^{-2b-1}$. For the final term in (\ref{d.psi}), we will now give the estimate:
\begin{align} \n
\int \int u_R \hat{\psi}^2 x^{-2b-3} &= \int \int u_R \hat{\psi}^2 \frac{-\p_x}{2b+2} x^{-2b-2} \\ \n
& = \int \int \frac{2}{2b+2} u_R \hat{\psi} \hat{v} x^{-2b-2} + \int \int \frac{u_{Rx}}{2b+2} \hat{\psi}^2 x^{-2b-2} \\ \n
& \le  \Big[ \frac{1}{2} || u_R^{\frac{1}{2}}  \hat{\psi} x^{-b-\frac{3}{2}}||_{L^2}^2 + \frac{1}{2} \frac{4}{(2b+2)^2}  || u_R^{\frac{1}{2}}  \hat{v} x^{-b-\frac{1}{2}}||_{L^2}^2 \Big]\\
& \hspace{30 mm} + \frac{||u_{Rx}x||_{L^\infty}}{2b+2}  \frac{\sup |u_R|}{\inf |u_R|} \int \int u_R \hat{\psi}^2 x^{-2b-3}.
\end{align}

By collecting terms and rearranging, we obtain: 
\begin{align}
\Big[1 - \frac{1}{2} -  \frac{||u_{Rx}x||_{L^\infty}}{2b+2}  \frac{\sup |u_R|}{\inf |u_R|} \Big] ||u_R^{\frac{1}{2}} \hat{\psi} x^{-b-\frac{3}{2}}||_{L^2}^2 \le \frac{2}{(2b+2)^2} ||u_R^{\frac{1}{2}} \hat{v} x^{-b-\frac{1}{2}}||_{L^2}^2. 
\end{align}

This then implies: 
\begin{align}
|| u_R^{\frac{1}{2}}  \hat{\psi}x^{-b-\frac{3}{2}}||_{L^2}^2 \le \frac{1}{1-\mathcal{O}(\delta)} \frac{4}{(2b+2)^2} || u_R^{\frac{1}{2}} \hat{v} x^{-b-\frac{1}{2}}||_{L^2}^2. 
\end{align}

Inserting this into (\ref{d.psi}), one arrives at: 
\begin{align} \n 
|\int \int  b(2b+1)(2b+2) \eps &u_R \hat{\psi}^2 x^{-2b-3}| \\ \n
&\le \frac{1}{1-\mathcal{O}(\delta)} \frac{4b(2b+1)(2b+2)}{(2b+2)^2}  \int \int \eps u_R \hat{v}^2 x^{-1-2b} \\ 
& \le  \int \int \frac{5b}{2} u_R \eps \hat{v}^2 x^{-1-2b},
\end{align}

so long as $b$ is sufficiently close to $0$, by the following calculation: 
\begin{align}
\lim_{b \rightarrow 0} \frac{(2b+1)(2b+2)}{(2b+2)^2} = \frac{1}{2}.
\end{align}

Thus, taking $b$ sufficiently small, and recalling the positive contributions from (\ref{d.psi}) and (\ref{d.phi.LN.2}), we have: 
\begin{align} \label{Yabove}
3b \int\int \eps u_R \hat{v}^2 x^{-2b-1} &- \frac{5b}{2} \int \int \eps u_R \hat{v}^2 x^{-2b-1} = \frac{b}{2} \int \int \eps u_R \hat{v}^2 x^{-2b-1}.
\end{align}

The remaining terms from (\ref{d.phi.LN.2}) and (\ref{d.psi}) are then estimated in terms of (\ref{Yabove}) using the smallness of $\mathcal{O}(\delta)$. Summarizing, we have established control over:
\begin{align}
\int \int \eps u_R \hat{v}_x \cdot \Big[\hat{v} x^{-2b} -2b \hat{\psi} x^{-2b-1} \Big] \gtrsim \int \int b \eps \hat{v}^2 x^{-1-2b},
\end{align}

for a constant independent of small $\delta$ and $b$. We will now move to the second term from (\ref{Svhat}), for which we recall estimates (\ref{PE0.1}) and (\ref{PE4}):
\begin{align} \n
| \int \int \eps v_{Rx} \hat{u} \cdot \Big[ \hat{v} x^{-2b} - 2b \hat{\psi} x^{-2b-1} \Big] | &\le \sqrt{\eps} ||v_{Rx} x^{\frac{3}{2}}||_{L^\infty} ||\frac{\hat{u}}{x^{\frac{3}{4}-b}}||_{L^2} ||\sqrt{\eps}\frac{\hat{v}}{x^{\frac{3}{4}-b}}||_{L^2} \\
& \le \sqrt{\eps} ||\hat{u}_x x^{\frac{1}{4}-b}||_{L^2} ||\sqrt{\eps}\hat{v}_x x^{\frac{1}{4}-b}||_{L^2}.
\end{align}

For the third term from (\ref{Svhat}), we use Young's inequality and estimates (\ref{PE1}), (\ref{PE4.new.2}):
\begin{align} \n
| \int \int \eps v_R \hat{v}_y &\Big[ \hat{v} x^{-2b} -2b \hat{\psi} x^{-2b-1} \Big] | \\ \n
& \le ||v_R x^{\frac{1}{2}}||_{L^\infty} \Big[ ||\hat{v}_y x^{\frac{1}{2}-b}||_{L^2}^2 + ||\sqrt{\eps}\hat{v} x^{-b-\frac{1}{2}}||_{L^2}^2 + ||\sqrt{\eps} \hat{\psi} x^{-b-\frac{3}{2}}||_{L^2}^2 \Big] \\
& \le \mathcal{O}(\delta) \Big[ ||\hat{v}_y x^{\frac{1}{2}-b}||_{L^2}^2 + ||\sqrt{\eps}\hat{v} x^{-b-\frac{1}{2}}||_{L^2}^2 + ||\sqrt{\eps} \hat{\psi} x^{-b-\frac{3}{2}}||_{L^2}^2 \Big] .
\end{align}

For the final term from (\ref{Svhat}), we use Young's inequality and estimates (\ref{PE1}), (\ref{PE4.new.2}):
\begin{align}
|\int \int\eps v_{Ry} \hat{v} &\cdot \Big[ \hat{v} x^{-2b} - 2b \hat{\psi} x^{-2b-1} \Big]|  \\ \n
& \lesssim ||v_{Ry} x||_{L^\infty} \Big[ ||\sqrt{\eps} \hat{v} x^{-b-\frac{1}{2}}||_{L^2}^2 + b ||\sqrt{\eps}  \hat{\psi} x^{-b-\frac{3}{2}}||_{L^2}^2 \Big]. 
\end{align}

Summarizing these last few terms, we obtain: 
\begin{align} \label{Svhat.2}
|(\ref{Svhat.1}) | \gtrsim \int \int b \eps \hat{v} x^{-1-2b} + \mathcal{O}(\delta) \Big[ ||\{\hat{u}, \sqrt{\eps}v \} x^{\frac{1}{2}-b}||_{L^2}^2 + || \hat{v}_y x^{\frac{1}{2}-b} ||_{L^2}^2 \Big].
\end{align}

The final task is to turn to the right-hand side. Reading from (\ref{weak.5}), and (\ref{hats.1}) - (\ref{hats.2}):
\begin{align} \n
\int \int \hat{f} \cdot \phi_y + \eps \hat{g} \cdot \phi_x &= \int \int \hat{f} \cdot \hat{u}^{(n)} x^{-2b} + \eps \hat{g} \cdot [\hat{v}^{(n)} x^{-2b} + \hat{\psi}^{(n)} \p_x x^{-2b}] \\ \label{UQ.RHS}
& \xrightarrow{n \rightarrow \infty} \int \int \hat{f} \cdot \hat{u}^{(n)} x^{-2b} + \eps \hat{g} \cdot [\hat{v} x^{-2b} + \hat{\psi} \p_x x^{-2b}],
\end{align}

where we have passed to the limit using again the definition of $X_1$ from (\ref{norm.x0}). Combining (\ref{UQ.1}), (\ref{goodB}), (\ref{Svhat.2}), and (\ref{UQ.RHS}), one obtains the desired result, estimate (\ref{hat.Energy}).

\end{proof}

We now repeat the positivity estimate, with a correspondingly weaker weight in order to close the above energy estimate. We refer the reader to Proposition \ref{prop.pos} for a comparison. 

\begin{lemma} Fix any $0 < b < 1$. Let $\delta, \eps$ be sufficiently small relative to universal constants, and $\eps << \delta$. Then for $[\hat{u}, \hat{v}] \in Z$ solutions to (\ref{hats.1}) - (\ref{hats.2}) with boundary conditions (\ref{hats.bc.1}) satisfy the following estimate:
\begin{align} \label{pos.hat.1}
||\{\hat{u}_x, \sqrt{\eps} \hat{v}_x \} x^{\frac{1}{2}-b}||_{L^2}^2 \lesssim ||\hat{u}_y x^{-b}||_{L^2}^2 + ||\{\sqrt{\eps} \hat{v}, \hat{u} \} x^{-\frac{1}{2}-b}||_{L^2}^2 + \mathcal{W}_{1,P,b}. 
\end{align}
\end{lemma}
\begin{proof}

The estimate will follow upon applying the multiplier $\hat{v} x^{1-2b}$ to the system (\ref{hats.v}). In order to proceed formally, we must start with the weak formulation given in (\ref{weak.5}), and select the test function:  
\begin{align} \label{jos.0}
\phi = \hat{v}^{(n)} x^{1-2b}, \hspace{3 mm} [\hat{u}^{(n)}, \hat{v}^{(n)}] \xrightarrow{X_1} [\hat{u}, \hat{v}],
\end{align}

where $X_1$ is defined in (\ref{norm.x0}). Turning to the weak formulation in (\ref{weak.5}), we will first expand the second-order terms: 
\begin{align} \n
\int \int \nabla^2_\eps \hat{\psi} : \nabla^2_\eps \phi &= \int \int \nabla^2_\eps \hat{\psi} : \nabla^2_\eps \hat{v}^{(n)} x^{1-2b} \\ \n
& = \int \int \hat{\psi}_{yy} \hat{v}^{(n)}_{yy} x^{1-2b} + 2\eps \hat{\psi}_{xy} \p_x \Big( \hat{v}^{(n)}_y x^{1-2b} \Big) + \eps^2 \hat{\psi}_{xx} \p_{xx} \Big( \hat{v}^{(n)}  x^{1-2b} \Big) \\ \label{Q.P.1}
& =  \int \int -\hat{u}_{y} \hat{v}^{(n)}_{yy} x^{1-2b} + 2\eps \hat{v}_{y} \p_x \Big( \hat{v}^{(n)}_y x^{1-2b} \Big) + \eps^2 \hat{v}_{x} \p_{xx} \Big( \hat{v}^{(n)}  x^{1-2b} \Big)
\end{align}

We first arrive at the first two terms from (\ref{Q.P.1}): 
\begin{align} \n
\int \int -\hat{u}_{y} &\hat{v}^{(n)}_{yy} x^{1-2b} -2 \eps \hat{u}_x \hat{v}^{(n)}_{xy} x^{1-2b} -2 \eps \hat{u}_x \hat{v}^{(n)}_y x^{-2b} \\ \label{jos.1}
& = \int \int - \hat{u}^{(n)}_{y} \p_x[ \hat{u}_{y} x^{1-2b}] - 2\eps \hat{u}^{(n)}_{x} \p_x [ \hat{u}_x  x^{1-2b} ] -2\eps \hat{u}_x \hat{v}^{(n)}_y x^{-2b}.
\end{align}

Referring to the definition of $X_1$ in (\ref{norm.x0}), according to (\ref{jos.0}), we may pass to the limit as $n \rightarrow \infty$, and appeal to the estimates in (\ref{evo.low}) and (\ref{evo.mid}), to obtain: 
\begin{align} \n
(\ref{jos.1})& \xrightarrow{n \rightarrow \infty} \int \int - \hat{u}_{y} \p_x[ \hat{u}_{y} x^{1-2b}] - 2 \eps \hat{u}_{x} \p_x [ \hat{u}_x  x^{1-2b}]  - 2\eps \hat{u}_x \hat{v}_y x^{-2b} \\ \n
& = \int \int -\frac{(1-2b)}{2} \hat{u}_y^2 x^{-2b} + (1+2b) \eps \hat{u}_x^2 x^{-2b} + \lim_{M \rightarrow \infty} \int_{x = M} \Big[ \frac{1}{2} \hat{u}_y^2 x^{1-2b} - \eps \hat{u}_x^2 x^{1-2b} \Big] \\
& =   \int \int -\frac{(1-2b)}{2} \hat{u}_y^2 x^{-2b} + (1+2b) \eps \hat{u}_x^2 x^{-2b}. 
\end{align}

Again referring to the definition in (\ref{norm.x0}), the third term from (\ref{Q.P.1}) is treated by: 
\begin{align} \n
\int \int \eps^2 &\hat{v}_x \p_{xx} \Big( \hat{v}^{(n)} x^{1-2b} \Big) = - \int \int \eps^2 \hat{v}_{xx} \p_x \Big( \hat{v}^{(n)} x^{1-2b} \Big) \\ \label{iq.1}
& \xrightarrow{n \rightarrow \infty}  - \int \int \eps^2 \hat{v}_{xx} \p_x \Big( \hat{v} x^{1-2b} \Big) = \int \int -\eps^2 \hat{v}_{xx} \hat{v}_x x^{1-2b} - \int \int \eps^2 \hat{v}_{xx} \hat{v} (1-2b) x^{-2b}. 
\end{align}

Integrating by parts the first term on the right-hand side of (\ref{iq.1}), and appealing to estimate (\ref{evo.mid}):
\begin{align} \n
 \int \int -\eps^2 \hat{v}_{xx} \hat{v}_x x^{1-2b} &= \int \int \eps^2 \frac{1-2b}{2} |\hat{v}_x|^2 x^{-2b} - \lim_{M \rightarrow \infty} \int_{x = M} \frac{\eps^2}{2} \hat{v}_x^2 x^{1-2b} \\ \label{iq.2}
 & =  \int \int \eps^2 \frac{1-2b}{2} |\hat{v}_x|^2 x^{-2b}.
\end{align}

Integrating by parts the second term on the right-hand side of (\ref{iq.1}), and again appealing to estimates (\ref{evo.low}) - (\ref{evo.high}) for the $M$-limit below:
\begin{align} \n
- \int \int \eps^2 \hat{v}_{xx} \hat{v} (1-2b) x^{-2b} &= \int \int \eps^2 (1-2b) \hat{v}_x \p_x \Big[ \hat{v} x^{-2b}  \Big] + \lim_{M \rightarrow \infty} \int_{x = M} \eps^2 (1-2b) \hat{v}_x \hat{v} x^{-2b}\\ \n
& =  \int \int \eps^2 (1-2b) \hat{v}_x^2 x^{-2b} - \int \int \eps^2 2b(1-2b) \hat{v}_x \hat{v} x^{-2b-1} \\ \n
& =  \int \int \eps^2 (1-2b) \hat{v}_x^2 x^{-2b} + \int \int \eps^2 b(1-2b) \hat{v}^2 \p_x x^{-2b-1} \\ \n
& \hspace{5 mm} - \lim_{M \rightarrow \infty} \eps^2 b(1-2b) \hat{v}^2 x^{-2b-1} \\ \label{iq.3}
& = \int \int \eps^2 (1-2b) \hat{v}_x^2 x^{-2b} - \int \int \eps^2 b(1-2b) (2b+1) \hat{v}^2  x^{-2b-2}. 
\end{align}

Combining the above estimates:  
\begin{align}
(\ref{iq.1}) = \int \int \frac{3}{2}(1-2b) \eps^2 \hat{v}_x^2 x^{-2b}  - b(2b+1)(1-2b) \eps^2 \hat{v}^2 x^{-2b-2}.
\end{align}

Hence, summarizing (\ref{jos.1}) - (\ref{iq.3}):
\begin{align}
|\lim_{n \rightarrow \infty} \int \int \nabla^2_\eps \hat{\psi} : \nabla^2_\eps \phi| = | \int \int \nabla^2_\eps \hat{\psi} : \nabla^2_\eps (\hat{v} x^{1-2b})| \lesssim \int \int [\eps^2 \hat{v}_x^2 + \eps \hat{u}_x^2 + \hat{u}_y^2 ]x^{-2b}.
\end{align}

We will now turn to the profile terms from $S_u$, which upon consultation with (\ref{weak.5}), the definition in (\ref{norm.x0}), and (\ref{jos.0}), read:
\begin{align} \n
\int \int  \Big[u_R \hat{u}_x &+ u_{Rx} \hat{u} + u_{Ry}\hat{v} + v_R \hat{u}_y \Big] \cdot \hat{u}^{(n)}_x x^{1-2b} \\  \label{hatSU2}
&\xrightarrow{n \rightarrow \infty} \int \int  \Big[u_R \hat{u}_x + u_{Rx} \hat{u} + u_{Ry}\hat{v} + v_R \hat{u}_y \Big] \cdot \hat{u}_x x^{1-2b}.
\end{align}

We now turn our attention to (\ref{hatSU2}). The first term yields the desired positivity: 
\begin{align}
\int \int u_R \hat{u}_x^2 x^{1-2b} \gtrsim \min u_R \int \int \hat{u}_x^2 x^{1-2b}. 
\end{align}

Next, by (\ref{PE0.4}), (\ref{PE5}):
\begin{align} \n
|\int \int u_{Rx} \hat{u} \hat{u}_x x^{1-2b}| &\le ||u_{Rx} x||_{L^\infty} ||\hat{u}_x x^{\frac{1}{2}-b}||_{L^2} ||\hat{u} x^{-\frac{1}{2}-b}||_{L^2} \\
& \le \mathcal{O}(\delta) ||\hat{u}_x x^{\frac{1}{2}-b}||_{L^2}^2. 
\end{align}

Next, we shall split $u_R = u^P_R + u^E_{R}$, and use estimate (\ref{PE0.5}) and (\ref{PE3}) for (\ref{new.new.1}) below and (\ref{PE5}) for (\ref{new.new.2}) below:
\begin{align} \label{new.new.1}
&| \int \int u^P_{Ry} \hat{v} \hat{u}_x x^{1-2b} | \le ||y u^P_{Ry}||_{L^\infty} ||\hat{v}_y x^{\frac{1}{2}-b}||_{L^2}^2, \\ \n
&| \int \int \sqrt{\eps} u^E_{RY} \hat{v} \hat{u}_x x^{1-2b} | \le ||u^E_{RY} x^{\frac{3}{2}}||_{L^\infty} || \sqrt{\eps} \hat{v} x^{-\frac{1}{2}-b}||_{L^2} ||\hat{u}_x x^{\frac{1}{2}-b}||_{L^2} \\ \label{new.new.2}
& \hspace{40 mm} \lesssim \sqrt{\eps} \Big[ ||\sqrt{\eps} \hat{v} x^{-\frac{1}{2}-b}||_{L^2}^2 + ||\hat{u}_x x^{\frac{1}{2}-b}||_{L^2}^2 \Big].
\end{align}

For the fourth term from (\ref{hatSU2}), by estimates (\ref{PE1}) and (\ref{PE4.new.2}):
\begin{align} \n
| \int \int v_R \hat{u}_y \hat{u}_x x^{1-2b} | &\le ||v_R x^{\frac{1}{2}}||_{L^\infty} ||\hat{u}_y x^{-b}||_{L^2} ||\hat{u}_x x^{\frac{1}{2}-b}||_{L^2} \\
& \le \mathcal{O}(\delta)  ||\hat{u}_y x^{-b}||_{L^2} ||\hat{u}_x x^{\frac{1}{2}-b}||_{L^2}. 
\end{align}

Summarizing the last four calculations: 
\begin{align} \n
| (\ref{hatSU2}) | \gtrsim \int \int \hat{u}_x^2 x^{1-2b} &- \mathcal{O}(\delta) \Big[ ||\hat{u} x^{-\frac{1}{2}-b}||_{L^2}^2 \\ 
& + ||\hat{u}_y x^{-b}||_{L^2}^2 + ||\sqrt{\eps} \hat{v} x^{-\frac{1}{2}-b}||_{L^2}^2 \Big].
\end{align}

The final three terms appearing on the right-hand side above all appear on the right-hand side of estimate (\ref{pos.hat.1}). Turning now to the profile terms, from $S_v$, for which we read (\ref{weak.5}) with $\phi = \hat{v}^{(n)} x^{1-2b}$, appeal to (\ref{norm.x0}) and (\ref{jos.0}), giving ultimately:
\begin{align} \n
\int \int \eps  \Big[u_R \hat{v}_x &+ v_{Rx} \hat{u} + v_R \hat{v}_y + v_{Ry} \hat{v} \Big] \cdot \p_x[\hat{v}^{(n)} x^{1-2b} ] \\ \label{Svlimit.g}
& \xrightarrow{n \rightarrow \infty} \int \int \eps  \Big[u_R \hat{v}_x + v_{Rx} \hat{u} + v_R \hat{v}_y + v_{Ry} \hat{v} \Big] \cdot \p_x[\hat{v} x^{1-2b} ].
\end{align}

We will treat each term in (\ref{Svlimit.g}). For the first term from (\ref{Svlimit.g}):
\begin{align} \n
 \int \int \eps u_R \hat{v}_x &\Big( \hat{v}_x x^{1-2b} + (1-2b) \hat{v} x^{-2b} \Big)  \\ \n
& = \int \int \eps u_R \hat{v}_x^2 x^{1-2b} + b(1-2b) u_R \eps \hat{v}^2 x^{-1-2b} - \int \int \frac{1-2b}{2} \eps u_{Rx} \hat{v}^2 x^{-2b} \\
& \gtrsim \int \int \eps \hat{v}_x^2 x^{1-2b} + b \int \int \eps \hat{v}^2 x^{-1-2b}. 
\end{align}

Above we have used (\ref{PE0.4}) and (\ref{PE5}). For the second term, we integrate by parts:
\begin{align} \n
\int \int \eps v_{Rx} \hat{u} \p_x[ \hat{v} x^{1-2b}] &= 
\int \int -\eps \p_x \Big( v_{Rx} \hat{u} \Big) \cdot \hat{v} x^{1-2b} + \lim_{M \rightarrow \infty} \int_{x = M} v_{Rx} \hat{u} \hat{v} x^{1-2b} \\ \n
&= \int \int -\eps v_{Rxx} \hat{u} \hat{v} x^{1-2b} - \eps v_{Rx} \hat{v} \hat{u}_x x^{1-2b} \\ \n
& \le \sqrt{\eps} ||v_{Rx} x^{\frac{3}{2}}, v_{Rxx} x^{\frac{5}{2}}||_{L^\infty}\Big[ ||\hat{u}x^{-\frac{1}{2}-b}||_{L^2}^2  \\ 
& \hspace{20 mm} + ||\hat{u}_x x^{\frac{1}{2}-b}||_{L^2}^2 + ||\sqrt{\eps} \hat{v} x^{-\frac{1}{2}-b}||_{L^2}^2 \Big].
\end{align}

The above $M-$limit vanishes according to estimates (\ref{evo.low}), and we have used estimates (\ref{PE0.1}) and (\ref{PE4}). For the third term, we recall estimates (\ref{PE1}), (\ref{PE4.new.2}):
\begin{align} \n
\int \int \eps v_R \hat{v}_y & \hat{v}_x x^{1-2b} + c_0 \eps v_R \hat{v}_y \hat{v} x^{-2b}  \\
& \le \sqrt{\eps} ||v_R x^{\frac{1}{2}}||_{L^\infty} \Big[ ||\hat{v}_y x^{\frac{1}{2}-b}||_{L^2}^2 + ||\sqrt{\eps} \hat{v}_x x^{\frac{1}{2}-b}||_{L^2}^2 + ||\sqrt{\eps} \hat{v} x^{-\frac{1}{2}-b}||_{L^2}^2 \Big].
\end{align}

For the fourth term, we integrate by parts and appeal to (\ref{evo.low}), (\ref{PE0.1}) - (\ref{PE1}), and (\ref{PE4}):
\begin{align} \n
\int \int \eps v_{Ry} \hat{v} \cdot \p_x[ \hat{v} x^{1-2b}] &= -\int \int \eps \p_x[ v_{Ry} \hat{v}] \cdot\hat{v} x^{1-2b} + \lim_{M \rightarrow \infty} \int_{x = M} \eps v_{Ry} \hat{v}^2 x^{1-2b} \\ \n
&= \int \int -\eps v_{Rxy} \hat{v}^2 x^{1-2b} - \eps v_{Ry} \hat{v}_x \hat{v} x^{1-2b} \\
& \le ||v_{Ry}x, v_{Rxy} x^2 ||_{L^\infty}\Big[ ||\sqrt{\eps} \hat{v}_x x^{\frac{1}{2}-b}||_{L^2}^2 + ||\sqrt{\eps}\hat{v} x^{-\frac{1}{2}-b}||_{L^2}^2 \Big]. 
\end{align}

Summarizing these four terms, 
\begin{align} \n
|(\ref{Svlimit.g})| &\gtrsim \int \int \eps \hat{v}_x^2 x^{1-2b} + b \int \int \eps \hat{v}^2 x^{-1-2b} \\
 & - \mathcal{O}(\delta) \Big[ ||\hat{u} x^{-\frac{1}{2}-b}, \hat{u}_x x^{\frac{1}{2}-b}||_{L^2}^2 + ||\sqrt{\eps} \hat{v} x^{-\frac{1}{2}-b}, \sqrt{\eps} \hat{v}_x x^{\frac{1}{2}-b}||_{L^2}^2 \Big].
\end{align}

On the right-hand side, appealing again to (\ref{norm.x0}), (\ref{jos.0}), and the definitions of $\hat{f}, \hat{g}$ in (\ref{hats.1}) - (\ref{hats.2}), one obtains: 
\begin{align} \n
 \int \int \hat{f} \hat{u}_x^{(n)} x^{1-2b} &+\hat{g} \Big[ \hat{v}^{(n)}_x x^{1-2b} + (1-2b) \hat{v}^{(n)} x^{-2b} \Big] \\
 & \xrightarrow{ n \rightarrow \infty}  \int \int \hat{f} \hat{u}_x x^{1-2b} +\hat{g} \Big[ \hat{v}_x x^{1-2b} + (1-2b) \hat{v} x^{-2b} \Big],
\end{align}

Placing the above estimates together yields the estimate (\ref{pos.hat.1}). 

\end{proof}

We will now introduce some notation, which is a natural adaptation of what is found in Section \ref{Section.Z} to the weaker weight of $x^{-b}$. The reader should recall the definitions of the cutoff functions introduced in (\ref{zeta}) - (\ref{rho}). The energy norms are defined as follows:
\begin{align} \label{norm.x0b}
||u,v||_{X_{1,b}}^2 &:= ||u_y x^{-b}||_{L^2}^2 + ||\{\sqrt{\epsilon}v_x, v_y \} x^{\frac{1}{2}-b}||_{L^2}^2\\ \label{norm.x1b}
||u,v||_{X_{2,b}}^2 &:= ||  u_{xy} \cdot \rho_2 x^{1-b}||_{L^2}^2 + ||  \{ \sqrt{\epsilon} v_{xx},  v_{xy} \} \cdot \rho_2^{\frac{3}{2}} x^{\frac{3}{2}-b}||_{L^2}^2, \\ \label{norm.x2b}
||u,v||_{X_{3,b}}^2 &:= ||  u_{xxy} \cdot \rho_3^2 x^{2-b}||_{L^2}^2 +||  \{ \sqrt{\epsilon} v_{xxx},  v_{xxy} \} \cdot \rho_3^{\frac{5}{2}} x^{\frac{5}{2}-b}||_{L^2}^2.
\end{align}

\begin{definition} The norms $Y_{2,b}, Y_{3,b}$ are strengthenings of $X_{2,b}, X_{3,b}$ near the boundary, $x = 1$, and defined through: 
\begin{align} \label{norm.Y2b}
&||u,v||_{Y_{2,b}}^2 := ||u_{xy} x^{1-b}||_{L^2}^2 + ||\{\sqrt{\epsilon} v_{xx}, v_{xy} \} x^{\frac{3}{2}-b}||_{L^2}^2 + ||u_{yy}||_{L^2(x \le 2000)}, \\ \label{norm.Y3b}
&||u,v||_{Y_{3,b}}^2 := || u_{xxy} \cdot \zeta_3 x^{2-b}||_{L^2}^2 + || \{ \sqrt{\epsilon} v_{xxx}, v_{xxy} \} \cdot \zeta_3 x^{\frac{5}{2}-b}||_{L^2}^2.
\end{align}
\end{definition}

\begin{definition} The norm $Z_b$ is defined through: 
\begin{align} \nonumber
||u,v||_{Z_b} := &||u,v||_{X_{1,b} \cap X_{2,b} \cap X_{3,b}} + \epsilon^{N_2} ||u,v||_{Y_{2,b}} + \epsilon^{N_3} ||u,v||_{Y_{3,b}} \\ \nonumber 
&+ \epsilon^{N_4} ||ux^{\frac{1}{4}-b} , \sqrt{\eps} v x^{\frac{1}{2}-b}||_{L^\infty} + \epsilon^{N_5} \sup_{x \ge 20} ||\sqrt{\eps} v_x x^{\frac{3}{2}-b} , u_x x^{\frac{5}{4}-b} ||_{L^\infty} \\ \label{norm.Zb}
& + \eps^{N_6}  \sup_{x \ge 20} ||u_y x^{\frac{1}{2}-b}||_{L^2_y} + \epsilon^{N_7} \Big[\int_{20}^\infty x^{4-b} ||\sqrt{\eps} v_{xx}||_{L^\infty_y}^2 dx \Big]^{\frac{1}{2}} .
\end{align}
\end{definition}

Next, we record the second and third order versions of the energy and positivity estimates, which mimic Propositions \ref{prop.ho.1}, \ref{prop.ho.2}, \ref{prop.ho.3}, \ref{prop.ho.4}. We will omit most details, and record only those differences which arise. 

\begin{lemma}[Second-Order Energy Estimate] Fix any $0 < b < 1$. Let $\delta, \eps$ be sufficiently small relative to universal constants, and $\eps << \delta$. Then for $[\hat{u}, \hat{v}] \in Z$ solutions to (\ref{hats.1}) - (\ref{hats.2}):
\begin{align}
|| \hat{u}_{xy} \rho_2 x^{1-b}||_{L^2}^2 \lesssim \mathcal{O}(\delta)|| \{ \sqrt{\eps} \hat{v}_{xx}, \hat{v}_{xy}\} \rho_2^{\frac{3}{2}} x^{\frac{3}{2}-b}||_{L^2}^2 + ||\hat{u},\hat{v}||_{X_{1,b}}^2 +\mathcal{W}_{1,b} + \mathcal{W}_{2,E,b},
\end{align}
where (recall the definition of $\rho_2$ from (\ref{rho})):
\begin{align} \label{calw2b.E}
&\mathcal{W}_{2,E,b} := \int \int \hat{f}_x  \hat{u}_x \rho_2^2 x^{2-2b} + \int \int \eps \hat{g}_{x} \hat{v}_x \rho_2^2 x^{2-2b}, \\ \label{calw2b.P}
&\mathcal{W}_{2,P,b} = \int \int \hat{f}_x \hat{u}_{xx} \rho_2^3 x^{3-2b} +  \int \int \eps \hat{g}_{x}\hat{v}_{xx} \rho_2^3 x^{3-2b}, \\ \label{calw2b}
&\mathcal{W}_{2,b} := \mathcal{W}_{2,E,b} + \mathcal{W}_{2,P,b}.
\end{align}
\end{lemma}
\begin{proof}
Differentiating the weak formulation gives: 
\begin{align} \n
\int \int \nabla_\epsilon^2 \hat{\psi_x} : \nabla_\epsilon^2 \phi - \int \int \p_x S_u(\hat{u}, \hat{v}) \cdot \phi_y &+ \eps \p_x S_v(\hat{u}, \hat{v}) \cdot \phi_x \\ \label{weak.6}  &= \int \int \eps^{\frac{n}{2}+\gamma} \Big[ - \p_x \hat{f} \phi_y + \eps \p_x \hat{g} \phi_x \Big],
\end{align}

For the second-order energy estimate, we select $\phi = \rho_2^2 \hat{v}^{(n)} x^{2-2b}$, where: 
\begin{align} \label{UQ.lim}
[\hat{u}^{(n)}, \hat{v}^{(n)}] \in C^\infty_{0,D}, \hspace{3 mm} [\hat{u}^{(n)}, \hat{v}^{(n)}] \xrightarrow{X_1} [\hat{u}, \hat{v}]. 
\end{align}

Let us turn to the highest-order terms:
\begin{align} \n
\int \int \nabla^2_\eps \hat{\psi}_x : \nabla^2_\eps \phi &= \int \int \nabla^2_\eps \hat{v} : \nabla^2_\eps \rho_2^2 \hat{v}^{(n)} x^{2-2b} \\ \n
& = \int \int \hat{v}_{yy} \rho_2^2 \hat{v}^{(n)}_{yy} x^{2-2b} + 2\eps \hat{v}_{xy} \p_x[\hat{v}^{(n)}_y \rho_2^2 x^{2-2b}] + \eps^2 \hat{v}_{xx} \p_{xx}[\rho_2^2 \hat{v}^{(n)} x^{2-2b}] \\ \n
& =  \int \int  \hat{u}_{xy} \rho_2^2 \hat{u}^{(n)}_{xy} x^{2-2b} + 2\eps \hat{v}_{xy} \p_x[\hat{v}^{(n)}_y \rho_2^2 x^{2-2b}] + \eps^2 \hat{v}_{xx} \p_{xx}[\rho_2^2 \hat{v}^{(n)} x^{2-2b}] \\ \label{UQP.1}
& = \int \int -\p_x[ \hat{u}_{xy} \rho_2^2 x^{2-2b}] \hat{u}^{(n)}_y - 2\eps \hat{v}_{xxy} \hat{v}^{(n)}_y \rho_2^2 x^{2-2b} - \eps^2 \hat{v}_{xxx} \p_x[\rho_2^2 \hat{v}^{(n)} x^{2-2b}]
\end{align}

One now checks according to the definition (\ref{norm.x0}), that (\ref{UQ.lim}) suffices to pass to the limit in the above identity, which upon integrating by parts in $x$ yields:
\begin{align} \n
(\ref{UQP.1}) \xrightarrow{n \rightarrow \infty} &\int \int -\p_x[ \hat{u}_{xy} \rho_2^2 x^{2-2b}] \hat{u}_y - 2\eps \hat{v}_{xxy} \hat{v}_y \rho_2^2 x^{2-2b} - \eps^2 \hat{v}_{xxx} \p_x[\rho_2^2 \hat{v} x^{2-2b}] \\ \n
& = \int \int [\hat{u}_{xy}^2 + \eps \hat{u}_{xx}^2 + \eps^2 \hat{v}_{xx}^2] \rho_2^2 x^{2-2b} + J. 
\end{align}

where $|J| = |c_0 \eps^2 \hat{v}_x^2 \p_{xx}(\rho_2^2 x^{2-2b}) + c_1 \eps^2 \hat{v}^2 \p_{x}^4 (\rho_2^2 x^{2-2b})|  \lesssim ||u,v||_{X_{1,b}}^2$. From here, repeating the calculations in Proposition \ref{prop.ho.1} gives the desired result, where the required integrations by parts are justified upon using that $b > 0$, combined with the estimates in (\ref{evo.low}) - (\ref{evo.high}). These justifications are analogous to those in Lemma \ref{L1U}, and so we omit the details. 

\end{proof}

\begin{lemma}[Second-Order Positivity Estimate]  Fix any $0 < b < 1$. Let $\delta, \eps$ be sufficiently small relative to universal constants, and $\eps << \delta$. Then for $[\hat{u}, \hat{v}] \in Z$ solutions to (\ref{hats.1}) - (\ref{hats.2}):
\begin{align}
||\{ \sqrt{\eps}\hat{v}_{xx}, \hat{v}_{xy} \} \rho_2^{\frac{3}{2}} x^{\frac{3}{2}-b}||_{L^2}^2 \lesssim ||\hat{u}_{xy} \rho_2 x^{1-b}||_{L^2}^2 + ||\hat{u},\hat{v}||_{X_{1,b}}^2 +\mathcal{W}_{1,b} + \mathcal{W}_{2,b}. 
\end{align}
\end{lemma}
\begin{proof}

We start again with the weak formulation in (\ref{weak.6}). Fix a large $0 < L < \infty$. We then make the selection: $\phi = \hat{v}^{(n)}_x \cdot \rho_2^3 x_L^{3-2b}$, where, referring to (\ref{Nweight}), the weight $x_L$ is defined via: $x_L := \Big( a_L \ast \phi_L \Big) \chi\Big(\frac{x}{10L} \Big)$.  Define the domain: $\Omega_L := \{x: 3 < x < 50L + 100 \}$, so that $\hat{v}_x \cdot \rho_2^3 x_L^{3-2b} = 0$ on $\Omega_L^C$. The sequence $\hat{v}^{(n)}$ is selected according to: 
\begin{align}
[\hat{u}^{(n)}_x, \hat{v}^{(n)}_x] \in C^\infty_{0,D}(\Omega_L), \hspace{5 mm}  [\hat{u}^{(n)}_x, \hat{v}^{(n)}_x] \xrightarrow{H^1(\Omega_L)} [\hat{u}_x, \hat{v}_x]. 
\end{align}

The existence of such a sequence is guaranteed due to the standard Sobolev space theory, because we are now in the un-weighted setting. It is now straightforward to repeat all estimates in Proposition \ref{prop.ho.2} using the test function $\phi$. Upon doing so, we pass to the limit first as $n \rightarrow \infty$, and then as $L \rightarrow \infty$ to obtain the desired estimate. 

\end{proof}

\begin{lemma}[Third-Order Energy Estimate]  Fix any $0 < b < 1$. Let $\delta, \eps$ be sufficiently small relative to universal constants, and $\eps << \delta$. Then for $[\hat{u}, \hat{v}] \in Z$ solutions to (\ref{hats.1}) - (\ref{hats.2}):
\begin{align}
|| \hat{u}_{xxy} \rho_3^2 x^{2-b}||_{L^2}^2 \lesssim \mathcal{O}(\delta)|| \{ \sqrt{\eps} \hat{v}_{xxx}, \hat{v}_{xxy}\} \rho_3^{\frac{5}{2}} x^{\frac{5}{2}-b}||_{L^2}^2 + ||\hat{u},\hat{v}||_{X_{1,b} \cap X_{2,b}}^2 + \sum_{i = 1}^2 \mathcal{W}_{i,b} + W_{3,E,b},
\end{align}
where 
\begin{align}  \label{calw3b.E}
&\mathcal{W}_{3,E,b} := \int \int  \hat{f}_{xx} \hat{u}_{xx} \rho_3^4 x^{4-2b} + \int \int \eps \hat{g}_{xx}  \hat{v}_{xx} \rho_3^4  x^{4-2b}, \\ \label{calw3b.P}
&\mathcal{W}_{3,P,b} := \int \int  \hat{f}_{xx} \hat{u}_{xxx} \rho_3^5  x^{5-2b}  + \int \int \eps \hat{g}_{xx} \hat{v}_{xxx} \rho_3^5 x^{5-2b}, \\ \label{calw3b}
&\mathcal{W}_{3,b} := W_{3,E,b} + W_{3,P,b}.
\end{align}
\end{lemma}
\begin{proof}

The first step is to differentiate the weak formulation (\ref{weak.6}) yet again, which formally takes place using difference quotients, yielding: 
\begin{align} \n
\int \int \nabla_\epsilon^2 \hat{\psi_{xx}} : \nabla_\epsilon^2 \phi - \int \int \p_{xx} S_u(\hat{u}, \hat{v}) \cdot \phi_y &+ \eps \p_{xx} S_v(\hat{u}, \hat{v}) \cdot \phi_x \\ \label{weak.7}  &= \int \int \eps^{\frac{n}{2}+\gamma} \Big[ - \p_{xx} \hat{f} \phi_y + \eps \p_{xx} \hat{g} \phi_x \Big],
\end{align}

Fix any $L$ large, finite. The selection of test function is now $\phi := \hat{v}^{(n)}_x \rho_3^4 x_L^{4-2b}$, where the sequence: 
\begin{align}
[\hat{u}^{(n)}_{xx}, \hat{v}^{(n)}_{xx}] \in C^\infty_{0,D}(\Omega_L), \hspace{5 mm} [\hat{u}^{(n)}_{xx}, \hat{v}^{(n)}_{xx}] \xrightarrow{H^1(\Omega_L)} [\hat{u}_{xx}, \hat{v}_{xx}]. 
\end{align}

From here, repeating the estimates given in Proposition \ref{prop.ho.3}, and sending $n \rightarrow \infty$ and then $L \rightarrow \infty$ gives the desired result. 

\end{proof}

\begin{lemma}[Third-Order Positivity Estimate]  Fix any $0 < b < 1$. Let $\delta, \eps$ be sufficiently small relative to universal constants, and $\eps << \delta$. Then for $[\hat{u}, \hat{v}] \in Z$ solutions to (\ref{hats.1}) - (\ref{hats.2}):
\begin{align}
|| \{ \sqrt{\eps} \hat{v}_{xxx}, \hat{v}_{xxy}\} \rho_3^{\frac{5}{2}} x^{\frac{5}{2}-b}||_{L^2}^2 \lesssim || \hat{u}_{xxy} \rho_3^2 x^{2-b}||_{L^2}^2 + ||\hat{u},\hat{v}||_{X_{1,b} \cap X_{2,b}}^2 + \sum_{i = 1}^3 \mathcal{W}_{i,b}.
\end{align}
\end{lemma}
\begin{proof}

Again, fix any $L$ large, finite. The selection of the test function is now $\phi := \hat{v}^{(n)}_{xx} \rho_3^5 x_L^{5-2b}$, where the sequence $[\hat{u}^{(n)}, \hat{v}^{(n)}]$ is selected according to: 
\begin{align}
[\hat{u}^{(n)}_{xx}, \hat{v}^{(n)}_{xx}] \in C^\infty_{0,D}(\Omega_L), \hspace{5 mm}  [\hat{u}^{(n)}_{xx}, \hat{v}^{(n)}_{xx}] \xrightarrow{H^1(\Omega_L)} [\hat{u}_{xx}, \hat{v}_{xx}]. 
\end{align}

From here, repeating the estimates in Proposition \ref{prop.ho.4}, and sending $n \rightarrow \infty$ and then $L \rightarrow \infty$ gives the desired result. 

\end{proof}

Piecing together the above set of estimates, 
\begin{proposition} Let $\delta, \eps$ be sufficiently small relative to universal constants, and $\eps << \delta << b$. Then for $[\hat{u}, \hat{v}] \in Z$ solutions to (\ref{hats.1}) - (\ref{hats.2}): 
\begin{align} \label{hat.un.1}
||\hat{u}, \hat{v}||_{X_{1,b} \cap X_{2,b} \cap X_{3,b}}^2 \lesssim \mathcal{W}_{1,b} + \mathcal{W}_{2,b} + \mathcal{W}_{3,b},
\end{align}
where $\mathcal{W}_{i,b}$ have been defined in (\ref{calw1b}), (\ref{calw2b}), (\ref{calw3b}).  
\end{proposition}

By repeating the analysis in Section \ref{Section.Z}, one has: 
\begin{lemma} Let $\delta, \eps$ be sufficiently small relative to universal constants, and $\eps << \delta << b$. Then for $[\hat{u}, \hat{v}] \in Z$ solutions to (\ref{hats.1}) - (\ref{hats.2}): 
\begin{align} \label{W.Z.0}
||\hat{u}, \hat{v}||_{Z_b}^2 \lesssim \eps^{\frac{n}{2}+\gamma - \omega(N_i)} ||\hat{u}, \hat{v}||_{Z_b}^4 + ||\hat{u}, \hat{v}||_{X_{1,b} \cap X_{2,b} \cap X_{3,b}}^2. 
\end{align}
\end{lemma}

Due to (\ref{hat.un.1}), we will now turn to estimating $\mathcal{W}_{i,b}$
\begin{lemma} Let $\mathcal{W}_{1,b}, \mathcal{W}_{2,b}, \mathcal{W}_{3,b}$ be as in (\ref{calw1}), (\ref{calw2}), (\ref{calw3}). Then:
\begin{align} \label{W.Z.1}
| \mathcal{W}_{1,b} + \mathcal{W}_{2,b} + \mathcal{W}_{3,b} | \lesssim C(b) \eps^{\frac{n}{2}+\gamma - \omega(N_i)} ||\hat{u}, \hat{v}||_{Z_b}^2,
\end{align}
where $C(b) \uparrow \infty$ as $b \downarrow 0$. 
\end{lemma}
\begin{proof}

We will work with the expression: 
\begin{align} \n
\hat{f} &= \eps^{\frac{n}{2}+\gamma} \Big[ u^{(1)} u^{(1)}_x - u^{(2)} u^{(2)}_x + v^{(1)} u^{(1)}_y - v^{(2)} u^{(2)}_y  \Big]\\ 
& = \eps^{\frac{n}{2}+\gamma} \Big[ \hat{u} u^{(1)}_x + u^{(2)} \hat{u}_x + \hat{v} u^{(1)}_y + v^{(2)} \hat{u}_y \Big], \\ \n
\hat{g} & = \eps^{\frac{n}{2}+\gamma} \Big[ u^{(1)} v^{(1)}_x - u^{(2)} v^{(2)}_x + v^{(1)} v^{(1)}_y - v^{(2)} v^{(2)}_y \Big] \\ \label{hat.g.ex}
& = \eps^{\frac{n}{2}+\gamma} \Big[\hat{u} v^{(1)}_x + u^{(2)} \hat{v}_x + \hat{v} v^{(1)}_y + v^{(2)} \hat{v}_y \Big]. 
\end{align}

Concerning $\mathcal{W}_{1,b}$, let us bring particular attention to the following term from $\int \int| \hat{f} | \cdot |\hat{u}| x^{-2b}$:
\begin{align} \n
\int \int \eps^{\frac{n}{2}+\gamma}&[ \hat{v} u^{(1)}_y + v^{(2)} \hat{u}_y] \cdot |\hat{u}| x^{-2b} \\ \n
& \le \eps^{\frac{n}{2}+\gamma} ||\hat{v} x^{\frac{1}{2}-b}||_{L^\infty} ||u^{(1)}_y||_{L^2} ||\hat{u} x^{-\frac{1}{2}-b}||_{L^2} \\ \n
& \hspace{20 mm} + \eps^{\frac{n}{2}+\gamma} ||v^{(2)} x^{\frac{1}{2}}||_{L^\infty} ||\hat{u}_y x^{-b}||_{L^2} ||\hat{u} x^{-\frac{1}{2}-b}||_{L^2} \\ \n
& \le \eps^{\frac{n}{2}+\gamma} \Big[ ||\hat{v} x^{\frac{1}{2}-b}||_{L^\infty} ||u^{(1)}_y||_{L^2} ||\hat{u}_x x^{\frac{1}{2}-b}||_{L^2} \\ \n
& \hspace{20 mm} + ||v^{(2)} x^{\frac{1}{2}}||_{L^\infty} ||\hat{u}_y x^{-b}||_{L^2} ||\hat{u}_x x^{\frac{1}{2}-b}||_{L^2} \Big] \\ \label{mdiff}
& \le C(b) \eps^{\frac{n}{2}+\gamma - \omega(N_i)} ||u^{(i)}, v^{(i)}||_Z ||\hat{u}, \hat{v} ||_{Z_b}^2. 
\end{align}

The above term requires the weight of $x^{-2b}, b > 0$, in order to apply the Hardy inequality. Indeed, this was not required for the existence proof (see calculation (\ref{order.NL})), because the structure of $vu_y \cdot u$ enabled us to integrate by parts, unlike in the present situation. The remaining terms in $\mathcal{W}_{1,b}$, and all terms in $\mathcal{W}_{2,b}, \mathcal{W}_{3,b}$ are treated nearly identically to the Lemma \ref{LemmaW}, and so we omit repeating those calculations. 

\end{proof}

\begin{corollary} Fix $0 < b < 1$ sufficiently small, relative to universal constants. Suppose $\eps, \delta$ are sufficiently small, such that $\eps << \delta << b$. Then $\hat{u}, \hat{v} = 0$. 
\end{corollary}
\begin{proof}

Combining estimate (\ref{W.Z.1}) and (\ref{W.Z.0}) with estimate (\ref{W.Z.2}) yields: 
\begin{align}
||\hat{u}, \hat{v}||_{Z_b}^2 \lesssim C(b) \eps^{\frac{n}{2}+\gamma - \omega(N_i)} ||\hat{u}, \hat{v}||_{Z_b}^2. 
\end{align}

For $\eps$ sufficiently small, this then implies $||\hat{u}, \hat{v}||_{Z_b} = 0$. Upon consultation with the norm $Z_b$, and (\ref{hats.bc.1}), this implies that $\hat{u}, \hat{v} = 0$. 
\end{proof}

\begin{remark} We have controlled the second and third order energy norms, (\ref{norm.x1b}) - (\ref{norm.x2b}) in order to treat the term $\int \int \hat{v} u^{(1)}_y |\hat{u}| x^{-2b}$, which appears in (\ref{mdiff}). This term forces us to control $||\hat{v} x^{\frac{1}{2}-b}||_{L^\infty}$. One cannot get around placing this term in $L^\infty$ (for instance by integrating by parts from $u^{(1)}_y$) because this produces suboptimal decay rates, according to (\ref{u0}) - (\ref{Uf}).  
\end{remark}

This then establishes Theorem \ref{thm.e.u}, and controlling $[u,v] \in Z$ then immediately establishes the main result, Theorem \ref{thm.m.1}. 

\vspace{5 mm}

\textbf{Acknowledgements:} The author thanks Yan Guo for many valuable discussions regarding this research. The author also thanks Bjorn Sandstede for introducing him to the paper \cite{BKL}.

\end{document}